\documentclass[a4paper, 11pt]{amsart}

\usepackage[utf8]{inputenc}
\usepackage[english]{babel}
\usepackage[all]{xy}
\usepackage{enumitem}
\usepackage[dvipsnames]{xcolor}
\usepackage{wrapfig, graphicx}
\usepackage{amsfonts,amstext,amssymb,amsmath,amsthm,
	fancybox}
\usepackage{pifont}
\usepackage[left=2.5cm,right=2.5cm,top=2.5cm,bottom=2.5cm]{geometry}
\usepackage[colorlinks=true, pdfstartview=FitV, linkcolor=altblue, citecolor=altred, urlcolor=blue,pagebackref=false]{hyperref} %Pour un sommaire interractif.
\usepackage{fourier-orns}
\usepackage{array}
\usepackage{fancyhdr}
\usepackage{setspace}
\usepackage{tikz}
\usetikzlibrary{patterns}
\usepackage{mathrsfs}
\usepackage{dsfont}
\usepackage{mathdots}
\usepackage[colorinlistoftodos, french, bordercolor=black, linecolor=black, textsize=small]{todonotes}
\newcolumntype{M}[1]{>{\centering\arraybackslash}m{#1}}
\renewcommand{\epsilon}{\varepsilon}
\renewcommand{\phi}{\varphi}
\renewcommand{\kappa}{\varkappa}
\newcommand\R{\mathbb{R}}
\newcommand\Q{\mathbb{Q}}

\newcommand\C{\mathbb{C}}
\renewcommand\S{\mathbb{S}}
\newcommand\N{\mathbb{N}}
\newcommand\Z{\mathbb{Z}}
\newcommand\T{\mathbb{T}}
\newcommand{\inv}[1]{\frac{1}{#1}}
\renewcommand{\hat}{\widehat}
\renewcommand{\tilde}{\widetilde}
\renewcommand{\bar}{\overline}
\renewcommand{\leq}{\leqslant}
\renewcommand{\geq}{\geqslant}
\renewcommand{\Re}{\operatorname{Re}}
\renewcommand{\Im}{\operatorname{Im}}
\newcommand{\norme}[1]{\left\Vert #1\right\Vert}
\newcommand{\prodscalbis}[1]{\left\langle #1\right\rangle}
\newcommand{\prodscal}[2]{\left\langle #1\,\middle|\,#2\right\rangle}
\newcommand{\ensemble}[1]{\left\lbrace #1\right\rbrace}
\newcommand{\valabs}[1]{\left| #1\right|}

\newcommand{\indicatrice}{\mathds{1}}
\newcommand{\privede}[1]{\setminus\ensemble{#1}}

\DeclareMathOperator{\Id}{Id}

\DeclareMathOperator{\Ima}{Im}

\DeclareMathOperator{\diag}{diag}
\DeclareMathOperator{\dist}{dist}
\DeclareMathOperator{\Vect}{Vect}

\DeclareMathOperator{\com}{com}

\newcommand\F{\mathcal{F}}
\newcommand\E{\mathbf{E}}
\renewcommand{\Q}{\mathbf{Q}}

\renewcommand{\P}{\mathcal{P}}

\newcommand{\V}{\mathcal{V}}
\newcommand{\G}{\mathcal{G}}
\renewcommand{\H}{\mathcal{H}}
\renewcommand{\N}{\mathcal{N}}
\newcommand\B{\mathcal{B}}
\newcommand{\freq}{\mathbf{n}\cdot\boldsymbol{\zeta}}
\newcommand{\K}{\mathcal{K}}

\DeclareMathOperator{\osc}{osc}
\DeclareMathOperator{\ev}{ev}
\DeclareMathOperator{\app}{app}

\DeclareMathOperator{\sign}{sign}
\DeclareMathOperator{\inc}{in}
\DeclareMathOperator{\out}{out}
\DeclareMathOperator{\spectre}{sp}
\DeclareMathOperator{\res}{res}

\definecolor{altblue}{RGB}{0, 0, 180}
\definecolor{altred}{RGB}{200, 0, 0}
\definecolor{altorange}{RGB}{243, 136, 0}
\definecolor{altgreen}{RGB}{0,180, 0}

\newtheorem{theorem}{Theorem}[section]
\newtheorem{lemma}[theorem]{Lemma}
\newtheorem{proposition}[theorem]{Proposition}

\newtheorem{definition}[theorem]{Definition}
\newtheorem{assumption}{Assumption}
\theoremstyle{remark}
\newtheorem{remark}[theorem]{Remark}

\newtheorem{example}[theorem]{Example}

\numberwithin{equation}{section}
\title[Multiphase geometric optics for quasilinear boundary value problems]{Weakly nonlinear multiphase geometric optics for hyperbolic quasilinear boundary value problems: construction of a leading profile}

\author{Corentin Kilque}
\address{Institut de Mathématiques de Toulouse ; UMR5219 \\ Université de Toulouse ; CNRS \\
	UPS, F-31062 Toulouse Cedex 9, France}
\email{corentin.kilque@math.univ-toulouse.fr}

\begin{document}

\maketitle

\begin{abstract}
	We investigate in this paper the existence of the leading profile of a WKB expansion for quasilinear initial boundary value problems with a highly oscillating forcing boundary term. The framework is weakly nonlinear, as the boundary term is of order $ O(\epsilon) $ where the frequencies are of order $ O(1/\epsilon) $. We consider here multiple phases on the boundary, generating a countable infinite number of phases inside the domain, and we therefore use an almost periodic functional framework. The major difficulties of this work are the lack of symmetry in the leading profile equation and the occurrence of infinitely many resonances (opposite to the simple phase case studied earlier) The leading profile is constructed as the solution of a quasilinear problem, which is solved using a priori estimates without loss of derivatives. The assumptions of this work are illustrated with the example of isentropic Euler equations in space dimension two.
\end{abstract}

\tableofcontents

\section{Introduction}
We consider in this paper hyperbolic quasilinear initial boundary value problems with a highly oscillating forcing boundary term. We are interested in constructing, in the high frequency asymptotic, an approximate solution to this problem in the form of a WKB expansion. This is the type of question studied in hyperbolic geometric optics. The general idea is to consider an hyperbolic system, of which the source term, the initial term or the boundary term (in the case of boundary value problems) is highly oscillatory, namely with frequencies of order $ 1/\epsilon $ and to look for an approximate solution to the system in the form of an asymptotic expansion. First the equation satisfied by the terms of the asymptotic expansion needs to be formally derived, and then to be solved in a suitable functional space. Once this formal series has been constructed, one may prove that the truncated sums actually approaches the exact solution in the high frequencies asymptotic. The present paper addresses the first part of this framework, and more precisely we prove existence and uniqueness for the leading profile of the asymptotic expansion.

The study of hyperbolic geometric optics goes back to \cite{Lax1957Asymptotic} for the study of the linear Cauchy problem. When the system is nonlinear, the multiplicity of phases in the source term, initial term or boundary term is important since nonlinear interactions between phases may occur. In the case of only one phase, the construction of an asymptotic expansion was first performed by \cite{ChoquetBruhat1964Ondes}. For the justification of this asymptotic expansion we can refer to \cite{JolyRauch1992Justification} in the semi-linear case and \cite{Gues1993Developpement} in the quasi-linear case. The first study of the multiphase case for the Cauchy problem goes back to \cite{HunterMajdaRosales1986Resonantly}. The question has then been largely resolved by J.L. Joly, G. Métivier and J. Rauch, see in particular \cite{JolyMetivierRauch1993Generic}, \cite{JolyMetivierRauch1994Coherent} and \cite{JolyMetivierRauch1995Coherent}. The natural question is to obtain results for boundary value problems, similar to the ones for the Cauchy problems. In \cite{Chikhi1991Reflexion}, the author deals with a semi-linear boundary value problem for a system of two equations, the general case of multiple equations being treated for example in \cite{Williams1996Boundary} and \cite{Williams2000Boundary}. The quasi-linear case but with only one phase on the boundary is treated notably in  \cite{Williams2002Singular}, \cite{CoulombelGuesWilliams2011Resonant}, \cite{CoulombelWilliams2013Reflecting} (the latter taking interest into the justification, which is not addressed in this paper) and \cite{Hernandez2015Resonant}. This work is an extension to the multiphase case: we deal with the same quasi-linear boundary value problem, but with multiple phases on the boundary.

Because of the multiple frequencies on the boundary, the nonlinearity of the problem generates a countable infinite number of phases inside the domain, forcing us to consider an almost-periodic framework, the group of frequencies being, in general, not finitely generated. This almost-periodic functional framework has been previously used to construct approximate solutions to  systems with multiple phases, for semi-linear systems in the context of Wiener algebras by \cite{JolyMetivierRauch1994Coherent} for the Cauchy problem and \cite{Williams1996Boundary} for the boundary value problem. For quasi-linear systems, Bohr-Besicovich spaces are generally used, notably by
\cite{JolyMetivierRauch1995Coherent} for the Cauchy problem. In this work we attempt to achieve the next step, namely to obtain a similar result as the one of \cite{JolyMetivierRauch1995Coherent}, for quasi-linear boundary value problems. We adapt the functional framework of \cite{JolyMetivierRauch1995Coherent} to the context of boundary value problems, by considering functions that are quasi-periodic with respect to the tangential fast variables and almost-periodic with respect to the normal fast variable. Concerning the regularity, we choose a Sobolev control for the (slow and fast) tangential variables, and a uniform control for the normal variables.
The leading profile of the WKB expansion is then obtained as the solution of a quasilinear problem which takes into account the potentially infinite number of resonances between the phases. We solve this quasilinear problem in a classical way by proving estimates without loss of regularity. The example of gas dynamics is used all along the paper to illustrate the general assumptions that will be made during the analysis. The main difference between this paper and \cite{JolyMetivierRauch1995Coherent} is the absence of symmetry in the problem. Indeed, starting with an evolution problem in time, we modify it to obtain a propagation problem in the normal variable $ x_d $, with respect to which the system is not hyperbolic. In \cite{JolyMetivierRauch1995Coherent}, these symmetries are used for the a priori estimates to handle the resonance terms that appear in the equations. Even though it is relatively easy in our problem to create  symmetries for the self-interaction terms, it is more delicate for the resonance terms, which, unlike the case of \cite{CoulombelGuesWilliams2011Resonant}, are in infinite number. The last assumption of the paper is made to deal with this issue, and in essence controls the lack of symmetry of resonance terms. The notions associated to it appear in \cite[Chapter 11]{Rauch2012Hyperbolic}.

The proof of existence of a leading profile is divided in three parts: formal derivation of the equation satisfied by the leading profile, reduction and decoupling of these equations, and finally energy estimates on these equations.
Formal derivation of the WKB cascade is quite classical in geometric optics, and consists on formally replacing the series in the exact system (system \eqref{eq systeme 1} in the following). As usual, this cascade is decoupled using projectors and an operator on the space of profiles. The said operator remains formal in this paper, and the projectors require a small divisors assumption to be rigorously defined. The second part of the proof takes interest into decoupling and reducing the system to a system for the oscillating resonant modes, a system for each oscillating non-resonant mode, and a system for the evanescent part. In order to do that, extending the system into modes, we begin by showing that the mean value satisfies a decoupled system with zero source term and boundary term, and is therefore zero. Next we prove that the outgoing modes are also zero, deriving energy estimates for them. For this purpose we use a suitable scalar product on the space of profiles, that requires a compact support in the normal direction. Therefore a finite speed propagation is proven beforehand. Once there are only incoming and evanescent modes, it is easy to determine a boundary condition for each mode from the original boundary condition, and therefore decouple the system. The equation satisfied by the evanescent part gives a formula for it using the double trace on the boundary, so its construction is quite straightforward. However we need to check that the constructed solution is actually in the space of evanescent profiles. To construct the oscillating parts we show a priori estimates without loss of regularity for the linearized oscillating systems, which allow us to prove the well-posedness of these linearized systems, and then by an iterative scheme the existence of solutions to the original systems. To derive the energy estimates we use an alternative scalar product that takes advantage of propagation in the normal direction. Several terms need to be handled. The transport and Burgers ones are quite classical to treat, and the assumption on the set of resonances is used to address the resonant ones. Since it is quite usual, we do not give the details of the construction of the solution to the linearized systems and the iterative schemes.
Reassembling the constructed profiles we finally get the leading profile solution to the initial system. 

The paper is organized as follows.  After the first section devoted to this introduction, the second one introduces the problem and states the usual assumptions on the system. First the problem studied in this work is precisely described, and the example of Euler equations that will be used throughout the analysis is introduced, and then characteristic frequencies and the strict hyperbolicity assumption are looked at. Finally interest is made on properties and assumption about the boundary condition. Assumptions of this section ensure that the initial boundary problem is well posed locally in time for the exact solution. However, due to the high frequencies in the forcing term, we do not know if the lifespan of the exact solution is uniform with respect to the small wavelength.  Third section is devoted to the functional framework of the paper. After a motivation of this framework with a formal study of the frequencies created inside the domain and some assumptions on it, we describe the spaces of profiles which will be used, and introduce scalar products on these functional spaces. After this rather long introduction of the problem and assumptions, the ansatz of the expansion and the main result are stated in section 4. The proof is then divided in two sections. The fifth one is a formal derivation of the equations satisfied by the leading profile. First the cascade of equations for the profiles is obtained by a formal WKB study. It gives rise to a certain fast problem, that is resolved in a second part, which allows to finally write the decoupled equations for the leading profile. The last section of this paper is the core of the proof. First some coefficients associated to resonances are introduced and the last assumption of this work is made about these coefficients to deal with the lack of symmetry in resonance terms. We proceed by making rigorous the results of the fifth section analysis which will be used after. Next part is devoted to reducing and decoupling the equations into an equation for the evanescent part, and equations for the resonant oscillating part and each non-resonant oscillating part. A lot of the techniques used in this part are used in the next two ones, that achieves the main step of the proof, namely proving energy estimates on the linearized equations for the oscillating resonant and non-resonant parts. These estimates are used in the following part to construct an oscillating solution, using an iterative scheme. It is also proven that the constructed evanescent part belongs to the space of evanescent profiles. Finally a conclusion and some perspectives are drawn. In the appendix is presented three technical proofs that have been postponed, and are about notions of the second section.

\bigskip

In all the paper the letter $ C $ denotes a positive constant that may vary during the analysis, possibly without any mention being made, and for every matrix $ M $, the notation $ \,^tM $ refers to its real transpose.

\section{Notations and assumptions}

\subsection{Presentation of the problem}

Given a time $ T>0 $ and an integer $ d\geq 2 $, let $ \Omega_T $ be the domain $\Omega_T:=(-\infty,T]\times \R^{d-1}\times \R_+$ and $\omega_T:=(-\infty,T]\times \R^{d-1}$ its boundary.
We denote as $ t\in(-\infty,T] $ the time variable, $ x=(y,x_d)\in\R^{d-1}\times\R_+ $ the space variable, with $ y\in\R^{d-1} $ the tangential variable and $ x_d\in\R_+ $ the normal variable, and at last $ z=(t,x)=(t,y,x_d) $. We also denote by $ z'=(t,y)\in\omega_T $ the variable of the boundary $ \ensemble{x_d=0} $. For $ i=1,\dots,d $, we denote by $ \partial_i $ the operator of partial derivative with respect to $ x_i $. Finally we denote as $ \alpha\in\R^{d+1} $ and $ \zeta\in\R^d $ the dual variables of $ z\in\Omega_T $ and $ z'\in\omega_T $.  We consider the following problem
\begin{equation}\label{eq systeme 1}
\left\lbrace \begin{array}{lr}
L(u^{\epsilon},\partial_z)\,u^{\epsilon}:=\partial_tu^{\epsilon}+\displaystyle\sum_{i=1}^dA_i(u^{\epsilon})\,\partial_iu^{\epsilon}=0&\qquad \mbox{in } \Omega_T, \\[5pt]
B\,u^{\epsilon}_{|x_d=0}=\epsilon\, g^{\epsilon}&\qquad \mbox{on } \omega_T,  \\[10pt]
u^{\epsilon}_{|t\leq 0}=0,&
\end{array}
\right.
\end{equation}
where the unknown $u^{\epsilon}$ is a function from $\Omega_T$ to an open set $\mathcal{O}$ of $\R^N$ containing zero, with $ N\geq 1 $, the matrices $A_j$ are regular functions of $\mathcal{O}$ with values in $\mathcal{M}_N(\R)$ and the matrix $B$ belongs to $\mathcal{M}_{M\times N}(\R)$ and is of maximal rank. The integer $M$ is made precise in Assumption \ref{hypothese UKL} below. To simplify the notations and clarify the proofs we consider here linear boundary conditions, but it would be possible to deal with non-linear ones. Furthermore we assume the boundary to be noncharacteristic, that is the following assumption is made.

\begin{assumption}[Noncharacteristic boundary]\label{hypothese bord non caract}
	For all $ u $ in $ \mathcal{O} $, the matrix $ A_d(u) $ is invertible.
\end{assumption}

The dependence on $ \epsilon>0 $ of system \eqref{eq systeme 1} comes from the source term $ \epsilon\,g^{\epsilon} $ on the boundary $ \omega_T $, where the quasi-periodic function $ g^{\epsilon} $ is defined, for $ z' $ in $ \omega_T $, as
\begin{equation}\label{eq def g epsilon 1}
g^{\epsilon}(z')=G\left( z',\frac{z'\cdot\zeta_1}{\epsilon},\dots,\frac{z'\cdot\zeta_m}{\epsilon}\right),
\end{equation}
where $ G $ is a function of the Sobolev space $ H^{\infty}(\R^d\times\T^m) $, with $ m\geq2 $, that vanishes for negative times $ t $ and of zero mean with respect to $ \theta $ in $ \T^m $, and where $ \zeta_1,\dots,\zeta_{m} $ are frequencies of $ \R^{d}\privede{0} $. Here the notation $ \T $ stands for the torus $ \R/2\pi\Z $. We denote by $ \boldsymbol{\zeta} $ the $ m $-tuple $ \boldsymbol{\zeta}:=(\zeta_1,\dots,\zeta_m) $. The function $ G $ being periodic and of zero mean with respect to $\theta$, we may write
\begin{equation}\label{eq def G_n}
G(z',\theta)=\sum_{\mathbf{n}\in\Z^m\privede{0}}G_{\mathbf{n}}(z')\,e^{i\mathbf{n}\cdot\theta},
\end{equation}
where $ G_{\mathbf{n}} $ is in $ H^{\infty}(\R^d) $ and is zero for negative times $ t $, for all $ \mathbf{n} $ in $ \Z^m\privede{0} $. The framework of weakly non-linear geometric optics is chosen here, namely we expect the leading profile in the asymptotic expansion to be of order $ \epsilon $, which explains the $ \epsilon $ factor in front of $ g^{\epsilon} $ in the boundary condition. Note that without loss of generality, we can assume that $ \zeta_1,\dots,\zeta_m $ are linearly independent over $ \mathbb{Q} $.

Condition $u^{\epsilon}_{|t\leq 0}=0$ in \eqref{eq systeme 1} expresses the nullity of the initial conditions. The time of existence $ T>0 $ is not fixed at first and is likely to become sufficiently small to ensure existence of a leading profile.
 
The study of \cite{Williams2002Singular}, \cite{CoulombelGuesWilliams2011Resonant} and \cite{Hernandez2015Resonant} is here extended to several phases on the boundary. No assumption on the group of boundary frequencies generated by the frequencies $ (\zeta_1,\dots,\allowbreak\zeta_m) $ is made, apart from it being finitely generated. In particular it may not be discrete.

We want to approximate the exact solution  to \eqref{eq systeme 1}, in the limit where $ \epsilon $ goes to $ 0 $, by an approximate solution that behaves as $\epsilon$ in range, and $ 1/\epsilon $ in frequency. This is the weakly nonlinear geometric optics framework, see \cite{Rauch2012Hyperbolic} and \cite{Metivier2009Optics}. Recall that in this paper we do not prove stability, i.e. that the approximate solution converges in some sense to the exact one, since we do not know if the latter exists on a time interval independent of $ \epsilon $. To obtain this kind of result, we first have to make several suitable assumptions about the original problem. The rest of this section is devoted to these assumptions, and focuses on the characteristic frequencies associated with the system. Let us first detail the example that inspires the general framework developed in this paper.

\begin{example}\label{exemple Euler 1}
	The isentropic compressible Euler equations in two dimensions provide a system of the form of \eqref{eq systeme 1}. Under regularity assumptions on the solution, the associated boundary value problem reads
	\begin{equation}\label{eq systeme Euler}
	\left\lbrace \begin{array}{lr}
	\partial_tV^{\epsilon}+A_1(V^{\epsilon})\,\partial_1V^{\epsilon}+A_2(V^{\epsilon})\,\partial_2V^{\epsilon}=0&\qquad \mbox{in } \Omega_T, \\[5pt]
	B\,V^{\epsilon}_{|x_d=0}=\epsilon\, g^{\epsilon}&\qquad \mbox{on } \omega_T,  \\[10pt]
	V^{\epsilon}_{|t\leq 0}=0,&
	\end{array}
	\right.
	\end{equation}
	with $ V^{\epsilon}=(v^{\epsilon},\mathbf{u}^{\epsilon})\in\R^3 $, where $ v^{\epsilon}\in\R^*_+ $ represents the fluid volume, and $ \mathbf{u}^{\epsilon}\in\R^2 $ its velocity, and where the functions $ A_1 $ and $ A_2 $ are defined on $ \R^*_+\times\R^2 $ as
	\begin{equation}\label{eq Euler def A1 et A2}
	A_1(V):=\left(\begin{array}{ccc}
	\mathbf{u}_1 & -v & 0 \\[5pt]
	-c(v)^2/v & \mathbf{u}_1 & 0 \\[5pt]
	0 & 0 &\mathbf{u}_1 
	\end{array}\right),\qquad
	A_2(V):=\left(\begin{array}{ccc}
	\mathbf{u}_2 & 0 &-v  \\[5pt]
	0& \mathbf{u}_2 & 0 \\[5pt]
	-c(v)^2/v  & 0 &\mathbf{u}_2
	\end{array}\right),
	\end{equation}
	with $ c(v)>0 $ representing the sound velocity in the fluid, which depends on its volume $ v $. The noncharacteristic boundary Assumption \ref{hypothese bord non caract} for system \eqref{eq systeme Euler} is now discussed. In this article, we consider geometric optics expansions for system \eqref{eq systeme 1} constructed as perturbations around the equilibrium 0, performing a change of variables if necessary. For the Euler system the natural coefficients $ A_1,A_2 $ are rather used, and a perturbation around the equilibrium $ V_0=(v_0,0,u_0) $ is considered, where $ v_0>0 $ is a fixed volume, and $ (0,u_0) $ is an incoming \emph{subsonic} velocity, that is such that $ 0<u_0<c_0 $, where we denote $ c_0:=c(v_0) $. 
	
	The Assumption \ref{hypothese bord non caract} concerns in this case the invertibility of the matrix $ A_2(V) $ for $ V=(v,u_1,u_2)\in\R_+^*\times\R^2 $ in the neighborhood $ V_0 $. The determinant of the matrix $ A_2(V) $ is given by  $ \det A_2(V)=u_2\,(u_2^2-c(v)^2) $, which is nonzero if the velocity $ u_2 $ satisfies $ 0<u_2<c(v) $. The equilibrium $ V_0 $ verifying this condition, every  small enough neighborhood $ \mathcal{O} $ of $ V_0 $ suits to satisfy  Assumption \ref{hypothese bord non caract}.
\end{example}

The rest of the section is dedicated to the characteristic frequencies related to the problem and the associated assumptions.

\subsection{Strict hyperbolicity}

The following definition introduces the notion of characteristic frequency.

\begin{definition}
	For $\alpha=(\tau,\eta,\xi)\in\R\times\R^{d-1}\times\R$, the symbol $L(0,\alpha)$ associated with $L(0,\partial_z)$ is defined as
	\begin{equation*}
	L(0,\alpha):=\tau I+\sum_{i=1}^{d-1}\eta_iA_i(0)+\xi A_d(0).
	\end{equation*}
	Then we define its characteristic polynomial as $ p(\tau,\eta,\xi):=\det L\big(0,(\tau,\eta,\xi)\big) $. We say that $\alpha\in\R^{1+d}$ is a \emph{characteristic frequency} if it is a root of the polynomial $p$, and we denote by $ \mathcal{C} $ the set of characteristic frequencies.
\end{definition}

The following assumption, called \emph{strict hyperbolicity} (see  \cite[Definition 1.2]{BenzoniSerre2007Multi}), is made. Assumptions \ref{hypothese bord non caract} of non-characteristic boundary and of hyperbolicity (whether strict or with constant multiplicity) are very usual, see e.g. \cite{Williams1996Boundary,CoulombelGuesWilliams2011Resonant,JolyMetivierRauch1995Coherent}, and related to the structure of the problem. Assumption of hyperbolicity of constant multiplicity, which is more general than Assumption \ref{hypothese stricte hyp} of strict hyperbolicity, is sometimes preferred like in \cite{CoulombelGuesWilliams2011Resonant,JolyMetivierRauch1995Coherent}. We chose here to work with the latter for technical reasons.

\begin{assumption}[Strict hyperbolicity]\label{hypothese stricte hyp}
	There exist real functions $\tau_1<\cdots<\tau_N$, analytic with respect to $ (\eta,\xi) $ in $\R^d\setminus\{0\}$, such that for all
	$(\eta,\xi)\in\R^d\setminus\ensemble{0}$ and for all 
	$\tau\in\R$, the following factorisation is verified
	\[p(\tau,\eta,\xi)=\det\Big(\tau I+\sum_{i=1}^{d-1}\eta_iA_i(0)+\xi A_d(0)\Big)=\prod_{k=1}^N\big(\tau-\tau_k(\eta,\xi)\big),\]
	where the eigenvalues $-\tau_k(\eta,\xi)$ of the matrix $A(\eta,\xi):=\sum_{i=1}^{d-1}\eta_iA_i(0)+\xi A_d(0)$ are therefore simple. Consequently, for all $(\eta,\xi)\in\R^d \setminus\ensemble{0}$, 
	 the following decompositions of $\C^N$ into dimension 1 eigenspaces hold
	\begin{align}\label{eq decomp C^n ker L(alpha)}
	\C^N&=\ker L\big(0,\tau_1(\eta,\xi),\eta,\xi\big)\oplus\cdots\oplus\ker L\big(0,\tau_N(\eta,\xi),\eta,\xi\big),\\\label{eq decomp C^n ker L(alpha) tilde}
	\C^N&=A_d(0)^{-1}\,\ker L\big(0,\tau_1(\eta,\xi),\eta,\xi\big)\oplus\cdots\oplus A_d(0)^{-1}\,\ker L\big(0,\tau_N(\eta,\xi),\eta,\xi\big).
	\end{align}
	For $k=1,\dots,N$ and for $(\eta,\xi)$ in $\R^d\setminus\{0\}$,
	we define the projectors $\pi_k(\eta,\xi)$ and $ \tilde{\pi}_{k}(\eta,\xi) $, respectively associated with decompositions \eqref{eq decomp C^n ker L(alpha)} and \eqref{eq decomp C^n ker L(alpha) tilde}. 
	
	For $ k=1,\dots,N $ and $ (\eta,\xi)\in\R^d\privede{0} $, we also denote as $ E_k(\eta,\xi) $ a unitary eigenvector generating the eigenspace $ \ker L\big(0,\tau_k(\eta,\xi),\eta,\xi\big) $, so that
	\begin{equation}\label{eq base C^N adaptee espace propres stricte hyp}
	E_1(\eta,\xi),\dots,E_N(\eta,\xi)
	\end{equation}
	is a real normal basis of $ \C^N $ adapted to decomposition \eqref{eq decomp C^n ker L(alpha)}. Observe that the family
	\begin{equation}\label{eq base C^N tilde adaptee espace propres stricte hyp}
	A_d(0)^{-1}E_1(\eta,\xi),\dots,A_d(0)^{-1}E_N(\eta,\xi)
	\end{equation}
	is therefore a real normal basis of $ \C^N $ adapted to decomposition \eqref{eq decomp C^n ker L(alpha) tilde}.
\end{assumption}

\begin{remark} 
	\begin{enumerate}[label=\roman*), leftmargin=0.7cm]
		\item  We will be led further on to consider the modified operator
		\begin{equation*}
			\tilde{L}(0,\partial_z):=A_d(0)^{-1}\,L(0,\partial_z). 
		\end{equation*}
		This justifies the introduction of the modified symbol $ \tilde{L}(0,\alpha):=A_d(0)^{-1}\,L(0,\alpha) $, and thus of the projector $ \tilde{\pi}_{k}(\eta,\xi) $, the symbols $ L(0,\alpha) $ and $ \tilde{L}(0,\alpha) $ having different ranges.
		\item Since the matrix $A(\eta,\xi)=\sum_{i=1}^{d-1}\eta_iA_i(0)+\xi A_d(0)$ is real for $(\eta,\xi)\in\R^d$ and the eigenvalues $\tau_1,\dots,\tau_N$ are real,  decompositions \eqref{eq decomp C^n ker L(alpha)} and \eqref{eq decomp C^n ker L(alpha) tilde} also hold in $\R^N$, but we are interested in the ones of $\C^N$ since some functions that will be studied are complex valued.
		\item If $ \alpha=(\tau,\eta,\xi) $ is a characteristic frequency, then by definition and according to Assumption \ref*{hypothese stricte hyp}, the triplet $ (\tau,\eta,\xi) $ satisfies
		\begin{equation*}
		\prod_{k=1}^N\big(\tau-\tau_{k}(\eta,\xi)\big)=0.
		\end{equation*}
		There exists therefore an integer $ k $ between 1 and $ N $ such that $ \tau=\tau_k(\eta,\xi) $. In other words, the characteristic manifold $ \mathcal{C} $ is the union of the $ N $ hypersurfaces given by $ \ensemble{\tau=\tau_k(\eta,\xi)} $, $ k=1,\dots,N $.
	\end{enumerate} 
\end{remark}

\begin{remark}
	One can verify that in Assumption \ref{hypothese stricte hyp}, the functions $ \tau_k $ for $ k=1,\dots,N $ are positively homogeneous of degree 1 in $ \R^d\privede{0} $. The projectors $ \pi_{k} $ and $ \tilde{\pi}_k $ for $ k=1,\dots,N $ are therefore positively homogeneous of degree 0 in $ \R^d\privede{0} $.
\end{remark}

\begin{example}
	Returning to Example \ref{exemple Euler 1}, for system \eqref{eq systeme Euler} linearized around $ V_0=(v_0,0,u_0) $, the characteristic polynomial $ p $ reads
	\begin{equation*}
	p(\tau,\eta,\xi)=\det\left(\begin{array}{ccc}
	\tau+\xi\,u_0 & -v_0\,\eta & -v_0\,\xi \\[5pt]
	-c_0^2\,\eta/v_0 & \tau+ \xi\, u_0 & 0 \\[5pt]
	-c_0^2\,\xi/v_0 & 0 & \tau+\xi\,u_0
	\end{array}\right)=(\tau+\xi\,u_0)\,\big((\tau+\xi\,u_0)^2-c_0^2\,(\eta^2+\xi^2)\big).
	\end{equation*}
	Thus the eigenvalues of the matrix $ A(\eta,\xi)=\eta\,A_1(V_0)+\xi\,A_2(V_0) $ are the additive inverse of the roots with respect to $ \tau $ of the polynomial $ p $, given by
	\begin{equation}\label{eq Euler def tau 1 2 3}
	\tau_1(\eta,\xi):=-u_0\,\xi-c_0\,\sqrt{\eta^2+\xi^2},\quad \tau_2(\eta,\xi):=-u_0\,\xi,\quad \tau_3(\eta,\xi):=-u_0\,\xi+c_0\,\sqrt{\eta^2+\xi^2}.
	\end{equation}
	The functions $ \tau_1,\tau_2 $ and $ \tau_3 $ are analytic and distinct in $ \R^2\privede{0} $. System \eqref{eq systeme Euler} is therefore strictly hyperbolic, which means that it satisfies Assumption \ref{hypothese stricte hyp} of strict hyperbolicity.
	We have represented in Figure \ref{figure frequences caract Euler} the characteristic frequencies $ \alpha=(\tau,\eta,\xi)\in\R^3 $ for system \eqref{eq systeme Euler}.
\end{example}

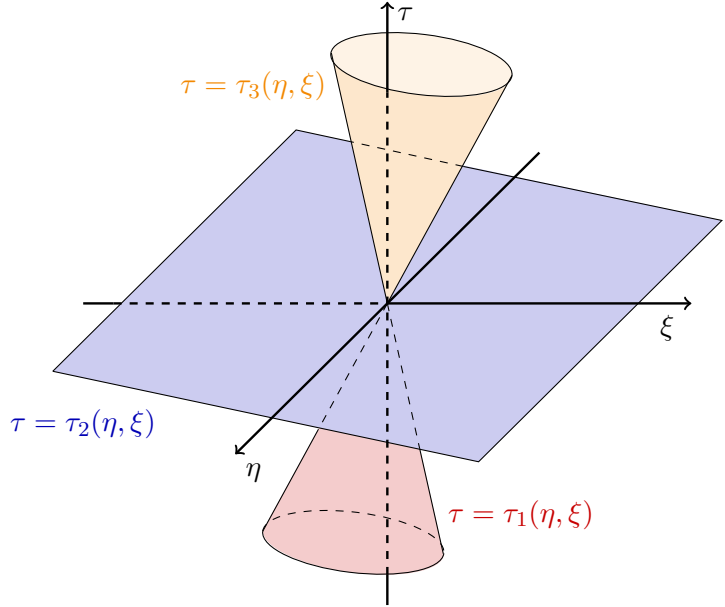
\begin{figure}[!ht]
	\centering
	\begin{tikzpicture}[scale=0.8]
	%Zone remplies
	\fill[altred!20, rotate=-8] (0,0) -- (-3/2,-4) arc(-180:0:1.5 and 0.5) -- (3/2,-4); 
	\fill[altblue!20] (-5.5,-1.125) --  (1.5,-2.625) -- (5.5,1.375) -- (-1.5,2.875);
	\fill[altorange!20, rotate=-8] (0,0) -- (-3/2,4) arc(-180:0:1.5 and 0.5) -- (3/2,4); 
	\fill[altorange!10, rotate around={-7:(0.56,3.96)}] (0.56,3.96) ellipse(1.5 and 0.5);
	%Axes
	\draw[line width = 0.9pt] (-5,0) -- (-4.5,0);
	\draw[line width = 0.9pt, dashed] (-4.5,0) -- (0,0);
	\draw[line width = 0.9pt,->] (0,0) -- (5,0);
	\draw[line width = 0.9pt,] (0,-5) -- (0,-4.5);
	\draw[line width = 0.9pt, dashed] (0,-4.5) -- (0,3.5);
	\draw[line width = 0.9pt,->] (0,3.5) -- (0,5);
	\draw[line width = 0.9pt,->] (2.5,2.5) -- (-2.5,-2.5);
	%Noms
	\draw (-2.2,-2.8) node{$ \eta $};
	\draw (4.6,-0.4) node{$ \xi $};
	\draw (0.3,4.8) node{$ \tau $};
	\draw[altred] (2.2,-3.5) node{$ \tau=\tau_1(\eta,\xi) $};
	\draw[altblue] (-5,-2) node{$ \tau=\tau_2(\eta,\xi) $};
	\draw[altorange] (-2.2,3.6) node{$ \tau=\tau_3(\eta,\xi) $};
	%Cônes
	\draw[] (0,0) -- (2.05,3.78);
	\draw[] (0,0) -- (-0.92,4.15);
	\draw[, dashed] (0,0) -- (-1.14,-2.1);
	\draw[, dashed] (0,0) -- (0.53,-2.4);
	\draw[] (-1.14,-2.1) -- (-2.05,-3.78);
	\draw[] (0.53,-2.4) -- (0.92,-4.15);
	\draw[rotate around={-7:(-0.56,-3.96)}, ] (-2.06,-3.96) arc(-180:0:1.5 and 0.5);
	\draw[rotate around={-7:(-0.56,-3.96)}, , dashed] (-2.06,-3.96) arc(180:0:1.5 and 0.5);
	\draw[rotate around={-7:(0.56,3.96)}, ] (0.56,3.96) ellipse(1.5 and 0.5);
	%Plan
	\draw[] (-5.5,-1.125) --  (1.5,-2.625) -- (5.5,1.375) -- (1.25,2.27);
	\draw[] (-0.6,2.68) -- (-1.5,2.875 )--(-5.5,-1.125) ;
	\draw[dashed] (1.25,2.27) -- (-0.6,2.68);
	\end{tikzpicture}
	\caption{Characteristic frequencies for the isentropic compressible Euler system \eqref{eq systeme Euler}}
	\label{figure frequences caract Euler}
\end{figure}

We now define projectors derived from the $ \C^N $ decomposition \eqref{eq decomp C^n ker L(alpha)}, that we extend to noncharacteristic frequencies. We also determine some equalities between the kernel and range of the projectors $ \pi_{\alpha} $ and $ \tilde{\pi}_{\alpha} $ and of the matrices $ L(0,\alpha) $ and $ \tilde{L}(0,\alpha) $. The proof is based on the one of \cite[Lemma 3.2]{CoulombelGues2010Geometric}.

\begin{definition}\label{def pi alpha}
	Let $\alpha=(\tau,\eta,\xi)\in\R^{1+d} \privede{0}$  be a characteristic frequency and $ k $ the integer between $ 1 $ and $ N $ such that $\tau=\tau_k(\eta,\xi)$. We denote by $\pi_{\alpha}:=\pi_k(\eta,\xi)$ (resp. $ \tilde{\pi}_{\alpha}:=\tilde{\pi}_k(\eta,\xi) $) the projection from $\C^N$ onto the eigenspace $\ker L\big(0,\tau_k(\eta,\xi),\eta,\xi\big)$ (resp. the subspace $A_d(0)^{-1}\,\ker L\big(0,\tau_k(\eta,\xi),\eta,\xi\big)$) according to decomposition \eqref{eq decomp C^n ker L(alpha)} (resp. \eqref{eq decomp C^n ker L(alpha) tilde}). If the frequency $\alpha\in\R^{1+d}\setminus\ensemble{0}$ is not characteristic, we denote $ \pi_{\alpha}=\tilde{\pi}_{\alpha}:=0 $ and if $ \alpha=0 $ we denote $ \pi_0:=\tilde{\pi}_0:=I $. For all $ \alpha $ in $ \R^{d+1} $, we can verify that $ \pi_{\alpha} $ satisfies
	\begin{equation}\label{eq freq ker L ima pi}
	\ker L(0,\alpha) = \ker \tilde{L}(0,\alpha) = \Ima \pi_{\alpha},
	\end{equation}
	and
	\begin{equation}\label{eq freq ima L ker pi}
	\Ima L(0,\alpha)=\ker \pi_{\alpha},
	\end{equation}
	and that the projector $ \tilde{\pi}_{\alpha} $ satisfies
	\begin{equation}\label{eq freq ima L tilde ker pi tilde}
	 \Ima \tilde{L}(0,\alpha)=\ker \tilde{\pi}_{\alpha},
	\end{equation}
	 recalling that $ \tilde{L}(0,\alpha) $ refers to the modified symbol $ \tilde{L}(0,\alpha):=A_d(0)^{-1}\,L(0,\alpha) $.
	 
	For all $ \alpha\in\R^{1+d}\setminus\ensemble{0} $, we denote by $ Q_{\alpha} $ the partial inverse of the matrix $ L(0,\alpha) $, namely the unique matrix $ Q_{\alpha} $ such that $ Q_{\alpha}\,  L(0,\alpha) =L(0,\alpha)\, Q_{\alpha}=I-\pi_{\alpha} $. If $ \alpha=0 $, we define $ Q_{\alpha}:=I $. 
\end{definition}

\begin{proof}
	Consider $ \alpha=(\tau,\eta,\xi) $ in $ \R^{d+1} $. Equation \eqref{eq freq ker L ima pi} is satisfied by definition of $ \pi_{\alpha} $, and equality of the kernels $ \ker L(0,\alpha) $  and $ \ker \tilde{L}(0,\alpha) $, the matrix $ A_d(0) $ being invertible. Regarding equation \eqref{eq freq ima L ker pi}, we first note that by the rank-nullity theorem and by definition of  $ \pi_{\alpha} $, the subspaces $ \Ima L(0,\alpha) $ and $ \ker \pi_{\alpha} $ have the same dimension. We denote by $ k_0 $ the integer between 1 and $ N $ such that $ \tau=\tau_{k_0}(\eta,\xi) $. We consider then an element $ L(0,\alpha)\,X $ of $ \Ima L(0,\alpha) $, with $ X $ in $ \C^N $ that we decompose according to \eqref{eq decomp C^n ker L(alpha)}, as
	$
	X=\sum_{k= 1}^N\pi_{k}(\eta,\xi)\,X.
	$
	For $ k=1,\dots,N $, the projector $ \pi_k(\eta,\xi) $ admitting the eigenspace of the matrix $ A(\eta,\xi) $ associated with the eigenvalue $ -\tau_k(\eta,\xi) $ as range, we have
	\begin{align*}
	L(0,\alpha)\,X&=\sum_{k= 1}^N\Big(\tau_{k_0}(\eta,\xi) \,I+A(\eta,\xi)\Big)\pi_{k}(\eta,\xi)\,X\\
	&=\sum_{k\neq k_0}\big(\tau_{k_0}(\eta,\xi)-\tau_{k}(\eta,\xi)\big)\,\pi_{k}(\eta,\xi)\,X.
	\end{align*}
	Equation \eqref{eq freq ker L ima pi} being satisfied, we deduce that $ L(0,\alpha)\,X $ belongs to
	\begin{equation*}
	\bigoplus_{k\neq k_0}\ker L\big(0,(\tau_k(\eta,\xi),\eta,\xi)\big),
	\end{equation*}
	which, by definition of the projectors $ \pi_k $, is equal to the kernel of $ \pi_{k_0}(\eta,\xi)=\pi_{\alpha} $. With the equality of dimensions, equation \eqref{eq freq ima L ker pi} is therefore verified. The proof of equation \eqref{eq freq ima L tilde ker pi tilde} is similar: we consider $ X $ in $ \C^N $ that we decompose as
	$
	X=\sum_{k= 1}^N\pi_{k}(\eta,\xi)\,X,
	$
	and then we write
	\begin{align*}
	\tilde{L}(0,\alpha)\,X&=A_d(0)^{-1}\sum_{k= 1}^N\Big(\tau_{k_0}(\eta,\xi) \,I+A(\eta,\xi)\Big)\pi_{k}(\eta,\xi)\,X\\
	&=A_d(0)^{-1}\sum_{k\neq k_0}\big(\tau_{k_0}(\eta,\xi)-\tau_{k}(\eta,\xi)\big)\,\pi_{k}(\eta,\xi)\,X,
	\end{align*}
	so that $ \tilde{L}(0,\alpha)\,X $ belongs to
	\begin{equation*}
	\bigoplus_{k\neq k_0}A_d(0)^{-1}\ker L\big(0,(\tau_k(\eta,\xi),\eta,\xi)\big)=\ker \tilde{\pi}_{\alpha}.
	\end{equation*}
	Once again by equality of dimensions it leads to equation \eqref{eq freq ima L tilde ker pi tilde}.
	\end{proof}

\begin{remark}\label{remarque proj pi et Q bornes}
	\begin{enumerate}[label=\roman*), leftmargin=0.7cm]
	\item For every $ k=1,\dots, N $, the projectors $\pi_k(\eta,\xi)$ and $ \tilde{\pi}_k(\eta,\xi) $ are positively homogeneous of degree 0 in $(\eta,\xi)\in\R^{d}\setminus\ensemble{0}$. Furthermore, by strict hyperbolicity, the basis $ E_1(\eta,\xi),\dots,E_N(\eta,\xi) $ and $ A_d(0)^{-1}E_1(\eta,\xi), \dots,A_d(0)^{-1}E_N(\eta,\xi) $ are analytic with respect to $(\eta,\xi)\in\R^{d}\setminus\ensemble{0}$, and the maps $(\eta,\xi)\mapsto\pi_k(\eta,\xi)$ and $(\eta,\xi)\mapsto\tilde{\pi}_k(\eta,\xi)$ are therefore analytic in $ \R^d\setminus\ensemble{0} $. Thus, by compactness of the sphere $ \S^{d-1} $, for all $ k=1,\dots,N $, the projectors $\pi_k(\eta,\xi)$ and $\tilde{\pi}_k(\eta,\xi)$ are uniformly bounded with respect to $ (\eta,\xi)\in\R^d\setminus\ensemble{0} $. The projectors $\pi_{\alpha}$ and $ \tilde{\pi}_{\alpha} $ are therefore bounded with respect to $\alpha$ in $\R^{1+d}$. 
	\item Unlike the projectors $ \pi_{k} $ and $ \tilde{\pi}_k $, $ k=1,\dots,N $, the projectors $ \pi_{\alpha} $ and $ \tilde{\pi}_{\alpha} $ are homogeneous of degree 0 with respect to $ \alpha $ in $ \R^{d+1} $, and not only \emph{positively} homogeneous. Indeed, the claim is obvious if $ \alpha $ is zero or noncharacteristic, and if $ \alpha $ is a nonzero characteristic frequency, and $ \lambda $ a nonzero real number, then, since $ \ker L(0,\lambda\,\alpha)=\ker L(0,\alpha) $ and $ \Ima L(0,\lambda\,\alpha)=\Ima L(0,\alpha) $, we have $ \pi_{\lambda\,\alpha}=\pi_{\alpha} $. The proof is the same for $ \tilde{\pi}_{\alpha} $.
	\end{enumerate}
\end{remark}

\subsection{The uniform Kreiss-Lopatinskii condition and some preliminary results}

We define the following space of frequencies
\begin{align*}
\Xi&:=\{\zeta=(\sigma=\tau-i\gamma,\eta)\in(\C\times\R^{d-1})\backslash\{0\} \mid \gamma\geq 0\},\\
\Sigma&:=\ensemble{\zeta\in\Xi\mid \tau^2+\gamma^2+|\eta|^2=1},\\
\Xi_0&:=\{\zeta\in\Xi \mid \gamma=0\},\\
\Sigma_0&:=\Xi_0\cap\Sigma.
\end{align*}
We also define the matrix that we get when applying the Laplace-Fourier transform to the operator $L(0,\partial_z)$. For all $\zeta=(\sigma,\eta)\in\Xi$, let
\[\mathcal{A}(\zeta):=-i\,A_d(0)^{-1}\Big(\sigma I+\sum_{i=1}^{d-1}\eta_j\, A_j(0)\Big).\]
The noncharacteristic boundary Assumption \ref{hypothese bord non caract} is used here to define the matrix $ \mathcal{A}(\zeta) $. We note that if $ \zeta=(\tau,\eta)\in\Xi_0 $, and if $ i\xi $ is an imaginary eigenvalue of $ \mathcal{A}(\zeta) $, then the frequency $ (\tau,\eta,\xi) $ is a real characteristic frequency, and vice versa.

Hersh lemma \cite[Lemma 1]{Hersh1963Mixed} ensures that for $ \zeta $ in $ \Xi\backslash\Xi_0 $, the matrix $\mathcal{A}(\zeta)$ has no eigenvalue of zero real part, and that the stable subspace associated with the eigenvalues of negative real part, denoted by $E_-(\zeta)$, is of constant dimension, denoted $p$. Furthermore, the integer $ p $ is  obtained as the number of positive eigenvalues of the matrix $ A_d(0) $.
We denote by $E_+(\zeta)$ the unstable subspace $\mathcal{A}(\zeta)$ associated with eigenvalues of positive real part, that is of dimension $N-p$.

In \cite{Kreiss1970Initial} (see also \cite[Theorem 3.5]{ChazarainPiriou1982Introduction} and \cite[Lemma 4.5]{BenzoniSerre2007Multi}) it is shown that the stable and unstable subspaces $E_{\pm}$ extend continuously to the whole space $\Xi$ in the strictly hyperbolic case (Assumption \ref{hypothese stricte hyp}).
We still denote by $ E_{\pm} $ the extensions to $ \Xi $. The main assumption of this work may now be stated, which, along with Assumptions \ref{hypothese bord non caract} and \ref{hypothese stricte hyp}, ensures that system \eqref{eq systeme 1} is well posed locally in time. Indeed the three assumptions \ref{hypothese bord non caract}, \ref{hypothese stricte hyp} and \ref{hypothese UKL} are stable under small perturbations around the equilibrium. Just like Assumptions \ref{hypothese bord non caract} and \ref{hypothese stricte hyp}, the following assumption is very structural to the problem.

\begin{assumption}[Uniform Kreiss-Lopatinskii condition]\label{hypothese UKL}
	For all $\zeta\in\Xi$, we have
	\[\ker B\cap E_-(\zeta)=\{0\}.\]
	In particular, it forces the rank of the matrix $B$ to be equal to the dimension of $E_-(\zeta)$, namely $M=p$.
\end{assumption}

\begin{remark}\label{remarque B inverse bornee}
	Historically, the first given definition of the uniform Kreiss-Lopatinskii condition did not involve the extension of $ E_- $ to $ \Xi_0 $. The original definition states that, for all $ \zeta\in\Xi\setminus\Xi_0 $, 
	\begin{equation*}
		\ker B\cap E_-(\zeta)=\ensemble{0},
	\end{equation*}
	and that the linear map $ \left(B_{|E_-(\zeta)}\right)^{-1} $ is uniformly bounded with respect to $ \zeta\in\Xi\setminus\Xi_0$, see for instance \cite{Sarason1965Hyperbolic}.
	Indeed, the space $ E_-(\zeta) $ being homogeneous of degree zero and continuous with respect to $ \zeta\in\Xi $, and by compactness of the unitary sphere $ \Sigma $, we note that Assumption \ref{hypothese UKL} implies that the linear map $ \left(B_{|E_-(\zeta)}\right)^{-1} $ is uniformly bounded with respect to $ \zeta\in\Xi $.
\end{remark}

It has already been discussed that for $ \zeta\in\Xi\setminus\Xi_0 $, the matrix $ \mathcal{A}(\zeta) $ has no imaginary eigenvalue. We now commit to describe more precisely the matrix $\mathcal{A}(\zeta)$ for $ \zeta $ in $ \Xi_0 $ as well as the continuous extension to $ \Xi_0 $ of the spaces $E_{\pm}(\zeta)$. The following result, proved by Kreiss \cite{Kreiss1970Initial} for the strictly hyperbolic case that is of interest here, Métivier \cite{Metivier2000Block} for the constantly hyperbolic case, and extended by Métivier and Zumbrun \cite{MetivierZumbrun2005Hyperbolic} to an even more general framework, gives a very useful decomposition of the matrix $\mathcal{A}(\zeta)$ when $\zeta$ belongs to $\Xi_0$.

\begin{proposition}[Block structure]\label{prop struct bloc}
	When Assumption \ref{hypothese stricte hyp} is satisfied, for all $\underline{\zeta}\in\Xi$, there exist a neighborhood $\mathcal{V}$ of $\underline{\zeta}$ in $\Xi$, an integer $L\geq 1$, a partition $N=\rho_1+\dots+\rho_L$ and an invertible matrix $T$ analytic in $\mathcal{V}$ such that for all $\zeta\in\mathcal{V}$, we have
	\[T(\zeta)\,\mathcal{A}(\zeta)\,T(\zeta)^{-1}=\diag\big(\mathcal{A}_1(\zeta),\dots,\mathcal{A}_L(\zeta)\big),\]
	where for all $j$ the matrix $\mathcal{A}_j(\zeta)$ is of size $\rho_j$ and satisfies one of the following properties:
	\begin{enumerate}[label=\roman*),itemsep=0pt]
		\item the real part of the matrix $\mathcal{A}_j(\zeta)$, defined by $(\mathcal{A}_j(\zeta)+\mathcal{A}_j(\zeta)^*)/2$, is positive-definite,
		\item the real part of the matrix $\mathcal{A}_j(\zeta)$ is negative-definite,
		\item $\rho_j=1$, $\mathcal{A}_j(\zeta)$ is imaginary when $\gamma$ is zero and $\partial_{\gamma}\mathcal{A}_j(\underline{\zeta})\in\R^*$,
		\item $\rho_j>1$, the coefficients of $\mathcal{A}_j(\zeta)$ are imaginary when $\gamma$ is zero, there exists $\xi_j\in \R$ such that 
		\[\mathcal{A}_j(\underline{\zeta})=\left(\begin{array}{ccc}
		i\,\xi_j & i & 0 \\
		& \ddots & i \\
		0 & & i\,\xi_j
		\end{array}
		\right),\]
		and the bottom left coefficient of  $\partial_{\gamma}\mathcal{A}_j(\underline{\zeta})$ is real and non zero.
	\end{enumerate}
\end{proposition}

This result, commonly referred to as "block structure" \cite[Section 5.1.2]{BenzoniSerre2007Multi}, is fundamental for the proof of Proposition \ref{prop proj bornes} below. In the aim of describing the subspaces $E_{\pm}(\zeta)$ for $\zeta\in\Xi_0$, the vector fields associated with each real characteristic phase are now defined.

\begin{definition}\label{def sortant rentrant alpha X alpha}
	Let $\alpha=(\tau,\eta,\xi)\in\R^{d+1}\backslash\ensemble{0}$ be a characteristic frequency, and $ k $ the integer between 1 and $ N $ such that $\tau=\tau_k(\eta,\xi)$. The group velocity $ \mathbf{v}_{\alpha} $ associated with $ \alpha $ is defined as
	\begin{equation*}
	\mathbf{v}_{\alpha}:=\nabla_{\eta,\xi}\,\tau_k(\eta,\xi).
	\end{equation*}
	We shall say that $\alpha$ is glancing (resp. incoming, outgoing) if $\partial_{\xi}\tau_k(\eta,\xi)$ is zero (resp. negative, positive).  
	Then the vector field $X_{\alpha}$ associated with $\alpha$ is defined as
	\begin{equation}\label{eq champ vecteur X_alpha}
	X_{\alpha}:=\partial_t-\mathbf{v}_{\alpha}\cdot\nabla_{x}=\partial_t-\nabla_{\eta}\tau_k(\eta,\xi)\cdot\nabla_{y}-\partial_{\xi}\tau_k(\eta,\xi)\,\partial_{x_d}.
	\end{equation}
	The vector field $ X_{\alpha} $ is represented in Figure  \ref{figure champ de vecteur r,s,g} in the glancing, incoming and outgoing case.
\end{definition}

\begin{figure}%[!ht]
	\centering
	\begin{tikzpicture}[scale=0.8]
	\fill[black!15] (0.5,-0.5) -- (0.5,5.5) -- (-2.5,2.5) -- (-2.5,-3.5)--cycle;
	\draw[line width = 1pt, dashed] (-1,0) -- (0,0);
	\draw[line width = 1pt,->] (0,0) -- (5,0);
	\draw[line width = 1pt,->] (0,-1) -- (0,5);
	\draw[line width = 1pt,->] (0.5,0.5) -- (-2.5,-2.5);
	\draw (5,-0.4) node{$ x_d $};
	\draw (0.3,4.8) node{$ y $};
	\draw (-2.1,-2.5) node{$ t $};
	\draw (-3.2,2.7) node{$ x_d=0 $};
	\draw[altred,->,line width = 1pt] (-0.3,2) -- (-0.8,-1.2); 
	\draw[altred] (-0.9,0.5) node{$ g $};
	\draw[altred,->,line width = 1pt] (3,-1.8) -- (0.5,-2); 
	\draw[altred] (1.8,-1.6) node{$ o $};
	\draw[altred,->,line width = 1pt] (1,3.5) -- (2.5,2.8); 
	\draw[altred] (1.7,3.5) node{$ i $};
	\end{tikzpicture}
	\caption{Incoming ($ i $), outgoing ($ o $) and glancing ($ g $) vector field.}
	\label{figure champ de vecteur r,s,g}
\end{figure}
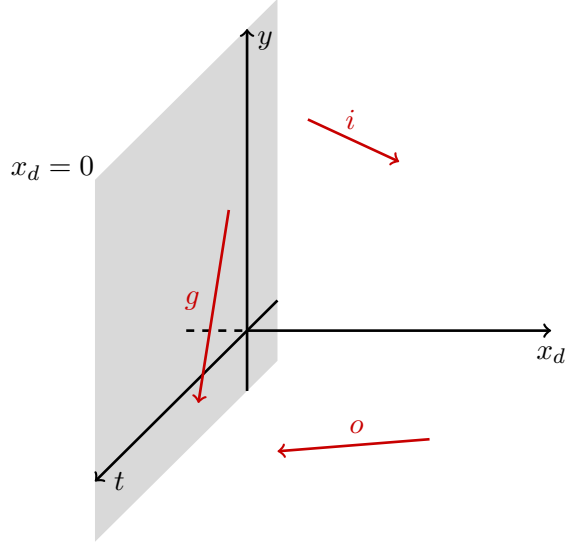

\begin{example}\label{exemple Euler UKL}
	We start by giving an example of a boundary condition for Example \ref{exemple Euler 1} satisfying the uniform Kreiss-Lopatinskii condition. For this purpose we look for a matrix $ B $ in $ \mathcal{M}_{2,3}(\R) $ of maximal rank, that generates strictly dissipative boundary conditions (see \cite[Definition 9.2]{BenzoniSerre2007Multi}), namely such that its kernel, which is of dimension 1, is generated by a nonzero vector $ E $ satisfying
	\begin{equation*}
		\,^tE\, S(V)\,A_2(V)\,E<0
	\end{equation*}
	for all $ V $ in the neighborhood of $ V_0 $, where the matrix $ S(V) $ refers to a Friedrichs symmetrizer of the system. Such strictly dissipative boundary conditions satisfy in particular the uniform Kreiss-Lopatinskii condition, see \cite[Proposition 4.4]{BenzoniSerre2007Multi}. In our example, the following symmetrizer may be considered
	\begin{equation*}
		S(V)=\diag\big(c(v)^2,v^2,v^2\big).
	\end{equation*}
	Recall that, in the notations of the example, a symmetrizer $ S(V) $ is a positive definite matrix such that the matrices $ S(V)\,A_1(V)  $ and $ S(V)\,A_2(V) $ are symmetric for all $ V $ in a neighborhood of $ V_0 $. It is then determined that a suitable vector $ E $ is given by $ E=(v_0,0,u_0) $, since in that case we have
	\begin{equation*}
		\,^tE\, S(V_0)\,A_2(V_0)\,E=u_0\,v_0^2\,(u_0^2-c_0^2)
	\end{equation*}
	the right-hand side quantity being negative by assumption on $ V_0 $, so it stays negative in a neighborhood of $ V_0 $. Thus a matrix $ B $ of maximal rank whose kernel is generated by $ E $ is for example given by
	\begin{equation*}
		B:=\left(\begin{array}{ccc}
			0 & v_0 & 0 \\
			-u_0 & 0 & v_0
		\end{array}\right),
	\end{equation*}
	which gives an example of a boundary condition satisfying the uniform Kreiss-Lopatinskii condition for Example \ref{exemple Euler 1} of compressible isentropic Euler equations in dimension 2.
	
	Interest is now made on the eigenvalues of the matrix $ \mathcal{A}(\tau,\eta) $ for the system of Example \ref{exemple Euler 1}. Their expressions, for $ (\tau,\eta)\in\R^2\privede{0} $, depend on the sign of $ \tau^2-\eta^2\,(c_0^2-u_0^2) $, as represented in Figure \ref{figure exemple Euler 1}.
	
	If $ |\tau|>\sqrt{c_0^2-u_0^2}\,|\eta| $, i.e. if $ \zeta=(\tau,\eta) $ is in the so-called \emph{hyperbolic} region $ \mathcal{H} $ (\cite[Definition 2.1]{Benoit2014Geometric}), then the matrix $ \mathcal{A}(\zeta) $ admits three simple imaginary eigenvalues given by
	\begin{subequations}\label{eq ex Euler v.p. hyp}
	\begin{align}
	i\,\xi_1(\tau,\eta)&:=i\,\frac{\tau\, u_0+\sign(\tau)\,c_0\,\sqrt{\tau^2-\eta^2\,(c_0^2-u_0^2)}}{c_0^2-u_0^2},\\
	i\,\xi_2(\tau,\eta)&:=i\,\frac{\tau \,u_0-\sign(\tau)\,c_0\,\sqrt{\tau^2-\eta^2\,(c_0^2-u_0^2)}}{c_0^2-u_0^2},\\\label{eq freq def xi 3}
	i\,\xi_3(\tau,\eta)&:=i\,\frac{-\tau}{u_0},
	\end{align}
	\end{subequations}
	where $ \sign(x):=x/|x| $ for $ x\neq0 $.
	The number $ \xi_1(\tau,\eta) $ being real, the frequency $ \alpha_1(\tau,\eta):=\big(\tau,\eta,\xi_1(\tau,\eta)\big) $ is a real characteristic frequency. It is then determined that we have $ \tau=\tau_3\big(\eta,\xi_1(\tau,\eta)\big) $ if $ \tau>0 $ and $ \tau=\tau_1\big(\eta,\xi_1(\tau,\eta)\big) $ if $ \tau<0 $. A calculation gives, if $ \tau>0 $,
	\begin{align*}
	&\partial_{\xi}\tau_3\big(\eta,\xi_1(\tau,\eta)\big)=\sqrt{\dfrac{\tau^2-\eta^2(c_0^2-u_0^2)}{\eta^2+\xi_1^2(\tau,\eta)}},
	\intertext{and, if $ \tau<0 $,}
	 &\partial_{\xi}\tau_1\big(\eta,\xi_1(\tau,\eta)\big)=\sqrt{\dfrac{\tau^2-\eta^2(c_0^2-u_0^2)}{\eta^2+\xi_1^2(\tau,\eta)}}.
	\end{align*}
	Thus the frequency $ \alpha_1(\tau,\eta)=\big(\tau,\eta,\xi_1(\tau,\eta)\big) $ is always outgoing. Likewise, it is determined that the real characteristic frequency $ \alpha_2(\tau,\eta):=\big(\tau,\eta,\xi_2(\tau,\eta)\big) $ is always incoming, and the frequency $ \alpha_3(\tau,\eta):=\big(\tau,\eta,\xi_3(\tau,\eta)\big) $ is incoming as well.
	
	If $ \zeta $  is located in the so-called \emph{glancing} region $ \mathcal{G} $, i.e. if $ |\tau|=\sqrt{c_0^2-u_0^2}\,|\eta|  $, then the matrix $ \mathcal{A}(\zeta) $ admits one imaginary simple eigenvalue $ i\,\xi_3(\zeta) $ which is still given by formula \eqref{eq freq def xi 3}, and a double imaginary eigenvalue given by
	\begin{equation*}
	i\,\xi_1(\zeta)=i\,\xi_2(\zeta)=i\,\frac{\tau\,u_0}{c_0^2-u_0^2}.
	\end{equation*}
	In this case we still have $ \tau=\tau_3\big(\eta,\xi_1(\tau,\eta)\big) $ if $ \tau>0 $ and $ \tau=\tau_1\big(\eta,\xi_1(\tau,\eta)\big) $ if $ \tau<0 $, and regarding the characteristic frequency $ \alpha_3(\tau,\eta)=\big(\tau,\eta,\xi_3(\tau,\eta)\big) $, we still have $ \tau=\tau_2\big(\eta,\xi_3(\tau,\eta)\big) $. Thus it is determined that
	\begin{equation*}
	\partial_{\xi}\tau_1\big(\eta,\xi_1(\tau,\eta)\big)=\partial_{\xi}\tau_3\big(\eta,\xi_1(\tau,\eta)\big)=0,
	\end{equation*}
	and therefore, regardless of the sign of $ \tau $, the frequency $ \alpha_1(\tau,\eta) $ is glancing. As for it, the frequency $ \alpha_3(\tau,\eta) $ is always incoming.
	
	Finally if $ |\tau|<\sqrt{c_0^2-u_0^2}\,|\eta| $ and so if $ \zeta $ is in the so-called \emph{mixed} region $ \mathcal{EH} $, then the matrix $ \mathcal{A}(\zeta) $ has one simple imaginary eigenvalue $ i\,\xi_3(\zeta) $ given by formula \eqref{eq freq def xi 3}, and two simple eigenvalues of nonzero real part (symmetric with respect to the imaginary axis), that are still denoted by $ i\,\xi_1 $ et $ i\,\xi_2 $ and which are given by
	\begin{align*}
	i\,\xi_1(\tau,\eta)&:=i\,\frac{\tau\, u_0+i\,c_0\,\sign(\tau)\,\sqrt{\eta^2\,(c_0^2-u_0^2)-\tau^2}}{c_0^2-u_0^2},\\
	i\,\xi_2(\tau,\eta)&:=i\,\frac{\tau \,u_0-i\,c_0\,\sign(\tau)\,\sqrt{\eta^2\,(c_0^2-u_0^2)-\tau^2}}{c_0^2-u_0^2}.
	\end{align*}
	The real characteristic frequency $ \alpha_3(\tau,\eta)=\big(\tau,\eta,\xi_3(\tau,\eta)\big) $ is once again incoming.
\end{example}

\begin{figure}%[!ht]
	\centering
	\begin{tikzpicture}[scale=0.6]
	\fill[altred!10] (-4.5,3) -- (0,0) -- (4.5,3) --cycle;
	\fill[altred!10] (-4.5,-3) -- (0,0) -- (4.5,-3) --cycle;
	\fill[altblue!10] (-4.5,-3) -- (0,0) -- (-4.5,3) --cycle;
	\fill[altblue!10] (4.5,-3) -- (0,0) -- (4.5,3) --cycle;
	\draw[line width = 0.8pt, altred] (-5.25,3.5) -- (5.25,-3.5);
	\draw[line width = 0.8pt, altred] (-5.25,-3.5) -- (5.25,3.5);
	\draw[line width = 1pt,->] (-5.5,0) -- (5.5,0);
	\draw[line width = 1pt,->] (0,-4) -- (0,4.5);
	\draw (5,-0.4) node{$ \eta $};
	\draw (0.4,4.2) node{$ \tau $};
	\draw[altred] (5.25,4) node{$ \tau=\sqrt{c_0^2-u_0^2}\,\eta $};
	\draw[altred] (-5.25,4) node{$ \tau=-\sqrt{c_0^2-u_0^2}\,\eta $};
	\draw[altred] (5,-2.8) node{$ \mathcal{G} $};
	\draw[altred] (0.8,2) node{$ \mathcal{H} $};
	\draw[altblue] (-3,1) node{$ \mathcal{EH} $};
	\end{tikzpicture}
	\caption{Areas of $ \Xi_0 $ for the isentropic compressible Euler equations}
	\label{figure exemple Euler 1}
\end{figure}
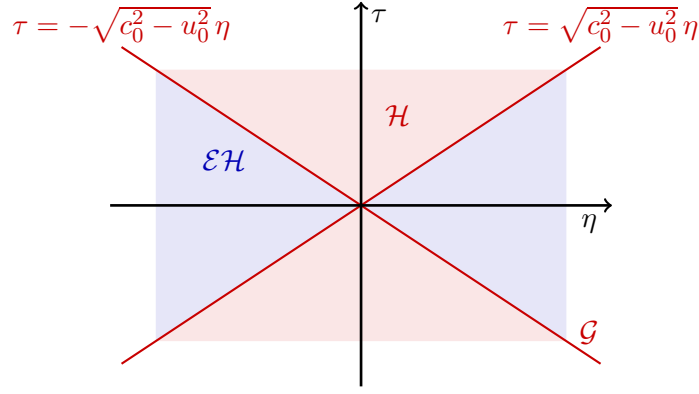

\bigskip

In the following, in order for the matrix factor of the partial derivative with respect to the normal variable $ x_d $ in the equations to be the identity matrix,  the modified operator $ \tilde{L}(u,\partial_z):=A_d(u)^{-1}\,L(u,\partial_z) $ shall be considered. For $ j=1,\dots,d-1 $, we denote $ \tilde{A}_j:=A_d^{-1}A_j $ and $ \tilde{A}_0:=A_d^{-1} $. 
The following lemma, which is a result of \cite{Lax1957Asymptotic} adapted by \cite[Lemma 2.11]{CoulombelGuesWilliams2011Resonant}, shows that, under suitable assumptions, the operator $ \tilde{\pi}_{\alpha}\,\tilde{L}(0,\partial_z)\,\pi_{\alpha} $ is given by a constant coefficient scalar transport operator, and therefore that the operator $ \tilde{\pi}_{\alpha}\,\tilde{L}(0,\partial_z) $ acts on polarized profiles (i.e. profiles $ U $ such that $ \pi_{\alpha}\,U=U $) as a much simpler operator.

\begin{lemma}[{\cite[Lax Lemma]{CoulombelGuesWilliams2011Resonant}}]\label{lemme Lax}
Let $\alpha=(\tau,\eta,\xi)\in\R^{1+d}\privede{0}$ be a real \emph{non glancing} characteristic frequency and $ k $ the integer between $ 1 $ and $ N $ such that $ \tau=\tau_k(\eta,\xi) $. Then we have
	\begin{equation*}
	\tilde{\pi}_{\alpha}\,\tilde{L}(0,\partial_z)\,\pi_{\alpha}=\frac{-1}{\partial_{\xi}\tau_k(\eta,\xi)}\,X_{\alpha}\,\tilde{\pi}_{\alpha}\,\pi_{\alpha},
	\end{equation*}
	where $X_{\alpha}$ is the vector field associated with $\alpha$ defined by \eqref{eq champ vecteur X_alpha}. Then we denote
	\begin{equation*}
	\tilde{X}_{\alpha}:=\frac{-1}{\partial_{\xi}\tau_k(\eta,\xi)}\,X_{\alpha}=\partial_{x_d}-\inv{\partial_{\xi}\tau_k(\eta,\xi)}\,\partial_t+\inv{\partial_{\xi}\tau_k(\eta,\xi)}\,\nabla_{\eta}\tau_k(\eta,\xi)\cdot\nabla_{y}.
	\end{equation*}
\end{lemma}

For the sake of completeness, the proof of  \cite{CoulombelGuesWilliams2011Resonant} is recalled here.

\begin{proof}
	According to identity \eqref{eq freq ker L ima pi}, we have
	\begin{equation}\label{eq freq Lax L pi = 0}
	\tilde{L}\big(0,(\tau_k(\eta,\xi),\eta,\xi)\big)\,\pi_k(\eta,\xi)=\Big(\tau_k(\eta,\xi)\,\tilde{A}_0(0)+\sum_{i=1}^{d-1}\eta_i\,\tilde{A}_i(0)+\xi\,I\Big)\,\pi_k(\eta,\xi)=0.
	\end{equation}
	The Dunford formula and the implicit function theorem ensure that in the strictly hyperbolic case, the projectors $ \pi_k $ as well as the real functions $ \tau_k $ are differentiable with respect to $ (\eta,\xi) $ in $ \R^d\privede{0} $ (they even depend analytically of $ (\eta,\xi) $). Thus identity \eqref{eq freq Lax L pi = 0} is differentiated with respect to $ \xi $ in a neighborhood of a frequency $ (\eta,\xi) $ in $ \R^d\privede{0} $ to obtain
	\begin{equation*}
	\Big(\partial_{\xi}\tau_k(\eta,\xi)\,\tilde{A}_0(0)+I\Big)\,\pi_k(\eta,\xi)+\Big(\tau_k(\eta,\xi)\,\tilde{A}_0(0)+\sum_{i=1}^{d-1}\eta_i\,\tilde{A}_i(0)+\xi\,I\Big)\,\partial_{\xi}\pi_k(\eta,\xi)=0,
	\end{equation*}
	and therefore, multiplying by $ \tilde{\pi}_k(\eta,\xi) $ on the left, according to identity \eqref{eq freq ima L tilde ker pi tilde}, we get
	\begin{equation}\label{eq freq Lax pi A d pi}
	\tilde{\pi}_k(\eta,\xi)\,\tilde{A}_0(0)\,\pi_k(\eta,\xi)=-\inv{\partial_{\xi}\tau_k(\eta,\xi)}\tilde{\pi}_k(\eta,\xi)\,\pi_k(\eta,\xi).
	\end{equation}
	Likewise, for $ i=1,\dots,d-1 $, equality \eqref{eq freq Lax L pi = 0} is differentiated with respect to $ \eta_i $ and next multiplied by $ \tilde{\pi}_k(\eta,\xi) $ to obtain
	\begin{equation*}
	\tilde{\pi}_k(\eta,\xi)\,\Big(\partial_{\eta_j}\tau_k(\eta,\xi)\,\tilde{A}_0(0)+\tilde{A}_i(0)\Big)\,\pi_k(\eta,\xi)=0.
	\end{equation*}
	With \eqref{eq freq Lax pi A d pi}, we thus get
	\begin{equation*}
	\tilde{\pi}_k(\eta,\xi)\,\tilde{A}_i(0)\,\pi_k(\eta,\xi)=\frac{\partial_{\eta_i}\tau_k(\eta,\xi)}{\partial_{\xi}\tau_k(\eta,\xi)}\tilde{\pi}_k(\eta,\xi)\,\pi_k(\eta,\xi),
	\end{equation*}
	which concludes the proof of the lemma.
\end{proof}

The following results use the classical Lax Lemma, whose proof is similar to the one of Lemma \ref{lemme Lax}. The result is recalled here.

\begin{lemma}[{\cite{Lax1957Asymptotic}}]\label{lemme Lax classique}
	Let $\alpha=(\tau,\eta,\xi)\in\R^{1+d}\privede{0}$ be a real characteristic frequency and $ k $ the integer between $ 1 $ and $ N $ such that $ \tau=\tau_k(\eta,\xi) $. Then we have
	\begin{equation*}
		\pi_{\alpha}\,L(0,\partial_z)\,\pi_{\alpha}=X_{\alpha}\,\pi_{\alpha},
	\end{equation*}
	where $X_{\alpha}$ is the vector field associated with $\alpha$ defined by \eqref{eq champ vecteur X_alpha}.
\end{lemma}

The first lemma below, quite standard, states that the group velocities $ \mathbf{v}_{\alpha} $ are bounded. The result presented here is not optimal, considering the constant $ C $ can be taken equal to 1, but it is sufficient for our analysis, and its proof is simpler.

\begin{lemma}\label{lemme vitesse de groupe bornees} There exists a positive constant $ C $ such that, for $ k=1\dots,N $ and $ (\eta,\xi) $ in $ \R^d\privede{0} $, we have
	\begin{equation*}
	\left|\nabla_{\eta,\xi}\tau_k(\eta,\xi)\right|\leq C \sup_{(\eta,\xi)\in\S^{d-1}}\rho\big(A(\eta,\xi)\big),
	\end{equation*}
	where we recall that $ A(\eta,\xi) $ has been defined for $ (\eta,\xi) $ in $ \R^d\privede{0} $ in Assumption \ref{hypothese stricte hyp}, and where $ \rho\big(A(\eta,\xi)\big) $ refers to the spectral radius of the matrix $ A(\eta,\xi) $.
	Then we denote by $ \mathcal{V}^* $ the finite quantity
	\begin{equation*}
	\mathcal{V}^*:= C\, \sup_{(\eta,\xi)\in\S^{d-1}}\rho\big(A(\eta,\xi)\big),
	\end{equation*} 
	which bounds the group velocities.
\end{lemma}

\begin{proof}
	First note that the quantity $ \mathcal{V}^* $ is actually finite. Indeed, according to Assumption \ref{hypothese stricte hyp}, we have
	\begin{equation*}
	\rho\big(A(\eta',\xi')\big)=\max_{k=1,\dots,N}\valabs{\tau_k(\eta',\xi')},
	\end{equation*}
	and the real functions $ \tau_1,\dots,\tau_N $ are analytic in $ \R^d\privede{0} $, and thus bounded on $ \S^{d-1} $.
	
	Now let $ (\eta,\xi) $ be in $ \R^d\privede{0} $ and $ k $ between 1 and $ N $. According to Lax Lemma \ref{lemme Lax classique}, we have, for $ (\eta',\xi') $ in $ \S^{d-1} $,
	\begin{equation*}
	\pi_{k}(\eta,\xi)\,A(\eta',\xi')\,\pi_{k}(\eta,\xi)=-d\tau_k(\eta,\xi)\cdot(\eta',\xi')\,\pi_{k}(\eta,\xi).
	\end{equation*}
	But since the following equality holds,
	\begin{equation*}
	\valabs{\nabla_{\eta,\xi}\tau_k(\eta,\xi)}=\sup_{(\eta',\xi')\in\S^{d-1}}\valabs{d\tau_k(\eta,\xi)\cdot(\eta',\xi')},
	\end{equation*}
	we obtain
	\begin{equation}\label{eq freq vg bornee 1}
	\valabs{\nabla_{\eta,\xi}\tau_k(\eta,\xi)}\leq \norme{\pi_{k}(\eta,\xi)}\sup_{(\eta',\xi')\in\S^{d-1}}\norme{A(\eta',\xi')}.
	\end{equation}
	On an other hand, because of Remark \ref{remarque proj pi et Q bornes}, there exists a positive constant $ C $ such that
	\begin{equation}\label{eq freq vg bornee 2}
	\norme{\pi_{k}(\eta,\xi)}\leq C,\quad k=1,\dots,N,
	\end{equation}
	uniformly with respect to $ (\eta,\xi) $ in $ \R^d\privede{0} $.
	Finally, Assumption \ref{hypothese stricte hyp} claims that the matrix $ A(\eta',\xi') $ is diagonalizable and well-conditioned, so there exists a positive constant $ C $ such that
	\begin{equation}\label{eq freq vg bornee 3}
	\sup_{(\eta',\xi')\in\S^{d-1}}\norme{A(\eta',\xi')}\leq C\sup_{(\eta',\xi')\in\S^{d-1}}\rho\big(A(\eta',\xi')\big).
	\end{equation}
	Equations \eqref{eq freq vg bornee 1}, \eqref{eq freq vg bornee 2}, and \eqref{eq freq vg bornee 3} then lead to the result.
\end{proof}

The second result quantitatively links the vector $ \tilde{\pi}_{\alpha}\,E_k(\eta,\xi) $ to the vector $ E_k(\eta,\xi) $, which will be useful in the following to get a control from below of the first vector. It is mentioned, for its second part, in  \cite{CoulombelGues2010Geometric}.

\begin{lemma}\label{lemme pi tilde E vitesse de groupe}
	Let $\alpha=(\tau,\eta,\xi)\in\R^{1+d}\privede{0}$ be a real characteristic frequency and $ k $ the integer between $ 1 $ and $ N $ such that $ \tau=\tau_k(\eta,\xi) $. Then we have
	\begin{equation*}
		\tilde{\pi}_{\alpha}\,E_k(\eta,\xi)=-\partial_{\xi}\tau_k(\eta,\xi)\, A_d(0)^{-1}E_k(\eta,\xi).
	\end{equation*}
In particular, if the frequency $ \alpha $ is not glancing, the projector $ \tilde{\pi}_{\alpha} $ induces an isomorphism from $ \Ima \pi_{\alpha} $ to $ \Ima \tilde{\pi}_{\alpha} $.
\end{lemma}

\begin{proof}
	First the vector $ E_k(\eta,\xi) $ is decomposed in basis \eqref{eq base C^N tilde adaptee espace propres stricte hyp} adapted to decomposition \eqref{eq decomp C^n ker L(alpha) tilde}:
	\begin{equation}\label{eq freq Lax 3 1}
		E_k(\eta,\xi)=\sum_{j=1}^N\lambda_j\,A_d(0)^{-1}E_j(\eta,\xi),
	\end{equation}
so that we have $ \tilde{\pi}_{\alpha}\,E_k(\eta,\xi)=\lambda_k\,A_d(0)^{-1}E_k(\eta,\xi) $. Thus the aim is to determine the coefficient $ \lambda_k $. Given that $ \pi_{\alpha}\,E_k(\eta,\xi)=E_k(\eta,\xi) $, and according to decomposition \eqref{eq freq Lax 3 1}, we have
\begin{equation*}
	A_d(0)\,\pi_{\alpha}\,E_k(\eta,\xi)=\sum_{j=1}^N\lambda_j\,E_j(\eta,\xi),
\end{equation*}
thus
\begin{equation*}
	\pi_{\alpha}\,A_d(0)\,\pi_{\alpha}\,E_k(\eta,\xi)=\lambda_k\,E_k(\eta,\xi).
\end{equation*}
And we conclude using Lax Lemma \ref{lemme Lax classique} which claims that $ \pi_{\alpha}\,A_d(0)\,\pi_{\alpha}=-\partial_{\xi}\tau_k(\eta,\xi)\,\pi_{\alpha} $.

To show that the projector $ \tilde{\pi}_{\alpha} $ induces an isomorphism from $ \Ima \pi_{\alpha} $ to $ \Ima \tilde{\pi}_{\alpha} $, the two spaces $ \Ima \pi_{\alpha} $ and $ \Ima \tilde{\pi}_{\alpha} $ having the same dimension, it is sufficient to prove that the intersection
\begin{equation*}
	\ker \tilde{\pi}_{\alpha}\cap\Ima \pi_{\alpha}
\end{equation*}
is trivial. So we consider a vector $ X $ of $ \C^N $ belonging to this intersection. Because $ X $ belongs to $ \Ima \pi_{\alpha} $, by definition of the vector $ E_k(\eta,\xi) $, it writes
\begin{equation*}
	X=\lambda\, E_k(\eta,\xi),
\end{equation*}
where $ k $ is the integer between $ 1 $ and $ N $ such that $ \alpha=\big(\tau_k(\eta,\xi),\eta,\xi\big) $ and $ \lambda\in\R $. According to the previous result, we have
\begin{equation*}
	\tilde{\pi}_{\alpha}\,X=-\partial_{\xi}\tau_k(\eta,\xi)\, A_d(0)^{-1}X.
\end{equation*}
But we also have $ \tilde{\pi}_{\alpha}\,X=0 $ by assumption and $ \partial_{\xi}\tau_k(\eta,\xi)\neq 0 $, the frequency $ \alpha $ being non-glancing. We therefore obtain $ X=0 $, which is the sought result.
\end{proof}

We are now in position to describe the decomposition of the stable subspace $ E_-(\zeta) $ for $ \zeta\in\Xi_0 $, which uses the strict hyperbolicity Assumption \ref{hypothese stricte hyp}.

\begin{proposition}[{\cite{Williams1996Boundary}, Proposition 3.4}]\label{prop decomp E_-}
	Consider $ \zeta=(\tau,\eta)\in\Xi_0 $. We denote by $ i\,\xi_j(\zeta) $ for $ j=1,\dots,\mathcal{M}(\zeta) $ the distinct complex eigenvalues of the matrix $ \mathcal{A}(\zeta) $, and if $ \xi_j(\zeta) $ is real, we shall denote by $ \alpha_j(\zeta):=(\tau,\eta,\xi_j(\tau,\eta)) $ the associated real characteristic frequency. If $ \xi_j(\zeta) $ is real, we also denote by $ k_j $ the integer between $ 1 $ and $ N $ such that $\tau=\tau_{k_j}(\eta,\xi_j(\zeta))$. Then the set $\ensemble{1,2,\dots,\mathcal{M}(\zeta)}$ decomposes as the disjoint union
	\begin{equation}\label{eq union disjointe M(zeta)}
	\ensemble{1,2,\dots,\mathcal{M}(\zeta)}=\mathcal{G}(\zeta)\cup\mathcal{I}(\zeta)\cup\P(\zeta)\cup\mathcal{O}(\zeta)\cup\mathcal{N}(\zeta),
	\end{equation}
	where the sets $\mathcal{G}(\zeta)$, $\mathcal{I}(\zeta)$, $\P(\zeta)$, $\mathcal{O}(\zeta)$ and $\mathcal{N}(\zeta)$ correspond to indexes $j$ such that respectively $\alpha_j(\zeta)$ is glancing, $\alpha_j(\zeta)$ is incoming, $\Im(\xi_j(\zeta))$ is positive, $\alpha_j(\zeta)$ is outgoing and $\Im(\xi_j(\zeta))$ is negative.
	
	Then the following decomposition of $E_-(\zeta)$ holds
	\begin{equation}\label{eq decomp E_-(zeta)}
	E_-(\zeta)=\bigoplus_{j\in\mathcal{G}(\zeta)}E^j_-(\zeta)\oplus \bigoplus_{j\in\mathcal{I}(\zeta)}E^j_-(\zeta) \oplus \bigoplus_{j\in\P(\zeta)}E^j_-(\zeta),
	\end{equation}
	where for each index $j$, the subspace $E_-^j(\zeta)$ is precisely described as follows.
	\begin{enumerate}[label=\roman*)]
		\item If $j\in \P(\zeta)$, the space $E^j_-(\zeta)$ is the generalized eigenspace $\mathcal{A}(\zeta)$ associated with the eigenvalue $i\,\xi_j(\zeta)$.
		\item If $j\in\mathcal{R}(\zeta)$, we have $E^j_-(\zeta)=\ker L\big(0,\alpha_j(\zeta)\big)$, 
		which is of dimension 1.
		\item If $j\in\mathcal{G}(\zeta)$, we denote by $n_j$ the algebraic multiplicity of the imaginary eigenvalue $i\xi_j(\zeta)$. For small positive $ \gamma $, the multiple eigenvalue $i\,\xi_j(\tau,\eta)$ splits into $n_j$ simple eigenvalues, denoted by $i\,\xi_j^k(\tau-i\gamma,\eta)$, $k=1,\dots,n_j$, all of nonzero real part. We denote by $\mu_j$ the number (independent of $\gamma>0$) of the eigenvalues $i\,\xi_j^k(\tau-i\gamma,\eta)$ of negative real part. Then $E_-^j(\zeta)$ is of dimension $ \mu_j $ and is generated by the vectors $w$ satisfying $[\mathcal{A}(\zeta)-i\xi_j(\zeta)]^{\mu_j}w=0$. Furthermore, if $n_j$ is even, $\mu_j=n_j/2$ and if $n_j$ is odd, $\mu_j$ is equal to $(n_j-1)/2$ or $(n_j+1)/2$.
	\end{enumerate}
	
	Likewise, the unstable subspace $E_+(\zeta)$ decomposes as
	\begin{equation}\label{eq decomp E_+(zeta)}
	E_+(\zeta)=\bigoplus_{j\in\mathcal{G}(\zeta)}E^j_+(\zeta)\oplus \bigoplus_{j\in\mathcal{O}(\zeta)}E^j_+(\zeta) \oplus \bigoplus_{j\in \mathcal{N}(\zeta)}E^j_+(\zeta),
	\end{equation}
	with similar description of the subspaces $E_+^j(\zeta)$. In particular, if the set $\mathcal{G}(\zeta)$ is empty, then
	\[\C^N=E_-(\zeta)\oplus E_+(\zeta).\]
\end{proposition}

\begin{remark}\label{remarque homogeneite xi, alpha etc}
	The notation $ \xi_j(\zeta) $ should not be taken for a function $ \xi_j $ depending on $ \zeta\in\Xi_0 $. Indeed for example the set $ \mathcal{M}(\zeta) $ depends on $ \zeta $. However, note that the matrix  $ \mathcal{A}(\zeta) $ is homogeneous of degree 1 with respect to $ \zeta $ in $ \Xi_0 $. Thus the number $ \mathcal{M}(\zeta) $ as well as the cardinality of the sets  $\mathcal{G}(\zeta)$, $\mathcal{I}(\zeta)\cup\mathcal{O}(\zeta)$ and $\P(\zeta)\cup\mathcal{N}(\zeta)$ depend only on the direction of $ \zeta $ in $ \Xi_0 $. We therefore assume that, $ \zeta\in\Xi_0 $ being fixed, for $ \lambda\in\R^* $, the indexes $ 1,\dots,\mathcal{M}(\lambda\,\zeta) $ are arranged in a way that, for $ j=1,\dots,\mathcal{M}(\zeta) $, we have
	\begin{equation*}
	\xi_j(\lambda\,\zeta)=\lambda\,\xi_j(\zeta) ,\quad \mbox{so that}\quad \alpha_j(\lambda\,\zeta)=\lambda\,\alpha_j(\zeta).
	\end{equation*}
	With this ordering, we note that if for $ \zeta\in\Xi_0 $, the frequency $ \alpha_{j}(\zeta) $ is glancing, incoming or outgoing (resp. $ \xi_j(\zeta) $ is of nonzero imaginary part), i.e. if $ j\in\mathcal{G}(\zeta)\cup\mathcal{O}(\zeta)\cup\mathcal{I}(\zeta) $ (resp. $ j\in\mathcal{P}(\zeta)\cup\mathcal{N}(\zeta) $), then for $ \lambda\in\R^* $, the frequency $ \alpha_{j}(\lambda\,\zeta)=\lambda\,\alpha_{j}(\zeta) $ is still glancing, incoming or outgoing (resp. $ \xi_j(\lambda\,\zeta) $ is still of nonzero imaginary part), that is to say $ j\in\mathcal{G}(\lambda\,\zeta)\cup\mathcal{O}(\lambda\,\zeta)\cup\mathcal{I}(\lambda\,\zeta) $ (resp. $ j\in\mathcal{P}(\lambda\,\zeta)\cup\mathcal{N}(\lambda\,\zeta) $). More precisely, if $ j\in\mathcal{P}(\zeta) $ (resp. $ \mathcal{N}(\zeta) $), then $ j\in\mathcal{N}(-\zeta) $ (resp. $ \mathcal{P}(-\zeta) $).
\end{remark}

\begin{definition}\label{def glancing}
	Consider $ \zeta\in\Xi_0 $. We say that $ \zeta $ is a glancing point and we denote $ \zeta\in\mathcal{G} $ if, with notations of Proposition \ref{prop decomp E_-}, there exists an index $ j $ between $ 1 $ and $ \mathcal{M}(\zeta) $ such that $ j\in\mathcal{G}(\zeta) $, in other words, if $ \zeta $ is such that there exists a real nonzero number $ \xi $ such that the frequency $ (\zeta,\xi) $ is characteristic and glancing.
\end{definition}

An assumption is now made, that helps to prove that the projectors associated with decomposition \eqref{eq decomp E_-(zeta)} are bounded uniformly with respect to $ \zeta $ in $ \Xi_0 $. This assumption has already been made in \cite{Sarason1965Hyperbolic,Williams1996Boundary}, and seems essential, see \cite{Williams2000Boundary}.

\begin{assumption}\label{hypothese mult vp}
	For all $k=1,\dots,N$, and for all $(\underline{\eta},\underline{\xi})\in\R^d\setminus\ensemble{0}$, we have
	\begin{equation*}
	\frac{\partial\tau_k}{\partial \xi}(\underline{\eta},\underline{\xi})=0\quad\Rightarrow\quad \frac{\partial^2\tau_k}{\partial \xi^2}(\underline{\eta},\underline{\xi})\neq 0.
	\end{equation*} 
\end{assumption}

\begin{remark}
	We will see during the proof of Proposition \ref{prop proj bornes} in appendix \ref{appendix preuve} that assumption \ref{hypothese mult vp} implies that for all $ \zeta $ in $\Xi_0$, for all index $j$ in $\mathcal{G}(\zeta)$, we have $n_j=2$, using the notations of Proposition \ref{prop decomp E_-}. We deduce that $\mu_j=1$ and that the component $E_-^j(\zeta)$ of the stable subspace $ E_-(\zeta) $ is of dimension 1 and given by $ \ker L\big(0,\alpha_{j}(\zeta)\big) $.
\end{remark}

\begin{definition}\label{def projecteurs Pi}
	For $\zeta\in\Xi_0$ and, using the notations of Proposition \ref{prop decomp E_-}, for an index $j$ in $\mathcal{G}(\zeta)\cup\mathcal{I}(\zeta)$, we denote by $\Pi^j_-(\zeta)$ the projection from $E_-(\zeta)$ on the component $E^j_-(\zeta)$ according to decomposition \eqref{eq decomp E_-(zeta)}. 

	We also denote by $ \Pi^{e}_-(\zeta) $ the projection from $E_-(\zeta)$ on the elliptic stable component $E_-^e(\zeta):=\oplus_{j\in\P(\zeta)}E^j_-(\zeta) $ according to decomposition \eqref{eq decomp E_-(zeta)}. 
	
	Finally, if $\zeta$ is not glancing, that is if the set $\mathcal{G}(\zeta)$ is empty, then according to Proposition \ref{prop decomp E_-} we have the following decomposition of $\C^N$
	\begin{equation}\label{eq decomp C^N E + E -}
	\C^N=
	\bigoplus_{j\in\mathcal{O}(\zeta)}E^j_+(\zeta) \oplus \bigoplus_{j\in \mathcal{N}(\zeta)}E^j_+(\zeta)\oplus
	 \bigoplus_{j\in\mathcal{I}(\zeta)}E^j_-(\zeta) \oplus \bigoplus_{j\in\P(\zeta)}E^j_-(\zeta).
	\end{equation}
	In that case we denote by $ \Pi^{e}_{\C^N}(\zeta) $ the projection from $ \C^N $ on the stable elliptic component $ E^e_-(\zeta) $ according to this decomposition, and by $ \Pi^{e,+}_{\C^N}(\zeta) $ the projection from $ \C^N $ on the unstable elliptic component $ E^e_+(\zeta):=\oplus_{j\in\mathcal{N}(\zeta)}E^j_+(\zeta) $ according to the same decomposition.
\end{definition}

The following proposition will be a key result in our analysis. It uses in a crucial way  Assumption \ref{hypothese mult vp}, as well as the strict hyperbolicity Assumption \ref{hypothese stricte hyp}. 

\begin{proposition}[{\cite{Williams1996Boundary}}]\label{prop proj bornes}
	Under Assumptions \ref{hypothese stricte hyp} and \ref{hypothese mult vp}, for $\zeta\in\Xi_0$, the projectors $\Pi^j_-(\zeta)$ for $j$ in $\mathcal{G}(\zeta)\cup\mathcal{I}(\zeta)$, and the projectors $\Pi^{e}_-(\zeta)$ are uniformly bounded with respect to $\zeta$ in $ \Xi_0 $. 
\end{proposition}

The proof of this result, omitted in \cite{Williams1996Boundary} and which requires some work, is postponed until Appendix \ref{appendix preuve}.

Thanks to Assumption \ref{hypothese mult vp} we are also able to prove the following result, which continues Lemma \ref{lemme pi tilde E vitesse de groupe}, and establishes a control from below over the normal component of the group velocity, and therefore over the vector $ \tilde{\pi}_{\alpha}\,E_k(\eta,\xi) $ for all $ \alpha=\big(\tau_k(\eta,\xi),\eta,\xi\big) $, involving the distance from $ \big(\tau_k(\eta,\xi),\eta\big) $ to the glancing set $ \mathcal{G} $. Its proof uses notations and results from the one of Proposition \ref{prop proj bornes}, and is therefore also skipped until Appendix \ref{appendix preuve}.

\begin{lemma}\label{lemme minoration vitesse de groupe}
There exists a positive constant $ C>0 $ such that, if the real frequency $ \alpha=(\tau,\eta,\xi) $ in $ \R^{1+d}\privede{0} $ is characteristic, and if $ k $ between 1 and $ N $ is such that $ \tau=\tau_k(\eta,\xi) $, then we have
\begin{equation*}
	\left|\partial_{\xi}\tau_k(\eta,\xi)\right|\geq C\frac{\dist\big((\tau,\eta),\mathcal{G}\big)^{1/2}}{|(\tau,\eta)|^{1/2}}.
\end{equation*}  
Using Lemma \ref{lemme pi tilde E vitesse de groupe}, we therefore obtain the following estimate
\begin{equation}\label{eq freq minoration pi tilde E alpha}
	\left|\tilde{\pi}_{\alpha}\,E_k(\eta,\xi)\right|\geq C\frac{\dist\big((\tau,\eta),\mathcal{G}\big)^{1/2}}{|(\tau,\eta)|^{1/2}}.
\end{equation}  
\end{lemma}

\section{Functional framework}

\subsection{Set of frequencies inside the domain}

To define the functional framework that will be used, we need first to determine a priori which frequencies may appear in the solution to \eqref{eq systeme 1}. For a detailed discussion of this analysis, reference is made to {\cite[Chapters 9 and 10]{Rauch2012Hyperbolic}} and \cite{MajdaArtola1988Mixed}. The presence on the boundary of the frequencies $ \zeta_1,\dots,\zeta_m $ creates, by nonlinear interaction, the following group of frequencies on the boundary
\begin{equation}\label{eq def frequences au bord}
\F_b:=\zeta_1\,\Z+\cdots+\zeta_m\,\Z\subset\R^d.
\end{equation}
The assumption is now made that this group does not contain any glancing point, which have been introduced in Definition \ref{def glancing}. This assumption is often made, and allows to avoid complications created by the glancing modes, see e.g. \cite{CoulombelGues2010Geometric,CoulombelGuesWilliams2011Resonant}.

\begin{assumption}\label{hypothese pas de glancing}
	We have
	\begin{equation*}
	\big(\F_b\privede{0}\big)\cap \mathcal{G}=\emptyset.
	\end{equation*}
	In other words, with the notations of Proposition \ref{prop decomp E_-}, for all $ \zeta\in\F_b\privede{0} $, the set $ \mathcal{G}(\zeta) $ of indexes $ j $ between $ 1 $ and $ \mathcal{M}(\zeta) $ such that the characteristic frequency $ \big(\zeta,\xi_j(\zeta)\big) $ is glancing, is an empty set.
\end{assumption}

However, attention must be paid on the fact that despite Assumption \ref{hypothese pas de glancing}, the set $ \F_b\privede{0} $ may contain frequencies arbitrary close to the set of glancing frequencies $ \mathcal{G} $, namely frequencies $ \zeta $ admitting a lifting inside the domain $ \alpha=\big(\tau_k(\eta,\xi),\eta,\xi\big) $ of which the normal component of the group velocity given by $ -\partial_{\xi}\tau_k(\eta,\xi) $ is arbitrary close to zero. This phenomenon is well illustrated in Example \ref{exemple Euler frequences au bord} of compressible isentropic Euler equations below.
In the following, we will need a control on the projectors $ \Pi^{e}_{\C^N}(\zeta) $ for $ \zeta $ in the group $ \F_b\privede{0} $, defined for $ \zeta $ non glancing. Indeed the norm of this projector increases when $ \zeta $ gets close to the glancing set $ \G $. This is why a small divisor assumption is now made, that gives a control over the distance between $ \zeta $ in $ \F_b $ and $ \G $ for large $ \zeta $, notably leading to Proposition \ref{prop controle exp t A Pi } below.

\begin{assumption}\label{hypothese petits diviseurs 1}
	There exists a real number $ a_1 $ and a positive constant $ c $ such that for all $ \zeta $ in $ \F_b\privede{0} $, we have
	\begin{equation*}
	\dist \big(\zeta,\G\big)\geq c\,|\zeta|^{-a_1}.
	\end{equation*}
\end{assumption}

Note that Assumption \ref{hypothese pas de glancing} is a consequence of Assumption \ref{hypothese petits diviseurs 1}, so could be omitted. However we have chosen to keep both assumptions because they play two different roles in the proofs. Small divisors Assumption \ref{hypothese petits diviseurs 1} is quite unusual, and plays a technical role in the proofs.

The operator $ L(0,\partial_{z}) $ being hyperbolic, the frequencies on the boundary $ \zeta\in\F_b\privede{0} $ are then lifted inside the domain into frequencies $ (\zeta,\xi) $. We will see that the polarization conditions for the leading profile cancel the modes associated with noncharacteristic frequencies. Therefore, since we are interested in bounded solutions, at this point only the incoming and evanescent characteristic frequencies lifted from frequencies on the boundary are created. Assumption \ref{hypothese pas de glancing} is used here to exclude the possibility of creating glancing frequencies $ \big(\zeta,\xi_{j}(\zeta)\big) $, that is with $ j\in\mathcal{G}(\zeta) $. Thus, at this stage, the set of frequencies $ \ensemble{0}\cup\F^{\inc}\cup\F^{\ev} $ has been obtained for the leading profile, where the sets $ \F^{\inc} $ and $ \F^{\ev} $ are given by
\begin{equation}\label{eq freq def F inc ev}
\mathcal{F}^{\inc}:=\Big\{\big(\zeta,\xi_j(\zeta)\big),\ \zeta\in \mathcal{F}_b\setminus\ensemble{0},\  j\in\mathcal{I}(\zeta)\Big\},\quad 
\F^{\ev}:=\Big\{\big(\zeta,\xi_j(\zeta)\big),\ \zeta\in \mathcal{F}_b\setminus\ensemble{0},\  j\in\P(\zeta)\Big\}.
\end{equation}
Apart from exceptional cases, the set $ \F^{\inc} $ is not finitely generated, which imposes an almost-periodic framework for the normal fast variable.

Interest is now made on resonances that may occur inside the domain. By nonlinear interaction, two frequencies $ \alpha_{j_p}(\zeta_p)=\big(\zeta_p,\xi_{j_p}(\zeta_p)\big) $ and $ \alpha_{j_q}(\zeta_q)=\big(\zeta_q,\xi_{j_q}(\zeta_q)\big) $ of $ \F^{\inc} $ may resonate to create a characteristic frequency $\alpha_{j_r}(\zeta_r)=\big(\zeta_r,\xi_{j_r}(\zeta_r)\big)  $ in the following way:
\begin{equation*}
n_p\,\alpha_{j_p}(\zeta_p)+n_q\,\alpha_{j_q}(\zeta_q)=n_r\,\alpha_{j_r}(\zeta_r),\quad n_p,n_q,n_r\in\Z\privede{0}.
\end{equation*}
If the index $ j_r $ belongs to the set $ \mathcal{O}(\zeta_r) $, that is to say if $ \alpha_{j_r}(\zeta_r) $ is an outgoing real characteristic frequency, a new frequency inside the domain is thus created, which does not already belong to the initial set $ \F^{\inc} $ defined above. The simplifying assumption that it does not occur is made, so there is no outgoing characteristic frequency created through a resonant triplet. More precisely, we assume that the outgoing and the incoming frequencies do not resonate one with the other. The set of outgoing frequencies $ \F^{\out} $ is defined as
\begin{equation}\label{eq freq def F out}
\mathcal{F}^{\out}:=\Big\{\big(\zeta,\xi_j(\zeta)\big),\ \zeta\in \mathcal{F}_b\setminus\ensemble{0},\  j\in\mathcal{O}(\zeta)\Big\}.
\end{equation}

\begin{assumption}\label{hypothese pas de sortant}
	\begin{enumerate}[label=\roman*), leftmargin=0.7cm]
		\item There does not exist a couple $ (\alpha_p,\alpha_q) $ of incoming characteristic frequencies $ \F^{\inc} $ and a couple of integers $ (n_p,n_q) $ such that the frequency
		\begin{equation*}
		n_p\,\alpha_p+n_q\,\alpha_q
		\end{equation*}
		is real, characteristic and outgoing.
		\item There does not exist a couple $ (\alpha_p,\alpha_q) $ of outgoing characteristic frequencies $ \F^{\out} $ and a couple of integers $ (n_p,n_q) $ such that the frequency
		\begin{equation*}
			n_p\,\alpha_p+n_q\,\alpha_q
		\end{equation*}
		is real, characteristic and incoming.
	\end{enumerate}
\end{assumption}

This is a strong assumption, that, up to our knowledge, cannot be found in the literature. It implies that there is no outgoing mode in the leading profile of the expansion. For example in \cite{CoulombelGuesWilliams2011Resonant}, there is no such assumption, but since there is only one phase on the boundary, there are a finite number of resonances, conversely to our case. When there are outgoing modes, they are coupled with incoming ones through the trace on the boundary. This is an issue since it seems that the suitable functional framework for outgoing modes, tailored for evolution in time, is different from the one for incoming modes, adapted to propagation in normal direction. It also complicates the iterative process used to construct a solution. We have chosen here to focus on the construction of a functional framework for incoming modes that allows to solve the problem.

Note that if three real characteristic frequencies  $ \alpha_p,\alpha_q,\alpha_r $ resonate as
\begin{equation*}
n_p\,\alpha_p+n_q\,\alpha_q=n_r\,\alpha_r,\quad n_p,n_q,n_r\in\Z,
\end{equation*}
then according to the previous assumption, the frequencies $ \alpha_p,\alpha_q,\alpha_r $ are either all incoming or all outgoing. On an other hand, despite Assumption \ref{hypothese pas de sortant}, there may exist a countable infinite number of resonances between incoming frequencies, as it is the case in Example \ref{exemple Euler frequences au bord} of compressible isentropic Euler equations in dimension 2.

At this stage, for a new frequency to be created from $ \F^{\inc} $, there must exist a resonance between two frequencies of $ \F^{\inc} $, that creates a real characteristic frequency which does not already belong to $ \F^{\inc} $. The frequencies in $ \F^{\inc} $ are incoming, and according to Assumption \ref{hypothese pas de sortant} above, a resonance between two incoming frequencies may only produce an incoming frequency, which already belongs to $ \F^{\inc} $. There is therefore no new frequency created, and the final set of frequencies inside the domain created by nonlinear interaction on the boundary and lifting inside the domain is given by
\begin{equation}
\mathcal{F}:=\ensemble{0}\cup\Big\{\alpha_j(\zeta),\ \zeta\in\mathcal{F}_b\setminus\ensemble{0},\  j\in\mathcal{I}(\zeta)\cup\P(\zeta)\Big\}.
\end{equation}

We expect for the leading profile of the solution to \eqref{eq systeme 1} to feature all frequencies in  $ \mathcal{F} $ created by lifting. It leads to consider, to maintain generality, all frequencies in $\mathcal{F}$. Yet it seems unlikely that the group generated by $ \mathcal{F} $ may be finitely generated, which a priori excludes an asymptotic expansion of the solution $ u^{\epsilon} $ in the form of quasi-periodic functions. Following \cite{JolyMetivierRauch1995Coherent} and \cite{CoulombelGuesWilliams2011Resonant} after,  a quasi-periodic framework is nevertheless considered for the tangential fast variables (the group of frequencies on the boundary being finitely generated), but an almost-periodic framework for the normal fast variable is considered. The next subsection is devoted to that question and describes the functional framework used in this analysis. This part is ended by verifying the different assumptions and assertions made in this subsection for Example \ref{exemple Euler 1} of compressible isentropic Euler equations in dimension 2.

\begin{example}\label{exemple Euler frequences au bord}
	The notations of Example \ref{exemple Euler 1} and those after are used. Assumptions \ref{hypothese pas de glancing},\ref{hypothese petits diviseurs 1} and \ref{hypothese pas de sortant} concern the group of frequencies on the boundary $ \F_b $, thus adequate frequencies on the boundary must be considered for Example \ref{exemple Euler 1}. To simplify the calculations, we take two frequencies $  \zeta^{1} $ and $ \zeta^{\delta} $ given by $ \zeta^1:=(c_0\,\eta_0,\eta_0) $ and $ \zeta^{\delta}:=( c_0\,\delta\,\eta_0,\eta_0) $, with $ \eta_0>0 $ and $ \delta $ an irrational number strictly larger than 1, so that $ \zeta^1 $ and $ \zeta^{\delta} $ are both in the hyperbolic region $ \mathcal{H} $. Recall that $ c_0=c(v_0)>0 $ refers to the sound velocity and that the equilibrium $ V_0 = (v_0,0,u_0) $ satisfies $ 0<u_0<c_0 $. The boundary frequencies lattice $ \F_b $ is therefore given in this example by
	\begin{equation*}
	\F_b=\ensemble{\big(c_0\,\eta_0\,(p+\delta\,q),\eta_0\,(p+q)\big)\mid p,q\in\Z}\subset\R^2.
	\end{equation*}
	We denote by $ \zeta_{p,q}:=(\tau_{p,q},\eta_{p,q}):=\big(c_0\,\eta_0\,(p+\delta q),\eta_0\,(p+q)\big) $ the frequency of $ \F_b $ given by $ p\,\zeta^1+q\,\zeta^{\delta} $, for $ p,q $ in $ \Z $.
	
	A nonzero frequency $ \zeta_{p,q} $ is glancing if and only if $ |\tau_{p,q}|=\sqrt{c_0^2-u_0^2}\,|\eta_{p,q}| $, that is to say if and only if \footnote{If $ q=0 $, then according to the relation $ |\tau_{p,q}|=\sqrt{c_0^2-u_0^2}\,|\eta_{p,q}| $, we have $ p^2=(1-u_0^2/c_0^2)\,p^2 $ so $ p $ is also zero, i.e. the frequency $ \zeta_{p,q} $ is zero, which is excluded by assumption.}
	\begin{equation}\label{eq ex relation glancing}
	\frac{p}{q}\in\ensemble{\frac{\sqrt{1-M^2}-\delta}{1-\sqrt{1-M^2}},\frac{-\sqrt{1-M^2}-\delta}{1+\sqrt{1-M^2}}}.
	\end{equation}
	We have denoted by $ M $ the Mach number given by $ M:=u_0/c_0 $, belonging to $ (0,1) $. If the two real numbers
	\begin{equation}\label{eq ex def reels irrationnels 1- M 2}
		K_-:=\frac{\sqrt{1-M^2}-\delta}{1-\sqrt{1-M^2}}\quad \text{et}\quad K_+:=\frac{-\sqrt{1-M^2}-\delta}{1+\sqrt{1-M^2}}
	\end{equation}
	are irrational, then there does not exist a relation of the form \eqref{eq ex relation glancing}, and there is therefore no glancing frequency in the group $ \F_b\privede{0} $. Thus we make the assumption that $ K_- $ and $ K_+ $ are irrational, so that the assumption \ref{hypothese pas de glancing} is verified. We may for example take $ M=\sqrt{3}/2 $ and $ \delta>1 $ irrational. We summarize now the different areas where the frequencies $ \zeta_{p,q} $ may be, depending on $ p $ and $ q $. First note that we have $ 0>K_+>K_->-\delta $. Recall that the frequency $ \zeta_{p,q} $ is in the hyperbolic region (resp. mixed region) if and only if $ |\tau_{p,q}|>\sqrt{c_0^2-u_0^2}\,|\eta_{p,q}| $ (resp. $ |\tau_{p,q}|<\sqrt{c_0^2-u_0^2}\,|\eta_{p,q}| $). We thus infer the classification given in Figure \ref{figure tableau}. The calculation steps are not detailed, but one case is treated in more details below.
	
	\begin{figure}
		\centering
		\begin{tabular}{|c || M{4.5cm} | M{4.5cm} | M{4.5cm} |}
			\hline & $ p/q>K_+ $ & $ K_+>p/q > -\delta  $ & $ p/q < -\delta  $  \\\hline\hline
		$ q > 0 $ & $ \zeta_{p,q} $ is in the hyperbolic region $ \mathcal{H} $ with $ \tau_{p,q}>0 $ & $ \zeta_{p,q} $ is in the mixed region $ \mathcal{EH} $ & $ \zeta_{p,q} $ is in the hyperbolic region $ \mathcal{H} $ with $ \tau_{p,q}<0 $\\
		\hline
		$ q < 0 $ & $ \zeta_{p,q} $ is in the hyperbolic region $ \mathcal{H} $ with $ \tau_{p,q}<0 $ & $ \zeta_{p,q} $ is in the mixed region $ \mathcal{EH} $ & $ \zeta_{p,q} $ is in the hyperbolic region $ \mathcal{H} $ with $ \tau_{p,q}>0 $\\ \hline
	\end{tabular}
\caption{Position of $ \zeta_{p,q} $ depending on $ p $ and $ q $.}
\label{figure tableau}
	\end{figure}

\bigskip 
	
	The remark concerning the group velocities that follows Assumption \ref{hypothese pas de glancing} is now illustrated. For that purpose a sequence of frequencies on the boundary that draw near the glancing set is considered, see Figure \ref{figure exemple Euler 2}. The real number $ K_+ $ given \eqref{eq ex def reels irrationnels 1- M 2} being irrational, there exist two sequences $ (p_k)_{k} $ and $ (q_k)_{k} $ of integers such that for $ k\geq 0 $, $ p_k $ and $ q_k $ are coprime, $ q_k>0 $, and such that
	\begin{equation*}
	\frac{p_k}{q_k}\xrightarrow[k\rightarrow+\infty]{>}\frac{-\sqrt{1-M^2}-\delta}{1+\sqrt{1-M^2}}=K_+.
	\end{equation*}
	Note that since $ K_+>K_->-\delta $, for all $ k\geq 0 $, we have $ p_k/q_k>K_->-\delta $, so on one hand we have $ p_k+\delta q_k>0 $, and on the other hand, independently of the sign of $ p_k+q_k $, we have $ p_k+\delta q_k>\sqrt{1-M^2}\,|p_k+q_k| $. The frequencies $ \zeta_{p_k,q_k} $ are therefore in the hyperbolic region $ \mathcal{H} $ with $ \tau>0 $, and draw near the glancing region $ \mathcal{G} $, see Figure \ref{figure exemple Euler 2}. Since $ \tau_{p_k,q_k}>0 $, according to Example \ref{exemple Euler UKL}, the last component of the group velocity associated with the frequency on the inside $ \alpha_1(\tau_{p_k,q_k},\eta_{p_k,q_k}) $ is given by
	\begin{align}
	\partial_{\xi}\tau_3\big(&\eta_{p_k,q_k},\xi_1(\tau_{p_k,q_k},\eta_{p_k,q_k})\big)\label{eq ex vitesse de groupe normale}
	\leq\sqrt{\frac{\tau_{p_k,q_k}^2-\eta_{p_k,q_k}^2\,(c_0^2-u_0^2)}{\eta_{p_k,q_k}^2}}\\
	&= c_0\,\sqrt{q_k\left(1+\sqrt{1-M^2}\right)\frac{p_k+\delta q_k-(p_k+q_k)(1-M^2)} {(p_k+q_k)^2}}\sqrt{\frac{p_k}{q_k}-\frac{-\sqrt{1-M^2}-\delta}{1+\sqrt{1-M^2}}}\label{eq ex quant maj vitesse de groupe}.
	\end{align}
	Note that the quantity under the first square root sign of \eqref{eq ex quant maj vitesse de groupe} is non negative, since for all $ k\geq 0 $, $ q_k>0 $ and $ p_k/q_k>K_- $, so that  $ p_k+\delta q_k-(p_k+q_k)(1-M^2)>0 $. This quantity being bounded, the quantity \eqref{eq ex quant maj vitesse de groupe} converges towards zero by construction of the integers $ (p_k,q_k) $.
	We see that the normal  group velocity \eqref{eq ex vitesse de groupe normale} of the hyperbolic frequencies $ \zeta_{p_k,q_k} $ converges towards zero as $ k $ goes to infinity. It shows that the normal group velocity of the hyperbolic frequencies, although nonzero, may be arbitrary close to zero.
	
	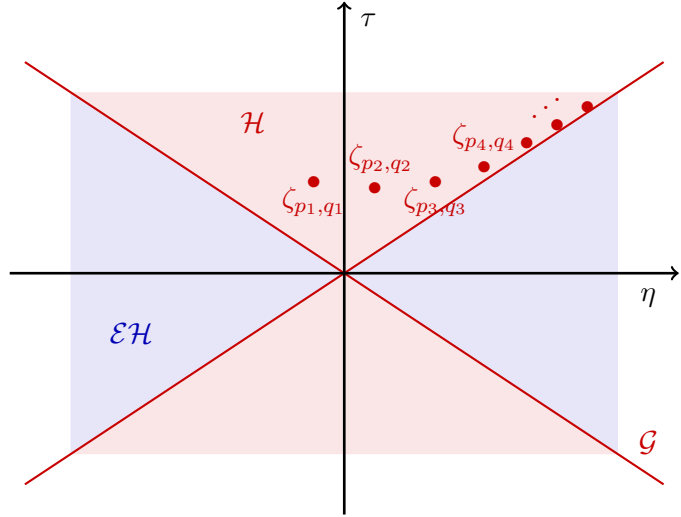
\begin{figure}[!hb]
		\centering
		\begin{tikzpicture}[scale=0.8]
		\fill[altred!10] (-4.5,3) -- (0,0) -- (4.5,3) --cycle;
		\fill[altred!10] (-4.5,-3) -- (0,0) -- (4.5,-3) --cycle;
		\fill[altblue!10] (-4.5,-3) -- (0,0) -- (-4.5,3) --cycle;
		\fill[altblue!10] (4.5,-3) -- (0,0) -- (4.5,3) --cycle;
		\draw[line width = 0.8pt, altred] (-5.25,3.5) -- (5.25,-3.5);
		\draw[line width = 0.8pt, altred] (-5.25,-3.5) -- (5.25,3.5);
		\draw[line width = 1pt,->] (-5.5,0) -- (5.5,0);
		\draw[line width = 1pt,->] (0,-4) -- (0,4.5);
		\draw (5,-0.4) node{$ \eta $};
		\draw (0.4,4.2) node{$ \tau $};
		\draw[altred] (5,-2.8) node{$ \mathcal{G} $};
		\draw[altred] (-1.5,2.5) node{$ \mathcal{H} $};
		\draw[altblue] (-3.5,-1) node{$ \mathcal{EH} $};
		\draw[altred] (-0.5,1.5) node{$ \bullet $};
		\draw[altred] (0.5,1.4) node{$ \bullet $};
		\draw[altred] (1.5,1.5) node{$ \bullet $};
		\draw[altred] (2.3,1.75) node{$ \bullet $};
		\draw[altred] (3,2.15) node{$ \bullet $};
		\draw[altred] (3.5,2.45) node{$ \bullet $};
		\draw[altred] (4,2.75) node{$ \bullet $};
		\draw[altred] (-0.5,1.6) node[below]{$ \zeta_{p_1,q_1} $};
		\draw[altred] (0.6,1.5) node[above]{$  \zeta_{p_2,q_2}  $};
		\draw[altred] (1.5,1.6) node[below]{$ \zeta_{p_3,q_3} $};
		\draw[altred] (2.3,1.85) node[above]{$ \zeta_{p_4,q_4} $};
		\draw[altred] (3.3,2.35) node[above]{$ \iddots $};
		%	\fill[red!10] (-4,5) -- (0,0) -- (4,5) --cycle;
		%	\fill[red!10] (-4,-5) -- (0,0) -- (4,-5) --cycle;
		%	\fill[blue!10] (-4,-5) -- (0,0) -- (-4,5) --cycle;
		%	\fill[blue!10] (4,-5) -- (0,0) -- (4,5) --cycle;
		%	\draw[line width = 0.8pt, red] (-4,5) -- (4,-5);
		%	\draw[line width = 0.8pt, red] (-4,-5) -- (4,5);
		%	\draw[line width = 1pt,->] (-4.5,0) -- (4.5,0);
		%	\draw[line width = 1pt,->] (0,-5.5) -- (0,6);
		%	\draw (4.5,-0.4) node{$ \eta $};
		%	\draw (0.4,5.7) node{$ \tau $};
		%	\draw[red] (2.8,-2.8) node{$ \mathcal{G} $};
		%	\draw[red] (0.8,4) node{$ \mathcal{H} $};
		%	\draw[blue] (-3,1) node{$ \mathcal{EH} $};
		%	\draw[red] (-0.5,2.9) node{$ \bullet $};
		%	\draw[red] (-1.3,3.1) node{$ \bullet $};
		%	\draw[red] (-2.05,3.4) node{$ \bullet $};
		%	\draw[red] (-2.6,3.8) node{$ \bullet $};
		%	\draw[red] (-3.1,4.2) node{$ \bullet $};
		%	\draw[red] (-3.5,4.6) node{$ \bullet $};
		%	\draw[red] (-3.8,4.9) node{$ \bullet $};
		%	\draw[red] (-0.5,3.1) node[below]{$ \zeta_{p_1,q_1} $};
		%	\draw[red] (-1.1,3.1) node[above]{$  \zeta_{p_2,q_2}  $};
		%	\draw[red] (-1.8,3.7) node[below left]{$ \zeta_{p_3,q_3} $};
		%	\draw[red] (-2.7,3.7) node[above right]{$ \zeta_{p_4,q_4} $};
		%	\draw[red] (-2.9,4.3) node[above]{$ \ddots $};
		\end{tikzpicture}
		\caption{Sequence of frequencies that draw near the glancing region, of which the normal group velocity goes to zero.}
		\label{figure exemple Euler 2}
	\end{figure}
	
	Interest is now made on Assumption \ref{hypothese pas de sortant} and into the  resonances between real characteristic frequencies. We recall the notations of Example \ref{exemple Euler UKL}, and we first determine that, in the hyperbolic region, the eigenvalues $ i\,\xi_1(\zeta) $, $ i\,\xi_2(\zeta) $ and $ i\,\xi_3(\zeta) $, defined by \eqref{eq ex Euler v.p. hyp}, are given, for $ \zeta=\zeta_{p,q} $ in $ \F_b\privede{0} $, by
	\begin{align*}
		i\,\xi_1(\zeta_{p,q})&=i\,\eta_0\,\frac{M\,(p+\delta q)+\sign(p+\delta q)\,\sqrt{(p+q)^2\,M^2+2pq(\delta-1)+q^2(\delta^2-1)}}{1-M^2},\\[5pt]
		i\,\xi_2(\zeta_{p,q})&=i\,\eta_0\,\frac{M\,(p+\delta q)-\sign(p+\delta q)\,\sqrt{(p+q)^2\,M^2+2pq(\delta-1)+q^2(\delta^2-1)}}{1-M^2},\\[5pt]
		i\,\xi_3(\zeta_{p,q})&=-i\,\eta_0\,\frac{p+\delta q}{M}.
	\end{align*}
	The case of the glancing region is excluded by assumption, and the one of the mixed region is included in the following, considering there is in this case only one imaginary eigenvalue, which is the linear  eigenvalue $ i\,\xi_3(\zeta) $.
	We first observe that, the eigenvalue $ i\,\xi_3(\zeta_{p,q}) $ begin linear, it generates resonances of the form
	\begin{equation*}
	\alpha_3(\zeta_{p,q})+\alpha_3(\zeta_{r,s})=\alpha_3(\zeta_{p+r,q+s}), \quad \forall p,q,r,s\in\Z.
	\end{equation*} 
	The frequency $ \alpha_3(\zeta_{p,q}) $ being always incoming, there are therefore already an infinite number of resonances between incoming frequencies. From now on the notation $ \alpha_{1,2} $ refers to one of the characteristic frequency $ \alpha_1 $ or $ \alpha_2 $. Since by linearity of $ \alpha_3 $, the resonance between two frequencies $ \alpha_3 $ and a frequency $ \alpha_{1,2} $ is impossible, the two following cases of resonance are still to be investigated:
	\begin{equation*}
	\alpha_{1,2}(\zeta_{p,q})+\alpha_{1,2}(\zeta_{r,s})=\alpha_3(\zeta_{p+r,q+s})\quad \text{et} \quad \alpha_{1,2}(\zeta_{p,q})+\alpha_{1,2}(\zeta_{r,s})=\alpha_{1,2}(\zeta_{p+r,q+s}).
	\end{equation*} 
	In the first case, it is equivalent to the relation
	\begin{multline}\label{eq ex equation resonance 1}
	\big[(p+q)^2-(r+s)^2\big]\,M^8+2\big[(p+q)^2-(r+s)^2\big]\big(2(rs-pq)(\delta-1)+(s^2-q^2)(\delta^2-1)\big)\,M^6\\
	+C_4(p,q,r,s,\delta)\,M^4
	+C_2(p,q,r,s,\delta)\,M^2+(p+r+\delta q+\delta s)^2=0
	\end{multline}
	where coefficients $ C_4(p,q,r,s,\delta) $  and $ C_2(p,q,r,s,\delta) $ are polynomial in their variables. Two cases may now occur, depending on whether the coefficient in front of $ M^8 $ in equation \eqref{eq ex equation resonance 1} is zero or not.
	\begin{enumerate}[label=\roman*), leftmargin=0.7cm]
		\item Either we have $ (p+q)^2-(r+s)^2\neq 0 $, in which case equation \eqref{eq ex equation resonance 1} is a polynomial equation of degree 4 in $ \mathbb{Q}[\delta] $ satisfied by $ M^2 $.
		\item Or we have $ (p+q)^2=(r+s)^2 $, and in this case \eqref{eq ex equation resonance 1} is a polynomial equation of degree at most 2 satisfied by $ M^2 $ in $ \mathbb{Q}[\delta] $. If once again the coefficients $ C_4(p,q,r,s,\delta) $  and $ C_2(p,q,r,s,\delta) $ in front of $ M^4 $ and $ M^2 $ are zero, then we get $ p+r+\delta q+\delta s=0 $. Therefore we have $ r=-p $ and $ s=-q $, that is to say $ \zeta_{r,s}=-\zeta_{p,q} $, so the studied resonance is actually self-interaction of $ \zeta_{p,q} $ with itself to generate the zero frequency. Thus, if $ (p+q)^2=(r+s)^2 $, the only cases of a real resonance are those where $ M^2 $ is a root of a polynomial of degree 1 or 2 in $ \mathbb{Q}[\delta] $.
	\end{enumerate}
	It has therefore been determined that for a resonance of the first type to occur (not of self-interaction type), then $ M^2 $ needs to be a root of a polynomial of degree at most 4 in $ \mathbb{Q}[\delta] $.

	For the second type of resonance, such a relation is verified if and if only if the following relation holds
	\begin{equation}\label{eq ex equation resonance 2}
	(ps-qr)^2(1-M^2)(\delta-1)^2=0,
	\end{equation}
	that is to say, since we have $ 0<M<1 $ and $ \delta>1 $, if and only if the two frequencies $ \zeta_{p,q} $ and $ \zeta_{r,s} $ are collinear. Then one may write $ \zeta_{p,q}=\lambda\,\zeta_{t,w} $ and $ \zeta_{r,s}=\mu\,\zeta_{t,w} $, with $ \zeta_{t,w} $ a hyperbolic frequency, and $ \lambda,\mu $ in $ \Z^* $. Next verify that, since we have $ \sign(\lambda\,(t+\delta w))=\sign(\lambda)\,\sign(t+\delta w) $, the two following relations hold $ \xi_1(\lambda\,\zeta_{t,w})=\lambda\,\xi_1(\zeta_{t,w})  $ and $ \xi_2(\lambda\,\zeta_{t,w})=\lambda\,\xi_2(\zeta_{t,w})  $. The same holds for the two frequencies $ \mu\,\zeta_{t,w} $ and $ (\lambda+\mu)\,\zeta_{t,w} $. The only two resonances that may occur are therefore
	\begin{equation*}
	\lambda\,\alpha_{1}(\zeta_{t,w})+\mu\,\alpha_{1}(\zeta_{t,w})=(\lambda+\mu)\,\alpha_1(\zeta_{t,w})\quad \text{et} \quad \lambda\,\alpha_{2}(\zeta_{t,w})+\mu\,\alpha_{2}(\zeta_{t,w})=(\lambda+\mu)\,\alpha_{2}(\zeta_{t,w}),
	\end{equation*} 
	which both are actually self-interaction of frequencies $ \alpha_{1}(\zeta_{t,w}) $ and $ \alpha_{2}(\zeta_{t,w}) $ with themselves: the evolution of the harmonics $ \lambda $ and $ \mu $ are coupled with the one of $ \lambda+\mu $. In particular, the three frequencies implied in this resonance are either all incoming or all outgoing. Thus, if $ M^2 $ is not a root of a polynomial of degree at most $ 4 $ in $ \mathbb{Q}[\delta] $, Assumption \ref{hypothese pas de sortant} is verified for the compressible isentropic Euler equations in dimension 2, with the group of frequencies on the boundary that has been considered.
	
	We finally dig into the small divisors Assumption \ref{hypothese petits diviseurs 1}. One can check that\footnote{Using the fact that the glancing set is constituted here of two lines, an elementary geometrical argument allows to reduce to the distance with respect to $ \eta $ only, where the constant $ C $ is given by 
		\begin{equation*}
			C:=\frac{\eta_0\,\sin \arctan (\sqrt{c_0^2-u_0^2})}{\sqrt{1-M^2}}.
	\end{equation*}}, depending on the sign of $ p,q $, the distance between $ \zeta_{p,q} $ and the glancing set $ \mathcal{G} $ is given by
	\begin{align*}
	\dist(\zeta_{p,q},\mathcal{G})&=C\left|(p+\delta q)\pm\sqrt{1-M^2}(p+q)\right|=Cq\big(1\pm\sqrt{1-M^2}\big)\left|\frac{p}{q}-K_{\pm}\right|.
	\end{align*}
	If $ p $ and $ q $ are not of the same size scale, then the same holds for $ p+\delta q $ and $ p+q $, so the previous distance can be lower bounded by a positive constant. We thus may in the following assume that 
	\begin{equation}\label{eq ex equivalence p and q}
		C_1|p|\leq |q| \leq C_2 |p|. 
	\end{equation}
	According to Roth theorem, see \cite[Theorem 2A]{Schmidt1991Diophantine}, if the real numbers $ K_+ $ and $ K_- $ given by \eqref{eq ex def reels irrationnels 1- M 2} are algebraic numbers (and irrational, which has been previously assumed), then they satisfy
	\begin{equation*}
	\left|K_{\pm}-\frac{p}{q}\right|\geq C\,|q|^{-(2+\epsilon)},
	\end{equation*}
	for all $ q\in\mathbb{N}^* $, $ p\in\Z $ and $ \epsilon>0 $. 
	So for all $ \zeta_{p,q} $ in $ \F_b\privede{0} $, we get, using \eqref{eq ex equivalence p and q},
	\begin{equation*}
	 \dist(\zeta_{p,q},\mathcal{G})\geq C|q|^{-3/2}\geq C|\zeta_{p,q}|^{-3/2},
	\end{equation*}
	and Assumption \ref{hypothese petits diviseurs 1} is therefore verified.
	
	In conclusion, for the compressible isentropic Euler equations in dimension 2 to satisfies Assumptions \ref{hypothese pas de glancing}, \ref{hypothese petits diviseurs 1} and \ref{hypothese pas de sortant}, it is therefore sufficient that the Mach number $ M $ and the parameter $ \delta>1 $ are such that $ K_+ $ and $ K_- $ are irrational algebraic numbers and that $ M^2 $ is not a polynomial solution of degree at most 4 in $ \mathbb{Q}[\delta] $. The set of solutions to such equations being countable, one may convince himself that the set of real numbers $ M $ satisfying these properties is not empty.
	One may for example choose $ M:=\sqrt{3}/2 $, which gives $ K_-=1-2\delta $ and $ K_+=-(1+2\delta)/3 $, and also choose $ \delta=\sqrt[7]{2} >1 $. In this way the real numbers $ K_+ $ and $ K_- $ are actually irrational algebraic numbers. On an other hand, a relation of the form \eqref{eq ex equation resonance 1} cannot be satisfied, except for trivial cases on $ p,q,r,s $, because one can check that it forms an algebraic equation of degree 4 in $ \delta $, which is not an algebraic number of degree 4.
\end{example}

\subsection{Spaces of profiles} 

According to the analysis of which frequencies may appear in the solution to \eqref{eq systeme 1}, we must define for the leading profile a functional framework that allows to consider superposition of waves of the form
\begin{equation*}\label{eq espaces forme onde}
a(z)\,e^{i\,\alpha\cdot z/\epsilon},
\end{equation*}
for $ \alpha $ a characteristic frequency in $ \F $.
Since we choose a quasi-periodic framework for the fast tangential variables, we write such a wave as
\begin{equation*}
a(z)\,e^{i\,n_1\,\zeta_1\cdot z'/\epsilon}\cdots e^{i\,n_m\,\zeta_m\cdot z'/\epsilon}\,e^{i\,\xi\,x_d/\epsilon},
\end{equation*}
where $ \alpha=(\zeta,\xi) $ with $ \zeta=n_1\,\zeta_1+\cdots+n_m\,\zeta_m\in\F_b $ and $ \xi\in\C $.
Next we denote by $ \theta=(\theta_1,\dots,\theta_m)\in\T^m $ the fast tangential variables which substitute to $ (z'\cdot\zeta_1/\epsilon,\dots,z'\cdot\zeta_m/\epsilon) $ and $ \psi_d\in\R_+ $ the fast normal variable substituting to $ x_d/\epsilon $.

For each integer $ s\geq 0 $ and for $ T>0 $, we denote by $H^s_+(\omega_T\times\T^m)$ the space of functions of $ (t,y,\theta)\in\omega_T\times\T^m $, zero for negative times $t$, of which all derivatives of order less or equal to $ s $ belong to $L^2(\omega_T\times\T^m)$. 

Now we describe the general space that will contain the oscillating and evanescent profiles spaces. We choose only a uniform control with respect to the fast and slow normal variables since it would be difficult to control derivatives of the leading profile with respect to these two variables.

\begin{definition}\label{def espace E}
	For an integer $ s\geq 0 $ and for $ T>0 $, we define the space $ \mathcal{E}_{s,T} $ as the set of functions $ U $ of $ (z',x_d,\theta,\psi_d)\in \omega_T\times\R_+\times\T^m\times\R_+ $, bounded continuous with respect to $ (x_d,\psi_d) $ in $ \R_+\times\R_+ $ with values in $ H^s_+(\omega_T\times\T^m) $, equipped with the obvious norm
	\begin{equation*}
	\norme{U}_{\mathcal{E}_{s,T}}:=\sup_{x_d>0,\psi_d>0}\norme{U(\,.\,,x_d,\,.\,,\psi_d)}_{H^s_+(\omega_T\times\T^m)}.
	\end{equation*}
\end{definition}

We may now introduce the space of oscillating profiles, corresponding to real frequencies $ \alpha_j(\zeta) $. Following \cite{JolyMetivierRauch1995Coherent}, we choose a quasi-periodic framework with respect to $ \theta $ and an almost-periodic one with respect to $ \psi_d $, namely we consider the closure of the space of trigonometric polynomials with respect to $ \psi_d $ in the space $ \mathcal{E}_{s,T} $ of quasi-periodic functions with respect to $ \theta $. See \cite[Chapters 3, 4]{Corduneanu2009Almost} for more details about almost-periodic functions with values in a Banach space.

\begin{definition}\label{def profils osc}
	We call a trigonometric polynomial with respect to $ \psi_d $ every function $ U $ of $ \mathcal{E}_{s,T} $ that writes as a finite sum in real numbers $ \xi $,
	\begin{equation*}
	U(z,\theta,\psi_d)=\sum_{\xi}U_{\xi}(z,\theta)\,e^{i\,\psi_d\,\xi},
	\end{equation*}
	with $ U_{\xi} $ in $ \mathcal{C}_b(\R^+_{x_d},H^s_+(\omega_T\times\T^m)) $ for all $ \xi $. 
	
	The space of oscillating profiles $ \P^{\osc}_{s,T} $ is then defined as the closure in $ \mathcal{E}_{s,T} $ of the set of trigonometric polynomials with respect to $ \psi_d $. This space is equipped with the norm of $ \mathcal{E}_{s,T} $.
\end{definition}

Concerning evanescent profiles, corresponding to frequencies $ \alpha_j(\zeta) $ with $ j\in\mathcal{P}(\zeta) $, we consider quasi-periodic functions with respect to $ \theta $.
The factors $ e^{i\,\xi\,\psi_d} $ with $ \Im \xi > 0 $ are expressed through a convergence to zero as $ \psi_d $ goes to infinity.

\begin{definition}\label{def profils ev}
	For $ s\geq 0 $ and $ T>0 $, the space $ \P^{\ev}_{s,T} $ of evanescent profiles is defined as the set of functions $ U $ of $ \mathcal{E}_{s,T} $, converging to zero in $ H^s(\omega_T\times\T^m) $ as $ \psi_d $ goes to infinity (for every fixed $ x_d\geq 0 $).
The space $ \P_{s,T}^{\ev} $ is equipped with the norm of $ \mathcal{E}_{s,T} $.
\end{definition}

We may now describe the space of profiles, constructed as the sum of an oscillating and an evanescent part.

\begin{definition}
	For $ T>0 $ and $ s\geq 0 $, we define the space of profiles  $ \P_{s,T} $ of regularity of order $ s $ as
	\begin{equation*}
	\P_{s,T}:=\P^{\osc}_{s,T}\oplus\P_{s,T}^{\ev},
	\end{equation*}
	equipped with the obvious norm.
	If $ U $ belongs to $ \P_{s,T} $, we denote by $ U^{\osc}\in\P_{s,T}^{\osc} $ and $ U^{\ev}\in\P_{s,T}^{\ev} $ the profiles such that $ U=U^{\osc}+U^{\ev} $.
\end{definition}

The proof of the fact that the spaces $ \P^{\osc}_{s,T} $ and $ \P^{\osc}_{s,T} $ are indeed in a direct sum is presented later, after the introduction of a scalar product used in the proof.

One can find in \cite{JolyMetivierRauch1995Coherent} a partial proof of the following result, that we recall here for the sake of clarity.

\begin{lemma}[{\cite[Lemma 6.1.2]{JolyMetivierRauch1995Coherent}}]\label{lemme propr algebre}
	For all $ T>0 $ and for $ s>(d+m)/2 $, the spaces $ \mathcal{E}_{s,T} $, $ \P^{\osc}_{s,T} $, $ \P^{\ev}_{s,T} $ and $ \P_{s,T} $ are all normed algebras. Furthermore, if $ U $ and $ V $ decomposes in $ \P_{s,T}=\P_{s,T}^{\osc}\oplus\P_{s,T}^{\ev} $ as $ U=U^{\osc}+U^{\ev} $ and $ V=V^{\osc}+V^{\ev} $, then the oscillating part of the profile $ UV $ is given by $ U^{\osc}V^{\osc} $ and and its evanescent part by $ U^{\osc}V^{\ev}+U^{\ev}V^{\osc}+U^{\ev}V^{\ev} $.
	
	Moreover, for $ T>0 $ and $ s\geq 0 $, the spaces $ \mathcal{E}_{s,T} $ and $ \P^{\osc}_{s,T} $ are Banach spaces.
\end{lemma}

\begin{proof}
	The algebra properties for $ \mathcal{E}_{s,T} $ and $ \P^{\ev}_{s,T} $ arise immediately from the one of $ H^s_+(\omega_T\times\T^m) $. The same holds for $ \P^{\osc}_{s,T} $ since the set of trigonometric polynomials is stable under multiplication. Finally, this algebra property for $ H^s_+(\omega_T\times\T^m) $ shows that if $ U $ belongs to $ \P^{\osc}_{s,T} $ and $ V $ to $ \P^{\ev}_{s,T} $, then the product $ UV $ belongs to $ \P^{\ev}_{s,T} $, so the space $ \P_{s,T} $ is also an algebra.
	
	As for them, the completeness properties are obvious.
\end{proof}

\subsection{Scalar products on the space of oscillating profiles}

We now define three scalar products that will be useful in the following, notably to obtain a priori estimates. This part is adapted from \cite{JolyMetivierRauch1995Coherent} to the framework of boundary value problems. We shall use a scalar product with the time variable $ t $ fixed (as in \cite{JolyMetivierRauch1995Coherent}, which is a priori adapted to the Cauchy problem) as well as a scalar product with the space variable $ x_d $ fixed, more adapted to the initial boundary problem.

For $ U,V $ two functions of $ \omega_T\times\R_+\times\T^m\times\R_+ $, we denote, when the formulas are licit, for $ x_d>0 $,
\begin{align}\label{eq def prod scal rentrant}
\prodscal{U}{V}_{\inc}(x_d)&:=\lim_{R\rightarrow+\infty}\inv{R}\int_{0}^{R}\prodscal{U}{V}_{L^2(\omega_T\times\T^m)}(x_d,\psi_d)\,d\psi_d,
\intertext{for $ 0<t<T $,}
\label{eq def prod scal sortant}
\prodscal{U}{V}_{\out}(t)
&:=
\lim_{R\rightarrow+\infty}\inv{R} \int_{0}^{R}  \prodscal{U}{V}_{L^2(\R^{d-1}\times\R_+\times\T^m)} (t,\psi_d)\,d\psi_d.
\intertext{and, if
 $ \K $ is a domain of $ \Omega_T $ bounded in the $ x_d $ direction,}
\label{eq def prod scal ppgat}
\prodscal{U}{V}_{\K}&:=\lim_{R\rightarrow+\infty}\inv{R}\int_{0}^{R}\prodscal{U}{V}_{L^2(\K\times\T^m)}(\psi_d)\,d\psi_d.
\end{align}
The first scalar product is suited to the study of incoming modes when the second one is for the outgoing modes, and the last one will be used to prove the finite speed propagation of the leading profile oscillating part.

If  $ U $ and $ V $ are trigonometric polynomials of $ \P^{\osc}_{s,T} $ of the form\footnote{The sums in $ \xi $ are necessarily countable.}
\begin{equation*}
U(z,\theta,\psi_d)=\sum_{\mathbf{n}\in\Z^m}\sum_{\xi\in\R}U_{\mathbf{n},\xi}(z)\,e^{i\,\mathbf{n}\cdot\theta}\,e^{i\,\xi\,\psi_d},\quad V(z,\theta,\psi_d)=\sum_{\mathbf{n}\in\Z^m}\sum_{\xi\in\R}V_{\mathbf{n},\xi}(z)\,e^{i\,\mathbf{n}\cdot\theta}\,e^{i\,\xi\,\psi_d},
\end{equation*}
then \cite[Theorem 3.4, Remark 4.17]{Corduneanu2009Almost} ensure that the scalar products $ \prodscal{U}{V}_{\inc}(x_d) $ and $ \prodscal{U}{V}_{\K} $ are well-defined and satisfies, for $ x_d\geq 0 $,
\begin{align}\label{eq prod scal pol trigo rentrants}
\prodscal{U}{V}_{\inc}(x_d)&=(2\pi)^m\sum_{\mathbf{n}\in\Z^m}\sum_{\xi\in\R}\prodscal{U_{\mathbf{n},\xi}}{V_{\mathbf{n},\xi}}_{L^2(\omega_T)}(x_d),
\intertext{and}\label{eq prod scal pol trigo ppgat}
\prodscal{U}{V}_{\K}&=(2\pi)^m\sum_{\mathbf{n}\in\Z^m}\sum_{\xi\in\R}\prodscal{U_{\mathbf{n},\xi}}{V_{\mathbf{n},\xi}}_{L^2(\K)}.
\end{align}
Indeed, for each function $ U $ of $ \P^{\osc}_{s,T} $, its trace with respect to $ \psi_d $ belongs to $ L^2(\K\times\T^m) $ 	for all $ \psi_d\geq  0 $, and its trace with respect to $ x_d,\psi_d $ belongs to $ L^2(\omega_T\times\T^m) $ for all $ x_d,\psi_d\geq 0 $.
If $ U $ and $ V $ are moreover of compact support with respect to $ x_d $ and if $ s\geq 1 $, then the traces of $ U $ and $ V $ with respect to $ t,\psi_d $ belong to $ L^2(\R^{d-1}\times\R_+\times\T^m) $, and the same results from \cite{Corduneanu2009Almost} ensure that the scalar product $ \prodscal{U}{V}_{\out}(t) $ is well-defined and satisfy, for $ t>0 $,
\begin{equation}
	\label{eq prod scal pol trigo sortants}
	\prodscal{U}{V}_{\out}(t)=(2\pi)^m\sum_{\mathbf{n}\in\Z^m}\sum_{\xi\in\R}\prodscal{U_{\mathbf{n},\xi}}{V_{\mathbf{n},\xi}}_{L^2(\R^{d-1}\times\R_+)}(t).
\end{equation}

In short, scalar products \eqref{eq def prod scal rentrant} and \eqref{eq def prod scal ppgat} (resp. \eqref{eq def prod scal sortant}) are well-defined on the space of profiles $ \P_{s,T}^{\osc} $, $ s\geq 0 $ (resp. for profiles of $ \P_{s,T}^{\osc} $ with compact support with respect to $ x_d $ with $ s\geq 1 $), and  formulas \eqref{eq prod scal pol trigo rentrants}, \eqref{eq prod scal pol trigo ppgat} and  \eqref{eq prod scal pol trigo sortants} are satisfied in this case.

The incoming scalar product \eqref{eq def prod scal rentrant} is used to prove the following result.

\begin{lemma}
	For all $ T>0 $, the spaces $ \P^{\osc}_{0,T} $ and $ \P^{\ev}_{0,T} $ are in direct sum.
\end{lemma}

\begin{proof}
	Consider a profile $ U $ in $ \P^{\osc}_{0,T} \cap \P^{\ev}_{0,T} $ that writes, because it is an oscillating profile,
	\begin{equation*}
		U(z,\theta,\psi_d)=\sum_{\mathbf{n}\in\Z^m}\sum_{\xi\in\R}U_{\mathbf{n},\xi}(z)\,e^{i\,\mathbf{n}\cdot\theta}\,e^{i\,\xi\,\psi_d},
	\end{equation*}
	the sum in $ \xi $ being countable.
	The profile $ U $ also being evanescent, for all $ x_d\geq 0 $, the function $ \psi_d\mapsto U(.,x_d,.,\psi_d) $ converges to zero in $ L^2(\omega_T\times\T^m) $ as $ \psi_d $ goes to infinity. Thus, for all $ x_d\geq 0 $, we have
	\begin{equation*}
		\prodscal{U}{U}_{\inc}(x_d)=0,
	\end{equation*}
	since the mean value (in terms of \eqref{eq def prod scal rentrant}) of a continuous function on $ \R_+ $ converging to zero at infinity is zero.
	But, since $ U $ is an oscillating profile, we have
	\begin{equation*}
		\prodscal{U}{U}_{\inc}(x_d)=(2\pi)^m\sum_{\mathbf{n}\in\Z^m}\sum_{\xi\in\R}\norme{U_{\mathbf{n},\xi}}^2_{L^2(\omega_T)}(x_d),
	\end{equation*}
so, for all $ \mathbf{n},\xi $, the function $ U_{\mathbf{n},\xi}(x_d) $ is zero in $ L^2(\omega_T\times\T^m) $, and the profile $ U $ is therefore zero as well.
\end{proof}

\section{Ansatz and main result}

We seek to construct an approximate solution to \eqref{eq systeme 1} under the form of a formal series $ u^{\epsilon,\app}(z,z'\cdot\zeta_1/\epsilon,\dots,z'\cdot\zeta_m/\epsilon,x_d/\epsilon) $, where $ u^{\epsilon,\app} $ is given by
\begin{equation}\label{eq def serie formelle}
u^{\epsilon,\app}(z,\theta,\psi_d):=\sum_{k\geq 1}\epsilon^k\,U_k(z,\theta,\psi_d),
\end{equation}
with at least $ U_1 $ in $ \P_{s,T} $ for some $ s\geq 0 $. As for them, correctors $ (U_k)_{k\geq 2} $ a priori exhibit frequencies that may not be characteristic. The convergence of the associated series then relies on a small divisor assumption which is different from the previously made small divisor assumption \ref{hypothese petits diviseurs 1}. Thus we only consider correctors as formal trigonometric series involving all frequencies in the group $ \prodscalbis{\F} $ generated by the set $ \F $.  

We are now in place to state the main result of this work. It is placed under Assumptions
\ref{hypothese bord non caract} to \ref{hypothese pas de sortant}, as well as Assumption \ref{hypothese type resonance} that will be made further on. We denote by $ h $ an integer larger or equal to $ (3+a_1)/2 $ where $ a_1 $ is the real number introduced in the small divisors Assumption \ref{hypothese petits diviseurs 1}. Then we denote $ s_0:=h+(d+m)/2 $.

\begin{theorem}\label{thm existence profils}
	Let  $ s $ be an integer such that $ s\geq s_0 $. Under previously listed assumptions, there exists a time $ T>0 $ such that system \eqref{eq obtention U 1} described below and that governs the evolution of the leading profile in the asymptotic expansion \eqref{eq def serie formelle} admits a unique solution $ U_1 $ in $ \P_{s,T} $.
\end{theorem}

We recall that we have considered a forcing term $ G $ in $ H^{\infty}(\R^d\times\T^m)  $, zero for negative times $ t $ and of zero mean with respect to $ \theta $ in $ \T^m $, but the infinite regularity assumption is made only for simplicity, and the estimates, and thus the existence time $ T $, only depend on the $ H^s(\R^d\times\T^m) $ norm of $ G $. More precisely, the existence time $ T $ depends on the operator $ L(0,\partial_z) $, the frequencies on the boundary $ \zeta_1,\dots,\zeta_{m} $, the order of regularity $ s $, and the $ H^s(\omega_T\times\T^m) $ norm of the forcing term $ G $. 

The formal WKB study shows that the function
\begin{equation*}
	z\mapsto \epsilon\, U_1(z,z'\cdot\zeta_1/\epsilon,\dots,z'\cdot\zeta_m/\epsilon,x_d/\epsilon)
\end{equation*}
is formally an approximate solution to system \eqref{eq systeme 1}.

Last two sections of the article are devoted to the proof of Theorem \ref{thm existence profils}. In section \ref{section 5} we start by formally deriving the cascade of equations that must be verified by the amplitudes $ (U_k)_{k\geq 1} $. By resolving, first formally and then rigorously for a part of it (in section \ref{section 6}), a fast problem, this cascade is triangularized and a system of equations for the leading profile is extracted from it. Next, in section \ref{section 6}, after a discussion about the different types of resonances that appear, the system is reduced to two decoupled systems for the oscillating and evanescent parts, the one on the oscillating part involving only the incoming phases. The oscillating system is even decoupled in a system for all resonant modes, and a system for each non-resonant mode. Then a priori estimates are proved for the linearized system for the oscillating parts, leading to the construction of solutions to these linearized systems. Such estimates prove in particular that the solution to \eqref{eq obtention U 1} is unique. An iterative scheme is then used to construct solutions to the nonlinear systems, and the evanescent part is finally determined. 

Unlike in \cite[Part 6]{JolyMetivierRauch1995Coherent}, from which the following is mainly inspired, there is no symmetry in the system, since it is not hyperbolic as a propagation system in the variable $ x_d $. This lack of symmetry is a genuine obstacle to deal with the resonance terms (that are in infinite number) in the a priori estimates. Assumption \ref{hypothese type resonance}, concerning all resonances with the possible exception of a finite number of them, allows to work around the problem and to obtain estimates for the associated terms. Assumption \ref{hypothese type resonance} will be carefully verified for the Euler system.

\section{Formal WKB study}\label{section 5}

\subsection{Cascade of equations for the profiles}

We seek to formally determine the equations the sequence of profiles $(U_n)_{n\geq 1}$ must satisfy for the formal series $ z\mapsto u^{\epsilon,\app}(z,z'\cdot\zeta_1/\epsilon,\dots,z'\cdot\zeta_m/\epsilon,x_d/\epsilon) $ given by \eqref{eq def serie formelle} to be solution to system \eqref{eq systeme 1}.  In the following we wish for the coefficient (a priori dependent on $ u^{\epsilon,\app} $) in factor of the partial derivative with respect to $ x_d $ to be the identity matrix, so that its differential is zero. The analogous property for the partial derivative in time is crucial in \cite{JolyMetivierRauch1995Coherent} from which we mainly draw our analysis. This choice is justified here by the particular role of the variable $ x_d $ in a priori estimates for the principal profile. This is why we are interested from now on in the following equivalent system
\begin{equation}\label{eq systeme 1 tilde}
\left\lbrace \begin{array}{lr}
\tilde{L}(u^{\epsilon,\app},\partial_z)\,u^{\epsilon,\app}
=0&\qquad \mbox{in } \Omega_T, \\[5pt]
B\,u^{\epsilon,\app}_{|x_d=0}=\epsilon\, g^{\epsilon}&\qquad \mbox{on } \omega_T,  \\[10pt]
u^{\epsilon,\app}_{|t\leq 0}=0&
\end{array}
\right.
\end{equation}
where we have denoted
\begin{equation*}
\tilde{L}(u,\partial_z):=A_d(u)^{-1}\,L(u,\partial_z)=\tilde{A}_0(u)\,\partial_t+\displaystyle\sum_{i=1}^{d-1}\tilde{A}_i(u)\,\partial_i+\partial_d,
\end{equation*}
with $ \tilde{A}_i(u):=A_d(u)^{-1}\,A_i(u) $ for $ i=0,\dots,d-1 $, and $ A_0(u):=I $. In the following we may use the notation $ \partial_0:=\partial_t $.

\subsubsection{WKB Cascade inside the domain}
We are now able to write the equations verified by the profiles $ U_k $, $ k\geq 1 $, by formally replacing $ u^{\epsilon,\app} $ by its formal expansion
\begin{equation*}
	u^{\epsilon,\app}=\sum_{k\geq 1}\epsilon^k\,U_k
\end{equation*}
in system \eqref{eq systeme 1 tilde}.
First, we note that the following Taylor expansion is verified, for $i=0,\dots,d-1$:
\begin{equation*}
\tilde{A}_i(u^{\epsilon,\app})=\tilde{A}_i(0)+\epsilon\,d\tilde{A}_i(0)\cdot U_1+\sum_{k\geq 2}\epsilon^k\Big[d\tilde{A}_i(0)\cdot U_k+G^i_{k-1}\Big],
\end{equation*}
where, for $k\geq 2 $, $G^i_{k-1}$ only depends on $U_1,\dots,U_{k-1}$.
The operator $\tilde{L}\big(u^{\epsilon,\app},\partial_z\big)$ thus writes
\begin{equation}\label{eq obtention devt L tilde}
\tilde{L}\big(u^{\epsilon,\app},\partial_z\big)=\tilde{L}(0,\partial_z)+\sum_{k\geq 1}\epsilon^k\tilde{L}_k(U_k,\partial_{z'}),
\end{equation}
with
\begin{equation*}
\tilde{L}_1\big(U_1,\partial_{z'}\big):=\sum_{i=0}^{d-1}d\tilde{A}_i(0)\cdot U_1\, \partial_i,\qquad 
\tilde{L}_k\big(U_k,\partial_{z'}\big):=\sum_{i=0}^{d-1}\Big(d\tilde{A}_i(0)\cdot U_k+G^i_{k-1}\Big)\,\partial_i,\quad \forall k\geq 2.
\end{equation*}
For $ k\geq 2 $, despite the fact that only the $ U_k $ dependency is indicated in the operators $ \tilde{L}_k(U_k,\partial_{z'}) $, these ones also depend on the profiles $ U_1,\dots,U_{k-1} $, via the functions $ G^i_{k-1} $. The operator $ \tilde{L}_1(U_1,\partial_{z'}) $ depends however only on $ U_1 $.

We see here the benefit of considering the modified operator $ \tilde{L}(u,\partial_z) $: there is an  $ x_d $ derivative only in the leading operator $ \tilde{L}(0,\partial_z) $, and not in the other operators $ \tilde{L}_k(U_k,\partial_{z'}) $.
Furthermore, we verify that 
\begin{align}\label{eq obtention devt L u}
\tilde{L}(u^{\epsilon,\app}&,\partial_z)\Big[u^{\epsilon,\app}(z,z'\cdot\zeta_1/\epsilon,\dots,z'\cdot\zeta_m/\epsilon,x_d/\epsilon )\Big]\\\nonumber
=&\Big[\tilde{L}(u^{\epsilon,\app},\partial_z)\,u^{\epsilon}+\inv{\epsilon}\sum_{j=1}^m\tilde{L}(u^{\epsilon,\app},\zeta_j)\,\partial_{\theta_j}u^{\epsilon,\app}+\inv{\epsilon}\partial_{\psi_d}u^{\epsilon,\app}\Big]\\
\nonumber&(z,z'\cdot\zeta_1/\epsilon,\dots,z'\cdot\zeta_m/\epsilon,x_d/\epsilon ).
\end{align}
where, for $ j=1,\dots,m $, the symbol $ \tilde{L}(u^{\epsilon,\app},\zeta_j) $ is defined by $ \sum_{i=0}^{d-1}\zeta_j^i\,\tilde{A}_i(u^{\epsilon,\app}) $ with $ \zeta_j=(\zeta_j^0,\dots,\zeta_j^{d-1}) $. Expansion \eqref{eq obtention devt L tilde} of the operator $ \tilde{L}(u^{\epsilon,\app},\partial_z) $ leads to the analogous expansions of the operators $ \tilde{L}(u^{\epsilon,\app},\zeta_j) $ for $ j=1,\dots,m $:
\begin{equation}\label{eq obtention devt L tilde zeta}
\tilde{L}\big(u^{\epsilon,\app},\zeta_j\big)=\tilde{L}(0,\zeta_j)+\sum_{k\geq 1}\epsilon^k\tilde{L}_k(U_k,\zeta_j)
\end{equation}
where
\begin{equation*}
\tilde{L}_1\big(U_1,\zeta_j\big):=\sum_{i=0}^{d-1}\zeta_j^i\,d\tilde{A}_i(0)\cdot U_1,\qquad 
\tilde{L}_k\big(U_k,\zeta_j\big):=\sum_{i=0}^{d-1}\zeta_j^i\Big(d\tilde{A}_i(0)\cdot U_k+G^i_{k-1}\Big),\quad \forall k\geq 2.
\end{equation*}
Thus, according to expansions \eqref{eq obtention devt L tilde}, \eqref{eq obtention devt L u} and \eqref{eq obtention devt L tilde zeta}, the following asymptotic expansion holds
\begin{align*}
\tilde{L}(u^{\epsilon,\app},\partial_z)&\Big[u^{\epsilon,\app}(z,z'\cdot\zeta_1/\epsilon,\dots,z'\cdot\zeta_m/\epsilon,x_d/\epsilon )\Big]\\
&=\inv{\epsilon}\Big\lbrace\sum_{j=1}^m\tilde{L}(0,\zeta_j)\,\partial_{\theta_j}u^{\epsilon,\app}+\partial_{\psi_d}u^{\epsilon,\app}\Big\rbrace+\tilde{L}(0,\partial_z)\,u^{\epsilon}+\sum_{j=1}^m\tilde{L}_1(U_1,\zeta_j)\,\partial_{\theta_j}u^{\epsilon,\app}\\
\nonumber&\quad+\sum_{k\geq 1}\epsilon^{k}\Big\{\tilde{L}_k(U_k,\partial_{z'})\,u^{\epsilon,\app}+\sum_{j=1}^m\tilde{L}_{k+1}(U_{k+1},\zeta_j)\,\partial_{\theta_j}u^{\epsilon,\app}\Big\},
\end{align*}
where the right hand side is evaluated in $ (z,z'\cdot\zeta_1/\epsilon,\dots,z'\cdot\zeta_m/\epsilon,x_d/\epsilon ) $.
The operator $ L(u^{\epsilon,\app},\partial_z) $ applied to $ u^{\epsilon,\app}(z,z'\cdot\zeta_1/\epsilon,\dots,z'\cdot\zeta_m/\epsilon,x_d/\epsilon )$ is therefore given by the formal series
\begin{multline}\label{eq obtention serie W k}
\tilde{L}(u^{\epsilon,\app},\partial_z)\Big[u^{\epsilon,\app}(z,z'\cdot\zeta_1/\epsilon,\dots,z'\cdot\zeta_m/\epsilon,x_d/\epsilon )\Big]
=\sum_{k\geq 0}\epsilon^k\,W_k(z,z'\cdot\zeta_1/\epsilon,\dots,z'\cdot\zeta_m/\epsilon,x_d/\epsilon ),
\end{multline}
where, if the variables $ \theta $ and $ \psi_d $ are substituted to $ (z'\cdot\zeta_1/\epsilon,\dots,z'\cdot\zeta_m/\epsilon) $ and $ x_d/\epsilon  $, the amplitudes $ (W_k)_{k\geq 0} $ of the formal series \eqref{eq obtention serie W k} are given by
\begin{subequations}\label{eq obtention def W}
\begin{equation}\label{eq obtention def W 0}
W_0:=\Big\{\sum_{j=1}^m\tilde{L}(0,\zeta_j)\,\partial_{\theta_j}+\partial_{\psi_d}\Big\}U_1,
\end{equation}
\begin{equation}\label{eq obtention def W 1}
W_1:=\Big\{\sum_{j=1}^m\tilde{L}(0,\zeta_j)\,\partial_{\theta_j}+\partial_{\psi_d}\Big\}U_2+\Big\{\tilde{L}(0,\partial_z)+\sum_{j=1}^m\tilde{L}_1(U_1,\zeta_j)\,\partial_{\theta_j}\Big\}U_1,
\end{equation}
and for $k\geq 2$,
\begin{multline}\label{eq obtention def W k}
W_k:=\Big\{\sum_{j=1}^m\tilde{L}(0,\zeta_j)\,\partial_{\theta_j}+\partial_{\psi_d}\Big\}U_{k+1}
+\Big\{\tilde{L}(0,\partial_z)+\sum_{j=1}^m\tilde{L}_1(U_1,\zeta_j)\,\partial_{\theta_j}\Big\}U_k\\
+\sum_{l=1}^{k-1}\Big\{\tilde{L}_{k-l}(U_{k-l},\partial_{z'})+\sum_{j=1}^m\tilde{L}_{k-l+1}(U_{k-l+1},\zeta_j)\,\partial_{\theta_j}\Big\}U_l.
\end{multline}
\end{subequations}
Formulas \eqref{eq obtention def W 0} and \eqref{eq obtention def W 1} correspond to the analogous ones in \cite[(1.33), (1.46)]{CoulombelGuesWilliams2011Resonant} in the case $ m=1 $.

Thus, for the formal series \eqref{eq def serie formelle} to be solution to \eqref{eq systeme 1}, the formal series \eqref{eq obtention serie W k} must be zero, or equivalently
\begin{equation*} 
W_k=0,\qquad \forall k\geq 0.
\end{equation*}

We note that each equation $ W_k=0 $ involves the fast operator
\begin{equation*}
\mathcal{L}(\partial_{\theta},\partial_{\psi_d}):=\sum_{j=1}^m\tilde{L}(0,\zeta_j)\,\partial_{\theta_j}+\partial_{\psi_d},
\end{equation*}
which is linear and has constant coefficients, as customary in weakly nonlinear geometric optics, see for example \cite[Section 9.4]{Rauch2012Hyperbolic}.
The subject of the following part is to study this operator in order to rewrite equations \eqref{eq obtention def W} in an equivalent manner. Before that the WKB cascades on the boundary and at initial time are determined.

\subsubsection{WKB cascade on the boundary}

Since we want the formal series \eqref{eq def serie formelle} to satisfy the boundary condition
\begin{equation*}
B\,u^{\epsilon,\app}_{|x_d=0}=\epsilon\,g^{\epsilon},
\end{equation*}
the profiles $ (U_k)_{k\geq 1} $ must verify, using variables $ (z',\theta) $, the following boundary conditions
\begin{align*}
(B\,U_1)_{|x_d=0,\psi_d=0}&=G\\
(B\,U_k)_{|x_d=0,\psi_d=0}&=0,\quad k\geq 2.
\end{align*}

\subsubsection{Initial conditions}

In a similar manner, the profiles $ (U_k)_{k\geq 1} $ must satisfy
the following initial conditions
\begin{equation*}
(U_k)_{|t\leq 0}=0,\quad k\geq 1.
\end{equation*}

\subsection{Resolution of the fast problem $ \mathcal{L}(\partial_{\theta},\partial_{\psi_d})\,U=H $}
In this part we seek to resolve in the formal trigonometric series framework the equation
\begin{equation*}
 \mathcal{L}(\partial_{\theta},\partial_{\psi_d})\,U=H,
\end{equation*}
and more precisely, to formally determine the kernel and range of the operator $ \mathcal{L}(\partial_{\theta},\partial_{\psi_d}) $. We follow, in a formal manner, the analysis of {\cite[Part 3]{Lescarret2007Wave}}. Thus we consider $ U $ writing 
\begin{equation*}
U(z,\theta,\psi_d)=\sum_{\mathbf{n}\in\Z^m}\sum_{\xi\in\R}U^{\osc}_{\mathbf{n},\xi}(z)\,e^{i\,\mathbf{n}\cdot\theta}\,e^{i\,\xi\,\psi_d}+\sum_{\mathbf{n}\in\Z^m}U^{\ev}_{\mathbf{n}}(z,\psi_d)\,e^{i\,\mathbf{n}\cdot\theta},
\end{equation*}
and $ H $ writing
\begin{equation*}
H(z,\theta,\psi_d)=\sum_{\mathbf{n}\in\Z^m}\sum_{\xi\in\R}H^{\osc}_{\mathbf{n},\xi}(z)\,e^{i\,\mathbf{n}\cdot\theta}\,e^{i\,\xi\,\psi_d}+\sum_{\mathbf{n}\in\Z^m}H^{\ev}_{\mathbf{n}}(z,\psi_d)\,e^{i\,\mathbf{n}\cdot\theta},
\end{equation*}
where, for all $ \mathbf{n} $ in $ \Z^m $, the sum in $ \xi $ is countable.
Then, by definition of the fast operator $ \mathcal{L}(\partial_{\theta}, \partial_{\psi_d}) $, we get
\begin{multline*}
\mathcal{L}(\partial_{\theta},\partial_{\psi_d})\,U(z,\theta,\psi_d) =\sum_{\mathbf{n}\in\Z^m} \sum_{\xi\in\R} i\,\tilde{L}\big(0,(\mathbf{n}\cdot\boldsymbol{\zeta},\xi)\big) \,U^{\osc}_{\mathbf{n},\xi}(z) \,e^{i\,\mathbf{n}\cdot\theta}\,e^{i\,\xi\,\psi_d}\\
+\sum_{\mathbf{n}\in\Z^m}\big\{i\,\tilde{L}\big(0,(\mathbf{n}\cdot\boldsymbol{\zeta},0)\big)+\partial_{\psi_d}\big\} U^{\ev}_{\mathbf{n}}(z,\psi_d)\,e^{i\,\mathbf{n}\cdot\theta},
\end{multline*}
where we recall that $ \boldsymbol{\zeta} $ refers to the $ m $-tuple of elements of $ \R^d $ given by $ \boldsymbol{\zeta}=(\zeta_1,\dots,\zeta_m) $. Therefore, the profile $ U $ is a solution to $ \mathcal{L}(\partial_{\theta},\partial_{\psi_d})\,U=H $ if and only if, for all $ \mathbf{n} $ in $ \Z^m $ and for all $ \xi $ in $ \R $, we have
\begin{subequations}
\begin{align}
	i\,\tilde{L}\big(0,(\mathbf{n}\cdot\boldsymbol{\zeta},\xi)\big) \,U^{\osc}_{\mathbf{n},\xi}(z)&=H^{\osc}_{\mathbf{n},\xi}(z),\label{eq BKW form osc n zeta}
	\intertext{and}
	\Big(i\,\tilde{L}\big(0,(\mathbf{n}\cdot\boldsymbol{\zeta},0)\big)+\partial_{\psi_d}\Big)\, U^{\ev}_{\mathbf{n}}(z,\psi_d)&=H^{\ev}_{\mathbf{n}}(z,\psi_d).\label{eq BKW form ev n}
\end{align}
\end{subequations}
For $ \mathbf{n} $ in $ \Z^m $ and for $ \xi $ in $ \R $, equation \eqref{eq BKW form osc n zeta} admits a solution if and only if $ H^{\osc}_{\mathbf{n},\xi} $ belongs to the range of the matrix $ \tilde{L}(0,(\mathbf{n}\cdot\boldsymbol{\zeta},\xi)) $, that is to say, according to Definition \ref{def pi alpha}, the kernel $ \ker \tilde{\pi}_{(\mathbf{n}\cdot\boldsymbol{\zeta},\xi)} $. According to Definition \ref{def pi alpha} of the partial inverse $ Q_{\alpha} $, every solution is therefore of the form
\begin{equation*}
	U^{\osc}_{\mathbf{n},\xi}=X_{\mathbf{n},\xi}-i\,Q_{(\mathbf{n}\cdot\boldsymbol{\zeta},\xi)}\,A_d(0)\,H^{\osc}_{\mathbf{n},\xi},
\end{equation*}
with $ X_{\mathbf{n},\xi} $ an element of $ \Ima \pi_{(\mathbf{n}\cdot\boldsymbol{\zeta},\xi)} $, and thus satisfies
\begin{equation*}
	U^{\osc}_{\mathbf{n},\xi}=\pi_{(\mathbf{n}\cdot\boldsymbol{\zeta},\xi)}\,U^{\osc}_{\mathbf{n},\xi}-i\,Q_{(\mathbf{n}\cdot\boldsymbol{\zeta},\xi)}\,A_d(0)\,H^{\osc}_{\mathbf{n},\xi}.
\end{equation*}
As for it, the differential equation \eqref{eq BKW form ev n} admits a formal solution for every $ \mathbf{n} $ in $ \Z^m $. 
For $ \mathbf{n}=0 $, the solution is formally given by
\begin{equation*}
U^{\ev}_{\mathbf{0}}(z,\psi_d)=-\int_{\psi_d}^{+\infty}H^{\ev}_{\mathbf{0}}(z,s)\,ds,
\end{equation*}
and, for $ \mathbf{n} $ in $ \Z^m\privede{0} $, according to Duhamel's principle, by
\begin{align}\label{eq BKW form U n ev equa diff}
U^{\ev}_{\mathbf{n}}(z,\psi_d)=&e^{\psi_d\mathcal{A}(\mathbf{n}\cdot\boldsymbol{\zeta})}\,\Pi^{e}_{\C^N}(\mathbf{n}\cdot\boldsymbol{\zeta})\,U^{\ev}_{\mathbf{n}}(z,0)+\int_0^{\psi_d}e^{(\psi_d-s)\mathcal{A}(\mathbf{n}\cdot\boldsymbol{\zeta})}\,\Pi^{e}_{\C^N}(\mathbf{n}\cdot\boldsymbol{\zeta})\,H^{\ev}_{\mathbf{n}}(z,s)\,ds\\\nonumber
&-\int_{\psi_d}^{+\infty}e^{(\psi_d-s)\mathcal{A}(\mathbf{n}\cdot\boldsymbol{\zeta})}\,\big(I-\Pi^{e}_{\C^N}(\mathbf{n}\cdot\boldsymbol{\zeta})\big) \,H^{\ev}_{\mathbf{n}}(z,s)\,ds,
\end{align}
noting that $ i\,\tilde{L}\big(0,(\mathbf{n}\cdot\boldsymbol{\zeta},0)\big)=-\mathcal{A}(\mathbf{n}\cdot\boldsymbol{\zeta}) $. Indeed, according to Assumption \ref{hypothese pas de glancing}, the frequency $ \mathbf{n}\cdot\boldsymbol{\zeta} $ is not glancing and the projector $ \Pi^{e}_{\C^N}(\mathbf{n}\cdot\boldsymbol{\zeta}) $ is thus well defined. The Duhamel's principle then applies separately to $ \Pi^{e}_{\C^N}(\mathbf{n}\cdot\boldsymbol{\zeta})\,U^{\ev}_{\C^N} $ and $ (I-\Pi^{e}_{\C^N}(\mathbf{n}\cdot\boldsymbol{\zeta}))\,U^{\ev}_{\C^N} $.
The first term of the right-hand side of \eqref{eq BKW form U n ev equa diff} is therefore well-defined, and the integral of the second one converges since, according to Proposition \ref{prop controle exp t A Pi } proved in appendix, the matrix $ e^{t\mathcal{A}(\mathbf{n}\cdot\boldsymbol{\zeta})}\,\Pi^{e}_{\C^N}(\mathbf{n}\cdot\boldsymbol{\zeta}) $ is bounded by a decaying exponential for $ t\geq 0 $. However we do not know if the integral of the third term converges. Indeed, $ H^{\ev}_{\mathbf{n}} $ is only converging to zero at infinity, and the matrix $ e^{t\mathcal{A}(\mathbf{n}\cdot\boldsymbol{\zeta})}\,(I-\Pi^{e}_{\C^N}(\mathbf{n}\cdot\boldsymbol{\zeta})) $ is simply bounded for $ t\leq 0 $, according to Proposition \ref{prop controle exp t A Pi }. The issue is essentially the same for the integral defining $ U^{\ev}_{\mathbf{0}} $. The following result is deduced from this analysis.

\begin{lemma}[{\cite[Theorem 2.14]{Lescarret2007Wave}}]\label{lemme Lescaret}
	The equation $ \mathcal{L}(\partial_{\theta},\partial_{\psi_d})\,U=H $ admits a solution in the framework of formal trigonometric series if and only if $ \tilde{\E^i}\,H=0 $, and every solution is of the form
	\begin{equation*}
	U=\E\, U+\Q\,H,
	\end{equation*}
	where projectors $ \E $ and $ \tilde{\E^i} $ and operator $ \Q $ are formally defined further on.
	In particular we have
	\begin{equation*}
	\ker \mathcal{L}(\partial_{\theta},\partial_{\psi_d})=\Im \E \quad \mbox{and}\quad \Im \mathcal{L}(\partial_{\theta},\partial_{\psi_d})=\ker \tilde{\E^i}.
	\end{equation*}
	
	If $ U $ is given by
	\begin{equation*}
	U(z,\theta,\psi_d)=\sum_{\mathbf{n}\in\Z^m}\sum_{\xi\in\R}U^{\osc}_{\mathbf{n},\xi}(z)\,e^{i\,\mathbf{n}\cdot\theta}\,e^{i\,\xi\,\psi_d}+\sum_{\mathbf{n}\in\Z^m}U^{\ev}_{\mathbf{n}}(z,\psi_d)\,e^{i\,\mathbf{n}\cdot\theta},
	\end{equation*}
	then $ \tilde{\E^i}\,U $ is defined as
	\begin{equation}\label{eq def E i}
	\tilde{\E^i}\,U(z,\theta,\psi_d):=\sum_{\mathbf{n}\in\Z^m}\sum_{\xi\in\R} \tilde{\pi}_{(\mathbf{n}\cdot\boldsymbol{\zeta},\xi)}\, U^{\osc}_{\mathbf{n},\xi}(z)\,e^{i\,\mathbf{n}\cdot\theta}\,e^{i\,\xi\,\psi_d},
	\end{equation}
	$ \E\,U $ as
	\begin{align}\label{eq def E}
	\E\,U(z,\theta,\psi_d):=&\sum_{\mathbf{n}\in\Z^m}\sum_{\xi\in\R} \pi_{(\mathbf{n}\cdot\boldsymbol{\zeta},\xi)}\,\, U^{\osc}_{\mathbf{n},\xi}(z)\,e^{i\,\mathbf{n}\cdot\theta}\,e^{i\,\xi\,\psi_d} \\\nonumber
	 &+\sum_{\mathbf{n}\in\Z^m\privede{0}}e^{\psi_d\mathcal{A}(\mathbf{n}\cdot\boldsymbol{\zeta})}\,\Pi^{e}_{\C^N}(\mathbf{n}\cdot\boldsymbol{\zeta})\,U^{\ev}_{\mathbf{n}}(z,0)\,e^{i\,\mathbf{n}\cdot\theta},
	\end{align}
	and $ \Q\,U $ as
	\begin{align*}
	\Q\,U(z,\theta,\psi_d):=&-\sum_{\mathbf{n}\in\Z^m}\sum_{\xi\in\R} i\,Q_{(\mathbf{n}\cdot\boldsymbol{\zeta},\xi)}\,A_d(0)\,U^{\osc}_{\mathbf{n},\xi}(z)\,e^{i\,\mathbf{n}\cdot\theta}\,e^{i\,\xi\,\psi_d}-\int_{\psi_d}^{+\infty}U^{\ev}_{\mathbf{0}}(z,s)\,ds\\
	&+\sum_{\mathbf{n}\in\Z^m\privede{0}} \left(\int_0^{\psi_d}e^{(\psi_d-s)\mathcal{A}(\mathbf{n}\cdot\boldsymbol{\zeta})}\,\Pi^{e}_{\C^N}(\mathbf{n}\cdot\boldsymbol{\zeta})\,U^{\ev}_{\mathbf{n}}(z,s)\,ds\right.\\&\left. 
	\quad-\int_{\psi_d}^{+\infty}e^{(\psi_d-s)\mathcal{A}(\mathbf{n}\cdot\boldsymbol{\zeta})}\,\big(I-\Pi^{e}_{\C^N}(\mathbf{n}\cdot\boldsymbol{\zeta})\big) \,U^{\ev}_{\mathbf{n}}(z,s)\,ds\right)e^{i\,\mathbf{n}\cdot\theta}.
	\end{align*}
\end{lemma}

Note that, for now, the operator $ \Q $ and the projectors $ \E $ and $ \tilde{\E^i} $ are only formally defined. The projectors $ \E $ and $ \tilde{\E^i} $, that are the only one involved in the leading profile equations, can be defined in the space $ \P_{s,T} $, and this result will constitute a part of the following section. However, the operator $ \Q $ cannot be rigorously defined in the functional framework used here (the issues are the absence of some small divisor control - in the same manner as in \cite{JolyMetivierRauch1995Coherent}, as well as a lack of exponential decay for the evanescent profiles).

\subsection{System of equations satisfied by the leading profile}

According to expressions \eqref{eq obtention def W 0} and \eqref{eq obtention def W 1} of the amplitudes $ W_0 $ and $ W_1 $, and using the previous Lemma \ref{lemme Lescaret}, we get the following system of equations for the leading profile $ U_1 $, simply denoted from now on by $ U $:
\begin{subequations}\label{eq obtention U 1}
	\begin{align}
	\E\, U&=U \label{eq obtention E U}  \\
	\tilde{\E^i}\Big[\tilde{L}(0,\partial_z)\,U+\sum_{j=1}^m\tilde{L}_1(U,\zeta_j)\,\partial_{\theta_j}U\Big]&=0 \label{eq obtention E i U}\\
	B\,U_{|x_d=0,\psi_d=0}&=G \label{eq obtention cond bord U} \\[5pt]
	U_{|t\leq 0}&=0. \label{eq obtention cond initiale U}
	\end{align}
\end{subequations}

We note that the leading profile $ U_1 $ is polarized, in the sense that it satisfies  equation \eqref{eq obtention E U}, so according to Formula \eqref{eq def E} defining projector $ \E $, only the characteristic frequencies occur in its Fourier expansion. We shall see in the next section that the oscillating part $ U^{\osc} $ of profile $ U $ satisfies the problem
\begin{subequations}\label{eq obtention U 1 osc}
	\begin{align}
	\E\, U^{\osc}&=U^{\osc}  \\
	\tilde{\E^i}\Big[\tilde{L}(0,\partial_z)\,U^{\osc}+\sum_{j=1}^m\tilde{L}_1(U^{\osc},\zeta_j)\,\partial_{\theta_j}U^{\osc}\Big]&=0 \\
	B\,\big(U^{\osc}+U^{\ev}\big)_{|x_d=0,\psi_d=0}&=G \label{eq obtention U 1 1 osc cond bord}\\[5pt]
	U^{\osc}_{|t\leq 0}&=0. 
	\end{align}
\end{subequations}
The question is to know whether or not boundary condition \eqref{eq obtention U 1 1 osc cond bord} determines on its own the trace $ U^{\osc}_{|x_d=0,\psi_d=0} $. As already explained in \cite{CoulombelGuesWilliams2011Resonant} and \cite{CoulombelWilliams2017Mach}, the answer depends on the existence of a resonance between two incoming frequencies that generates an outgoing frequency. Such a resonance pattern is excluded by Assumption \ref{hypothese pas de sortant}. In this case the boundary condition \eqref{eq obtention U 1 1 osc cond bord} also determines the trace $ U^{\ev}_{|x_d=0,\psi_d=0} $, which, according to the polarization type condition, immediately leads to the construction of the evanescent part of $ U $.

\section{Construction of the leading profile}\label{section 6}

Now that the system that must be verified by the leading profile $ U $ has been formally determined, we are in position to construct a solution to it. We begin in a first part by a discussion about the different types of resonances that may appear in the system. In particular, the technical assumption is made that the resonances for which the lack of symmetry is not controlled are in finite number. This assumption is made to deal with the lack of symmetry in the resonances terms compared to the case of \cite{JolyMetivierRauch1995Coherent}. In a second part, the projectors appearing in system \eqref{eq obtention U 1} as well as some results of section \ref{section 5} are made rigorous, using the small divisors Assumption \ref{hypothese petits diviseurs 1}. Then we proceed to the three main steps of the proof, namely the decoupling of system \eqref{eq obtention U 1}, derivation of energy estimates for the linearized systems and construction of the solution. First step is achieved in section \ref{subsection reducing} and consists in reducing the system \eqref{eq obtention U 1} to a system for the evanescent part, a system for the oscillating incoming resonant part and a system for each oscillating incoming non-resonant part. Writing equation \eqref{eq obtention E i U} in extension with modes, we isolate the system satisfied by the mean value to show that it is zero. Then it is proven, using an energy estimate on the oscillating outgoing part, that every outgoing mode is zero. To obtain this estimate, we use a scalar product defined only for profiles with bounded $ x_d $-support. Therefore we need to show beforehand that a solution to \eqref{eq obtention U 1} propagates in the normal direction with finite velocity. The proof of this result is postponed after the introduction of the techniques used in it, in section \ref{subsection estimate resonant part}. Finally, using the fact that incoming modes are zero, the boundary condition \eqref{eq obtention cond bord U} can be decoupled for each evanescent and incoming oscillating mode, which will conclude the decoupling of the system. The derivation of a priori estimates without loss of derivatives is performed in sections \ref{subsection estimate resonant part} and \ref{subsection est Burgers}, one for the resonant part and one for each non-resonant part, namely Burgers type equations. Each non-resonant mode must be treated separately to avoid a factor unbounded with respect to the frequency, but part \ref{subsection est Burgers} presents no additional difficulty since it reuses techniques displayed in the previous one. The derivation of a priori estimates for the linearized oscillating resonant system is presented in section \ref{subsection estimate resonant part}, beginning by the $ L^2 $ estimate. It is obtained taking the incoming modes suited scalar product between the linearized propagation equation and a modified profile. Four terms need to be addressed, the transport and Burgers type (corresponding to self-interaction) ones are treated classically (with an integration by part and using symmetry in the self-interaction terms), while the two resonant ones are handled using the technical assumption on resonance terms. More precisely this assumption asserts that these resonance terms (except for a finite number of them, that are treated separately) are such that the lack of symmetry in it is controlled in a way that the techniques used for the Burgers type terms can be adapted. The same method for energy estimates is used with a different scalar product to prove the finite speed propagation in this section. Then the estimates for derivatives are obtained classically using commutator estimates. Section \ref{subsection construction solution} is devoted to construction of the solution. The oscillating part is constructed using an usual procedure which is not detailed, and consists in proving existence of a linearized solution with a finite difference scheme using the a priori estimates previously derived, and then existence of the sought solution using an iterative scheme. Uniqueness is deduced from the a priori estimates. As for the evanescent part, its expression is prescribed by a polarization type condition, and we prove that the constructed profile belongs to the space of evanescent profiles. It will achieve the proof of Theorem \ref{thm existence profils}, and finally section \ref{subsection conclusion} draws a conclusion and some perspectives.

\subsection{Resonance coefficient and additional assumption}\label{notations 2}

The sets defined below permit to gather the characteristic frequencies according to collinearity.

\begin{definition}\label{def base Z m}
	We consider the subset of $ \Z^m\privede{0} $, denoted by $ \mathcal{B}_{\Z^m} $, constituted of all $ m $-tuples of coprime integers of which the first  nonzero term is positive:
	\begin{equation*}
	\mathcal{B}_{\Z^m}:=\ensemble{(n_1,\dots,n_m)\in\Z^m\privede{0}\,\middle|\,\begin{aligned}
		&n_1\wedge\cdots\wedge n_m=1,\\
		&\exists k\in\ensemble{0,\dots,m-1},\, n_1,\dots,n_k=0,n_{k+1}>0
		\end{aligned}}.
	\end{equation*}
	One can verify that for all $ \mathbf{n} $ of $ \Z^m\privede{0} $, there exists a unique element $ \mathbf{n}_0 $ of $ \mathcal{B}_{\Z^M} $ and a unique nonzero integer $ \lambda $ such that $ \mathbf{n}=\lambda\,\mathbf{n}_0 $.
\end{definition} 

Then we introduce the following notation for real characteristic frequencies lifted from frequencies on the boundary.

\begin{definition}\label{def C(n)}
	For $ \mathbf{n} $ in $ \Z^m\privede{0} $, we denote by $ \mathcal{C}(\mathbf{n}) $ the finite set of real numbers $ \xi $ such that the frequency $ (\mathbf{n}\cdot\boldsymbol{\zeta},\xi) $ is real and characteristic, namely
	\begin{equation*}
	\mathcal{C}(\mathbf{n}):=\ensemble{\xi\in\R \mid (\mathbf{n}\cdot\boldsymbol{\zeta},\xi)\in\mathcal{C}}.
	\end{equation*}
	We also denote by $ \mathcal{C}_{\inc}(\mathbf{n}) $  (resp. $ \mathcal{C}_{\out}(\mathbf{n}) $) the set of real numbers $ \xi $ such that the frequency $ (\mathbf{n}\cdot\boldsymbol{\zeta},\xi) $ is real, characteristic and incoming (resp. outgoing), namely
	\begin{equation*}
	\mathcal{C}_{\inc}(\mathbf{n})=\ensemble{\xi_j(\freq) \mid j\in\mathcal{I}(\mathbf{n}\cdot\boldsymbol{\zeta})},\quad \mathcal{C}_{\out}(\mathbf{n})=\ensemble{\xi_j(\freq) \mid j\in\mathcal{O}(\mathbf{n}\cdot\boldsymbol{\zeta})},
	\end{equation*}
	with notations of Proposition \ref{prop decomp E_-}.
\end{definition}

We recall that according to Assumption \ref{hypothese pas de glancing}, there is no glancing frequency in $ \F $, so the disjoint union
\begin{equation*}
\mathcal{C}(\mathbf{n})=\mathcal{C}_{\inc}(\mathbf{n})\cup \mathcal{C}_{\out}(\mathbf{n})
\end{equation*}
is satisfied for all $ \mathbf{n} $ in $ \Z^m\privede{0} $. All \emph{real} characteristic frequencies have been considered here, but there may also exist non-real characteristic frequencies lifted from $ \mathbf{n}\cdot\boldsymbol{\zeta} $.

\begin{remark}
	One can check that, according to Remark \ref{remarque homogeneite xi, alpha etc}, the sets $ \mathcal{C} $, $ \mathcal{C}_{\inc} $ and $ \mathcal{C}_{\out} $ are homogeneous of degree 1. Thus, if $ \mathbf{n} $ belongs to $ \Z^m\privede{0} $ and $ \xi $ to $ \mathcal{C}(\mathbf{n}) $, and if $ \mathbf{n}_0 $ in $ \B_{\Z^m} $ and $ \lambda $ in $ \Z^* $ are such that $ \mathbf{n}=\lambda\,\mathbf{n}_0 $, then there exists $ \xi_0 $ in $ \mathcal{C}(\mathbf{n}_0) $ such that $ \xi=\lambda\,\xi_0 $.
\end{remark}

We now introduce some notations for the resonances.

\begin{definition}\label{def E}
	For $ \mathbf{n} $ in $ \Z^m\privede{0} $ and $ \xi $ in $ \mathcal{C}(\mathbf{n}) $, we denote by $ E(\mathbf{n},\xi) $ the vector of the basis $ E_1,\dots,E_N $ of $ \C^N $ given by \eqref{eq base C^N adaptee espace propres stricte hyp} that generates the line $ \ker L\big(0,(\mathbf{n}\cdot\boldsymbol{\zeta},\xi)\big) $.
\end{definition}

\begin{remark}
	Note that for $ \mathbf{n} $ in $ \Z^m\privede{0} $, $ \xi $ in $ \mathcal{C}(\mathbf{n}) $ and $ \lambda $ in $ \Z^* $, since the linear subspaces $ \ker L\big(0,(\mathbf{n}\cdot\boldsymbol{\zeta},\xi)\big) $ and $ \ker L\big(0,(\lambda\mathbf{n}\cdot\boldsymbol{\zeta},\lambda\xi)\big) $ are equal, we infer $ E(\mathbf{n},\xi)=E(\lambda\mathbf{n},\lambda\xi) $, so the vector $ E(\mathbf{n},\xi) $ is homogeneous of degree 0. 
\end{remark}

The following definition is based on \cite[Chapter 11]{Rauch2012Hyperbolic}.

\begin{definition}\label{def coefficient gamma}
	Let $ \mathbf{n}_p $, $ \mathbf{n}_q $ be two elements of $ \Z^m\privede{0} $, and let $ (\xi_p,\xi_q) $ in $ \mathcal{C}(\mathbf{n}_p)\times\mathcal{C}(\mathbf{n}_q) $ be such that the frequency
	\begin{equation*}
	(\mathbf{n}_p\cdot\boldsymbol{\zeta},\xi_p)+(\mathbf{n}_q\cdot\boldsymbol{\zeta},\xi_q)=:(\mathbf{n}_r\cdot\boldsymbol{\zeta},\xi_r)
	\end{equation*}
	is real and characteristic (i.e. such that there is a resonance). Then the resonance coefficient $ \Gamma\big((\mathbf{n}_p,\xi_p),\allowbreak(\mathbf{n}_q,\xi_q)\big) $ is defined by the equation
	\begin{equation*}
	\tilde{\pi}_{(\mathbf{n}_r\cdot\boldsymbol{\zeta},\xi_r)}\,\tilde{L}_1\big(E(\mathbf{n}_p,\xi_p),\mathbf{n}_q\cdot\boldsymbol{\zeta}\big)\,E(\mathbf{n}_q,\xi_q)
	=\Gamma\big((\mathbf{n}_p,\xi_p),(\mathbf{n}_q,\xi_q)\big)\,\tilde{\pi}_{(\mathbf{n}_r\cdot\boldsymbol{\zeta},\xi_r)}\,E(\mathbf{n}_r,\xi_r).
	\end{equation*}
	This coefficient exists by definition of the projectors $ \tilde{\pi}_k $, for $ k=1,\dots,N $ and according to Lemma \ref{lemme pi tilde E vitesse de groupe}. 
\end{definition}

\begin{remark}
	\begin{enumerate}[label=\roman*), leftmargin=0.7cm]
	\item 
Since all quantities involved in the definition of $ \Gamma $ are homogeneous of degree 0 or 1, the coefficient $ \Gamma $ is homogeneous of degree 1, i.e. for all $ \mathbf{n}_p $, $ \mathbf{n}_q $ in $ \Z^m\privede{0} $, $ (\xi_p,\xi_q) $ in $ \mathcal{C}(\mathbf{n}_p)\times\mathcal{C}(\mathbf{n}_q) $ such that the frequency $ (\mathbf{n}_p\cdot\boldsymbol{\zeta},\xi_p)+(\mathbf{n}_q\cdot\boldsymbol{\zeta},\xi_q) $ is real characteristic (that is such that there is a resonance) and for all $ \lambda $ in $ \Z^* $, we have
	\begin{align}\label{eq propr homogeneite Gamma}
	 \Gamma\big((\lambda\mathbf{n}_p,\lambda\xi_p),(\lambda\mathbf{n}_q,\lambda\xi_q)\big)&=\lambda\,\Gamma\big((\mathbf{n}_p,\xi_p),(\mathbf{n}_q,\xi_q)\big).
	\end{align}
By definition and for the same reason, we also have, for $ \mathbf{n}_0 $ in $ \Z^m\privede{0} $ and $ \xi_0 $ in $ \mathcal{C}(\mathbf{n}_0) $, and for $ \lambda_1,\lambda_2 $ in $ \Z^* $,
\begin{equation}\label{eq propr gamma AI}
	\Gamma\big(\lambda_1(\mathbf{n}_0,\xi_0),\lambda_2(\mathbf{n}_0,\xi_0)\big)=\lambda_2\,\Gamma\big((\mathbf{n}_0,\xi_0),(\mathbf{n}_0,\xi_0)\big).
\end{equation}
	\item Since according to Remark \ref{remarque proj pi et Q bornes} the projectors $ \tilde{\pi}_{\alpha} $ are bounded and the vectors $ E(\mathbf{n},\xi) $ are of norm 1, for all $ \mathbf{n} $ in $ \Z^m\privede{0} $ and all $ \xi $ in $ \mathcal{C}(\mathbf{n}) $, we have
	\begin{equation*}
	\left|\Gamma\big((\mathbf{n},\xi),(\mathbf{n},\xi)\big)\right|\leq C\frac{|\mathbf{n}|}{\left|\tilde{\pi}_{(\mathbf{n}\cdot\boldsymbol{\zeta},\xi)}\,E(\mathbf{n},\xi)\right|}.
	\end{equation*}
	Therefore, according to the lower bound \eqref{eq freq minoration pi tilde E alpha} of Lemma \ref{lemme minoration vitesse de groupe} and the small divisors Assumption \ref{hypothese petits diviseurs 1}, for all $ \mathbf{n} $ in $ \Z^m\privede{0} $ and all $ \xi $ in $ \mathcal{C}(\mathbf{n}) $, we have
	\begin{equation}\label{eq est gamma n n}
		\left|\Gamma\big((\mathbf{n},\xi),(\mathbf{n},\xi)\big)\right|\leq C|\mathbf{n}|^h,
	\end{equation}
	where $ h $ is an integer larger than  $ (3+a_1)/2 $ with notation of Assumption \ref{hypothese petits diviseurs 1}.
\item The quantity $ \tilde{\pi}_{(\mathbf{n}_r\cdot\boldsymbol{\zeta},\xi_r)}\,\tilde{L}_1\big(E(\mathbf{n}_p,\xi_p),\mathbf{n}_q\cdot\boldsymbol{\zeta}\big)\,E(\mathbf{n}_q,\xi_q) $ being homogeneous of degree 1 with respect to $ \mathbf{n}_q $, the resonance coefficient $ \Gamma\big((\mathbf{n}_p,\xi_p),(\mathbf{n}_q,\xi_q)\big) $ formally corresponds to a partial derivative with respect to the fast tangential variables, applied to the profile associated with the frequency $ (\mathbf{n}_q\cdot\zeta,\xi_q) $.
	\end{enumerate}
\end{remark}

The coefficients $ \Gamma $ defined above shall appear in the computations to obtain a priori estimates for system \eqref{eq obtention U 1}. In particular, when these coefficients present some symmetry property, the associated resonance is easy to control in the a priori estimates. Thus we discriminate the resonances satisfying this symmetry property from the others.

\begin{definition}\label{def ens resonances} Fix a constant $ C_0>0 $.
	Let $ \mathbf{n}_r $ be in $ \B_{\Z^m} $, and $ \xi _r$ in $ \mathcal{C}(\mathbf{n}_r) $. We consider the set of 7-tuples $ (\lambda_p,\lambda_q,\lambda_r,\mathbf{n}_p,\mathbf{n}_q,\xi_p,\xi_q) $ with $ \lambda_p,\lambda_q,\lambda_r $ in $ \Z^* $, $ \mathbf{n}_p,\mathbf{n}_q $ in $ \B_{\Z^m} $, $ (\xi_p,\xi_q) $ in $ \mathcal{C}(\mathbf{n}_p)\times\mathcal{C}(\mathbf{n}_q) $, $ (\mathbf{n}_p\cdot\boldsymbol{\zeta},\xi_p) $ and $ (\mathbf{n}_q\cdot\boldsymbol{\zeta},\xi_q) $ non collinear and $ \lambda_p,\lambda_q,\lambda_r $ coprime numbers, that resonate to give the resonance $ \lambda_r\,(\mathbf{n}_r\cdot\boldsymbol{\zeta},\xi) $ in the following way
	\begin{equation*}
	\lambda_p\,(\mathbf{n}_p\cdot\boldsymbol{\zeta},\xi_p)+ \lambda_q\,(\mathbf{n}_q\cdot\boldsymbol{\zeta},\xi_q)= \lambda_r\,(\mathbf{n}_r\cdot\boldsymbol{\zeta},\xi_r).
	\end{equation*}
	This set is written as the disjoint union
	\begin{equation*}
	\mathcal{R}_1(\mathbf{n}_r,\xi_r)\sqcup 	
	\mathcal{R}_2(\mathbf{n}_r,\xi_r),
	\end{equation*}
	where the sets
	$	\mathcal{R}_1(\mathbf{n}_r,\xi_r) $ and  $ 	
	\mathcal{R}_2(\mathbf{n}_r,\xi_r) $ are defined as follows.
	\begin{enumerate}[label=\roman*)]
		\item The set $ \mathcal{R}_1(\mathbf{n}_r,\xi_r) $ is constituted of 7-tuples $ (\lambda_p,\lambda_q,\lambda_r,\mathbf{n}_p,\mathbf{n}_q,\xi_p,\xi_q) $ satisfying 
		\begin{equation}\label{eq propr res type 1}
		\left|\Gamma\big((\lambda_p\,\mathbf{n}_p,\lambda_p\,\xi_p),(\lambda_q\,\mathbf{n}_q,\lambda_q\,\xi_q)\big)+\Gamma\big((\lambda_p\,\mathbf{n}_p,\lambda_p\,\xi_p),(-\lambda_r\,\mathbf{n}_r,-\lambda_r\,\xi_r)\big)\right|
		\leq C_0 \big|(\lambda_p\,\mathbf{n}_p,\lambda_p\,\xi_p)\big|
		\end{equation}
		where $ C_0>0 $ is the constant which have been fixed in the beginning and which does not depend on $ \lambda_p$, $\lambda_q$, $\lambda_r$, $\mathbf{n}_p$, $\mathbf{n}_q$, $\mathbf{n}_r$, $\xi_p$, $\xi_q$, $\xi_r $. These resonances are said to be of type 1.
		\item The 7-tuples $ (\lambda_p,\lambda_q,\lambda_r,\mathbf{n}_p,\mathbf{n}_q,\xi_p,\xi_q) $ which do not satisfy the previous property constitute the set of type 2 resonances, denoted by $ \mathcal{R}_2(\mathbf{n}_r,\xi_r) $.
	\end{enumerate}
\end{definition}

\begin{remark}\label{remarque type resonance}
	\begin{enumerate}[label=\roman*), leftmargin=0.7cm]
		\item Note that the sets
		$ \mathcal{R}_1(\mathbf{n},\xi) $ and  $ 	
		\mathcal{R}_2(\mathbf{n},\xi) $ depend on the constant $ C_0>0 $ fixed at the beginning, although this dependence is not indicated.
		\item According to Assumption \ref{hypothese pas de sortant}, note that if the frequency $ (\freq,\xi) $ is outgoing (resp. incoming), then both sets
		$	\mathcal{R}_1(\mathbf{n},\xi) $ and  $ 	
		\mathcal{R}_2(\mathbf{n},\xi) $ are constituted only of 7-tuples corresponding to outgoing (resp. incoming) frequencies.
		\item Note that since the coefficients $ \Gamma $ are not symmetrical, the type of a resonance
		\begin{equation*}
			\lambda_p\,(\mathbf{n}_p\cdot\boldsymbol{\zeta},\xi_p)+ \lambda_q\,(\mathbf{n}_q\cdot\boldsymbol{\zeta},\xi_q)= \lambda_r\,(\mathbf{n}_r\cdot\boldsymbol{\zeta},\xi_r),
		\end{equation*}
		depends on the way it is written. However, since condition \eqref{eq propr res type 1} is symmetrical in $ (q,r) $, the resonance \begin{equation*}
			\lambda_p\,(\mathbf{n}_p\cdot\boldsymbol{\zeta},\xi_p)+ \lambda_q\,(\mathbf{n}_q\cdot\boldsymbol{\zeta},\xi_q)= \lambda_r\,(\mathbf{n}_r\cdot\boldsymbol{\zeta},\xi_r),
		\end{equation*}
	is of type 1 if and only if the resonance  \begin{equation*}
		\lambda_p\,(\mathbf{n}_p\cdot\boldsymbol{\zeta},\xi_p)-\lambda_r\,(\mathbf{n}_r\cdot\boldsymbol{\zeta},\xi_r) = -\lambda_q\,(\mathbf{n}_q\cdot\boldsymbol{\zeta},\xi_q),
	\end{equation*}
is of type 1.
		\item Also note that if a resonance of the form
		\begin{equation*}
		\lambda_p\,(\mathbf{n}_p\cdot\boldsymbol{\zeta},\xi_p)+ \lambda_q\,(\mathbf{n}_q\cdot\boldsymbol{\zeta},\xi_q)= \lambda_r\,(\mathbf{n}_r\cdot\boldsymbol{\zeta},\xi_r),
		\end{equation*}
		holds, then for $ k $ in $ \Z^* $, the following resonance relation is also satisfied
		\begin{equation*}
		k\,\lambda_p\,(\mathbf{n}_p\cdot\boldsymbol{\zeta},\xi_p)+ k\,\lambda_q\,(\mathbf{n}_q\cdot\boldsymbol{\zeta},\xi_q)= k\,\lambda_r\,(\mathbf{n}_r\cdot\boldsymbol{\zeta},\xi_r).
		\end{equation*}
		This explains the choice made in Definition \ref{def ens resonances} to consider only 3-tuples $ (\lambda_p,\lambda_q,\lambda_r) $ of coprime integers and $ m $-tuples $ \mathbf{n}_p $, $ \mathbf{n}_q $ and $ \mathbf{n}_r $ of $ \B_{\Z^m} $.
	\end{enumerate}
\end{remark}

The previous definition leads to the last assumption of this work. It is made for technical reasons, but, up to our knowledge, it is not a necessary assumption.

\begin{assumption}\label{hypothese type resonance}
There exists a constant $ C_0>0 $ such that the sets $	\mathcal{R}_1(\mathbf{n},\xi) $ and  $ 	
\mathcal{R}_2(\mathbf{n},\xi) $, for $ \mathbf{n} $ in $ \Z^m\privede{0} $ and $ \xi $ in $ \mathcal{C}(\mathbf{n}) $, defined in Definition \ref{def ens resonances} satisfy the two following properties.
\begin{enumerate}[label=\alph*), leftmargin=0.7cm]
	\item The sets of incoming resonances of type 2 and outgoing resonances of types 1 and 2
	\begin{equation*}
		\bigcup_{\substack{\mathbf{n}\in\B_{\Z^m}\\\xi\in\mathcal{C}_{\inc}(\mathbf{n})}}\mathcal{R}_2(\mathbf{n},\xi),\quad 
		\bigcup_{\substack{\mathbf{n}\in\B_{\Z^m}\\\xi\in\mathcal{C}_{\out}(\mathbf{n})}}\big(\mathcal{R}_1(\mathbf{n},\xi)\cup\mathcal{R}_2(\mathbf{n},\xi)\big),
	\end{equation*}
	are finite sets.
	\item For all incoming frequency $ (\mathbf{n},\xi) $ of $ \B_{\Z^m}\times\mathcal{C}_{\inc}(\mathbf{n}) $ such that the set $ \mathcal{R}_1(\mathbf{n},\xi) $ is nonempty, the following lower bound holds
	\begin{equation}\label{eq propr minoration pi tilde E R1}
		\left|\tilde{\pi}_{(\mathbf{n}\cdot\boldsymbol{\zeta},\xi)}E(\mathbf{n},\xi)\right|\geq \inv{C_0}.
	\end{equation} 
\end{enumerate}
\end{assumption}

\begin{remark}
	\begin{enumerate}[label=\roman*), leftmargin=0.7cm]
		\item The self-interaction between two collinear frequencies always constitute a resonance, but these terms should not be an issue in the analysis, since they induce terms of Burgers type, which are commonly treated in the estimates.  However the resonances of type 2 are difficult to control, that is why a finiteness assumption is made on this set, whereas property \eqref{eq propr res type 1} satisfied by resonances of type 1 allows to treat an infinite number of them (to the prize of a uniform control). Such an infinity of resonances appears irremediably in Example \ref{exemple Euler 1} of compressible isentropic Euler equations in dimension 2. It constitute one of the main additional difficulty addressed here in comparison to the monophase case of \cite{CoulombelGuesWilliams2011Resonant}.
		\item  We already know that for $ (\mathbf{n},\xi) $ in $ \B_{\Z^m}\times\mathcal{C}(\mathbf{n}) $, the vector $ \tilde{\pi}_{(\mathbf{n}\cdot\boldsymbol{\zeta},\xi)}E(\mathbf{n},\xi) $ is bounded, according to Remark \ref{remarque proj pi et Q bornes}, and that it is allowed to go to zero but in a controlled way, according to estimate \eqref{eq freq minoration pi tilde E alpha} and Assumption \ref{hypothese petits diviseurs 1}. In the case of an infinite number of resonances, namely for incoming resonances of type 1, we also need to make sure that these vectors do not go to zero, for a technical reason explained below. This is why we assume the uniform lower bound \eqref{eq propr minoration pi tilde E R1} of Assumption \ref{hypothese type resonance}. This assumption excludes the possibility of the existence of a sequence of frequencies $ (\mathbf{n},\xi) $ such that $ \mathcal{R}_1(\mathbf{n},\xi) $ is nonempty, converging to the glancing set $ \mathcal{G} $.
	\end{enumerate}
\end{remark}

According to Assumption \ref{hypothese pas de sortant}, it has already been established that the sets of incoming and outgoing frequencies $ \F^{\inc} $ and $ \F^{\out} $ defined in \eqref{eq freq def F inc ev} and \eqref{eq freq def F out} do not resonate with each other. This decomposition of the frequencies set in sets that do not resonate with each other is now to be refined, which will allow to decouple the studied system according to these sets.

\begin{definition}\label{def ensembles F 1 2 in out}
	In this definition we confuse the frequency $ (\mathbf{n}\cdot\boldsymbol{\zeta},\xi) $ with the couple $ (\mathbf{n},\xi) $. Let $ C_0>0 $ be the constant fixed in Assumption \ref{hypothese type resonance}.  We denote by $ \F^{\out}_{\res} $ the set of outgoing frequencies $ (\mathbf{n},\xi) $ of $ \B_{\Z^m}\times\mathcal{C}_{\out}(\mathbf{n}) $ involved in resonances of type 1 or 2, namely such that $ \mathcal{R}_1(\mathbf{n},\xi)\cup\mathcal{R}_2(\mathbf{n},\xi) $ is nonempty. Then the following disjoint union holds
	\begin{equation}\label{eq propr def F out 1 2}
		\ensemble{(\mathbf{n},\xi) \in\B_{\Z^m}\times\mathcal{C}_{\out}(\mathbf{n}) }=\F^{\out}_{\res}\sqcup\bigsqcup_{(\mathbf{n},\xi) \in(\B_{\Z^m}\times\mathcal{C}_{\out}(\mathbf{n}))\setminus\F^{\out}_{\res}}\ensemble{(\mathbf{n},\xi) },
	\end{equation}
	where the set involved in the disjoint union do not resonate with each other. The set $ \F^{\inc}_{\res} $ is defined in a similar way for incoming frequencies, so that the following decomposition holds
	\begin{equation}\label{eq propr def F in 1 2}
		\ensemble{(\mathbf{n},\xi) \in\B_{\Z^m}\times\mathcal{C}_{\inc}(\mathbf{n}) }=\F^{\inc}_{\res}\sqcup\bigsqcup_{(\mathbf{n},\xi) \in(\B_{\Z^m}\times\mathcal{C}_{\inc}(\mathbf{n}))\setminus\F^{\inc}_{\res}}\ensemble{(\mathbf{n},\xi) },
	\end{equation}
	where the set involved in the disjoint union do not resonate with each other. 	
\end{definition}

\begin{remark}
	In Assumption \ref{hypothese type resonance}, the bound \eqref{eq propr minoration pi tilde E R1} a priori applies to couples $ (\mathbf{n},\xi) $ such that the set $ \mathcal{R}_1(\mathbf{n},\xi) $ is nonempty. But since according to Assumption \ref{hypothese type resonance} there is only a finite number of type 2 resonances, we can assume without loss of generality that this bound also applies to couples $ (\mathbf{n},\xi) $ such that $ \mathcal{R}_1(\mathbf{n},\xi) $ is empty but $ \mathcal{R}_2(\mathbf{n},\xi) $ is not, namely to all elements of $ \F^{\inc}_{\res} $. Therefore for all $ (\mathbf{n},\xi) $ in $ \F^{\inc}_{\res} $, the following bound holds
	\begin{equation}\label{eq est gamma n n resonance}
		\left|\Gamma\big((\mathbf{n},\xi),(\mathbf{n},\xi)\big)\right|\leq C|\mathbf{n}|.
	\end{equation}
Note that the previous estimate differs from \eqref{eq est gamma n n} by a linear control and not an algebraic control of degree $ h $.
\end{remark}

Finally the projectors analogous to $ \E $ and $ \tilde{\E^i} $, selecting only the incoming resonant frequencies, are defined, and we verify after that Assumption \ref{hypothese type resonance} for the Euler equations example considered in this paper.

\begin{definition}\label{def proj E res}
	For all formal trigonometric series $ U $ writing
	\begin{equation*}
		U(z,\theta,\psi_d)=\sum_{\mathbf{n}\in\Z^m}\sum_{\xi\in\R}U^{\osc}_{\mathbf{n},\xi}(z)\,e^{i\,\mathbf{n}\cdot\theta}\,e^{i\,\xi\,\psi_d}+\sum_{\mathbf{n}\in\Z^m}U^{\ev}_{\mathbf{n}}(z,\psi_d)\,e^{i\,\mathbf{n}\cdot\theta},
	\end{equation*}
	the series $ \tilde{\E^i}^{\inc}_{\res}\,U $ is defined as
	\begin{equation}\label{eq def E i res in}
		\tilde{\E^i}^{\inc}_{\res}\,U(z,\theta,\psi_d):=\sum_{(\mathbf{n}_0,\xi_0)\in\F^{\inc}_{\res}}\sum_{\lambda\in\Z^*} \tilde{\pi}_{(\lambda\mathbf{n}_0\cdot\boldsymbol{\zeta},\lambda\xi_0)}\, U^{\osc}_{\lambda\mathbf{n}_0,\lambda\xi_0}(z)\,e^{i\,\lambda\mathbf{n}_0\cdot\theta}\,e^{i\,\lambda\xi_0\,\psi_d},
	\end{equation}
	and $ \E^{\inc}_{\res}\,U $ as
	\begin{align}\label{eq def E res in}
		\E^{\inc}_{\res}\,U(z,\theta,\psi_d):=\sum_{(\mathbf{n}_0,\xi_0)\in\F^{\inc}_{\res}}\sum_{\lambda\in\Z^*} \pi_{(\lambda\mathbf{n}_0\cdot\boldsymbol{\zeta},\lambda\xi_0)}\, U^{\osc}_{\lambda\mathbf{n}_0,\lambda\xi_0}(z)\,e^{i\,\lambda\mathbf{n}_0\cdot\theta}\,e^{i\,\lambda\xi_0\,\psi_d}.
	\end{align}
\end{definition}

\begin{example}
	We return to Example \ref{exemple Euler 1} of compressible isentropic Euler equations in dimension 2, for which we check the last assumption of this work, namely Assumption \ref{hypothese type resonance} about the control of resonances. Recall the notations and results of Example \ref{exemple Euler 1} and those after, and notably the analysis of the resonances made in Example \ref{exemple Euler frequences au bord}, and consider a Mach number $ M $ satisfying the previously made assumptions. Is has been shown that if $ M^2 $ is not an algebraic number of degree less than 4 in $ \mathbb{Q}[\delta] $, then the only resonances (except for the self-interactions) occurring are those involving the linear frequency $ \alpha_3(\zeta) $, which are in infinite number, even with collinearity, and between incoming frequencies. It will be shown that there exists a constant $ C_0>0 $ such that these resonances satisfy property \eqref{eq propr res type 1} and such that all frequencies $ \alpha_{3}(\zeta_{p,q}) $ satisfy the lower bound \eqref{eq propr minoration pi tilde E R1}, proving that system \eqref{eq systeme Euler} verifies Assumption \ref{hypothese type resonance}.
	First we look for the coefficients $ \Gamma $ for this type of resonances.
	Consider $ (p,q) $ and $ (r,s) $ in $ \Z^2\privede{0} $. The aim is to determine the coefficient 
	\begin{equation*}
	\Gamma\big((p,q),\xi_3(\zeta_{p,q}),(r,s),\xi_3(\zeta_{r,s})\big)
	\end{equation*}
	relative to the resonance
	\begin{equation*}
	\alpha_3(\zeta_{p,q})+\alpha_3(\zeta_{r,s})=\alpha_3(\zeta_{p+r,q+s}),
	\end{equation*}
	denoted more briefly by $ \Gamma\big((p,q),(r,s)\big) $. Since, for $ \zeta $ in $ \R^2\privede{0} $, the real characteristic frequency $ \alpha_3(\zeta)=(\tau,\eta,\xi_3(\zeta)) $ satisfies $ \tau=\tau_2\big(\eta,\xi_3(\zeta)\big) $, we obtain
	\begin{align*}
	&E\big((p,q),\xi_3(\zeta_{p,q})\big)=E_2\big(\eta_{p,q},\xi_3(\zeta_{p,q})\big),\quad E\big((r,s),\xi_3(\zeta_{r,s})\big)=E_2\big(\eta_{r,s},\xi_3(\zeta_{r,s})\big),\\[5pt] &E\big((p+r,q+s),\xi_3(\zeta_{p+r,q+s})\big)=E_2\big(\eta_{p+r,q+s},\xi_3(\zeta_{p+r,q+s})\big),
	\end{align*}
	where $ E_2(\eta,\xi) $ is the vector of basis \eqref{eq base C^N adaptee espace propres stricte hyp} of $ \C^N $ given in this example by
	\begin{equation*}
	E_2(\eta,\xi):=\inv{\sqrt{\eta^2+\xi^2}}\left(\begin{array}{c}
	0 \\ \xi \\ -\eta 
	\end{array}\right),\quad (\eta,\xi)\in\R^2\privede{0}.
	\end{equation*}
	Thus the associated vector of basis \eqref{eq base C^N tilde adaptee espace propres stricte hyp} is given by
	\begin{equation*} A_2(V_0)^{-1}\,E_2(\eta,\xi)=\inv{(u_0^2-c_0^2)\sqrt{\eta^2+\xi^2}}\left(\begin{array}{c}
	-\eta\,v_0 \\ \xi\,\frac{u_0^2-c_0^2}{u_0} \\ -\eta\,u_0 
	\end{array}\right),\quad (\eta,\xi)\in\R^2\privede{0}.
	\end{equation*}
	Also computing the vectors $ E_1(\eta,\xi) $ and $ E_3(\eta,\xi) $ for $ (\eta,\xi)\in\R^2\privede{0} $, it is determined that the projector $ \tilde{\pi}_{((p+r,q+s)\cdot\boldsymbol{\zeta},\xi_3(\zeta_{p+r,q+s}))} $ occurring in the coefficient $ \Gamma\big((p,q),(r,s)\big) $ is given in this example by $ \tilde{\pi}_2(\eta_{p+r,q+s},\xi_3(\zeta_{p+r,q+s})) $ where, for $ (\eta,\xi) $ in $ \R^2\privede{0} $,
	\begin{equation*}
	\tilde{\pi}_2(\eta,\xi)=\inv{u_0(u_0^2-c_0^2)(\eta^2+\xi^2)}\left(\begin{array}{ccc}
	-u_0\,c_0^2\,\eta^2 & -u_0^2\,v_0\,\eta\,\xi & u_0^2\, v_0\,\eta^2 \\[5pt]
	\eta\,\xi\,c_0^2\,(u_0^2-c_0^2)/v_0 & \xi^2\,u_0\,(u_0^2-c_0^2) & -\eta\,\xi\,u_0\,(u_0^2-c_0^2) \\[5pt]
	-\eta^2\,c_0^2\,u_0^2/v_0 & -\eta\,\xi\,u_0^3 & u_0^3\,\eta^2 
	\end{array}\right).
	\end{equation*}
	Thus the vector $ \tilde{\pi}_2(\eta,\xi)\,E_2(\eta,\xi) $ is given by $ u_0\,A_2(V_0)^{-1}\,E_2(\eta,\xi) $, for $ (\eta,\xi) $ in $ \R^2\privede{0} $. It ensures in particular that Assumption \ref{hypothese type resonance} is verified, since the following uniform lower bound holds
	\begin{align*}
		\left|\tilde{\pi}_{\alpha_{3}(\zeta_{p,q})}E\big((p,q),\xi_3(\zeta_{p,q})\big)\right|&=\left|\tilde{\pi}_2\big(\eta_{p,q},\xi_3(\zeta_{p,q})\big)\,E_2\big(\eta_{p,q},\xi_3(\zeta_{p,q})\big)\right|\\
		&=\left|u_0\,A_2(V_0)^{-1}\,E_2\big(\eta_{p,q},\xi_3(\zeta_{p,q})\big)\right|\geq C\left| E_2\big((p,q),\xi_3(\zeta_{p,q})\big)\right|=C.
	\end{align*}
	Returning to the determination of coefficients $ \Gamma\big((p,q),(r,s)\big) $, by computing differentials $ d\tilde{A}_i(V_0) $, we finally get
	\begin{multline*}
	\tilde{L}_1\big(E_2\big(\eta_{p,q},\xi_3(\zeta_{p,q})\big),\zeta_{r,s}\big)\,E_2\big(\eta_{r,s},\xi_3(\zeta_{r,s})\big)\\
	=\frac{\tau_{p,q}\,\eta_{r,s}-\tau_{r,s}\,\eta_{p,q}}{\sqrt{\eta_{p,q}^2+\tau_{p,q}^2/u_0^2}\sqrt{\eta_{r,s}^2+\tau_{r,s}^2/u_0^2}}\left(\begin{array}{c}
	v_0\,\eta_{r,s}/[u_0\,(u_0^2-c_0^2)]\\[5pt]
	\tau_{r,s}/u_0^3 \\[5pt]
	\eta_{r,s}/(u_0^2-c_0^2)
	\end{array}\right).
	\end{multline*}
	The formula $ \xi_3(\zeta)=-\tau/u_0 $, for $ \zeta=(\tau,\eta) $ in $ \R^2\privede{0} $ has been used here. Then we have
	\begin{multline*}
	\tilde{\pi}_2(\eta_{p+r,q+s},\xi_3(\zeta_{p+r,q+s}))\,\tilde{L}_1\big(E_2\big(\eta_{p,q},\xi_3(\zeta_{p,q})\big),\zeta_{r,s}\big)\,E_2\big(\eta_{r,s},\xi_3(\zeta_{r,s})\big)\\
	=\frac{(\tau_{p,q}\,\eta_{r,s}-\tau_{r,s}\,\eta_{p,q})(\tau_{p+r,q+s}\,\tau_{r,s}+u^2\,\eta_{p+r,q+s}\,\eta_{r,s})}{u_0^2\,(u_0^2-c_0^2)\sqrt{\eta_{p,q}^2+\tau_{p,q}^2/u_0^2}\sqrt{\eta_{r,s}^2+\tau_{r,s}^2/u_0^2}\big(\eta_{p+r,q+s}^2+\tau_{p+r,q+s}^2/u_0^2\big)}\left(\begin{array}{c}
	\eta_{p+r,q+s}\,v_0/u_0\\[5pt]
	\tau_{p+r,q+s}\,(u_0^2-c_0^2)/u_0^3 \\[5pt]
	\eta_{p+r,q+s}
	\end{array}\right).
	\end{multline*}
	We deduce from the relation $ \tilde{\pi}_2(\eta,\xi)\,E_2(\eta,\xi) = u_0\,A_2(V_0)^{-1}\,E_2(\eta,\xi) $ the following formula for the pursued coefficient $ \Gamma\big((p,q),(r,s)\big) $:
	\begin{equation}\label{eq ex Euler formule gamma}
	\Gamma\big((p,q),(r,s)\big)\\
	=-\frac{(\tau_{p,q}\,\eta_{r,s}-\tau_{r,s}\,\eta_{p,q})(\tau_{p+r,q+s}\,\tau_{r,s}+u_0^2\,\eta_{p+r,q+s}\,\eta_{r,s})}{u_0^4\,\sqrt{\eta_{p,q}^2+\tau_{p,q}^2/u_0^2}\sqrt{\eta_{r,s}^2+\tau_{r,s}^2/u_0^2}\sqrt{\eta_{p+r,q+s}^2+\tau_{p+r,q+s}^2/u_0^2}}.
	\end{equation}
	
	We now check estimate \eqref{eq propr res type 1} with these coefficients. Let $ (p,q) $, $ (r,s) $ and $ (t,w) $ be  in $ \Z^2\privede{0} $ such that $ (p,q)+ (r,s) + (t,w)=(0,0) $. One can verify, using the formulas $ \tau_{p,q}=c_0\,(p+\delta q)\,\eta_0 $, $ \eta_{p,q}=(p+q)\,\eta_0 $, $ p+r=-t $ and $ q+s=-w $, that we get
	\begin{multline*}
		\Gamma\big((p,q),(r,s)\big)\\
		=\frac{(1-\delta)(ps-qr)\big((t+\delta w)\,(r+\delta s)+M^2\,(t+w)\,(r+s)\big)}{\sqrt{(p+q)^2+(p+\delta q)^2/M^2}\sqrt{(r+s)^2+(r+\delta s)^2/M^2}\sqrt{(t+w)^2+(t+\delta w)^2/M^2}}\frac{\eta_0\,c_0^3}{u_0^4}.
	\end{multline*}
	Since $ ps-qr=-(pw-qt) $, we finally have
	\begin{equation*}
	\Gamma\big((p,q),(r,s)\big)+\Gamma\big((p,q),(t,w)\big)=0,
	\end{equation*}
	so estimate \eqref{eq propr res type 1} is in particular trivially satisfied and therefore so is Assumption \ref{hypothese type resonance}.
	Note that in this example, the set $ \F^{\out}_{\res} $ is empty and the set $ \F^{\inc}_{\res} $ is given by
	\begin{equation*}
		\F^{\inc}_{\res}=\ensemble{\alpha_{3}(\zeta)\,\middle|\,\zeta\in\F_b}.
	\end{equation*}
	It concludes the analysis in this paper of the example of compressible isentropic Euler equations in dimension 2, which, with the chosen parameters, satisfies all assumptions of this work.
\end{example}

\begin{remark}
Note that the coefficient $ \Gamma $ determined above corresponds to the one in \cite[(11.5.9)]{Rauch2012Hyperbolic}.
\end{remark}

In the following subsection, some rigorous results on projectors $ \E $ and $ \tilde{\E^i} $ will be proved, using the small divisors Assumption \ref{hypothese petits diviseurs 1}.

\subsection{Rigorous definition of projectors $ \E $, $ \tilde{\E^i} $, $ \E^{\inc}_{\res} $ and $ \tilde{\E^i}^{\inc}_{\res} $}

This part follows \cite[Section 6.2]{JolyMetivierRauch1995Coherent}. Before considering the projectors $ \E $, $ \tilde{\E^i} $, $ \E^{\inc}_{\res} $ and $ \tilde{\E^i}^{\inc}_{\res} $, we state the following controls over spectral projectors, that will be used to rigorously define the projectors $ \E $ and $ \E^{\inc}_{\res} $. The proof of these controls uses notations and results from the one of Proposition \ref{prop proj bornes}, and is therefore postponed after it, in Appendix \ref{appendix preuve}.

\begin{proposition}\label{prop controle exp t A Pi }
	Under Assumption \ref{hypothese petits diviseurs 1}, there exists a constant $ c_1>0 $ and a real number $ b_1 $ such that, for all $ \zeta $ in $ \F_b\privede{0} $, the following estimates hold
	\begin{subequations}
	\begin{align}
	&\label{eq controle exp t A Pi - - t positif} \left|e^{t\mathcal{A}(\zeta)} \, \Pi^{e}_-(\zeta)\right|\leq c_1\,e^{-c_1\,t\,|\zeta|^{-b_1}}\leq c_1, &\forall t\geq 0,\\[5pt]
	&\label{eq controle exp t A Pi C^N - t positif} \left|e^{t\mathcal{A}(\zeta)}\,\Pi^{e}_{\C^N}(\zeta)\right|\leq c_1\,|\zeta|^{b_1}\,e^{-c_1\,t\,|\zeta|^{-b_1}}, &\forall t\geq 0,\\[5pt]
	&\label{eq controle exp t A Pi C^N - t negatif} \left|e^{t\mathcal{A}(\zeta)}\big(I-\Pi^{e}_{\C^N}(\zeta)\big)\right| \leq c_1\,|\zeta|^{b_1}, &\forall t\leq 0.
	\end{align}
	\end{subequations}
\end{proposition}

We are now in position to rigorously define projectors $ \E $ and $ \tilde{\E^i} $. The result concerning the oscillating part comes from \cite[Proposition 6.2.1]{JolyMetivierRauch1995Coherent} and \cite[Proposition 2.2]{CoulombelGuesWilliams2011Resonant}, but the proof is recalled here. The result concerning the evanescent part is simpler, and reduces to prove that some series converges.

\begin{lemma}[{\cite[Proposition 6.2.1]{JolyMetivierRauch1995Coherent}}]
	For all $ T>0 $ and $ s\geq 0 $, the projectors $ \tilde{\E^i} $ and $ \E $ defined by \eqref{eq def E i} and \eqref{eq def E} on the space of trigonometric polynomials each admit a unique extension from the space $ \P^{\osc}_{s,T} $ to itself. Moreover, for $ T_0>0 $, their norm is uniformly bounded with respect to $ T $ in $ ]0,T_0] $.
	
	On an other hand, for $ T>0 $ and $ s\geq 0 $, the projector $ \E $ is well-defined from the space $ \P^{\ev}_{s+\lceil b_1 \rceil,T} $ to the space $ \P^{\ev}_{s,T} $. Furthermore it is uniformly bounded with respect to $ s $ and $ T $. Recall that $ b_1 $ refers to the real number of Proposition \ref{prop controle exp t A Pi }. 
	
	Finally, for $ T>0 $ and $ s\geq 0 $, the projectors $ \tilde{\E^i}^{\inc}_{\res} $  and $ \E^{\inc}_{\res} $ defined as \eqref{eq def E i res in} and \eqref{eq def E res in} on the space of trigonometric polynomials each admit a unique extension from the space $ \P^{\osc}_{s,T} $ to itself. Moreover, for $ T_0>0 $, their norm is uniformly bounded with respect to $ T $ in $ ]0,T_0] $.
\end{lemma}

\begin{proof}
	First the oscillating case is investigated. We consider $ U^{\osc} $ a trigonometric polynomial writing
	\begin{equation*}
	U^{\osc}(z,\theta,\psi_d)=\sum_{\mathbf{n}\in\Z^m}\sum_{\xi\in\R}U^{\osc}_{\mathbf{n},\xi}(z)\,e^{i\,\mathbf{n}\cdot\theta}\,e^{i\,\xi\,\psi_d},
	\end{equation*}
	where each sum in $ \xi $ is finite, and we denote, for $ \mathbf{n}\in\Z^m $,
	\begin{equation}\label{eq proj E forme U n}
	U^{\osc}_{\mathbf{n}}(z,\psi_d):=\sum_{\xi}U^{\osc}_{\mathbf{n},\xi}(z)\,e^{i\,\xi\,\psi_d}.
	\end{equation}
	According to formula \eqref{eq def E} for the projector $ \E $ and since the projectors $ \pi_{\alpha} $ are zero for every noncharacteristic frequency $ \alpha $, we obtain, for $ \mathbf{n}\in\Z^m $,
	\begin{equation*}
	\E\,U^{\osc}_{\mathbf{n}}(z,\psi_d)=\sum_{\xi\in\mathcal{C}(\mathbf{n})}\pi_{(\mathbf{n}\cdot\boldsymbol{\zeta},\xi)}\,U^{\osc}_{\mathbf{n},\xi}(z)\,e^{i\,\xi\,\psi_d},
	\end{equation*} 
	so that, according to Remark \ref{remarque proj pi et Q bornes} ensuring that the projectors $ \pi_{\alpha} $ are uniformly bounded,
	\begin{equation*}
	\left|\E\,U^{\osc}_{\mathbf{n}}(z,\psi_d)\right|^2\leq C\,N\sum_{\xi\in\mathcal{C}(\mathbf{n})}\left|U^{\osc}_{\mathbf{n},\xi}(z)\right|^2,
	\end{equation*}
	using the fact that the cardinality of $ \mathcal{C}(\mathbf{n}) $ is at most $ N $.
	On an other hand, according to \eqref{eq proj E forme U n}, for $ \mathbf{n} $ in $ \Z^m $ and $ \xi $ in $ \R $, we have
	\begin{equation*}
	U^{\osc}_{\mathbf{n},\xi}(z)=\lim_{R\rightarrow +\infty}\inv{R}\int_0^R U^{\osc}_{\mathbf{n}}(z,\psi_d)\,e^{-i\,\xi\,\psi_d}\,d\psi_d,
	\end{equation*}
	so that, using Cauchy-Schwarz inequality, 
	\begin{equation*}
	\left|\E\,U^{\osc}_{\mathbf{n}}(z,\psi_d)\right|^2\leq C\,N^2\,\lim_{R\rightarrow +\infty}\inv{R}\int_0^R\left| U^{\osc}_{\mathbf{n}}(z,\psi_d)\right|^2\,d\psi_d.
	\end{equation*}
	Then Fatou's lemma is applied to the sum with respect to $ \mathbf{n} $ in $ \Z^m $ and the integration with respect to $ z' $ in $ \omega_T $ to get
	\begin{align*}
	\sum_{\mathbf{n}\in\Z^m}\norme{\E\,U^{\osc}_{\mathbf{n}}(.,x_d,\psi_d)}_{L^2(\omega_T)}^2&\leq C\,N^2\,\liminf_{R\rightarrow +\infty}\inv{R}\int_0^R \sum_{\mathbf{n}\in\Z^m}\norme{U^{\osc}_{\mathbf{n}}(.,x_d,\psi_d)}_{L^2(\omega_T)}^2\,d\psi_d,
	\intertext{thus}
	\norme{\E\,U(.,x_d,.,\psi_d)}^2_{L^2(\omega_T\times\T^m)}&\leq C\,N^2\, \sup_{\psi_d>0} \norme{U^{\osc}(.,x_d,.,\psi_d)}_{L^2(\omega_T\times\T^m)}^2,
	\end{align*}
	that is to say
	\begin{equation*}
	\norme{\E\,U}_{\mathcal{E}_{0,T}}\leq \sqrt{C}\,N\norme{U}_{\mathcal{E}_{0,T}}.
	\end{equation*}
	The projector $ \E $ thus admits a uniformly bounded unique extension from $ \P^{\osc}_{0,T} $ to $ \P^{\osc}_{0,T} $. The result for the space $ \P^{\osc}_{s,T} $ for $ s\geq 1 $ is obtained by observing that the projector $ \E $ commutes with the partial derivatives with respect to $ z' $ and $ \theta $. The same argument applies to the projector $ \tilde{\E^i} $, which concludes the proof relative to the oscillating part.
	
	Concerning the evanescent part of $ \E $, it must be proved that if $ U^{\ev} $ writing
	\begin{equation*}
	U^{\ev}(z,\theta,\psi_d)=\sum_{\mathbf{n}\in\Z^m}U^{\ev}_{\mathbf{n}}(z,\psi_d)\,e^{i\,\mathbf{n}\cdot\theta}
	\end{equation*}
	belongs to $ \P^{\ev}_{s+\lceil b_1 \rceil,T} $, then $ \E\,U^{\ev} $ is in $ \P_{s,T}^{\ev} $. By definition of  the norm of $ \mathcal{E}_{s,T} $ and according to the Parseval's identity, we obtain
	\begin{align*}
	\norme{\E\,U^{\ev}}_{\mathcal{E}_{s,T}}^2&=\sup_{x_d>0,\psi_d>0}\sum_{\mathbf{n}\in\Z^m\privede{0}}\sum_{l=0}^s(1+|\mathbf{n}|^2)^{s-l}\norme{e^{\psi_d\mathcal{A}(\mathbf{n}\cdot\boldsymbol{\zeta})}\,\Pi^{e}_{\C^N}(\mathbf{n}\cdot\boldsymbol{\zeta})\,U^{\ev}_{\mathbf{n}}(.,0)}^2_{H^l_+(\omega_T)}\\
	&\leq \sup_{x_d>0,\psi_d>0}\sum_{\mathbf{n}\in\Z^m\privede{0}}\sum_{l=0}^s(1+|\mathbf{n}|^2)^{s-l}\,c_1^2\,|\mathbf{n}\cdot\boldsymbol{\zeta}|^{2\,b_1}\norme{U^{\ev}_{\mathbf{n}}(.,0)}^2_{H^l_+(\omega_T)},
	\end{align*}
	according to estimate \eqref{eq controle exp t A Pi C^N - t positif} of Proposition \ref{prop controle exp t A Pi } and recalling the notations of this result. It leads to the following estimate
	\begin{equation*}
	\norme{\E\,U^{\ev}}_{\mathcal{E}_{s,T}}\leq C \norme{U^{\ev}}_{\mathcal{E}_{s+\lceil b_1 \rceil,T}}.
	\end{equation*}
We investigate now the convergence towards zero in $ H^s(\omega_T\times\T^m) $ of the profile $ \E\,U^{\ev}(x_d,\psi_d) $, for every fixed $ x_d\geq 0 $. Consider $ \epsilon>0 $. Using the convergence of the following sum,
\begin{equation*}
	\sum_{\mathbf{n}\in\Z^m\privede{0}}\sum_{l=0}^{s}(1+|\mathbf{n}|^2)^{s-l}\,|\mathbf{n}\cdot\boldsymbol{\zeta}|^{2\,b_1}\norme{U^{\ev}_{\mathbf{n}}(.,x_d,0)}^2_{H^l_+(\omega_T)},
\end{equation*}
which is bounded by $ C\norme{U^{\ev}}_{\mathcal{E}_{s+\lceil b_1 \rceil,T}}^2 $, there exists $ M>0 $ such that
\begin{equation*}
	\sum_{|\mathbf{n}|>M}\sum_{l=0}^{s}(1+|\mathbf{n}|^2)^{s-l}\,|\mathbf{n}\cdot\boldsymbol{\zeta}|^{2\,b_1}\norme{U^{\ev}_{\mathbf{n}}(.,x_d,0)}^2_{H^l_+(\omega_T)}\leq \epsilon.
\end{equation*}
 Then we have
\begin{align*}
	\norme{\E\,U^{\ev}(x_d,\psi_d)}_{H^s(\omega_T\times\T^m)}^2&=\sum_{0<|\mathbf{n}|\leq M}\sum_{l=0}^s(1+|\mathbf{n}|^2)^{s-l}\norme{e^{\psi_d\mathcal{A}(\mathbf{n}\cdot\boldsymbol{\zeta})}\,\Pi^{e}_{\C^N}(\mathbf{n}\cdot\boldsymbol{\zeta})\,U^{\ev}_{\mathbf{n}}(.,0)}^2_{H^l_+(\omega_T)}\\
	&\quad +\sum_{|\mathbf{n}|>M}\sum_{l=0}^s(1+|\mathbf{n}|^2)^{s-l}\norme{e^{\psi_d\mathcal{A}(\mathbf{n}\cdot\boldsymbol{\zeta})}\,\Pi^{e}_{\C^N}(\mathbf{n}\cdot\boldsymbol{\zeta})\,U^{\ev}_{\mathbf{n}}(.,0)}^2_{H^l_+(\omega_T)}.
\end{align*}
According to estimate \eqref{eq controle exp t A Pi C^N - t positif} of Proposition \ref{prop controle exp t A Pi } and by construction of $ M $, the second sum of the right-hand side is less or equal to $ c_1^2\,\epsilon $. For the first one, according to the same estimate \eqref{eq controle exp t A Pi C^N - t positif}, we have
\begin{multline*}
	\sum_{0<|\mathbf{n}|\leq M}\sum_{l=0}^s(1+|\mathbf{n}|^2)^{s-l}\norme{e^{\psi_d\mathcal{A}(\mathbf{n}\cdot\boldsymbol{\zeta})}\,\Pi^{e}_{\C^N}(\mathbf{n}\cdot\boldsymbol{\zeta})\,U^{\ev}_{\mathbf{n}}(.,0)}^2_{H^l_+(\omega_T)}\\
	\leq \sum_{0<|\mathbf{n}|\leq M}\sum_{l=0}^sc_1^2\,(1+|\mathbf{n}|^2)^{s-l}\,|\mathbf{n}\cdot\boldsymbol{\zeta}|^{2\,b_1}\,e^{-2\,c_1\,\psi_d\,|\mathbf{n}\cdot\boldsymbol{\zeta}|^{-b_1}}\norme{U^{\ev}_{\mathbf{n}}(.,0)}^2_{H^l_+(\omega_T)}.
\end{multline*}
The right-hand side is a finite sum of functions of $ \psi_d $ converging to zero at infinity, so there exists $ B>0 $ such that for $ |\psi_d|>B $, the right-hand side is less or equal to $ \epsilon $. We get finally
\begin{equation*}
	\norme{\E\,U^{\ev}(x_d,\psi_d)}_{H^s(\omega_T\times\T^m)}^2\leq (1+c_1^2)\,\epsilon,
\end{equation*}
for all $ \psi_d $ such that $ |\psi_d|>B $, and the aimed convergence follows.
	The profile $ \E\,U^{\ev} $ thus satisfies the condition of Definition \ref{def profils ev} of the evanescent profiles of $ \P^{\ev}_{s,T} $, concluding the proof.
	
	Concerning the projectors $ \E^{\inc}_{\res} $ and $ \tilde{\E^i}^{\inc}_{\res} $, the proof is analogous to the one for the oscillating part of $ \E $ and $ \tilde{\E} $.
\end{proof}

Now that the projectors $ \E $, $ \tilde{\E^i} $, $ \E^{\inc}_{\res} $ and $ \tilde{\E^i}^{\inc}_{\res} $ are well-defined, it can be proved that the kernel in $ \P_{s,T} $ of the fast operator $ \mathcal{L}(\partial_{\theta},\partial_{\psi_d}) $ is actually given by the range of the projector $ \E $.

\begin{definition}
	For $ s\geq 0 $ and $ T>0 $, we denote by $ \N_{s,T} $ the range in $ \P_{s,T} $ of $ \P_{s+\lceil b_1 \rceil,T} $ projected by $ \E $. We also denote $ \N^{\osc}_{s,T}:=\N_{s,T}\cap\P_{s,T}^{\osc} $ and $ \N^{\ev}_{s,T}:=\N_{s,T}\cap\P_{s,T}^{\ev} $.
\end{definition}

\begin{lemma}[{\cite[Lemma 6.2.3.]{JolyMetivierRauch1995Coherent}}]
	The space $ \N_{s,T} $ is the kernel in $ \P_{s,T} $ of the operator $ \mathcal{L}(\partial_{\theta},\partial_{\psi_d}) $.
\end{lemma}

\begin{proof}
	Since $ \N_{s,T} $ is equal to the kernel $ \ker I-\E $ in $ \P_{s,T} $, it must be shown that the equality of kernels $ \ker I-\E=\ker \mathcal{L}(\partial_{\theta},\partial_{\psi_d}) $ holds in $ \P_{s,T} $. Let $ U=U^{\osc}+U^{\ev} $ be in $ \ker \mathcal{L}(\partial_{\theta},\partial_{\psi_d}) $, and write
	\begin{equation*}
		U(z,\theta,\psi_d)=U^*(z,\psi_d)+\sum_{\mathbf{n}\in\Z^m\privede{0}}U_{\mathbf{n}}(z,\psi_d)\,e^{i\,\mathbf{n}\cdot\theta},
	\end{equation*}
	where $ U_{\mathbf{n}} $ decomposes in $ \P_{s,T}=\P^{\osc}_{s,T}\oplus\P^{\ev}_{s,T} $ as $ U_{\mathbf{n}}=U^{\osc}_{\mathbf{n}}+U^{\ev}_{\mathbf{n}} $ for all $ \mathbf{n} $ in $ \Z^m\privede{0} $.
	Then one gets the following differential equations, 
	\begin{equation*}
	\partial_{\psi_d}\,U^*=0,\qquad \text{and} \qquad \big(-\mathcal{A}(\mathbf{n}\cdot\boldsymbol{\zeta})+\partial_{\psi_d}\big)\,U_{\mathbf{n}}=0,\quad \forall\mathbf{n}\in\Z^m\privede{0}.
	\end{equation*}
	Thus, on one hand, it follows $ U^*(z,\psi_d)=U^*(z) $, which therefore belongs to $ \P_{s,T}^{\osc} $. On the other hand, for every $ \mathbf{n} $ in $ \Z^m $, the amplitude $ U_{\mathbf{n}} $ admits the following expansion, according to  decomposition \eqref{eq decomp C^N E + E -} of $ \C^N $ into subspaces stable under the action of $ \mathcal{A}(\freq) $,
	\begin{equation*}
		U_{\mathbf{n}}=\Pi_{\C^N}^e(\mathbf{n}\cdot\boldsymbol{\zeta})\,U_{\mathbf{n}}+\Pi_{\C^N}^{e,+}(\mathbf{n}\cdot\boldsymbol{\zeta})\,U_{\mathbf{n}}+\sum_{\xi\in\mathcal{C}(\mathbf{n})}\pi_{(\freq,\xi)}\,U_{\mathbf{n}},
	\end{equation*}
	where, recalling the notations of Definition \ref{def projecteurs Pi}, $ \Pi_{\C^N}^e(\mathbf{n}\cdot\boldsymbol{\zeta})\,U_{\mathbf{n}} $ (resp. $ \Pi^{e,+}_{\C^N}(\freq)\,U_{\mathbf{n}} $) belongs to the stable (resp. unstable) elliptic component $ E^e_-(\freq) $ (resp. $ E^e_+(\freq) $), and for $ \xi=\xi_j(\freq)\in\mathcal{C}(\mathbf{n}) $, $ \pi_{(\freq,\xi)}\,U_{\mathbf{n}} $ belongs to the subspace $ \ker L\big(0,\alpha_j(\freq)\big) $ (recall that according to proposition \ref{prop decomp E_-}, for $ j $ in $ \mathcal{I}(\freq) $, we have $ E^j_-(\freq)=\ker L\big(0,\alpha_j(\freq)\big) $, and a similar result holds for $ j $ in $ \mathcal{O}(\freq) $). Using that $ \mathcal{A}(\freq)=-i\tilde{L}(0,\freq) $ and the property of the projectors $ \pi_{(\freq,\xi)} $, we get the following differential equations for each part,
	\begin{subequations}
		\begin{align}
		\partial_{\psi_d}\,\big(\Pi_{\C^N}^e(\mathbf{n}\cdot\boldsymbol{\zeta})\,U_{\mathbf{n}}\big)&=\mathcal{A}(\freq)\,\big(\Pi_{\C^N}^e(\mathbf{n}\cdot\boldsymbol{\zeta})\,U_{\mathbf{n}}\big),\label{eq obtention eq diff ell stable}\\
		\partial_{\psi_d}\,\big(\Pi_{\C^N}^{e,+}(\mathbf{n}\cdot\boldsymbol{\zeta})\,U_{\mathbf{n}}\big)&=\mathcal{A}(\freq)\,\big(\Pi_{\C^N}^{e,+}(\mathbf{n}\cdot\boldsymbol{\zeta})\,U_{\mathbf{n}}\big),\label{eq obtention eq diff ell instable}\\		
		\partial_{\psi_d}\,\big(\pi_{(\freq,\xi)}\,U_{\mathbf{n}}\big)&=i\,\xi\,\big(\pi_{(\freq,\xi)}\,U_{\mathbf{n}}\big),\qquad \forall\xi\in\mathcal{C}(\mathbf{n}).\label{eq obtention eq diff oscillant}
	\end{align}
	\end{subequations}
	Resolving equation \eqref{eq obtention eq diff ell stable}, one get
	\begin{equation*}
		\Pi_{\C^N}^e(\mathbf{n}\cdot\boldsymbol{\zeta})\,U_{\mathbf{n}}(z,\psi_d)=e^{\psi_d\mathcal{A}(\freq)}\,\Pi_{\C^N}^e(\mathbf{n}\cdot\boldsymbol{\zeta})\,U_{\mathbf{n}}(z,0),
	\end{equation*}
	which goes to zero as $ \psi_d $ goes to infinity, so belongs to $ \P_{s,T}^{\ev} $. In the same way, we get, with \eqref{eq obtention eq diff ell instable}, 
	\begin{equation*}
		\Pi_{\C^N}^{e,+}(\mathbf{n}\cdot\boldsymbol{\zeta})\,U_{\mathbf{n}}(z,\psi_d)=e^{\psi_d\mathcal{A}(\freq)}\,\Pi_{\C^N}^{e,+}(\mathbf{n}\cdot\boldsymbol{\zeta})\,U_{\mathbf{n}}(z,0).
	\end{equation*}
	But since $ \Pi^{e,+}_{\C^N}(\freq) $ is the projector on the unstable elliptic component, if $ \Pi_{\C^N}^{e,+}(\mathbf{n}\cdot\boldsymbol{\zeta})\,U_{\mathbf{n}}(z,0) $ is nonzero then $ \Pi_{\C^N}^{e,+}(\mathbf{n}\cdot\boldsymbol{\zeta})\,U_{\mathbf{n}} $ is unbounded, so we have 
	\begin{equation*}
		\Pi_{\C^N}^{e,+}(\mathbf{n}\cdot\boldsymbol{\zeta})\,U_{\mathbf{n}}=0.
	\end{equation*}
	Finally, \eqref{eq obtention eq diff oscillant} gives, for $ \xi $ in $ \mathcal{C}(\mathbf{n}) $, 
	\begin{equation*}
		\pi_{(\freq,\xi)}\,U_{\mathbf{n}}(z,\psi_d)=e^{i\,\xi\,\psi_d}\,\pi_{(\freq,\xi)}\,U_{\mathbf{n}}(z,0),
	\end{equation*}
	which belongs to $ \P_{s,T}^{\osc} $. To summarize, $ U_{\mathbf{n}}=U^{\osc}_{\mathbf{n}}+U^{\ev}_{\mathbf{n}} $ is given by
		\begin{equation*}
			U^{\osc}_{\mathbf{n}}(z,\psi_d)=\sum_{\xi\in\mathcal{C}(\mathbf{n})}\pi_{(\freq,\xi)}\,U_{\mathbf{n}}(z,0)\,e^{i\,\xi\,\psi_d},\qquad
		U^{\ev}_{\mathbf{n}}(z,\psi_d)=e^{\psi_d\mathcal{A}(\freq)}\,\Pi_{\C^N}^{e,+}(\mathbf{n}\cdot\boldsymbol{\zeta})\,U_{\mathbf{n}}(z,0),
		\end{equation*}
	so the equality $ \E\,U=U $ clearly holds.
	
	Conversely, if $ U^{\ev}=\E\,U^{\ev} $, it immediately leads to $ \mathcal{L}(\partial_{\theta},\partial_{\psi_d})\,U^{\ev}=0 $. On an other hand, supposing $ U^{\osc}=\E\,U^{\osc} $, we consider a sequence $ (U^{\osc}_{\nu})_{\nu} $ of trigonometric polynomials converging in $ \mathcal{E}_{s+\lceil b_1 \rceil,T} $ towards $ U^{\osc} $. By continuity of the projector $ \E $, the sequence of trigonometric polynomials $ (\E\,U^{\osc}_{\nu})_{\nu} $ converges in $ \mathcal{E}_{s,T} $ towards $ U^{\osc} $. But one can check immediately that these trigonometric polynomials satisfy $ \mathcal{L}(\partial_{\theta},\partial_{\psi_d})\,\E\,U^{\osc}_{\nu}=0 $, then passing to the limit yields to $ \mathcal{L}(\partial_{\theta},\partial_{\psi_d})\,U^{\osc}=0 $.
\end{proof}

\begin{remark}\label{remarque forme U osc polarise}
	In the proof above, it has been proven in particular that if $ U=U^{\osc}+U^{\ev} $ belongs to $ \N^{\osc}_{s,T}\oplus\N^{\ev}_{s,T} $, then the profile $ U^{\osc} $ writes
	\begin{equation}\label{eq obtention forme U osc noyau E}
	U^{\osc}(z,\theta,\psi_d)=U^*(z)+\sum_{\mathbf{n}\in\Z^m\privede{0}}\sum_{\xi\in\mathcal{C}(\mathbf{n})}U^{\osc}_{\mathbf{n},\xi}(z)\,e^{i\,\mathbf{n}\cdot\theta}\,e^{i\,\xi\,\psi_d},
	\end{equation}
	with, for $ \mathbf{n} $ in $ \Z^m\privede{0} $ and $ \xi $ in $ \mathcal{C}(\mathbf{n}) $, $ \pi_{(\mathbf{n}\cdot\boldsymbol{\zeta},\xi)}\,U^{\osc}_{\mathbf{n},\xi}=U^{\osc}_{\mathbf{n},\xi} $, and the profile $ U^{\ev} $ writes
	\begin{equation}\label{eq obtention forme U ev noyau E}
	U^{\ev}(z,\theta,\psi_d)=\sum_{\mathbf{n}\in\Z^m\privede{0}}e^{\psi_d\mathcal{A}(\mathbf{n}\cdot\boldsymbol{\zeta})}\,\Pi^{e}_{\C^N}(\mathbf{n}\cdot\boldsymbol{\zeta})\,U^{\ev}_{\mathbf{n}}(z,0)\,e^{i\,\mathbf{n}\cdot\theta}.
	\end{equation}
\end{remark}

\bigskip

The previous remark leads to the following result, which links the norm $ \mathcal{C}_b(\R^+_{\psi_d},L^2(\omega_T\times\T^m)) $ of a profile of $ \N^{\osc}_{0,T} $ and its incoming scalar product \eqref{eq def prod scal rentrant} with itself. This result will be used in the following to deduce from a priori estimates on the scalar product a priori estimates on the norm $ \P_{s,T}^{\osc} $. It is analogous to \cite[Lemma 6.2.4]{JolyMetivierRauch1995Coherent}, in a weaker form (because of a lack of symmetry in our context).

\begin{lemma}\label{lemme correspondance norme prod scal}
	There exists a constant $ C>0 $ such that for every profile $ U^{\osc} $ of $ \N^{\osc}_{0,T} $, we have, for $ x_d\geq 0 $,
	\begin{equation*}
		C\norme{U^{\osc}}^2_{\mathcal{C}_b(\R^+_{\psi_d},L^2(\omega_T\times\T^m))}(x_d)\leq \prodscal{U^{\osc}}{U^{\osc}}_{\inc}(x_d)\leq\norme{U^{\osc}}^2_{\mathcal{C}_b(\R^+_{\psi_d},L^2(\omega_T\times\T^m))}(x_d).
	\end{equation*}
\end{lemma}

\begin{proof}
	The second inequality is obvious by definition of the scalar product $ \prodscal{.}{.}_{\inc} $, since we have
	\begin{align*}
		\prodscal{U^{\osc}}{U^{\osc}}_{\inc}(x_d)&=\lim_{R\rightarrow+\infty}\inv{R}\int_{0}^{R}\norme{U^{\osc}}^2_{L^2(\omega_T\times\T^m)}(x_d,\psi_d)\,d\psi_d\leq \sup_{\psi_d\geq 0}\norme{U^{\osc}}^2_{L^2(\omega_T\times\T^m)}(x_d,\psi_d).
	\end{align*}
	On the other hand, according to Remark \ref{remarque forme U osc polarise}, if $ U^{\osc} $ belongs to $ \N_{0,T}^{\osc} $, then the profile writes
	\begin{equation*}
		U^{\osc}(z,\theta,\psi_d)=U^*(z)+\sum_{\mathbf{n}\in\Z^m\privede{0}}\sum_{\xi\in\mathcal{C}(\mathbf{n})}U^{\osc}_{\mathbf{n},\xi}(z)\,e^{i\,\mathbf{n}\cdot\theta}\,e^{i\,\xi\,\psi_d},
	\end{equation*}
	where, for $ \mathbf{n} $ in $ \Z^m\privede{0} $ and $ \xi $ in $ \mathcal{C}(\mathbf{n}) $, each amplitude satisfies $ U^{\osc}_{\mathbf{n},\xi}=\pi_{(\mathbf{n}\cdot\boldsymbol{\zeta},\xi)}\,U^{\osc}_{\mathbf{n},\xi} $. The Parseval's identity then gives
	\begin{align*}
		\norme{U^{\osc}}^2_{L^2(\omega_T\times\T^m)}(x_d,\psi_d)&=\norme{U^*}^2_{L^2(\omega_T)}+\sum_{\mathbf{n}\in\Z^m\privede{0}}\norme{ \sum_{\xi\in\mathcal{C}(\mathbf{n})}U^{\osc}_{\mathbf{n},\xi}(z)\,e^{i\,\xi\,\psi_d}}^2_{L^2(\omega_T)}.
		\intertext{Therefore, since for all $ \mathbf{n} $ in $ \Z^m\privede{0} $, the set $ \mathcal{C}(\mathbf{n}) $ is of cardinality at most $ N $, we have}
		\norme{U^{\osc}}^2_{L^2(\omega_T\times\T^m)}(x_d,\psi_d)&\leq \norme{U^*}^2_{L^2(\omega_T)}+N\sum_{\mathbf{n}\in\Z^m\privede{0}} \sum_{\xi\in\mathcal{C}(\mathbf{n})}\norme{U^{\osc}_{\mathbf{n},\xi}}^2_{L^2(\omega_T)}(x_d)\\
		&\leq C \prodscal{U^{\osc}}{U^{\osc}}_{\inc}(x_d),
	\end{align*}
	according to formula \eqref{eq prod scal pol trigo rentrants}. The first inequality of Lemma \ref{lemme correspondance norme prod scal}  follows finally by passing to the supremum in $ \psi_d\geq 0 $.
\end{proof}

\subsection{Reducing the system}\label{subsection reducing}

It is shown in this part that in every solution to system \eqref{eq obtention U 1} there occur only incoming modes (in particular every solution is of zero mean), and every solution is supported in a finite interval in $ x_d $. We also show that system \eqref{eq obtention U 1} decouples according to the oscillating and the evanescent part, and even, for the oscillating part, according to the set $ \F^{\inc}_{\res} $ of resonant modes and each non resonant mode. More precisely, the following result is proved. Recall that $ s_0 $ is given by $ s_0=h+(d+m)/2 $ where $ h $ is an integer greater or equal to $ (3+a_1)/2 $ occurring in estimate \eqref{eq est gamma n n}, with $ a_1$ the real number of Assumption \ref{hypothese petits diviseurs 1}.

\begin{proposition}\label{prop equivalence systeme}
	Consider $ T>0 $, and $ s>s_0 $. Every solution $ U $ in $ \P_{s,T} $ of system \eqref{eq obtention U 1} is such that its oscillating part $ U^{\osc} $ features only incoming modes. Furthermore, system \eqref{eq obtention U 1} on $ U=U^{\osc}+U^{\ev} $ in $ \P_{s,T} $
	is equivalent to the following decoupled systems, the first one involving the resonant incoming modes,
	\begin{subequations}\label{eq red syst osc res}
		\begin{align}
			\E^{\inc}_{\res}\, U_{\res}^{\osc}&=U_{\res}^{\osc} \label{eq red syst osc E U = U res}  \\
			\tilde{\E^i}^{\inc}_{\res}\Big[\tilde{L}(0,\partial_z)\,\beta_TU_{\res}^{\osc}+\sum_{j=1}^m\tilde{L}_1(\beta_TU_{\res}^{\osc},\zeta_j)\,\partial_{\theta_j}\beta_TU_{\res}^{\osc}\Big]&=0 \label{eq red syst osc E i U res}\\
			\big(U_{\res}^{\osc}\big)_{|x_d=0,\psi_d=0}&=H_{\res}^{\osc} \label{eq red syst osc cond bord res} \\[5pt]
			\big(U_{\res}^{\osc}\big)_{|t\leq 0}&=0, \label{eq red syst osc cond init res}
		\end{align}
	\end{subequations}
	then the system verified by each non resonant incoming mode, for $ (\mathbf{n}_0,\xi_0) $ in  $\big(\B_{\Z^m}\times\mathcal{C}_{\inc}(\mathbf{n}_0)\big)\setminus\F^{\inc}_{\res} $,
	\begin{subequations}\label{eq red syst osc non res}
		\begin{align}
			\tilde{X}_{(\mathbf{n}_0\cdot\boldsymbol{\zeta},\xi_0)}S_{\mathbf{n}_0,\xi_0}+\Gamma\big((\mathbf{n}_0,\xi_0),(\mathbf{n}_0,\xi_0)\big)S_{\mathbf{n}_0,\xi_0}\partial_{\Theta}S_{\mathbf{n}_0,\xi_0}&=0\label{eq red syst osc Burgers non res}\\
			\big(S_{\mathbf{n}_0,\xi_0}\big)_{|x_d=0}&=h_{\mathbf{n}_0,\xi_0}\label{eq red syst osc cond bord non res}\\
			\big(S_{\mathbf{n}_0,\xi_0}\big)_{|t\leq0}&=0,\label{eq red syst osc cond init non res}
		\end{align}
	\end{subequations}
	and finally the system for the evanescent part $ U^{\ev} $,
	\begin{subequations}\label{eq red syst ev}
		\begin{align}
			\E\, U^{\ev}&=U^{\ev} \label{eq red syst ev E U = U}  \\
			U^{\ev}_{|x_d=0,\psi_d=0}&=H^{\ev}, \label{eq red syst ev cond bord} 
		\end{align}
	\end{subequations}
	where, if the solution $ U^{\osc} $ (occurring only incoming modes and being polarized) writes
	\begin{equation*}
		U^{\osc}(z,\theta,\psi_d)=\sum_{\substack{\mathbf{n}_0\in\B_{\Z^m}\\\xi_0\in\mathcal{C}_{\inc}(\mathbf{n}_0)}}\sum_{\lambda\in\Z^*}\sigma
		_{\lambda,\mathbf{n}_0,\xi_0}(z)\,e^{i\,\lambda\mathbf{n}_0\cdot\theta}\,e^{i\,\lambda\xi_0\,\psi_d}\,E(\mathbf{n}_0,\xi_0),
	\end{equation*}
	where $ \sigma_{\lambda,\mathbf{n}_0,\xi_0} $ are scalar functions, then the resonant part $ U^{\osc}_{\res} $ is given by
	\begin{equation*}
		U_{\res}^{\osc}(z,\theta,\psi_d)=\sum_{(\mathbf{n}_0,\xi_0)\in\F^{\inc}_{\res}}\sum_{\lambda\in\Z^*}\sigma_{\lambda,\mathbf{n}_0,\xi_0}(z)\,e^{i\,\lambda\mathbf{n}_0\cdot\theta}\,e^{i\,\lambda\xi_0\,\psi_d}\,E(\mathbf{n}_0,\xi_0),
	\end{equation*}
	and the scalar component $ S_{\mathbf{n}_0,\xi_0} : \Omega_T\times\T \rightarrow \C $ for each non resonant direction $ (\mathbf{n}_0,\xi_0) $ in  $\big(\B_{\Z^m}\times\mathcal{C}_{\inc}(\mathbf{n}_0)\big)\setminus\F^{\inc}_{\res} $, is given by
	\begin{equation*}
		S_{\mathbf{n}_0,\xi_0}(z,\Theta):=\sum_{\lambda\in\Z^*}\sigma_{\lambda,\mathbf{n}_0,\xi_0}(z)\,e^{i\lambda\,\Theta},
	\end{equation*}
	where the function $ \beta_{T} $ of $ x_d $, of class $ \mathcal{C}^{\infty} $, equals 1 on $ [0,\V^*T] $ and 0 on $ [2\V^*T,+\infty) $ (where $ \V^* $ has been defined in Lemma \ref{lemme vitesse de groupe bornees}), and where $ H^{\osc}_{\res} $, $ h_{\mathbf{n}_0,\xi_0} $ for  $ (\mathbf{n}_0,\xi_0) $ in  $\big(\B_{\Z^m}\times\mathcal{C}_{\inc}(\mathbf{n}_0)\big)\setminus\F^{\inc}_{\res} $, and $ H^{\ev} $ are defined from $ G $ by the formulas 
	\begin{subequations}\label{eq red def H}
		\begin{align}\label{eq red def H osc res}
			H^{\osc}_{\res}(z',\theta)&:=\sum_{(\mathbf{n}_0,\xi_0)\in\F^{\inc}_{\res}}\sum_{\lambda\in\Z^*}\Pi_-^{j(\lambda\mathbf{n}_0,\lambda\xi_0)}(\lambda\mathbf{n}_0\cdot\boldsymbol{\zeta})\,\big(B_{|E_-(\lambda\mathbf{n}_0\cdot\boldsymbol{\zeta})}\big)^{-1}G_{\lambda\mathbf{n}_0}(z')\,e^{i\lambda\mathbf{n}_0\cdot\theta},
			\\\label{eq red def H osc non res}h_{\mathbf{n}_0,\xi_0}(z',\Theta)&:=\sum_{\lambda\in\Z^*}\prodscal{\Pi_-^{j(\lambda\mathbf{n}_0,\lambda\xi_0)}(\lambda\mathbf{n}_0\cdot\boldsymbol{\zeta})\,\big(B_{|E_-(\lambda\mathbf{n}_0\cdot\boldsymbol{\zeta})}\big)^{-1}G_{\lambda\mathbf{n}_0}(z')}{E(\mathbf{n}_0,\xi_0)}_{\C^N}e^{i\lambda\Theta},
			\\H^{\ev}(z',\theta)&:=\sum_{\mathbf{n}\in\Z^m\privede{0}}\Pi_-^e(\freq)\,\big(B_{|E_-(\freq)}\big)^{-1}G_{\mathbf{n}}(z')\,e^{i\mathbf{n}\cdot\theta}\label{eq red def H ev},
		\end{align}
	\end{subequations}
	where, for $ \mathbf{n} $ in $ \Z^m\privede{0} $ and $ \xi $ in $ \mathcal{C}_{\inc}(\mathbf{n}) $, $ j(\mathbf{n},\xi) $ is the index such that $ \xi=\xi_{j(\mathbf{n},\xi)}(\freq) $. Recall that amplitudes $ G_{\mathbf{n}} $ of the function $ G $ have been defined by \eqref{eq def G_n}, and that projectors $ \E^{\inc}_{\res} $ and $ \tilde{\E^i}^{\inc}_{\res} $ have been introduced in Definition \ref{def proj E res}. Note that in these notations, solution $ U $ decomposes as
	\begin{equation*}
		U(z,\theta,\psi_d)=U^{\osc}_{\res}(z,\theta,\psi_d)+\sum_{\substack{(\mathbf{n}_0,\xi_0) \in\\(\B_{\Z^m}\times\mathcal{C}_{\inc}(\mathbf{n}_0))\setminus\F^{\inc}_{\res}}}S_{\mathbf{n}_0,\xi_0}(z,\mathbf{n}_0\cdot\theta+\xi_0\,\psi_d)\,E(\mathbf{n}_0,\xi_0)+U^{\ev}(z,\theta,\psi_d).
	\end{equation*}
\end{proposition}

To prove Theorem \ref{thm existence profils}, it is therefore equivalent to prove that there exists solutions $ U^{\osc}_{\res} $, $ S_{\mathbf{n}_0,\xi_0} $ and $ U^{\ev} $ to systems \eqref{eq red syst osc res},  \eqref{eq red syst osc non res} and $ \eqref{eq red syst ev} $.

In this part dedicated to the proof of Proposition \ref{prop equivalence systeme}, we consider a solution $ U $ to \eqref{eq obtention U 1} sufficiently regular, and we start by showing that its mean value $ U^* $ is zero, by extracting from \eqref{eq obtention U 1} a homogeneous linear hyperbolic system satisfied by it. To show that there is no outgoing mode, the scalar product \eqref{eq def prod scal sortant} for outgoing modes is used, which is defined for profiles of compact support with respect to $ x_d $. Thus we must prove before that the considered solution $ U $ is of compact support with respect to $ x_d $. Then the outgoing modes are isolated in equation \eqref{eq obtention U 1}, deducing that they are zero. First the left term of equation \eqref{eq obtention E i U} is rewritten.

\subsubsection{Rewriting the evolution equation} 

According to remark \ref{remarque forme U osc polarise}, since $ U $ satisfies the polarization condition \eqref{eq obtention E U}, and according to Remark \ref{remarque forme U osc polarise}, the profile writes $ U=U^{\osc}+U^{\ev} $, where
\begin{subequations}
	\begin{equation}\label{eq red pola osc}
	U^{\osc}(z,\theta,\psi_d)=U^*(z)+\sum_{\mathbf{n}\in\Z^m\privede{0}}\sum_{\xi\in\mathcal{C}(\mathbf{n})}U^{\osc}_{\mathbf{n},\xi}(z)\,e^{i\,\mathbf{n}\cdot\theta}\,e^{i\,\xi\,\psi_d},
	\end{equation}
	and
	\begin{equation}\label{eq red pola ev}
	U^{\ev}(z,\theta,\psi_d)=\sum_{\mathbf{n}\in\Z^m\privede{0}}e^{\psi_d\,\mathcal{A}(\mathbf{n}\cdot\boldsymbol{\zeta})}\,\Pi^e_{\C^N}(\mathbf{n}\cdot\boldsymbol{\zeta})\,U^{\ev}_{\mathbf{n}}(z,0)\,e^{i\mathbf{n}\cdot\theta},
	\end{equation}
\end{subequations}
with $ \pi_{(\freq,\xi)}\,U^{\osc}_{\mathbf{n},\xi}=U^{\osc}_{\mathbf{n},\xi} $ for all $ \mathbf{n},\xi $.
Then $ U^{\osc} $ is rewritten to take advantage of collinearities, using notations of Part \ref{notations 2}. 
Let $ \mathbf{n} $ be in $ \Z^m\privede{0} $, and $ \mathbf{n}_0 $ in $ \mathcal{B}_{\Z^m} $, $ \lambda $ in $ \Z^* $ such that $ \mathbf{n}=\lambda\,\mathbf{n}_0 $, and let also $ \xi $ be in $ \mathcal{C}(\mathbf{n}) $. Since the set $ \mathcal{C} $ is homogeneous of degree 1, we have  $ \mathcal{C}(\mathbf{n})=\lambda\,\mathcal{C}(\mathbf{n}_0) $, so there exists $ \xi_0 $ in $ \mathcal{C}(\mathbf{n}_0) $ such that $ \xi=\lambda\,\xi_0 $. By polarization of the profile $ U $, the amplitude $ U^{\osc}_{\mathbf{n},\xi} $ belongs to the kernel of $ L\big(0,(\mathbf{n}\cdot\boldsymbol{\zeta},\xi)\big) $, which is given according to Definition \ref{def E} by $ \Vect E(\mathbf{n},\xi)=\Vect E(\mathbf{n}_0,\xi_0) $.
One may thus write
\begin{equation*}
U^{\osc}_{\mathbf{n},\xi}(z)=\sigma_{\lambda,\mathbf{n}_0,\xi_0}(z)\,E(\mathbf{n}_0,\xi_0),
\end{equation*}
where $ \sigma_{\lambda,\mathbf{n}_0,\xi_0} $ is a scalar function defined on $ \Omega_T $. Since the profile $ U^{\osc} $ is assumed to be real, coefficients $ \sigma_{\lambda,\mathbf{n}_0,\xi_0} $ satisfy $ \sigma_{-\lambda,\mathbf{n}_0,\xi_0}=\bar{\sigma_{\lambda,\mathbf{n}_0,\xi_0} } $ for all $ \lambda $, $ \mathbf{n}_0 $ and $ \xi_0 $.
In this notation, the profile $ U^{\osc} $ writes
\begin{equation*}
U^{\osc}(z,\theta,\psi_d)=U^*(z)+\sum_{\substack{\mathbf{n}_0\in\B_{\Z^m}\\\xi_0\in\mathcal{C}(\mathbf{n}_0)}}\sum_{\lambda\in\Z^*}\sigma_{\lambda,\mathbf{n}_0,\xi_0}(z)\,e^{i\,\lambda\mathbf{n}_0\cdot\theta}\,e^{i\,\lambda\xi_0\,\psi_d}\,E(\mathbf{n}_0,\xi_0).
\end{equation*}
Note that according to identity \eqref{eq prod scal pol trigo sortants} and since the vectors $ E(\mathbf{n}_0,\xi_0) $ are of norm 1, when it is well-defined, the scalar product $ \prodscal{U^{\osc}}{U^{\osc}}_{\out}(t) $ is given in these notations by
\begin{equation*}
	\prodscal{U^{\osc}}{U^{\osc}}_{\out}(t) =(2\pi)^m\norme{U^*}_{L^2(\R^{d-1}\times\R_+)}^2(t)+(2\pi)^m\sum_{\substack{\mathbf{n}_0\in\B_{\Z^m}\\\xi_0\in\mathcal{C}(\mathbf{n}_0)}}\sum_{\lambda\in\Z^*}\norme{\sigma_{\lambda,\mathbf{n}_0,\xi_0}}_{L^2(\R^{d-1}\times\R_+)}^2(t).
\end{equation*}
Recall that the scalar product $ \prodscal{.}{.}_{\out} $ is defined only for profiles with compact support in the normal direction, but we will prove that every solution of \eqref{eq obtention U 1} is indeed compactly supported in the $ x_d $-direction before using this scalar product.

Since the projector $ \tilde{\E^i} $ occurring in equation \eqref{eq obtention E i U} only acts on oscillating profiles, the oscillating part of the term $ \tilde{L}(0,\partial_z)\,U+\sum_{j=1}^m\tilde{L}_1(U,\zeta_j)\,\partial_{\theta_j}U $ must be determined. On one hand, the oscillating part of $ \tilde{L}(0,\partial_z)\,U $ is given by
\begin{equation*}
\tilde{L}(0,\partial_z)\,U^{\osc}=\tilde{L}(0,\partial_z)\,U^*+\sum_{\substack{\mathbf{n}_0\in\B_{\Z^m}\\\xi_0\in\mathcal{C}(\mathbf{n}_0)}}\sum_{\lambda\in\Z^*}\tilde{L}(0,\partial_z)\,\sigma_{\lambda,\mathbf{n}_0,\xi_0}\,e^{i\,\lambda\mathbf{n}_0\cdot\theta}\,e^{i\,\lambda\xi_0\,\psi_d}\,E(\mathbf{n}_0,\xi_0).
\end{equation*}
These two terms correspond to terms \eqref{eq moy E i L U trsp moy} and \eqref{eq moy E i L U trsp osc} of equation \eqref{eq moy E i L U} below.

On the other hand, according to Lemma \ref{lemme propr algebre} concerning the algebra properties of the space of profiles $ \P_{s,T} $, the oscillating part of the quadratic term $ \sum_{j=1}^m\tilde{L}_1(U,\zeta_j)\,\partial_{\theta_j}U $ is given by
\begin{align}\label{eq moy partie osc L 1 U}
	\sum_{j=1}^m\tilde{L}_1(U^{\osc},\zeta_j)\,\partial_{\theta_j}U^{\osc}
	&=\sum_{\substack{\mathbf{n}_0\in\B_{\Z^m}\\\xi_0\in\mathcal{C}(\mathbf{n}_0)}}\sum_{\lambda\in\Z^*}\tilde{L}_1(U^*,i\,\lambda\mathbf{n}_0\cdot\boldsymbol{\zeta})\,E(\mathbf{n}_0,\xi_0)\,\sigma_{\lambda,\mathbf{n}_0,\xi_0}\,e^{i\,\lambda\mathbf{n}_0\cdot\theta}\,e^{i\,\lambda\xi_0\,\psi_d}\\
	\nonumber&+\sum_{\mathbf{n}_1,\mathbf{n}_2\in\B_{\Z^m}}\sum_{\substack{\xi_1\in\mathcal{C}(\mathbf{n}_1)\\\xi_2\in\mathcal{C}(\mathbf{n}_2)}}\sum_{\lambda_1,\lambda_2\in\Z^*}\tilde{L}_1(E(\mathbf{n}_1,\xi_1),i\,\lambda_2\mathbf{n}_2\cdot\boldsymbol{\zeta})\,E(\mathbf{n}_2,\xi_2)\,\\
	\nonumber&\quad\sigma_{\lambda_1,\mathbf{n}_1,\xi_1}\,\sigma_{\lambda_2,\mathbf{n}_2,\xi_2}\,e^{i\,(\lambda_1\mathbf{n}_1+\lambda_2\mathbf{n}_2)\cdot\theta}\,e^{i\,(\lambda_1\xi_1+\lambda_2\xi_2)\,\psi_d}.
\end{align}

The first term of the right hand side of equation \eqref{eq moy partie osc L 1 U} corresponds to term \eqref{eq moy E i L U res moy} of equation \eqref{eq moy E i L U} below. In the second term of the right hand side of equation \eqref{eq moy partie osc L 1 U}, since the projectors $\tilde{\pi}_{\alpha}$ appear in the projector $ \tilde{\E^i} $, only the frequencies $ \big((\lambda_1\mathbf{n}_1+\lambda_2\mathbf{n}_2)\cdot\boldsymbol{\zeta},\lambda_1\xi_1+\lambda_2\xi_2\big) $ that are characteristic will be preserved.
\begin{enumerate}[label=\roman*), leftmargin=0.8cm]
	\item If  $ \mathbf{n}_1=\mathbf{n}_2 $, $\xi_1=\xi_2$ and $ \lambda_1=-\lambda_2 $, the created frequency is zero, so it is characteristic. This non oscillating term corresponds to term \eqref{eq moy E i L U res 0} of equation \eqref{eq moy E i L U} below.
	\item  If $ \mathbf{n}_1=\mathbf{n}_2 $, $\xi_1=\xi_2$ and $ \lambda_1+\lambda_2\neq0 $, then the nonzero frequency obtained is given by $ (\lambda_1+\lambda_2)\,(\mathbf{n}_1\cdot\boldsymbol{\zeta},\xi_1) $ which is characteristic. This is called \emph{self-interaction} of the frequency $ (\mathbf{n}_1\cdot\boldsymbol{\zeta},\xi_1) $ with itself, and constitutes term \eqref{eq moy E i L U AI} of equation  \eqref{eq moy E i L U}.
	\item Finally, in the remaining cases, if the nonzero frequency obtained $ \lambda_1\,(\mathbf{n}_1\cdot\boldsymbol{\zeta},\xi_1)+ \lambda_2\,(\mathbf{n}_2\cdot\boldsymbol{\zeta},\xi_2) $ is characteristic, then it corresponds to a resonance in the sense of Definition \ref{def ens resonances}. Namely there exist $ \lambda_0 $ in $ \Z^* $, $ \mathbf{n}_0 $ in $ \B_{\Z^m} $ and $ \xi_0 $ in $ \mathcal{C}(\mathbf{n}_0) $ such that
	\begin{equation*}
	 \lambda_1\,(\mathbf{n}_1\cdot\boldsymbol{\zeta},\xi_1)+ \lambda_2\,(\mathbf{n}_2\cdot\boldsymbol{\zeta},\xi_2)= \lambda_0\,(\mathbf{n}_0\cdot\boldsymbol{\zeta},\xi_0),
	\end{equation*}  
	thus there exists $ \ell $ in $ \Z^* $ such that $ (\lambda_1,\lambda_2,\lambda_0)=\ell\,(\lambda_p,\lambda_q,\lambda_r) $ where the 7-tuple $ (\lambda_p,\lambda_q,\lambda_r, \allowbreak \mathbf{n}_p, \allowbreak\mathbf{n}_q,\xi_p,\xi_q) $ belongs to one of the sets $ \mathcal{R}_1(\mathbf{n}_0,\xi_0) $ or $ \mathcal{R}_2(\mathbf{n}_0,\xi_0) $.
	These resonances constitute terms \eqref{eq moy E i L U R1} and \eqref{eq moy E i L U R2} of equation \eqref{eq moy E i L U}.
\end{enumerate} 

According to the expression of the projector $ \tilde{\E^i} $ and since $ \tilde{\pi}_0=\Id $, the term $ \tilde{\E^i}\big[\tilde{L}(0,\partial_z)\,U+\sum_{j=1}^m\tilde{L}_1(U,\zeta_j)\,\partial_{\theta_j}U\big] $ is thus given by 
\begin{subequations}\label{eq moy E i L U}
\begin{align}
	\tilde{\E^i}\Big[&\tilde{L}(0,\partial_z)\,U+\sum_{j=1}^m\tilde{L}_1(U,\zeta_j)\,\partial_{\theta_j}U\Big]=\tilde{L}(0,\partial_z)\,U^*\label{eq moy E i L U trsp moy}\\
	&\quad+\sum_{\substack{\mathbf{n}_0\in\B_{\Z^m}\\\xi_0\in\mathcal{C}(\mathbf{n}_0)}}\sum_{\lambda\in\Z^*}\tilde{\pi}_{(\mathbf{n}_0\cdot\boldsymbol{\zeta},\xi_0)}\,\tilde{L}(0,\partial_z)\,E(\lambda\mathbf{n}_0,\lambda\xi_0)\,\sigma_{\lambda,\mathbf{n}_0,\xi_0}\,e^{i\,\lambda\mathbf{n}_0\cdot\theta}\,e^{i\,\lambda\xi_0\,\psi_d},\label{eq moy E i L U trsp osc}
	\intertext{constituting the transport terms of the mean value and the oscillating part, then the terms of resonances with the mean value as well as the resonances creating a zero frequency}
	&\quad+\sum_{\substack{\mathbf{n}_0\in\B_{\Z^m}\\\xi_0\in\mathcal{C}(\mathbf{n}_0)}}\sum_{\lambda\in\Z^*}\tilde{\pi}_{(\mathbf{n}_0\cdot\boldsymbol{\zeta},\xi_0)}\,\tilde{L}_1(U^*,i\,\lambda\mathbf{n}_0\cdot\boldsymbol{\zeta})\,E(\lambda\mathbf{n}_0,\lambda\xi_0)\,\sigma_{\lambda,\mathbf{n}_0,\xi_0}\,e^{i\,\lambda\mathbf{n}_0\cdot\theta}\,e^{i\,\lambda\xi_0\,\psi_d}\label{eq moy E i L U res moy}\\
	&\quad+\sum_{\substack{\mathbf{n}_0\in\B_{\Z^m}\\\xi_0\in\mathcal{C}(\mathbf{n}_0)}}\sum_{\lambda\in\Z^*}\tilde{L}_1(E(-\lambda\mathbf{n}_0,-\lambda\xi_0),i\,\lambda\mathbf{n}_0\cdot\boldsymbol{\zeta})\, E(\lambda\mathbf{n}_0,\lambda\xi_0)\,\sigma_{\lambda,\mathbf{n}_0,\xi_0}\,\sigma_{-\lambda,\mathbf{n}_0,\xi_0},\label{eq moy E i L U res 0}
	\intertext{and finally the self-interaction term}
	&\quad+\sum_{\substack{\mathbf{n}_0\in\B_{\Z^m}\\\xi_0\in\mathcal{C}(\mathbf{n}_0)}}\sum_{\lambda\in\Z^*}\sum_{\substack{\lambda_1,\lambda_2\in\Z^*\\\lambda_1+\lambda_2=\lambda}}\tilde{\pi}_{(\mathbf{n}_0\cdot\boldsymbol{\zeta},\xi_0)}\,\tilde{L}_1(E(\lambda_1\mathbf{n}_0,\lambda_1\xi_0),i\,\lambda_2\,\mathbf{n}_0\cdot\boldsymbol{\zeta})\,\label{eq moy E i L U AI}\\
	\nonumber&\qquad\qquad E(\lambda_2\mathbf{n}_0,\lambda_2\xi_0)\,\sigma_{\lambda_1,\mathbf{n}_0,\xi_0}\,\sigma_{\lambda_2,\mathbf{n}_0,\xi_0}\,e^{i\,\lambda\mathbf{n}_0\cdot\theta}\,e^{i\,\lambda\xi_0\,\psi_d},\intertext{and the resonances of types 1 and 2 terms}
	&\quad+\sum_{\substack{\mathbf{n}_0\in\B_{\Z^m}\\\xi_0\in\mathcal{C}(\mathbf{n}_0)}}\sum_{\substack{(\lambda_p,\lambda_q,\lambda_r,\mathbf{n}_p,\mathbf{n}_q,\\\xi_p,\xi_q)\in\mathcal{R}_1(\mathbf{n}_0,\xi_0)}}\sum_{\ell\in\Z^*}\tilde{\pi}_{(\mathbf{n}_0\cdot\boldsymbol{\zeta},\xi_0)}\,\tilde{L}_1(E(\ell\lambda_p\mathbf{n}_p,\ell\lambda_p\xi_p),i\ell\lambda_q\mathbf{n}_q\cdot\boldsymbol{\zeta})\,\label{eq moy E i L U R1}\\
	\nonumber&\qquad\qquad E(\ell\lambda_q\mathbf{n}_q,\ell\lambda_q\xi_q)\,\sigma_{\ell\lambda_p,\mathbf{n}_p,\xi_p}\,\sigma_{\ell\lambda_q,\mathbf{n}_q,\xi_q}\,e^{i\,\ell\,\lambda_r\,\mathbf{n}_0\cdot\theta}\,e^{i\,\ell\,\lambda_r\xi_0\,\psi_d}\\[5pt]
	%\intertext{et de type 2}
	&\quad+\sum_{\substack{\mathbf{n}_0\in\B_{\Z^m}\\\xi_0\in\mathcal{C}(\mathbf{n}_0)}}\sum_{\substack{(\lambda_p,\lambda_q,\lambda_r,\mathbf{n}_p,\mathbf{n}_q,\\\xi_p,\xi_q)\in\mathcal{R}_2(\mathbf{n}_0,\xi_0)}}\sum_{\ell\in\Z^*}\tilde{\pi}_{(\mathbf{n}_0\cdot\boldsymbol{\zeta},\xi_0)}\,\tilde{L}_1(E(\ell\lambda_p\mathbf{n}_p,\ell\lambda_p\xi_p),i\ell\lambda_q\mathbf{n}_q\cdot\boldsymbol{\zeta})\,\label{eq moy E i L U R2}\\
	\nonumber&\qquad\qquad E(\ell\lambda_q\mathbf{n}_q,\ell\lambda_q\xi_q)\,\sigma_{\ell\lambda_p,\mathbf{n}_p,\xi_p}\,\sigma_{\ell\lambda_q,\mathbf{n}_q,\xi_q}\,e^{i\,\ell\,\lambda_r\,\mathbf{n}_0\cdot\theta}\,e^{i\,\ell\,\lambda_r\xi_0\,\psi_d}.
\end{align}
\end{subequations}
The homogeneity of degree zero of the projectors $ \tilde{\pi}_{\alpha} $ has been used here. In the following, Definition \ref{def coefficient gamma} of coefficients $ \Gamma $ will be used to rewrite the different terms of \eqref{eq moy E i L U}. In equation \eqref{eq moy E i L U}, the vectors $ E(\mathbf{n},\xi) $ being homogeneous of degree 0, coefficients $ \lambda,\lambda_1,\lambda_2, \lambda_p,\lambda_q $ and $ \ell $ may or may not appear. They are indicated here because they will be useful in a computation below. They may however be removed without any mention being made.

\subsubsection{The mean value is zero}
We prove now that the mean value $ U^* $ is zero, by extracting the system verified by it. According to equation \eqref{eq moy E i L U}, the mean value of the term $ \tilde{\E^i}\,\big( \tilde{L}(0,\partial_z)\,U+\sum_{j=1}^m\tilde{L}_1(U,\zeta_j)\,\partial_{\theta_j}U\big) $ is a priori given by
\begin{equation}\label{eq moy moy E i U}
 \tilde{L}(0,\partial_z)\,U^*+\sum_{\substack{\mathbf{n}_0\in\B_{\Z^m}\\\xi_0\in\mathcal{C}(\mathbf{n}_0)}}\sum_{\lambda\in\Z^*}\tilde{L}_1(E(\mathbf{n}_0,\xi_0),i\lambda\mathbf{n}_0\cdot\boldsymbol{\zeta})\, E(\mathbf{n}_0,\xi_0)\,\sigma_{\lambda,\mathbf{n}_0,\xi_0}\,\sigma_{-\lambda,\mathbf{n}_0,\xi_0}.
\end{equation}
The change of variables $ \lambda=-\lambda $ then shows that term \eqref{eq moy moy E i U} is actually zero. Indeed one can compute
\begin{align*}
&\sum_{\substack{\mathbf{n}_0\in\B_{\Z^m}\\\xi_0\in\mathcal{C}(\mathbf{n}_0)}}\sum_{\lambda\in\Z^*}\tilde{L}_1(E(\mathbf{n}_0,\xi_0),i\lambda\mathbf{n}_0\cdot\boldsymbol{\zeta})\, E(\mathbf{n}_0,\xi_0)\,\sigma_{\lambda,\mathbf{n}_0,\xi_0}\,\sigma_{-\lambda,\mathbf{n}_0,\xi_0}\\
=&\sum_{\substack{\mathbf{n}_0\in\B_{\Z^m}\\\xi_0\in\mathcal{C}(\mathbf{n}_0)}}\sum_{\lambda\in\Z^*}\tilde{L}_1(E(\mathbf{n}_0,\xi_0),-i\lambda\mathbf{n}_0\cdot\boldsymbol{\zeta})\, E(\mathbf{n}_0,\xi_0)\,\sigma_{-\lambda,\mathbf{n}_0,\xi_0}\,\sigma_{\lambda,\mathbf{n}_0,\xi_0}\\
=&-\sum_{\substack{\mathbf{n}_0\in\B_{\Z^m}\\\xi_0\in\mathcal{C}(\mathbf{n}_0)}}\sum_{\lambda\in\Z^*}\tilde{L}_1(E(\mathbf{n}_0,\xi_0),i\lambda\mathbf{n}_0\cdot\boldsymbol{\zeta})\, E(\mathbf{n}_0,\xi_0)\,\sigma_{\lambda,\mathbf{n}_0,\xi_0}\,\sigma_{-\lambda,\mathbf{n}_0,\xi_0}
=0.
\end{align*}
The second term of \eqref{eq moy moy E i U} being zero, the non oscillating terms of \eqref{eq moy E i L U} are given by the term $ \tilde{L}(0,\partial_z)\,U^* $ only. Thus, using system \eqref{eq obtention U 1}, we see that the mean value $ U^* $ satisfies the decoupled system
\begin{equation*}
\left\lbrace \begin{array}{l}
\tilde{L}(0,\partial_z)\,U^*=0,\\[5pt]
B\,U^*_{|x_d=0}=0,\\[10pt]
U^*_{|t\leq 0}=0
\end{array}
\right.
\end{equation*}
since $ G $ is of zero mean value. The mean value $ U^* $ therefore satisfies a boundary value problem verifying the uniform Kreiss-Lopatinskii condition with a strictly hyperbolic operator $ \tilde{L}(0,\partial_z) $. According to \cite{Kreiss1970Initial}, the problem is thus well-posed so $ U^* $ is zero on $ \Omega_T $.

At this point we should note that equation \eqref{eq moy E i L U} can be decoupled between incoming and outgoing modes, thanks to Assumption \ref{hypothese pas de sortant} and the nullity of terms \eqref{eq moy E i L U trsp moy}, \eqref{eq moy E i L U res moy} and \eqref{eq moy E i L U res 0}. The difficulty in decoupling the system now relies on decoupling the boundary condition, and we use the fact that there is no outgoing mode to do it. In its turn the nullity of outgoing modes relies on the nullity of the mean value.

\subsubsection{Finite speed propagation}

It can be proved that if $ U $ is a smooth enough solution to \eqref{eq obtention U 1}, then it is supported in a finite interval in $ x_d $. More precisely the following result is verified.

\begin{lemma}\label{lemme propagation vitesse finie}
	Consider $ T>0 $ and $ s>s_0 $, and let $ U $ in $ \P_{s,T} $ be a solution to system \eqref{eq obtention U 1}. Then its oscillating part $ U^{\osc} $ is zero outside the dihedron $ \ensemble{(t,y,x_d)\in\Omega_T\,\middle|\,0\leq x_d\leq \mathcal{V}^*t} $ (see Figure \ref{figure zone de propagation}).
\end{lemma}

The proof of this lemma uses techniques developed below, so it is postponed, in order to focus on the derivation of a priori estimates. We use the fact that $ U $  travels at finite speed in the normal direction, according to Lemma \ref{lemme vitesse de groupe bornees}.

According to this result, in system \eqref{eq obtention U 1} and the associated linearized systems, the profile $ U^{\osc} $ can be replaced by $ \beta_{T}\,U^{\osc} $, where $ \beta_{T} $ is the function of $ \mathcal{C}_{0}^{\infty}(\R^+_{x_d}) $ introduced in Proposition \ref{prop equivalence systeme}, equating 1 on $ [0,\V^*T] $ and 0 on $ [2\V^*T,+\infty) $. In the following, the scalar product \eqref{eq def prod scal sortant} suited for outgoing profiles can be used, since it is well-defined for profile of compact support with respect to $ x_d $.

\subsubsection{There is no outgoing mode} 
The aim is now to determine the equations satisfied by the outgoing modes. According to equation \eqref{eq moy E i L U}, since the mean value $ U^* $ is zero, the following equality holds
\begin{subequations}\label{eq sortants evolution 1}
\begin{align}
&\sum_{\substack{\mathbf{n}_0\in\B_{\Z^m}\\\xi_0\in\mathcal{C}_{\out}(\mathbf{n}_0)}}\sum_{\lambda\in\Z^*}\tilde{X}_{(\mathbf{n}_0\cdot\boldsymbol{\zeta},\xi_0)}\,\sigma_{\lambda,\mathbf{n}_0,\xi_0}\,e^{i\,\lambda\mathbf{n}_0\cdot\theta}\,e^{i\,\lambda\xi_0\,\psi_d}\,\tilde{\pi}_{(\mathbf{n}_0\cdot\boldsymbol{\zeta},\xi_0)}\,E(\mathbf{n}_0,\xi_0)\label{eq sortants evolution 1 trsp}\\
+&\sum_{\substack{\mathbf{n}_0\in\B_{\Z^m}\\\xi_0\in\mathcal{C}_{\out}(\mathbf{n}_0)}}\sum_{\lambda\in\Z^*}\sum_{\substack{\lambda_1,\lambda_2\in\Z^*\\\lambda_1+\lambda_2=\lambda}}i\,\sigma_{\lambda_1,\mathbf{n}_0,\xi_0}\,\sigma_{\lambda_2,\mathbf{n}_0,\xi_0}\,\Gamma\big(\lambda_1(\mathbf{n}_0,\xi_0),\lambda_2(\mathbf{n}_0,\xi_0)\big)\label{eq sortants evolution 1 AI}\\
\nonumber&\qquad\qquad e^{i\,\lambda\mathbf{n}_0\cdot\theta}\,e^{i\,\lambda\xi_0\,\psi_d}\,\tilde{\pi}_{(\mathbf{n}_0\cdot\boldsymbol{\zeta},\xi_0)}\,E(\mathbf{n}_0,\xi_0)\\
+&\sum_{\substack{\mathbf{n}_0\in\B_{\Z^m}\\\xi_0\in\mathcal{C}_{\out}(\mathbf{n}_0)}}\sum_{\substack{(\lambda_p,\lambda_q,\lambda_r,\mathbf{n}_p,\mathbf{n}_q,\xi_p,\xi_q)\\\in\mathcal{R}_1(\mathbf{n}_0,\xi_0)\cup\mathcal{R}_2(\mathbf{n}_0,\xi_0)}}\sum_{\ell\in\Z^*}i\,\sigma_{\ell\lambda_p,\mathbf{n}_p,\xi_p}\,\sigma_{\ell\lambda_q,\mathbf{n}_q,\xi_q}\,\Gamma\big(\ell\lambda_p(\mathbf{n}_p,\xi_p),\ell\lambda_q(\mathbf{n}_q,\xi_q)\big)\label{eq sortants evolution 1 res}\\
\nonumber&\qquad\qquad e^{i\,\ell\,\lambda_r\mathbf{n}_0\cdot\theta}\,e^{i\,\ell\,\lambda_r\xi_0\,\psi_d}\,\tilde{\pi}_{(\mathbf{n}_0\cdot\boldsymbol{\zeta},\xi_0)}\,E(\mathbf{n}_0,\xi_0)=0.
\end{align}
\end{subequations}
The Lax Lemma \ref{lemme Lax} has been used here to rewrite the term \eqref{eq moy E i L U trsp osc} as the term \eqref{eq sortants evolution 1 trsp}, using that, by definition, we have $ E(\mathbf{n}_0,\xi_0)=\pi_{(\mathbf{n}_0\cdot\boldsymbol{\zeta},\xi_0)}\,E(\mathbf{n}_0,\xi_0) $, and Definition \ref{def coefficient gamma} of coefficients $ \Gamma $ has also been used to rewrite the terms \eqref{eq moy E i L U AI} and \eqref{eq moy E i L U R2} as \eqref{eq sortants evolution 1 AI} and \eqref{eq sortants evolution 1 res}. Note that according to Assumption \ref{hypothese pas de sortant}, all modes $ U^{\osc}_{\mathbf{n},\xi} $ involved in equation \eqref{eq sortants evolution 1} are outgoing modes, and also that the equations are now scalar up to a constant vector depending only on the directions.

Equation \eqref{eq sortants evolution 1} is coupled to the initial condition
\begin{equation}\label{eq sortants cond init}
\big(U^{\osc}_{\mathbf{n},\xi}\big)_{|t\leq 0}=0,\quad \mathbf{n}\in\Z^m\privede{0},\xi\in\mathcal{C}_{\out}(\mathbf{n}).
\end{equation}
We thus seek to solve the problem \eqref{eq sortants evolution 1}, \eqref{eq sortants cond init}. We will prove a priori estimates for this purpose, using the scalar product \eqref{eq def prod scal sortant}.
The decomposition \eqref{eq propr def F out 1 2} of Definition \ref{def ensembles F 1 2 in out} will be used, and the set $ \F^{\out}_{\res} $ of outgoing frequencies involved in resonances will be treated separately, which is finite according to Assumption \ref{hypothese type resonance}.

\emph{Non resonant modes.}
First the modes that are not involved in resonances are investigated, namely we consider $ \mathbf{n}_0 $ in $ \B_{\Z^m} $ and $ \xi_0 $ in $ \mathcal{C}_{\out}(\mathbf{n}_0) $ such that $ (\mathbf{n}_0,\xi_0) $ does not belong to $ \F^{\out}_{\res} $. The sets $ \mathcal{R}_1(\mathbf{n}_0,\xi_0) $ and $ \mathcal{R}_2(\mathbf{n}_0,\xi_0) $ are therefore empty, so, according to equation \eqref{eq sortants evolution 1}, we obtain
\begin{subequations}\label{eq sortants evolution 2}
	\begin{align}
	&\sum_{\lambda\in\Z^*}\tilde{X}_{(\mathbf{n}_0\cdot\boldsymbol{\zeta},\xi_0)}\,\sigma_{\lambda,\mathbf{n}_0,\xi_0}\,e^{i\,\lambda\mathbf{n}_0\cdot\theta}\,e^{i\,\lambda\xi_0\,\psi_d}\,\tilde{\pi}_{(\mathbf{n}_0\cdot\boldsymbol{\zeta},\xi_0)}\,E(\mathbf{n}_0,\xi_0)\label{eq sortants evolution 2 trsp}\\
	+&\sum_{\lambda\in\Z^*}\sum_{\substack{\lambda_1,\lambda_2\in\Z^*\\\lambda_1+\lambda_2=\lambda}}i\,\sigma_{\lambda_1,\mathbf{n}_0,\xi_0}\,\sigma_{\lambda_2,\mathbf{n}_0,\xi_0}\,\lambda_2\,\Gamma\big((\mathbf{n}_0,\xi_0),(\mathbf{n}_0,\xi_0)\big)\label{eq sortants evolution 2 AI}\\
	\nonumber&\qquad e^{i\,\lambda\mathbf{n}_0\cdot\theta}\,e^{i\,\lambda\xi_0\,\psi_d}\,\tilde{\pi}_{(\mathbf{n}_0\cdot\boldsymbol{\zeta},\xi_0)}\,E(\mathbf{n}_0,\xi_0)=0.
	\end{align}
\end{subequations}
Here we have used identity \eqref{eq propr gamma AI} to get the term  \eqref{eq sortants evolution 2 AI}.
Note that if, for $ z $ in $ \Omega_T $ and $ \Theta $ in $ \T $, we define
\begin{equation*}
	S_{\mathbf{n}_0,\xi_0}(z,\Theta):=\sum_{\lambda\in\Z^*}\sigma_{\lambda,\mathbf{n}_0,\xi_0}(z)\,e^{i\lambda\Theta},
\end{equation*}
then 
one can check that the real valued function $ S_{\mathbf{n}_0,\xi_0} $ satisfies the following scalar Burgers equation
\begin{equation*}
	\tilde{X}_{(\mathbf{n}_0\cdot\boldsymbol{\zeta},\xi_0)}S_{\mathbf{n}_0,\xi_0}+\Gamma\big((\mathbf{n}_0,\xi_0),(\mathbf{n}_0,\xi_0)\big)S_{\mathbf{n}_0,\xi_0}\partial_{\Theta}S_{\mathbf{n}_0,\xi_0}=0,
\end{equation*}
that could be solved classically. Indeed, recall that $ \tilde{X}_{(\mathbf{n}_0\cdot\boldsymbol{\zeta},\xi_0)} $, defined in Lemma \ref{lemme Lax}, is given by
\begin{equation*}
	\tilde{X}_{(\mathbf{n}_0\cdot\boldsymbol{\zeta},\xi_0)}=\frac{-1}{\partial_{\xi}\tau_{k(\mathbf{n}_0,\xi_0)}(\mathbf{n}_0\cdot\boldsymbol{\eta},\xi_0)}\partial_t+\inv{\partial_{\xi}\tau_{k(\mathbf{n}_0,\xi_0)}(\mathbf{n}_0\cdot\boldsymbol{\eta},\xi_0)}\nabla_{\eta}\tau_{k(\mathbf{n}_0,\xi_0)}(\mathbf{n}_0\cdot\boldsymbol{\eta},\xi_0)\cdot\nabla_y+\partial_{x_d}.
\end{equation*} 
If, for $ k=1,\dots,m $, we denote by $ \eta_k $ the last $ d-1 $ coordinates of $ \zeta_k $, then we have denoted by $ \boldsymbol{\eta} $ the $ m $-tuple $ \boldsymbol{\eta}:=(\eta_1,\dots,\eta_m) $, in a similar way than $ \boldsymbol{\zeta} $. In this notation, for each $ \mathbf{n}_0 $, the frequency $ \mathbf{n}_0\cdot\boldsymbol{\eta} $ is given by the $ d-1 $ last coordinates of $ \mathbf{n}_0\cdot\boldsymbol{\zeta} $.

We choose however to explain on this easy example the techniques that shall be applied in the following to equations that go beyond the scope of the mere Burgers equations.

We take the scalar product $ \prodscal{.}{.}_{\out} $ of equality \eqref{eq sortants evolution 2} with the quantity
\begin{equation*}
\sum_{\lambda\in\Z^*}\sigma_{\lambda,\mathbf{n}_0,\xi_0}\,e^{i\,\lambda\mathbf{n}_0\cdot\theta}\,e^{i\,\lambda\xi_0\,\psi_d}\,\frac{\tilde{\pi}_{(\mathbf{n}_0\cdot\boldsymbol{\zeta},\xi_0)}\,E(\mathbf{n}_0,\xi_0)}{\left|\tilde{\pi}_{(\mathbf{n}_0\cdot\boldsymbol{\zeta},\xi_0)}\,E(\mathbf{n}_0,\xi_0)\right|^2},
\end{equation*}
to obtain
\begin{subequations}\label{eq sortants prod 2}
\begin{align}
&\sum_{\lambda\in\Z^*}\prodscal{\tilde{X}_{(\mathbf{n}_0\cdot\boldsymbol{\zeta},\xi_0)}\,\sigma_{\lambda,\mathbf{n}_0,\xi_0}}{\sigma_{\lambda,\mathbf{n}_0,\xi_0}}_{L^2(\R^{d-1}\times\R_+)}(t)\label{eq sortants prod 2 trsp}\\
+&\sum_{\lambda\in\Z^*}\sum_{\substack{\lambda_1,\lambda_2\in\Z^*\\\lambda_1+\lambda_2=\lambda}}i\,\lambda_2\,\Gamma\big((\mathbf{n}_0,\xi_0),(\mathbf{n}_0,\xi_0)\big)\label{eq sortants prod 2 AI}\prodscal{\sigma_{\lambda_1,\mathbf{n}_0,\xi_0}\,\sigma_{\lambda_2,\mathbf{n}_0,\xi_0}}{\sigma_{\lambda,\mathbf{n}_0,\xi_0}}_{L^2(\R^{d-1}\times\R_+)}(t)=0.
\end{align}
\end{subequations}
Note that the scalar product is well defined since $ U^{\osc} $ is of compact support with respect to $ x_d $. An integration by parts shows that the transport term \eqref{eq sortants prod 2 trsp} satisfies
\begin{align*}
	2\Re\,\eqref{eq sortants prod 2 trsp}=&
	\sum_{\lambda\in\Z^*}
	\frac{-1}{\partial_{\xi}\tau_{k(\mathbf{n}_0,\xi_0)}(\mathbf{n}_0\cdot\boldsymbol{\eta},\xi_0)}\,\partial_t\norme{\sigma_{\lambda,\mathbf{n}_0,\xi_0}}^2_{L^2(\R^{d-1}\times\R_+)}(t)\\
	&-\sum_{\lambda\in\Z^*}\norme{\sigma_{\lambda,\mathbf{n}_0,\xi_0}}^2_{L^2(\R^{d-1})}(t,0).
\end{align*}
We have denoted by $ k(\mathbf{n}_0,\xi_0) $ the integer between $ 1 $ and $ N $ such that if $ (\tau,\eta,\xi_0):=(\mathbf{n}_0\cdot\boldsymbol{\zeta},\xi_0) $, then $ \tau=\tau_{k(\mathbf{n}_0,\xi_0)}(\eta,\xi_0) $. 
It leads to the following equality
\begin{align}\label{eq sortants derivee 1}
2\Re\,\eqref{eq sortants prod 2 trsp}=&
\frac{-1}{\partial_{\xi}\tau_{k(\mathbf{n}_0,\xi_0)}(\mathbf{n}_0\cdot\boldsymbol{\eta},\xi_0)}\,\frac{d}{dt}\sum_{\lambda\in\Z^*} \norme{\sigma_{\lambda,\mathbf{n}_0,\xi_0}}^2_{L^2(\R^{d-1}\times\R_+)}(t)\\\nonumber
&-\sum_{\lambda\in\Z^*}\norme{\sigma_{\lambda,\mathbf{n}_0,\xi_0}}^2_{L^2(\R^{d-1})}(t,0).
\end{align}

 Then the Burgers term \eqref{eq sortants prod 2 AI} is studied, and more precisely the following sums, that is $ S $ given by
\begin{equation*}
S:=\sum_{\lambda\in\Z^*}\sum_{\substack{\lambda_1,\lambda_2\in\Z^*\\\lambda_1+\lambda_2=\lambda}}i\,\lambda\prodscal{\sigma_{\lambda_1,\mathbf{n}_0,\xi_0}\,\sigma_{\lambda_2,\mathbf{n}_0,\xi_0}}{\sigma_{\lambda,\mathbf{n}_0,\xi_0}}_{L^2(\R^{d-1}\times\R_+)}(t,x_d).
\end{equation*}
and, for $ j=1,2 $,
\begin{align*}
S_j&:=\sum_{\lambda\in\Z^*}\sum_{\substack{\lambda_1,\lambda_2\in\Z^*\\\lambda_1+\lambda_2=\lambda}}i\,\lambda_j\,\prodscal{\sigma_{\lambda_1,\mathbf{n}_0,\xi_0}\,\sigma_{\lambda_2,\mathbf{n}_0,\xi_0}}{\sigma_{\lambda,\mathbf{n}_0,\xi_0}}_{L^2(\R^{d-1}\times\R_+)}(t,x_d),
\end{align*}
the term \eqref{eq sortants prod 2 AI} being given by $ \Gamma\big((\mathbf{n}_0,\xi_0),(\mathbf{n}_0,\xi_0)\big)\,S_2 $.
First, one can verify that $ S=S_1+S_2 $.
But, on one hand, we have immediately $ S_1=S_2 $. On the other hand, the following equality holds: 
\begin{align*}
S&=\sum_{\lambda\in\Z^*}\sum_{\substack{\lambda_1,\lambda_2\in\Z^*\\\lambda_1+\lambda_2=\lambda}}i\,\lambda\prodscal{\sigma_{\lambda_1,\mathbf{n}_0,\xi_0}\,\bar{\sigma_{\lambda,\mathbf{n}_0,\xi_0}}}{\bar{\sigma_{\lambda_2,\mathbf{n}_0,\xi_0}}}_{L^2(\R^{d-1}\times\R_+)}(t,x_d)
\intertext{then, with the  consecutive changes of variables $ \lambda_2=\lambda_1+\lambda_2 $ and $ \lambda_1=-\lambda_1 $,}
S&=\sum_{\lambda\in\Z^*}\sum_{\substack{\lambda_1,\lambda_2\in\Z^*\\\lambda_2-\lambda_1=\lambda}}i\,\lambda_2\prodscal{\sigma_{\lambda_1,\mathbf{n}_0,\xi_0}\,\bar{\sigma_{\lambda_2,\mathbf{n}_0,\xi_0}}}{\bar{\sigma_{\lambda,\mathbf{n}_0,\xi_0}}}_{L^2(\R^{d-1}\times\R_+)}(t,x_d)\\
&=\sum_{\lambda\in\Z^*}\sum_{\substack{\lambda_1,\lambda_2\in\Z^*\\\lambda_2+\lambda_1=\lambda}}i\,\lambda_2\prodscal{\sigma_{-\lambda_1,\mathbf{n}_0,\xi_0}\,\bar{\sigma_{\lambda_2,\mathbf{n}_0,\xi_0}}}{\bar{\sigma_{\lambda,\mathbf{n}_0,\xi_0}}}_{L^2(\R^{d-1}\times\R_+)}(t,x_d)
\intertext{finally, since we have $ \sigma_{-\lambda_1,\mathbf{n}_0,\xi_0}=\bar{\sigma_{\lambda_1,\mathbf{n}_0,\xi_0}} $ (the profile $ U^{\osc} $ being real), one gets}
S&=\sum_{\lambda\in\Z^*}\sum_{\substack{\lambda_1,\lambda_2\in\Z^*\\\lambda_2+\lambda_1=\lambda}}i\,\lambda_2\prodscal{\bar{\sigma_{\lambda_1,\mathbf{n}_0,\xi_0}} \,\bar{\sigma_{\lambda_2,\mathbf{n}_0,\xi_0}}}{\bar{\sigma_{\lambda,\mathbf{n}_0,\xi_0}}}_{L^2(\R^{d-1}\times\R_+)}(t,x_d)\\
&=-\bar{S_2}.
\end{align*}
It follows from $ S=2S_2 $ that $ S_2=0 $, so the term \eqref{eq sortants prod 2 AI} is zero. With equalities \eqref{eq sortants prod 2} and \eqref{eq sortants derivee 1}, we thus obtain
\begin{equation*}
\frac{d}{dt}\sum_{\lambda\in\Z^*} \norme{\sigma_{\lambda,\mathbf{n}_0,\xi_0}}^2_{L^2(\R^{d-1}\times\R_+)}(t)+\partial_{\xi}\tau_{k(\mathbf{n}_0,\xi_0)}(\mathbf{n}_0\cdot\boldsymbol{\eta},\xi_0)\sum_{\lambda\in\Z^*}\norme{\sigma_{\lambda,\mathbf{n}_0,\xi_0}}^2_{L^2(\R^{d-1})}(t,0)=0,
\end{equation*}
and therefore, with the initial condition \eqref{eq sortants cond init}, for $ t\geq 0 $, we get
\begin{equation*}
\sum_{\lambda\in\Z^*} \norme{\sigma_{\lambda,\mathbf{n}_0,\xi_0}}^2_{L^2(\R^{d-1}\times\R_+)}(t)+\partial_{\xi}\tau_{k(\mathbf{n}_0,\xi_0)}(\mathbf{n}_0\cdot\boldsymbol{\eta},\xi_0)\int_0^t\sum_{\lambda\in\Z^*}\norme{\sigma_{\lambda,\mathbf{n}_0,\xi_0}}^2_{L^2(\R^{d-1})}(\rho,0)\,d\rho=0.
\end{equation*}
Since the quantity $ \partial_{\xi}\tau_{k(\mathbf{n}_0,\xi_0)}(\mathbf{n}_0\cdot\boldsymbol{\eta},\xi_0) $ is positive (the frequency $ (\mathbf{n}_0\cdot\boldsymbol{\zeta},\xi_0) $ being outgoing),  we deduce that $ \sigma_{\lambda,\mathbf{n}_0,\xi_0} $ is zero for all $ \lambda $ in $ \Z^* $, and $ U^{\osc}_{\lambda\mathbf{n}_0,\lambda\xi_0} $ is therefore zero for $ \mathbf{n}_0 $ in $ \B_{\Z^m} $, $ \xi_0 $ in $ \mathcal{C}_{\out}(\mathbf{n}_0) $ such that $ \mathcal{R}_2(\mathbf{n}_0,\xi_0) $ and $ \mathcal{R}_2(\mathbf{n}_0,\xi_0) $ are empty, and $ \lambda $ in $ \Z^* $.

\emph{Resonant modes.} Outgoing modes involved in resonances are now investigated, namely the couples $ (\mathbf{n}_0,\xi_0) $ of the set $ \F^{\out}_{\res} $, that are coupled through equation \eqref{eq sortants evolution 1}, because of the resonances, and therefore must be treated all together.
From equation \eqref{eq sortants evolution 1} is deduced the equation for the resonant modes, involving, in addition to a transport and a Burgers terms, a resonant one. There holds
\begin{subequations}\label{eq sortants evolution 3}
	\begin{align}
	&\sum_{(\mathbf{n}_0,\xi_0)\in\F^{\out}_{\res}}\sum_{\lambda\in\Z^*}\tilde{X}_{(\mathbf{n}_0\cdot\boldsymbol{\zeta},\xi_0)}\,\sigma_{\lambda,\mathbf{n}_0,\xi_0}\,e^{i\,\lambda\mathbf{n}_0\cdot\theta}\,e^{i\,\lambda\xi_0\,\psi_d}\,\tilde{\pi}_{(\mathbf{n}_0\cdot\boldsymbol{\zeta},\xi_0)}\,E(\mathbf{n}_0,\xi_0)\label{eq sortants evolution 3 trsp}\\
	+&\sum_{(\mathbf{n}_0,\xi_0)\in\F^{\out}_{\res}}\sum_{\lambda\in\Z^*}\sum_{\substack{\lambda_1,\lambda_2\in\Z^*\\\lambda_1+\lambda_2=\lambda}}i\,\lambda_2\,\sigma_{\lambda_1,\mathbf{n}_0,\xi_0}\,\sigma_{\lambda_2,\mathbf{n}_0,\xi_0}\,\Gamma\big((\mathbf{n}_0,\xi_0),(\mathbf{n}_0,\xi_0)\big)\label{eq sortants evolution 3 AI}\\
	\nonumber&\qquad e^{i\,\lambda\mathbf{n}_0\cdot\theta}\,e^{i\,\lambda\xi_0\,\psi_d}\,\tilde{\pi}_{(\mathbf{n}_0\cdot\boldsymbol{\zeta},\xi_0)}\,E(\mathbf{n}_0,\xi_0)\\
	+&\sum_{(\mathbf{n}_0,\xi_0)\in\F^{\out}_{\res}}\sum_{\substack{(\lambda_p,\lambda_q,\lambda_r,\mathbf{n}_p,\mathbf{n}_q,\xi_p,\xi_q)\\\in\mathcal{R}_1(\mathbf{n}_0,\xi_0)\cup\mathcal{R}_2(\mathbf{n}_0,\xi_0)}}\sum_{\ell\in\Z^*}i\,\ell\,\sigma_{\ell\lambda_p,\mathbf{n}_p,\xi_p}\,\sigma_{\ell\lambda_q,\mathbf{n}_q,\xi_q}\,\Gamma\big(\lambda_p(\mathbf{n}_p,\xi_p),\lambda_q(\mathbf{n}_q,\xi_q)\big)\label{eq sortants evolution 3 res}\\
	\nonumber&\qquad e^{i\,\ell\,\lambda_r\mathbf{n}_0\cdot\theta}\,e^{i\,\ell\,\lambda_r\xi_0\,\psi_d}\,\tilde{\pi}_{(\mathbf{n}_0\cdot\boldsymbol{\zeta},\xi_0)}\,E(\mathbf{n}_0,\xi_0)=0.
	\end{align}
\end{subequations}
Identities \eqref{eq propr homogeneite Gamma} and \eqref{eq propr gamma AI} have been used here.
Two profiles constructed from the resonant modes are defined here, that will be used in the following
\begin{align*}
U^{\osc}_{\mathcal{R}^{\out}}(z,\theta,\psi_d) &:=\sum_{(\mathbf{n}_0,\xi_0)\in\F^{\out}_{\res}}\sum_{\lambda\in\Z^*}\sigma_{\lambda,\mathbf{n}_0,\xi_0}(z)\,e^{i\,\lambda\mathbf{n}_0\cdot\theta}\,e^{i\,\lambda\xi_0\,\psi_d}\,E(\mathbf{n}_0,\xi_0),\\
\tilde{U}^{\osc}_{\mathcal{R}^{\out}}(z,\theta,\psi_d) &:=\sum_{(\mathbf{n}_0,\xi_0)\in\F^{\out}_{\res}}\sum_{\lambda\in\Z^*}\sigma_{\lambda,\mathbf{n}_0,\xi_0}(z)\,e^{i\,\lambda\mathbf{n}_0\cdot\theta}\,e^{i\,\lambda\xi_0\,\psi_d}\,\tilde{\pi}_{(\mathbf{n}_0\cdot\boldsymbol{\zeta},\xi_0)}\,E(\mathbf{n}_0,\xi_0).
\end{align*}
Taking the scalar product $ \prodscal{.}{.}_{\out} $ of equality \eqref{eq sortants evolution 3} with the profile $ \tilde{U}^{\osc}_{\mathcal{R}^{\out}} $, one gets
\begin{subequations}\label{eq sortants prod 3}
	\begin{align}
	&\sum_{(\mathbf{n}_0,\xi_0)\in\F^{\out}_{\res}}\sum_{\lambda\in\Z^*} \left|\tilde{\pi}_{(\mathbf{n}_0\cdot\boldsymbol{\zeta},\xi_0)}\,E(\mathbf{n}_0,\xi_0)\right|^2\prodscal{\tilde{X}_{(\mathbf{n}_0\cdot\boldsymbol{\zeta},\xi_0)}\,\sigma_{\lambda,\mathbf{n}_0,\xi_0}}{\sigma_{\lambda,\mathbf{n}_0,\xi_0}}_{L^2(\R^{d-1}\times\R_+)}(t)\label{eq sortants prod 3 trsp}\\
	+&\sum_{(\mathbf{n}_0,\xi_0)\in\F^{\out}_{\res}}\sum_{\lambda\in\Z^*}\sum_{\substack{\lambda_1,\lambda_2\in\Z^*\\\lambda_1+\lambda_2=\lambda}}i\,\lambda_2\,\Gamma\big((\mathbf{n}_0,\xi_0),(\mathbf{n}_0,\xi_0)\big)\left|\tilde{\pi}_{(\mathbf{n}_0\cdot\boldsymbol{\zeta},\xi_0)}\,E(\mathbf{n}_0,\xi_0)\right|^2\label{eq sortants prod 3 AI}
	\\
	\nonumber&\qquad  
	\prodscal{\sigma_{\lambda_1,\mathbf{n}_0,\xi_0}\,\sigma_{\lambda_2,\mathbf{n}_0,\xi_0}}{\sigma_{\lambda,\mathbf{n}_0,\xi_0}}_{L^2(\R^{d-1}\times\R_+)}(t)\\
	+&\sum_{(\mathbf{n}_0,\xi_0)\in\F^{\out}_{\res}}\sum_{\substack{(\lambda_p,\lambda_q,\lambda_r,\mathbf{n}_p,\mathbf{n}_q,\xi_p,\xi_q)\\\in\mathcal{R}_1(\mathbf{n}_0,\xi_0)\cup\mathcal{R}_2(\mathbf{n}_0,\xi_0)}}\sum_{\ell\in\Z^*}i\,\ell\,\Gamma\big(\lambda_p(\mathbf{n}_p,\xi_p),\lambda_q(\mathbf{n}_q,\xi_q)\big)\left|\tilde{\pi}_{(\mathbf{n}_0\cdot\boldsymbol{\zeta},\xi_0)}\,E(\mathbf{n}_0,\xi_0)\right|^2\label{eq sortants prod 3 res}
	\\\nonumber&\qquad 
	\prodscal{\sigma_{\ell\lambda_p,\mathbf{n}_p,\xi_p}\,\sigma_{\ell\lambda_q,\mathbf{n}_q,\xi_q}}{\sigma_{\ell\lambda_r,\mathbf{n}_0,\xi_0}}_{L^2(\R^{d-1}\times\R_+)}(t) =0.
	\end{align}
\end{subequations}
The first two terms are treated in the same way as for the non resonant modes (see above), so we obtain
\begin{align}\label{eq sortants derivee 2}
2\Re\,\eqref{eq sortants prod 3 trsp}=&-\sum_{(\mathbf{n}_0,\xi_0)\in\F^{\out}_{\res}}\frac{\left|\tilde{\pi}_{(\mathbf{n}_0\cdot\boldsymbol{\zeta},\xi_0)}\,E(\mathbf{n}_0,\xi_0)\right|^2}{\partial_{\xi}\tau_{k(\mathbf{n}_0,\xi_0)}(\mathbf{n}_0\cdot\boldsymbol{\eta},\xi_0)}\,\frac{d}{dt}\sum_{\lambda\in\Z^*} \norme{\sigma_{\lambda,\mathbf{n}_0,\xi_0}}^2_{L^2(\R^{d-1}\times\R_+)}(t)\\
&-\sum_{(\mathbf{n}_0,\xi_0)\in\F^{\out}_{\res}}\sum_{\lambda\in\Z^*}\norme{\sigma_{\lambda,\mathbf{n}_0,\xi_0}}^2_{L^2(\R^{d-1})}(t,0)\nonumber\\
\leq &-C\,\frac{d}{dt}\prodscal{\tilde{U}^{\osc}_{\mathcal{R}^{\out}}}{\tilde{U}^{\osc}_{\mathcal{R}^{\out}}}_{\out}(t),\nonumber
\end{align}
with $ C>0 $, using that $ \F^{\out}_{\res} $ is finite and that the group velocity $ \partial_{\xi}\tau_{k(\mathbf{n}_0,\xi_0)}(\mathbf{n}_0\cdot\boldsymbol{\eta},\xi_0) $ is positive. On an other hand, with the same techniques than for the non resonant modes, one gets $ \eqref{eq sortants prod 3 AI}=0 $. 
Finally the resonance term \eqref{eq sortants prod 3 res} is investigated. Since the sets $ \F^{\out}_{\res} $ and $ \bigcup_{(\mathbf{n}_0,\xi_0)\in\F^{\out}_{\res}}\big(\mathcal{R}_1(\mathbf{n}_0,\xi_0) \cup\mathcal{R}_2(\mathbf{n}_0,\xi_0)\big) $ are finite, the following bound holds
\begin{equation*}
	\left|\Gamma\big(\lambda_p(\mathbf{n}_p,\xi_p),\lambda_q(\mathbf{n}_q,\xi_q)\big)\right|\left|\tilde{\pi}_{(\mathbf{n}_0\cdot\boldsymbol{\zeta},\xi_0)}\,E(\mathbf{n}_0,\xi_0)\right|^2\leq C,
\end{equation*}
where the constant $ C>0 $ is independent of $ \mathbf{n}_p,\mathbf{n}_q,\mathbf{n}_r,\xi_p,\xi_q $ and $ \xi_r $. Thus a term of the form $ \prodscal{f*g}{g} $ is obtained in \eqref{eq sortants prod 3 res}, which is estimated using Cauchy-Schwarz and Young inequalities, and the injection of $ L^2(\T^m) $ into $ L^1(\T^m) $, which gives the following estimate on the term \eqref{eq sortants prod 3 res}:
\begin{equation}
	\left|2\Re \eqref{eq sortants prod 3 res}\right|
	\label{eq sortants est res} 
	\leq C
	\norme{U^{\osc}_{\mathcal{R}^{\out}}}_{\mathcal{E}_{s,T}}\prodscal{\tilde{U}^{\osc}_{\mathcal{R}^{\out}}}{\tilde{U}^{\osc}_{\mathcal{R}^{\out}}}_{\out}(t).
\end{equation}
It follows from equations \eqref{eq sortants evolution 3}, \eqref{eq sortants derivee 2} and \eqref{eq sortants est res} the differential inequality
\begin{equation*}
\frac{d}{dt}\prodscal{\tilde{U}^{\osc}_{\mathcal{R}^{\out}}}{\tilde{U}^{\osc}_{\mathcal{R}^{\out}}}_{\out}(t)\leq C
\norme{U^{\osc}_{\mathcal{R}^{\out}}}_{\mathcal{E}_{s,T}}\prodscal{\tilde{U}^{\osc}_{\mathcal{R}^{\out}}}{\tilde{U}^{\osc}_{\mathcal{R}^{\out}}}_{\out}(t)
\end{equation*}
The initial conditions \eqref{eq sortants cond init} ensure that $ \big(\tilde{U}^{\osc}_{\mathcal{R}^{\out}}\big)_{|t=0}=0 $, so, for $ t\geq 0 $, we have
\begin{equation*}
\prodscal{\tilde{U}^{\osc}_{\mathcal{R}^{\out}}}{\tilde{U}^{\osc}_{\mathcal{R}^{\out}}}_{\out}(t)
=\sum_{(\mathbf{n}_0,\xi_0)\in\F^{\out}_{\res}}\sum_{\lambda\in\Z^*} \norme{\sigma_{\lambda,\mathbf{n}_0,\xi_0}}^2_{L^2(\R^{d-1}\times\R_+)}(t)=0.
\end{equation*}
Thus, for all $ (\mathbf{n}_0,\xi_0) $ in $ \F^{\out}_{\res} $ and $ \lambda $ in $ \Z^* $, and for all $ t\geq 0 $, the function $ \sigma_{\lambda,\mathbf{n}_0,\xi_0}(t,.) $ is zero, therefore the same holds for the outgoing amplitude $ U^{\osc}_{\lambda\mathbf{n}_0,\lambda\xi_0} $, for all $ (\mathbf{n}_0,\xi_0) $ in $ \F^{\out}_{\res} $ and $ \lambda $ in $ \Z^* $.

In conclusion, it has been proven that for every profile $ U $ regular enough solution to system \eqref{eq obtention U 1}, its mean value $ U^{*} $ as well as each of its outgoing modes $ U^{\osc}_{\mathbf{n},\xi} $, $ \mathbf{n}\in\Z^m\privede{0} $, $ \xi\in\mathcal{C}_{\out}(\mathbf{n}) $, are zero.

\subsubsection{Decoupling the system}

Because of the algebra property of the space of profiles $ \P_{s,T} $, and since the projectors $ \E $ and $ \tilde{\E^i} $ preserve the decomposition $ \P_{s,T}=\P_{s,T}^{\osc}\oplus\P_{s,T}^{\ev} $, equations \eqref{eq obtention E U} and $ \eqref{eq obtention E i U} $ decoupled according to the oscillating and evanescent parts, and the same holds for equation \eqref{eq obtention cond initiale U}. The evanescent part therefore satisfies the equation
\begin{equation*}
	\E\, U^{\ev}=U^{\ev}
\end{equation*}
and the oscillating part the equations
\begin{align*}
	\E\,U^{\osc}&=U^{\osc}\\
	\tilde{\E^i}\Big[\tilde{L}(0,\partial_z)\,U^{\osc}+\sum_{j=1}^m\tilde{L}_1(U^{\osc},\zeta_j)\,\partial_{\theta_j}U^{\osc}\Big]&=0.
\end{align*}
The second equation may be rewritten as, using notations for $ U^{\osc} $ that have been already introduced,
\begin{subequations}\label{eq sortants evolution 4}
	\begin{align}
		&\sum_{\substack{\mathbf{n}_0\in\B_{\Z^m}\\\xi_0\in\mathcal{C}_{\inc}(\mathbf{n}_0)}}\sum_{\lambda\in\Z^*}\tilde{X}_{(\mathbf{n}_0\cdot\boldsymbol{\zeta},\xi_0)}\,\sigma_{\lambda,\mathbf{n}_0,\xi_0}\,e^{i\,\lambda\mathbf{n}_0\cdot\theta}\,e^{i\,\lambda\xi_0\,\psi_d}\,\tilde{\pi}_{(\mathbf{n}_0\cdot\boldsymbol{\zeta},\xi_0)}\,E(\mathbf{n}_0,\xi_0)\label{eq sortants evolution 4 trsp}\\
		+&\sum_{\substack{\mathbf{n}_0\in\B_{\Z^m}\\\xi_0\in\mathcal{C}_{\inc}(\mathbf{n}_0)}}\sum_{\lambda\in\Z^*}\sum_{\substack{\lambda_1,\lambda_2\in\Z^*\\\lambda_1+\lambda_2=\lambda}}i\,\lambda_2\,\sigma_{\lambda_1,\mathbf{n}_0,\xi_0}\,\sigma_{\lambda_2,\mathbf{n}_0,\xi_0}\,\Gamma\big((\mathbf{n}_0,\xi_0),(\mathbf{n}_0,\xi_0)\big)\label{eq sortants evolution 4 AI}\\
		\nonumber&\qquad e^{i\,\lambda\mathbf{n}_0\cdot\theta}\,e^{i\,\lambda\xi_0\,\psi_d}\,\tilde{\pi}_{(\mathbf{n}_0\cdot\boldsymbol{\zeta},\xi_0)}\,E(\mathbf{n}_0,\xi_0)\\
		+&\sum_{\substack{\mathbf{n}_0\in\B_{\Z^m}\\\xi_0\in\mathcal{C}_{\inc}(\mathbf{n}_0)}}\sum_{\substack{(\lambda_p,\lambda_q,\lambda_r,\mathbf{n}_p,\mathbf{n}_q,\xi_p,\xi_q)\\\in\mathcal{R}_1(\mathbf{n}_0,\xi_0)\cup\mathcal{R}_2(\mathbf{n}_0,\xi_0)}}\sum_{\ell\in\Z^*}i\,\ell\,\sigma_{\ell\lambda_p,\mathbf{n}_p,\xi_p}\,\sigma_{\ell\lambda_q,\mathbf{n}_q,\xi_q}\,\Gamma\big(\lambda_p(\mathbf{n}_p,\xi_p),\lambda_q(\mathbf{n}_q,\xi_q)\big)\label{eq sortants evolution 4 res}\\
		\nonumber&\qquad e^{i\,\ell\,\lambda_r\mathbf{n}_0\cdot\theta}\,e^{i\,\ell\,\lambda_r\xi_0\,\psi_d}\,\tilde{\pi}_{(\mathbf{n}_0\cdot\boldsymbol{\zeta},\xi_0)}\,E(\mathbf{n}_0,\xi_0)=0.
	\end{align}
\end{subequations}
This equation decouples according to the set $ \F^{\inc}_{\res} $ of resonant modes, and each of the non resonant mode. For each mode $ (\mathbf{n}_0,\xi_0) $ in $ (\B_{\Z^m}\times\mathcal{C}_{\inc}(\mathbf{n}_0))\setminus\F^{\inc}_{\res} $, which is therefore such that the sets $ \mathcal{R}_1(\mathbf{n}_0,\xi_0) $ and $ \mathcal{R}_2(\mathbf{n}_0,\xi_0) $ are empty, we define
\begin{equation*}
	S_{\mathbf{n}_0,\xi_0}(z,\Theta):=\sum_{\lambda\in\Z^*}\sigma_{\lambda,\mathbf{n}_0,\xi_0}(z)\,e^{i\lambda\,\Theta}.
\end{equation*}
Then, according to equations \eqref{eq sortants evolution 4 trsp} and \eqref{eq sortants evolution 4 AI}, this function satisfies the following scalar Burgers equation
\begin{equation*}
	\tilde{X}_{(\mathbf{n}_0\cdot\boldsymbol{\zeta},\xi_0)}S_{\mathbf{n}_0,\xi_0}+\Gamma\big((\mathbf{n}_0,\xi_0),(\mathbf{n}_0,\xi_0)\big)\,S_{\mathbf{n}_0,\xi_0}\,\partial_{\Theta}S_{\mathbf{n}_0,\xi_0}=0.
\end{equation*}
 On an other hand the resonant modes of $ \F^{\inc}_{\res} $ satisfy the independent equation
\begin{align*}
	&\sum_{(\mathbf{n}_0,\xi_0)\in\F^{\inc}_{\res}}\sum_{\lambda\in\Z^*}\tilde{X}_{(\mathbf{n}_0\cdot\boldsymbol{\zeta},\xi_0)}\,\sigma_{\lambda,\mathbf{n}_0,\xi_0}\,e^{i\,\lambda\mathbf{n}_0\cdot\theta}\,e^{i\,\lambda\xi_0\,\psi_d}\,\tilde{\pi}_{(\mathbf{n}_0\cdot\boldsymbol{\zeta},\xi_0)}\,E(\mathbf{n}_0,\xi_0)\\
	+&\sum_{(\mathbf{n}_0,\xi_0)\in\F^{\inc}_{\res}}\sum_{\lambda\in\Z^*}\sum_{\substack{\lambda_1,\lambda_2\in\Z^*\\\lambda_1+\lambda_2=\lambda}}i\,\lambda_2\,\sigma_{\lambda_1,\mathbf{n}_0,\xi_0}\,\sigma_{\lambda_2,\mathbf{n}_0,\xi_0}\,\Gamma\big((\mathbf{n}_0,\xi_0),(\mathbf{n}_0,\xi_0)\big)\\
	\nonumber&\qquad e^{i\,\lambda\mathbf{n}_0\cdot\theta}\,e^{i\,\lambda\xi_0\,\psi_d}\,\tilde{\pi}_{(\mathbf{n}_0\cdot\boldsymbol{\zeta},\xi_0)}\,E(\mathbf{n}_0,\xi_0)\\
	+&\sum_{(\mathbf{n}_0,\xi_0)\in\F^{\inc}_{\res}}\sum_{\substack{(\lambda_p,\lambda_q,\lambda_r,\mathbf{n}_p,\mathbf{n}_q,\xi_p,\xi_q)\\\in\mathcal{R}_1(\mathbf{n}_0,\xi_0)\cup\mathcal{R}_2(\mathbf{n}_0,\xi_0)}}\sum_{\ell\in\Z^*}i\,\ell\,\sigma_{\ell\lambda_p,\mathbf{n}_p,\xi_p}\,\sigma_{\ell\lambda_q,\mathbf{n}_q,\xi_q}\,\Gamma\big(\lambda_p(\mathbf{n}_p,\xi_p),\lambda_q(\mathbf{n}_q,\xi_q)\big)\\
	\nonumber&\qquad e^{i\,\ell\,\lambda_r\mathbf{n}_0\cdot\theta}\,e^{i\,\ell\,\lambda_r\xi_0\,\psi_d}\,\tilde{\pi}_{(\mathbf{n}_0\cdot\boldsymbol{\zeta},\xi_0)}\,E(\mathbf{n}_0,\xi_0)=0,
\end{align*}
that may be rewritten, with already introduced notations (see Definition \ref{def proj E res}), as
\begin{equation*}
	\tilde{\E^i}^{\inc}_{\res}\Big[\tilde{L}(0,\partial_z)\,U^{\osc}_{\res}+\sum_{j=1}^m\tilde{L}_1(U^{\osc}_{\res},\zeta_j)\,\partial_{\theta_j}U^{\osc}_{\res}\Big]=0.
\end{equation*}
Note that by assumption on the set $ \F^{\inc}_{\res} $, all modes involved in this equation are part of the set $ \F^{\inc}_{\res} $.
Furthermore it is clear that the polarization condition $ \E\,U^{\osc}=U^{\osc} $ as well as the initial condition decouple in the same way. Therefore, to conclude the proof of Proposition \ref{prop equivalence systeme}, it must be shown that the boundary condition also decouples in this manner.

\subsubsection{Determination of the trace on the boundary}

It is possible to determine the traces on the boundary $ (U^{\osc}_{\res})_{|x_d=0,\psi_d=0} $, $ (S_{\mathbf{n}_0,\xi_0})_{x_d=0} $ and $ U^{\ev}_{|x_d=0,\psi_d=0} $ from the boundary condition \eqref{eq obtention cond bord U} using the fact that there are only incoming modes, which will prove the intended decoupling of the system. According to polarization conditions \eqref{eq red pola osc} and \eqref{eq red pola ev}, and since there are only incoming frequencies, for $ \mathbf{n} $ in $ \Z^m\privede{0} $, boundary condition \eqref{eq obtention cond bord U} writes
\begin{equation}\label{eq red cond bord inter 1}
B\Big[\sum_{\xi\in\mathcal{C}_{\inc}(\mathbf{n})}\pi_{(\mathbf{n}\cdot\boldsymbol{\zeta},\xi)}\,U^{\osc}_{\mathbf{n},\xi}(z',0)+\Pi^e_{\C^N}(\mathbf{n}\cdot\boldsymbol{\zeta})\,U^{\ev}_{\mathbf{n}}(z',0,0)\Big]=G_{\mathbf{n}}(z'),
\end{equation}
where the amplitudes $ G_{\mathbf{n}} $ have been defined with the formula \eqref{eq def G_n}.
For all $ \xi $ in $ \mathcal{C}_{\inc}(\mathbf{n}) $, the term $ \pi_{(\mathbf{n}\cdot\boldsymbol{\zeta},\xi)}\,U^{\osc}_{\mathbf{n},\xi} $ belongs to $ \ker L\big(0,(\freq,\xi)\big) $ which is included in $ E_-(\freq) $ according to Proposition \ref{prop decomp E_-}, since the frequency $ (\freq,\xi) $ is incoming. In the same way, according to the definition of the projector $ \Pi^e_{\C^N}(\mathbf{n}\cdot\boldsymbol{\zeta}) $, the term $ \Pi^e_{\C^N}(\mathbf{n}\cdot\boldsymbol{\zeta})\,U^{\ev}_{\mathbf{n}}(z,0) $ belongs to the space $ E_-(\freq) $. The vector on which acts the matrix $ B $ in \eqref{eq red cond bord inter 1} therefore belongs to $ E_-(\freq) $, and the matrix $ B $ restricted to this subspace is invertible according to the uniform Kreiss-Lopatinskii condition Assumption \ref{hypothese UKL}. It then follows by projecting on the spaces $ E_-^j(\freq) $ and $ E_-^e(\freq) $ the following boundary conditions
\begin{subequations}\label{eq red trace}
	\begin{align}
	&\pi_{(\mathbf{n}\cdot\boldsymbol{\zeta},\xi)}\,U^{\osc}_{\mathbf{n},\xi}(z',0)=\Pi_-^{j(\mathbf{n},\xi)}(\freq)\,\big(B_{|E_-(\freq)}\big)^{-1}G_{\mathbf{n}}(z'),\quad \xi\in\mathcal{C}_{\inc}(\mathbf{n}),\label{eq red trace osc}\\[5pt]
	&\Pi^e_{\C^N}(\mathbf{n}\cdot\boldsymbol{\zeta})\,U^{\ev}_{\mathbf{n}}(z',0,0)=\Pi_-^e(\freq)\,\big(B_{|E_-(\freq)}\big)^{-1}G_{\mathbf{n}}(z'),\label{eq red trace ev}
	\end{align}
\end{subequations}
where, for $ \xi $ in $ \mathcal{C}_{\inc}(\mathbf{n}) $, $ j(\mathbf{n},\xi) $ is the index such that $ \xi=\xi_{j(\mathbf{n},\xi)}(\freq) $. Therefore, according to \eqref{eq red trace} and the polarization conditions \eqref{eq red pola osc} and \eqref{eq red pola ev}, the profiles $ H_{\res}^{\osc} $ and $ H^{\ev} $ defined by \eqref{eq red def H} are such that $ (U^{\osc}_{\res})_{|x_d=0,\psi_d=0}=H^{\osc}_{\res} $ and $ U^{\ev}_{|x_d=0,\psi_d=0}=H^{\ev} $. On an other hand, since, according to \eqref{eq red trace osc}, we have, for all $ (\mathbf{n}_0,\xi_0) $ in $ (\B_{\Z^m}\times\mathcal{C}_{\inc}(\mathbf{n}_0))\setminus\F^{\inc}_{\res} $,
\begin{equation*}
	(\sigma_{\lambda,\mathbf{n}_0,\xi_0})_{|z_d=0}\,E(\mathbf{n}_0,\xi_0)=\Pi_-^{j(\lambda\mathbf{n}_0,\lambda\xi_0)}(\lambda\mathbf{n}_0\cdot\boldsymbol{\zeta})\,\big(B_{|E_-(\lambda\mathbf{n}_0\cdot\boldsymbol{\zeta})}\big)^{-1}G_{\lambda\mathbf{n}_0\cdot\boldsymbol{\zeta}},
\end{equation*}
and since the vectors $ E(\mathbf{n}_0,\xi_0) $ are of norm 1, the function $ h_{\mathbf{n}_0,\xi_0} $ defined by formula \eqref{eq red def H osc non res} satisfies $ (S_{\mathbf{n}_0,\xi_0})_{x_d=0}=h_{\mathbf{n}_0,\xi_0} $.
We finally check that the boundary terms $ H^{\osc}_{\res} $, $ h_{\mathbf{n}_0,\xi_0}  $ and $ H^{\ev} $ are controlled in $ H^s(\omega_T) $ by $ G $. On one hand, according to the uniform Kreiss-Lopatinskii condition \ref{hypothese UKL}, the inverse matrix $ (B_{|E_-(\freq)})^{-1} $ is uniformly bounded, see Remark \ref{remarque B inverse bornee}. On the other hand, according to Proposition \ref{prop proj bornes}, projectors $ \Pi_-^{j(\mathbf{n},\xi)}(\freq) $ and $ \Pi_-^e(\freq) $ are uniformly bounded with respect to $ \mathbf{n} $ in $ \Z^m\privede{0} $. According to formulas \eqref{eq def G_n} and \eqref{eq red def H} and the Parseval's identity, the sought control in $ L^2(\omega_T) $ is ensured. The control in $ H^s(\omega_T\times\T^m) $ for all $ s\geq 0 $ follows, using that the quantities $ \Pi_-^{j(\mathbf{n},\xi)}(\freq)\,\big(B_{|E_-(\freq)}\big)^{-1} $  and $ \Pi_-^e(\freq)\,\big(B_{|E_-(\freq)}\big)^{-1} $ do not depend on $ z' $ in $ \omega_T $. Therefore we obtain
\begin{equation}\label{eq red est H h avec G}
	\norme{H^{\osc}_{\res}}^2_{H^s(\omega_T\times\T^m)}+\sum_{\substack{(\mathbf{n}_0,\xi_0) \in\\ (\B_{\Z^m}\times\mathcal{C}_{\inc}(\mathbf{n}_0))\setminus\F^{\inc}_{\res}}}\norme{h_{\mathbf{n}_0,\xi_0}}^2_{H^s(\omega_T\times\T)}+\norme{H^{\ev}}^2_{H^s(\omega_T\times\T^m)}\leq C \norme{G}^2_{H^s(\omega\times\T^m)},
\end{equation}
where the positive constant $ C $ does not depend on $ T $ or $ s $.
This completes the proof of Proposition \ref{prop equivalence systeme}.

\subsection{A priori estimate on the linearized system for the oscillating resonant part}\label{subsection estimate resonant part}

According to Proposition \ref{prop equivalence systeme}, the study may be narrowed down to the one of systems \eqref{eq red syst osc res}, $ \eqref{eq red syst osc non res} $ and \eqref{eq red syst ev}. This part deals with the first one, and we will prove a priori estimates on the associated linearized system, which will be used to show the convergence of an iterative scheme. 
Recall that $ s_0 $ is given by $ s_0=h+(d+m)/2 $ where $ h $ is an integer greater than $ (3+a_1)/2 $, occurring in estimate \eqref{eq est gamma n n}, whith $ a_1$ the real number of Assumption \ref{hypothese petits diviseurs 1}.

\begin{proposition}\label{prop estim a priori}
	Consider $ s>s_0 $ and let $ U_{\res}^{\osc} $ be in $ \P^{\osc}_{s,T} $, $ V_{\res}^{\osc} $ in $ \N^{\osc}_{s,T} $ both involving only incoming resonant modes, and $ F_{\res}^{\osc} $ in $ \P^{\osc}_{s,T} $, satisfying the system
	\begin{subequations}\label{eq estim syst osc}
		\begin{align}
		\E^{\inc}_{\res}\, U_{\res}^{\osc}&=U_{\res}^{\osc} \label{eq estim syst osc E U = U}  \\
		\tilde{\E^i}^{\inc}_{\res}\Big[\tilde{L}(0,\partial_z)\,\beta_TU_{\res}^{\osc}+\sum_{j=1}^m\tilde{L}_1(\beta_TV_{\res}^{\osc},\zeta_j)\,\partial_{\theta_j}\beta_TU_{\res}^{\osc}\Big]&=\tilde{\E^i}\,F_{\res}^{\osc}\label{eq estim syst osc E i U}\\
		\big(U_{\res}^{\osc}\big)_{|x_d=0,\psi_d=0}&=H_{\res}^{\osc} \label{eq estim syst osc cond bord} \\[5pt]
		\big(U_{\res}^{\osc}\big)_{|t\leq 0}&=0, \label{eq estim syst osc cond init}
		\end{align}
	\end{subequations}
	where $ H^{\osc} $ is defined by equation \eqref{eq red def H osc res}.
	Then the profile $ U_{\res}^{\osc} $ satisfies the a priori estimate
	\begin{equation}\label{eq estim est a priori}
	\norme{U_{\res}^{\osc}}^2_{\mathcal{E}_{s,T}}\leq C_1\,e^{C(V)\,\mathcal{V}^*T}\norme{G}^2_{H^s(\omega_T\times\T^m)}
	+ \V^*T\,e^{C(V)\mathcal{V}^*T}\norme{F_{\res}^{\osc}}^2_{\mathcal{E}_{s,T}},
	\end{equation} 
	where $ C(V):=C_1\big(1+\norme{V_{\res}^{\osc}}^2_{\mathcal{E}_{s,T}}\big) $, with $ C_1>0 $ a positive constant depending only on the operator $ L(0,\partial_z) $ and of $ s $. 
	Recall that the real number $ \mathcal{V}^* $, which bounds the group velocities $ \mathbf{v}_{\alpha} $, has been defined in Lemma \ref{lemme vitesse de groupe bornees}.
\end{proposition}

Consider from now on an integer $ s>s_0 $.

\subsubsection{Rewriting the linearized oscillating system}

In system \eqref{eq estim syst osc} which is the linearization of system \eqref{eq red syst osc res} around $ V_{\res}^{\osc} $ in $ \P^{\osc}_{s,T} $, a source term $ F_{\res}^{\osc} $ in $ \P^{\osc}_{s,T} $ has been added, which will be useful to deduce from the $ L^2 $ estimate the higher order estimates, as well as in the iterative schemes used to construct solutions to the linearized system \eqref{eq estim syst osc} and to system \eqref{eq red syst osc res}. 
To simplify the equations, the function $ \beta_T $ will be omitted in the following.

The analysis conducted in the previous subsection is now reproduced to rewrite the left term of equality \eqref{eq estim syst osc E i U}. 
Since the profile $ U_{\res}^{\osc} $ satisfies the polarization condition \eqref{eq estim syst osc E U = U} and involves only incoming modes,
according to Remark \ref{remarque forme U osc polarise}, it writes
\begin{equation}\label{eq estim pola osc}
U_{\res}^{\osc}(z,\theta,\psi_d)=\sum_{(\mathbf{n}_0,\xi_0)\in\F^{\inc}_{\res}}\sum_{\lambda\in\Z^*}U^{\osc}_{\lambda\mathbf{n}_0,\lambda\xi_0}(z)\,e^{i\lambda\mathbf{n}_0\cdot\theta}\,e^{i\lambda\xi_0\,\psi_d}.
\end{equation}
with $ U^{\osc}_{\lambda\mathbf{n}_0,\lambda\xi_0}=\pi_{(\lambda\mathbf{n}_0\cdot\boldsymbol{\zeta},\lambda\xi_0)}\,U^{\osc}_{\lambda\mathbf{n}_0,\lambda\xi_0} $ for all $ \mathbf{n}_0,\xi_0,\lambda $. 
In the same way, since $ V_{\res}^{\osc} $ is in $ \N^{\osc}_{s,T} $ with only incoming resonant modes, we have
\begin{equation}\label{eq estim pola V osc}
	V_{\res}^{\osc}(z,\theta,\psi_d)=\sum_{(\mathbf{n}_0,\xi_0)\in\F^{\inc}_{\res}}\sum_{\lambda\in\Z^*}V^{\osc}_{\lambda\mathbf{n}_0,\lambda\xi_0}(z)\,e^{i\lambda\mathbf{n}_0\cdot\theta}\,e^{i\lambda\xi_0\,\psi_d}.
\end{equation}
with $ V^{\osc}_{\lambda\mathbf{n}_0,\lambda\xi_0}=\pi_{(\lambda\mathbf{n}_0\cdot\boldsymbol{\zeta},\lambda\xi_0)}\,V^{\osc}_{\lambda\mathbf{n}_0,\lambda\xi_0}  $ for all $ \mathbf{n}_0,\xi_0,\lambda $.
Once again, for $ (\mathbf{n}_0,\xi_0) $ in $ \F^{\inc}_{\res} $ and $ \lambda $ in $ \Z^* $, since the profiles $ U $ and $ V $ are polarized, we write
\begin{align*}
U^{\osc}_{\lambda\mathbf{n}_0,\lambda\xi_0}(z)&=\sigma_{\lambda,\mathbf{n}_0,\xi_0}(z)\,E(\mathbf{n}_0,\xi_0),\\[5pt]
V^{\osc}_{\lambda\mathbf{n}_0,\lambda\xi_0}(z)&=\omega_{\lambda,\mathbf{n}_0,\xi_0}(z)\,E(\mathbf{n}_0,\xi_0).
\end{align*} 
Note that according to identity \eqref{eq prod scal pol trigo rentrants}, the scalar product $ \prodscal{U_{\res}^{\osc}}{U_{\res}^{\osc}}_{\inc}(x_d) $ is given in this notation by
\begin{equation*}
\prodscal{U_{\res}^{\osc}}{U_{\res}^{\osc}}_{\inc}(x_d) =(2\pi)^m\sum_{(\mathbf{n}_0,\xi_0)\in\F^{\inc}_{\res}}\sum_{\lambda\in\Z^*}\norme{\sigma_{\lambda,\mathbf{n}_0,\xi_0}}_{L^2(\omega_T)}^2(x_d)  .
\end{equation*}
Since the projector $ \tilde{\E^i}^{\inc}_{\res} $ is applied to the source term $ F_{\res}^{\osc} $, one can assume without loss of generality that the latter writes
\begin{equation*}
	F_{\res}^{\osc}(z,\theta,\psi_d)=\sum_{(\mathbf{n}_0,\xi_0)\in\F^{\inc}_{\res}}\sum_{\lambda\in\Z^*}F_{\lambda,\mathbf{n}_0,\xi_0}(z)\,e^{i\mathbf{n}\cdot\theta}\,e^{i\xi\,\psi_d}.
\end{equation*}
We then denote, for $(\mathbf{n}_0,\xi_0) $ in $ \F^{\inc}_{\res} $ and $ \lambda $ in $ \Z^* $, by $ f_{\lambda,\mathbf{n}_0,\xi_0} $ the scalar function of $ \Omega_T $ such that
\begin{equation*}
	\tilde{\pi}_{(\mathbf{n}_0\cdot\boldsymbol{\zeta},\xi_0)}\,F_{\lambda,\mathbf{n}_0,\xi_0}=f_{\lambda,\mathbf{n}_0,\xi_0} \,\tilde{\pi}_{(\mathbf{n}_0\cdot\boldsymbol{\zeta},\xi_0)}\,E(\mathbf{n}_0,\xi_0)
\end{equation*}
so that $ \tilde{\E^i}^{\inc}_{\res}\,F_{\res}^{\osc} $ writes
\begin{equation*}
	\tilde{\E^i}^{\inc}_{\res}\,F_{\res}^{\osc} (z,\theta,\psi_d)=\sum_{(\mathbf{n}_0,\xi_0)\in\F^{\inc}_{\res}}\sum_{\lambda\in\Z^*}f_{\lambda,\mathbf{n}_0,\xi_0}(z)\,e^{i\mathbf{n}\cdot\theta}\,e^{i\xi\,\psi_d}\tilde{\pi}_{(\mathbf{n}_0\cdot\boldsymbol{\zeta},\xi_0)}\,E(\mathbf{n}_0,\xi_0).
\end{equation*}
According to estimate \eqref{eq propr minoration pi tilde E R1} of Assumption \ref{hypothese type resonance}, there exists a positive constant $ C $ such that for all $ (\mathbf{n}_0,\xi_0) $ in $ \F^{\inc}_{\res} $ and all $ \lambda $ in $ \Z^* $, we have
\begin{equation}\label{eq estim f F}
 \norme{f_{\lambda,\mathbf{n}_0,\xi_0}}_{L^2(\Omega_T)}\leq C \norme{F_{\lambda,\mathbf{n}_0,\xi_0}}_{L^2(\Omega_T)}.
\end{equation}

In this notation, the resonant incoming modes satisfy the following coupled equation, connecting the source term  $ \tilde{\E^i}^{\inc}_{\res}\,F_{\res}^{\osc} $
\begin{subequations}\label{eq estim E i L U}
	\begin{align}
	&\sum_{(\mathbf{n}_0,\xi_0)\in\F^{\inc}_{\res}}\sum_{\lambda\in\Z^*}f_{\lambda,\mathbf{n}_0,\xi_0}(z)\,e^{i\,\lambda\mathbf{n}_0\cdot\theta}\,e^{i\,\lambda\xi_0\,\psi_d}\,\tilde{\pi}_{(\mathbf{n}_0\cdot\boldsymbol{\zeta},\xi_0)}\,E(\mathbf{n}_0,\xi_0)
	\intertext{with the sum of a transport term, corresponding to $ \tilde{L}(0,\partial_z)\,U_{\res}^{\osc} $,}
	&\quad=\sum_{(\mathbf{n}_0,\xi_0)\in\F^{\inc}_{\res}}\sum_{\lambda\in\Z^*}\tilde{X}_{(\mathbf{n}_0\cdot\boldsymbol{\zeta},\xi_0)}\,\sigma_{\lambda,\mathbf{n}_0,\xi_0}\,e^{i\,\lambda\mathbf{n}_0\cdot\theta}\,e^{i\,\lambda\xi_0\,\psi_d}\,\tilde{\pi}_{(\mathbf{n}_0\cdot\boldsymbol{\zeta},\xi_0)}\,E(\mathbf{n}_0,\xi_0),\label{eq estim E i L U trsp osc}
	\intertext{a self-interaction term,}
	&\quad+\sum_{(\mathbf{n}_0,\xi_0)\in\F^{\inc}_{\res}}\sum_{\lambda\in\Z^*}\sum_{\substack{\lambda_1,\lambda_2\in\Z^*\\\lambda_1+\lambda_2=\lambda}}i\,\lambda_2\,\omega_{\lambda_1,\mathbf{n}_0,\xi_0}\,\sigma_{\lambda_2,\mathbf{n}_0,\xi_0}\,\Gamma\big((\mathbf{n}_0,\xi_0),(\mathbf{n}_0,\xi_0)\big)\label{eq estim E i L U AI}\\
	\nonumber&\qquad\qquad e^{i\,\lambda\mathbf{n}_0\cdot\theta}\,e^{i\,\lambda\xi_0\,\psi_d}\,\tilde{\pi}_{(\mathbf{n}_0\cdot\boldsymbol{\zeta},\xi_0)}\,E(\mathbf{n}_0,\xi_0),\intertext{and resonance terms of type 1,}
	&\quad+\sum_{(\mathbf{n}_0,\xi_0)\in\F^{\inc}_{\res}}\sum_{\substack{(\lambda_p,\lambda_q,\lambda_r,\mathbf{n}_p,\mathbf{n}_q,\\\xi_p,\xi_q)\in\mathcal{R}_1(\mathbf{n}_0,\xi_0)}}\sum_{\ell\in\Z^*}i\,\ell\,\omega_{\ell\lambda_p,\mathbf{n}_p,\xi_p}\,\sigma_{\ell\lambda_q,\mathbf{n}_q,\xi_q}\,\Gamma\big(\lambda_p(\mathbf{n}_p,\xi_p),\lambda_q(\mathbf{n}_q,\xi_q)\big)\label{eq estim E i L U R1}\\
	\nonumber&\qquad\qquad e^{i\,\ell\,\lambda_r\mathbf{n}_0\cdot\theta}\,e^{i\,\ell\,\lambda_r\xi_0\,\psi_d}\,\tilde{\pi}_{(\mathbf{n}_0\cdot\boldsymbol{\zeta},\xi_0)}\,E(\mathbf{n}_0,\xi_0),
	\intertext{and of type 2,}
	&\quad+\sum_{(\mathbf{n}_0,\xi_0)\in\F^{\inc}_{\res}}\sum_{\substack{(\lambda_p,\lambda_q,\lambda_r,\mathbf{n}_p,\mathbf{n}_q,\\\xi_p,\xi_q)\in\mathcal{R}_2(\mathbf{n}_0,\xi_0)}}\sum_{\ell\in\Z^*}i\,\ell\,\omega_{\ell\lambda_p,\mathbf{n}_p,\xi_p}\,\sigma_{\ell\lambda_q,\mathbf{n}_q,\xi_q}\,\Gamma\big(\lambda_p(\mathbf{n}_p,\xi_p),\lambda_q(\mathbf{n}_q,\xi_q)\big)\label{eq estim E i L U R2}\\
	\nonumber&\qquad\qquad e^{i\,\ell\,\lambda_r\mathbf{n}_0\cdot\theta}\,e^{i\,\ell\,\lambda_r\xi_0\,\psi_d}\,\tilde{\pi}_{(\mathbf{n}_0\cdot\boldsymbol{\zeta},\xi_0)}\,E(\mathbf{n}_0,\xi_0).
	\end{align}
\end{subequations}

Note that in terms \eqref{eq estim E i L U AI}, \eqref{eq estim E i L U R1} and \eqref{eq estim E i L U R2}, factors $ \lambda_2\,\Gamma\big((\mathbf{n}_0,\xi_0),(\mathbf{n}_0,\xi_0)\big) $ and $ \Gamma\big(\lambda_p(\mathbf{n}_p,\xi_p),\allowbreak\lambda_q(\mathbf{n}_q,\xi_q)\big) $ imply that something like a derivative with respect to $ \theta $ is applied to $ U $. To obtain estimates without loss of derivatives, one therefore needs the $ \theta $ derivative to apply to the coefficient $ V^{\osc}_{\res} $, which is the whole point of the following paragraph. Terms \eqref{eq estim E i L U AI} and \eqref{eq estim E i L U R2} are treated in the same way as the corresponding terms for the outgoing modes, whereas Assumption \ref{hypothese type resonance} is used to treat  term \eqref{eq estim E i L U R1} of resonances of type 1. 

\subsubsection{$ L^2 $ estimate}

This subsection is devoted to the proof of the following lemma.

\begin{lemma}\label{lemme estim est a priori L2}
	Consider $ s>s_0 $ and $ T>0 $, and let $ U_{\res}^{\osc} $ be in $ \P^{\osc}_{1,T} $, $ F^{\osc}_{\res} $ in $ \P^{\osc}_{0,T} $ and $ V_{\res}^{\osc} $ in $ \N^{\osc}_{s,T} $, only involving resonant incoming modes, satisfying system \eqref{eq estim syst osc}. Then the following estimate holds for $ x_d\geq0 $,
	\begin{equation}\label{eq estim est a priori L2}
		\frac{d}{dx_d}\prodscal{U_{\res}^{\osc}}{U_{\res}^{\osc}}_{\inc}(x_d)
		\leq C\prodscal{F^{\osc}_{\res}}{F_{\res}^{\osc}}_{\inc}(x_d)+C\left(1+\norme{V_{\res}^{\osc}}_{\mathcal{E}_{s,T}}\right)\prodscal{U_{\res}^{\osc}}{U_{\res}^{\osc}}_{\inc}(x_d).
	\end{equation}
\end{lemma}
 
 \begin{proof}
Consider the modified profile $ \tilde{U}_{\res}^{\osc} $ given by
\begin{equation*}
	\tilde{U}_{\res}^{\osc}(z,\theta,\psi_d):=\sum_{(\mathbf{n}_0,\xi_0)\in\F^{\inc}_{\res}}\sum_{\lambda\in\Z^*}\sigma_{\lambda,\mathbf{n}_0,\xi_0}(z)\,e^{i\,\lambda\mathbf{n}_0\cdot\theta}\,e^{i\,\lambda\xi_0\,\psi_d}\frac{\tilde{\pi}_{(\mathbf{n}_0\cdot\boldsymbol{\zeta},\xi_0)}\,E(\mathbf{n}_0,\xi_0)}{|\tilde{\pi}_{(\mathbf{n}_0\cdot\boldsymbol{\zeta},\xi_0)}\,E(\mathbf{n}_0,\xi_0)|^2}.
\end{equation*}
Note that despite the factor $ |\tilde{\pi}_{(\mathbf{n}_0\cdot\boldsymbol{\zeta},\xi_0)}\,E(\mathbf{n}_0,\xi_0)|^{-1} $ that could present a problem, the profile $ \tilde{U}^{\osc}_{\res} $ is well defined, since the set $ \F^{\inc}_{\res} $ satisfies the property \eqref{eq propr minoration pi tilde E R1} ensuring that these factors are uniformly lower bounded.
They have been introduced to balance the factors $ |\tilde{\pi}_{(\mathbf{n}_0\cdot\boldsymbol{\zeta},\xi_0)}\,E(\mathbf{n}_0,\xi_0)| $ which will occur in the estimate. These factors may not be uniformly bounded with respect to $ (\mathbf{n}_0,\xi_0) $ varying in the (potentially infinite) set of directions $ (\mathbf{n}_0,\xi_0) $ of $ \B_{\Z^m}\times\mathcal{C}_{\inc}(\mathbf{n}_0) $ such that $ \mathcal{R}_1(\mathbf{n}_0,\xi_0)\cup\mathcal{R}_2(\mathbf{n}_0,\xi_0)  $ is empty, justifying the choice to treat them separately below.

Taking the double of the real part of the scalar product \eqref{eq def prod scal rentrant} of equality \eqref{eq estim E i L U} with the profile $ \tilde{U}^{\osc}_{\res} $,
one gets an equality, with on one side the term
\begin{equation*}
2\Re \prodscal{\tilde{\E^i}^{\inc}_{\res}\,F^{\osc}_{\res}}{\tilde{U}^{\osc}_{\res}}_{\inc}(x_d),
\end{equation*}
which is estimated in the following way:
\begin{align}
	\nonumber\left|2\Re \prodscal{\tilde{\E^i}^{\inc}_{\res}\,F^{\osc}_{\res}}{\tilde{U}^{\osc}_{\res}}_{\inc}(x_d)\right|&=\left|2\Re \prodscal{\hat{F}^{\osc}_{\res}}{\hat{U}^{\osc}_{\res}}_{\inc}(x_d)\right|\\
	&\nonumber\leq C\prodscal{\hat{F}^{\osc}_{\res}}{\hat{F}^{\osc}_{\res}}_{\inc}^{1/2}(x_d)\prodscal{\hat{U}^{\osc}_{\res}}{\hat{U}^{\osc}_{\res}}_{\inc}^{1/2}(x_d)\\[5pt]
	&\leq 
	C\prodscal{F^{\osc}_{\res}}{F^{\osc}_{\res}}_{\inc}(x_d)+C\prodscal{U^{\osc}_{\res}}{U^{\osc}_{\res}}_{\inc}(x_d),\label{eq estim prod F}
\end{align}
where it has been denoted
\begin{align*}
\hat{F}^{\osc}_{(\mathbf{n}_0,\xi_0)}(z,\theta,\psi_d)&:=\sum_{\lambda\in\Z^*}f_{\lambda,\mathbf{n}_0,\xi_0}(z)\,e^{i\,\lambda\mathbf{n}_0\cdot\theta}\,e^{i\,\lambda\xi_0\,\psi_d}\,\frac{\tilde{\pi}_{(\mathbf{n}_0\cdot\boldsymbol{\zeta},\xi_0)}\,E(\mathbf{n}_0,\xi_0)}{\left|\tilde{\pi}_{(\mathbf{n}_0\cdot\boldsymbol{\zeta},\xi_0)}\,E(\mathbf{n}_0,\xi_0)\right|},
\intertext{and}
\hat{U}^{\osc}_{(\mathbf{n}_0,\xi_0)}(z,\theta,\psi_d) &:=\sum_{\lambda\in\Z^*}\sigma_{\lambda,\mathbf{n}_0,\xi_0}(z)\,e^{i\,\lambda\mathbf{n}_0\cdot\theta}\,e^{i\,\lambda\xi_0\,\psi_d}\,\frac{\tilde{\pi}_{(\mathbf{n}_0\cdot\boldsymbol{\zeta},\xi_0)}\,E(\mathbf{n}_0,\xi_0)}{\left|\tilde{\pi}_{(\mathbf{n}_0\cdot\boldsymbol{\zeta},\xi_0)}\,E(\mathbf{n}_0,\xi_0)\right|},
\end{align*}
so that the profile $ \hat{U}_{\res}^{\osc} $  is such that its scalar product with itself equals the one of $ U^{\osc}_{(\mathbf{n}_0,\xi_0)} $ with itself, and according to estimate \eqref{eq estim f F}, the scalar product of $ \hat{F}^{\osc}_{\res} $ with itself is bounded, up to a positive multiplicative constant, by the one of $ F^{\osc}_{\res} $ with itself. Since the lower bound \eqref{eq propr minoration pi tilde E R1} is in general not verified by the non resonant modes, the analogue of estimate \eqref{eq estim prod F} seems false, explaining why these modes cannot be treated in the same way as the resonant modes in this subsection, which leads us to go back to scalar equations for the first ones.

Now the right hand side terms of the equality obtained by taking the double of the real part of the scalar product of equality \eqref{eq estim E i L U} with $ \tilde{U}^{\osc}_{\res} $ are investigated. The analysis of the terms corresponding to terms \eqref{eq estim E i L U trsp osc}, \eqref{eq estim E i L U AI} and \eqref{eq estim E i L U R2} is analogous to the one made for the outgoing modes.

Concerning the transport term \eqref{eq estim E i L U trsp osc}, identity \eqref{eq prod scal pol trigo rentrants} and an integration by parts lead to
\begin{multline}\label{eq estim prod trsp}
	2\Re\prodscal{\eqref{eq estim E i L U trsp osc}}{\tilde{U}^{\osc}_{\res}}_{\inc}(x_d)=\frac{d}{dx_d}\prodscal{U_{\res}^{\osc}}{U_{\res}^{\osc}}_{\inc}(x_d)\\
	+(2\pi)^m\sum_{(\mathbf{n}_0,\xi_0)\in\F^{\inc}_{\res}}\sum_{\lambda\in\Z^*}\frac{-1}{\partial_{\xi}\tau_{k(\mathbf{n}_0,\xi_0)}(\mathbf{n}_0\cdot\boldsymbol{\eta},\xi_0)}\norme{\sigma_{\lambda,\mathbf{n}_0,\xi_0}}^2_{L^2(\R^{d-1})}(T).
\end{multline}
Note that since all modes are incoming here, the quantity $ -\partial_{\xi}\tau_{k(\mathbf{n}_0,\xi_0)}(\mathbf{n}_0\cdot\boldsymbol{\eta},\xi_0) $ is positive for all $ (\mathbf{n}_0,\xi_0) $, which will allow us to omit the second term on the right of the equality in the estimates below.

For the self-interaction term \eqref{eq estim E i L U AI} one can compute, 
\begin{subequations}
	\begin{align}
		\nonumber&\prodscal{\eqref{eq estim E i L U AI}}{\tilde{U}_{\res}^{\osc}}_{\inc}(x_d)\\
		&\quad=(2\pi)^m\sum_{(\mathbf{n}_0,\xi_0)\in\F^{\inc}_{\res}}\sum_{\lambda\in\Z^*}\sum_{\substack{\lambda_1,\lambda_2\in\Z^*\\\lambda_1+\lambda_2=\lambda}}i\lambda_2\,\Gamma\big((\mathbf{n}_0,\xi_0),(\mathbf{n}_0,\xi_0)\big)\prodscal{\omega_{\lambda_1,\mathbf{n}_0,\xi_0}\,\sigma_{\lambda_2,\mathbf{n}_0,\xi_0}}{\sigma_{\lambda,\mathbf{n}_0,\xi_0}}_{L^2(\omega_T)}(x_d)\label{eq estim inter 1 res}\\
		&\quad=(2\pi)^m\sum_{(\mathbf{n}_0,\xi_0)\in\F^{\inc}_{\res}}\sum_{\lambda\in\Z^*}\sum_{\substack{\lambda_1,\lambda_2\in\Z^*\\\lambda_1+\lambda_2=\lambda}}i\lambda\,\Gamma\big((\mathbf{n}_0,\xi_0),(\mathbf{n}_0,\xi_0)\big)\prodscal{\omega_{\lambda_1,\mathbf{n}_0,\xi_0}\,\sigma_{\lambda_2,\mathbf{n}_0,\xi_0}}{\sigma_{\lambda,\mathbf{n}_0,\xi_0}}_{L^2(\omega_T)}(x_d)\label{eq estim inter 2 res}\\
		&\quad-(2\pi)^m\sum_{(\mathbf{n}_0,\xi_0)\in\F^{\inc}_{\res}}\sum_{\lambda\in\Z^*}\sum_{\substack{\lambda_1,\lambda_2\in\Z^*\\\lambda_1+\lambda_2=\lambda}}i\lambda_1\,\Gamma\big((\mathbf{n}_0,\xi_0),(\mathbf{n}_0,\xi_0)\big)\prodscal{\omega_{\lambda_1,\mathbf{n}_0,\xi_0}\,\sigma_{\lambda_2,\mathbf{n}_0,\xi_0}}{\sigma_{\lambda,\mathbf{n}_0,\xi_0}}_{L^2(\omega_T)}(x_d).
	\end{align}
\end{subequations}
But with already detailed computations, one gets $
\eqref{eq estim inter 2 res}=-\bar{\eqref{eq estim inter 1 res}} $, so
\begin{multline*}
2\Re \prodscal{\eqref{eq estim E i L U AI}}{\tilde{U}^{\osc}_{\res}}_{\inc}(x_d)=\\
-(2\pi)^m\sum_{(\mathbf{n}_0,\xi_0)\in\F^{\inc}_{\res}}\sum_{\lambda\in\Z^*}\sum_{\substack{\lambda_1,\lambda_2\in\Z^*\\\lambda_1+\lambda_2=\lambda}}i\lambda_1\,\Gamma\big((\mathbf{n}_0,\xi_0),(\mathbf{n}_0,\xi_0)\big)\prodscal{\omega_{\lambda_1,\mathbf{n}_0,\xi_0}\,\sigma_{\lambda_2,\mathbf{n}_0,\xi_0}}{\sigma_{\lambda,\mathbf{n}_0,\xi_0}}_{L^2(\omega_T)}(x_d).
\end{multline*}
Note that this term differs from \eqref{eq estim inter 1 res} because of the coefficient $ \lambda_1 $ instead of $ \lambda_2 $, which makes the derivatives with respect to $ \theta $ apply on the coefficient $ V_{\res}^{\osc} $ instead of on the unknown $ U_{\res}^{\osc} $. Upper bound \eqref{eq est gamma n n resonance} therefore leads to
\begin{multline*}
	\left|2\Re \prodscal{\eqref{eq estim E i L U AI}}{\tilde{U}^{\osc}_{\res}}_{\inc}(x_d)\right|\\
	\leq C\sum_{(\mathbf{n}_0,\xi_0)\in\F^{\inc}_{\res}}\sum_{\lambda\in\Z^*}\sum_{\substack{\lambda_1,\lambda_2\in\Z^*\\\lambda_1+\lambda_2=\lambda}}|\lambda_1\mathbf{n}_0|\left|\prodscal{\omega_{\lambda_1,\mathbf{n}_0,\xi_0}\,\sigma_{\lambda_2,\mathbf{n}_0,\xi_0}}{\sigma_{\lambda,\mathbf{n}_0,\xi_0}}_{L^2(\omega_T)}(x_d)\right|.
\end{multline*}
The term on the right of the equality is of the form $ \prodscal{fg}{g} $, so we get
\begin{equation}
	\left|2\Re \prodscal{\eqref{eq estim E i L U AI}}{\tilde{U}^{\osc}_{\res}}_{\inc}(x_d)\right|\leq C
	\norme{V_{\res}^{\osc}}_{\mathcal{E}_{s,T}}\prodscal{U_{\res}^{\osc}}{U_{\res}^{\osc}}_{\inc}(x_d).\label{eq estim prod AI}
\end{equation}

For term \eqref{eq estim E i L U R2} of type 2 resonances, we write
\begin{multline*}
\prodscal{\eqref{eq estim E i L U R2}}{\tilde{U}^{\osc}_{\res}}_{\inc}(x_d)=
\sum_{(\mathbf{n}_0,\xi_0)\in\F^{\inc}_{\res}}\sum_{\substack{(\lambda_p,\lambda_q,\lambda_r,\mathbf{n}_p,\mathbf{n}_q,\\\xi_p,\xi_q)\in\mathcal{R}_2(\mathbf{n}_0,\xi_0)}} \sum_{\ell\in\Z^*} i\ell\,\Gamma\big(\lambda_p(\mathbf{n}_p,\xi_p),\lambda_q(\mathbf{n}_q,\xi_q)\big) \\
\prodscal{\omega_{\ell\lambda_p,\mathbf{n}_p,\xi_p}\,\sigma_{\ell\lambda_q,\mathbf{n}_q,\xi_q}}{\sigma_{\ell\lambda_r,\mathbf{n}_r,\xi_r}}_{L^2(\omega_T)}(x_d).
\end{multline*} 
Then the following upper bound is derived, for all $ (\lambda_p,\lambda_q,\lambda_r,\mathbf{n}_p,\mathbf{n}_q,\xi_p,\xi_q)\in\mathcal{R}_2(\mathbf{n}_0,\xi_0) $ with $ \mathbf{n}_0\in\B_{\Z^m} $, $ \xi_0\in\mathcal{C}_{\inc}(\mathbf{n}_0) $ (which constitutes a finite set, see Assumption \ref{hypothese type resonance}), 
\begin{equation*}
 \left|\Gamma\big(\lambda_p(\mathbf{n}_p,\xi_p),\lambda_q(\mathbf{n}_q,\xi_q)\big)\right|\leq C |\lambda_p|\left|\mathbf{n}_p,\xi_p\right|,
\end{equation*}
where the constant $ C>0 $ is independent of $ \mathbf{n}_p,\mathbf{n}_q,\mathbf{n}_0,\xi_p,\xi_q $ and $ \xi_0 $. Once again, with this bound, the derivative with respect to $ \theta $ no longer apply on $ U_{\res}^{\osc} $ but only on $ V_{\res}^{\osc} $. The following estimate is thus deduced:
\begin{equation}\label{eq estim prod R2}
\left|2\Re \prodscal{\eqref{eq estim E i L U R2}}{\tilde{U}^{\osc}_{\res}}_{\inc}(x_d)\right|\leq
C\norme{V_{\res}^{\osc}}_{\mathcal{E}_{s,T}}\prodscal{U_{\res}^{\osc}}{U_{\res}^{\osc}}_{\inc}(x_d).
\end{equation}

Finally term \eqref{eq estim E i L U R1} of type 1 resonances is investigated, which is treated following \cite[Chapter 11]{Rauch2012Hyperbolic}. Once again the aim is to have a derivative applying totally on $ V_{\res}^{\osc} $. First the set on which the sum \eqref{eq estim E i L U R1} is taken is parameterized in a different way. The set $ \mathcal{R}_1 $ of type 1 incoming resonant 6-tuples is defined as
\begin{equation*}
\displaystyle\mathcal{R}_1:=\ensemble{
	\begin{array}{c}
	\big(\ell\,\lambda_p\,\mathbf{n}_p,\ell\,\lambda_p\,\xi_p,\ell\,\lambda_q\,\mathbf{n}_q,\\
	\ell\,\lambda_q\,\xi_q,-\ell\,\lambda_r\,\mathbf{n}_0,-\ell\,\lambda_r\,\xi_0\big)
	\end{array}
	\,\middle|\,\begin{array}{c}
	\ell\in\Z^*,\,\mathbf{n}_0\in\B_{\Z^m},\,\xi_0\in\mathcal{C}_{\inc}(\mathbf{n}_0),\\(\lambda_p,\lambda_q,\lambda_r,\mathbf{n}_p,\mathbf{n}_q,\xi_p,\xi_q)\in\mathcal{R}_1(\mathbf{n}_0,\xi_0)
	\end{array}}.
\end{equation*}
Note that if $ (\mathbf{n}_p,\xi_p,\mathbf{n}_q,\xi_q,\mathbf{n}_r,\xi_r) $ is in $ \mathcal{R}_1 $, then $ \mathbf{n}_p+ \mathbf{n}_q+\mathbf{n}_r=0 $, and $ \xi_p+\xi_q+\xi_r=0 $. We also see that, according to remark \ref{remarque type resonance}, a 6-tuple $ (\mathbf{n}_p,\xi_p,\mathbf{n}_q,\xi_q,\mathbf{n}_r,\xi_r) $ is in $ \mathcal{R}_1 $ if and only if the symmetrical 6-tuple $ (\mathbf{n}_p,\xi_p,\mathbf{n}_r,\xi_r,\mathbf{n}_q,\xi_q) $ is in $ \mathcal{R}_1 $.
According to identity \eqref{eq propr homogeneite Gamma}, we have
\begin{align*}
	\nonumber\eqref{eq estim E i L U R1}=
	&\nonumber\sum_{(\mathbf{n}_0,\xi_0)\in\F^{\inc}_{\res}}\sum_{\substack{(\lambda_p,\lambda_q,\lambda_r,\mathbf{n}_p,\mathbf{n}_q,\\\xi_p,\xi_q)\in\mathcal{R}_1(\mathbf{n}_0,\xi_0)}}\sum_{\ell\in\Z^*}i\,\omega_{\ell\lambda_p,\mathbf{n}_p,\xi_p}\,\sigma_{\ell\lambda_q,\mathbf{n}_q,\xi_q}\,\Gamma\big((\ell\lambda_p\mathbf{n}_p,\ell\lambda_p\xi_p),(\ell\lambda_q\mathbf{n}_q,\ell\lambda_q\xi_q)\big)\\
	\nonumber&\qquad\qquad e^{i\,\ell\lambda_r\mathbf{n}_0\cdot\theta}\,e^{i\,\ell\lambda_r\xi_0\,\psi_d}\,\tilde{\pi}_{(\ell\lambda_r\mathbf{n}_0\cdot\boldsymbol{\zeta},\ell\lambda_r\xi_0)}\,E(\ell\lambda_r\mathbf{n}_0,\ell\lambda_r\xi_0)\\[5pt]
	=&\sum_{\substack{(\mathbf{n}_p,\xi_p,\mathbf{n}_q,\xi_q,\\\mathbf{n}_r,\xi_r)\in\mathcal{R}_1}}i\,\omega_{\mathbf{n}_p,\xi_p}\,\sigma_{\mathbf{n}_q,\xi_q}\,\Gamma\big((\mathbf{n}_p,\xi_p),(\mathbf{n}_q,\xi_q)\big)
	e^{-i\,\mathbf{n}_r\cdot\theta}\,e^{-i\xi_r\,\psi_d}\,\tilde{\pi}_{(\mathbf{n}_r\cdot\boldsymbol{\zeta},\xi_r)}\,E(\mathbf{n}_r,\xi_r).
\end{align*}
If $ \mathbf{n} $ in $ \Z^m\privede{0} $ and $ \xi $ in $ \mathcal{C}(\mathbf{n}) $ write as $ (\mathbf{n},\xi)=\lambda\,(\mathbf{n}_0,\xi_0) $ with $ \mathbf{n}_0\in\B_{\Z^m} $, $ \xi\in\mathcal{C}(\mathbf{n}_0) $ and $ \lambda\in\Z^* $, we have denoted
\begin{equation*}
\sigma_{\mathbf{n},\xi}:=\sigma_{\lambda,\mathbf{n}_0,\xi_0},\quad\text{and}\quad \omega_{\mathbf{n},\xi}:=\omega_{\lambda,\mathbf{n}_0,\xi_0}.
\end{equation*}
 Therefore we have
\begin{align*}
	\prodscal{\eqref{eq estim E i L U R1}}{\tilde{U}^{\osc}_{\res}}_{\inc}(x_d)
	&=\sum_{\substack{(\mathbf{n}_p,\xi_p,\mathbf{n}_q,\xi_q,\\\mathbf{n}_r,\xi_r)\in\mathcal{R}_1}}i\,\Gamma\big((\mathbf{n}_p,\xi_p),(\mathbf{n}_q,\xi_q)\big)\prodscal{\omega_{\mathbf{n}_p,\xi_p}\,\sigma_{\mathbf{n}_q,\xi_q}}{\sigma_{-\mathbf{n}_r,-\xi_r}}_{L^2(\omega_T)}(x_d)\\
	&=\sum_{\substack{(\mathbf{n}_p,\xi_p,\mathbf{n}_q,\xi_q,\\\mathbf{n}_r,\xi_r)\in\mathcal{R}_1}}i\,\Gamma\big((\mathbf{n}_p,\xi_p),(\mathbf{n}_q,\xi_q)\big)\bar{\prodscal{\omega_{-\mathbf{n}_p,-\xi_p}\,\sigma_{-\mathbf{n}_r,-\xi_r}}{\sigma_{\mathbf{n}_q,\xi_q}}}_{L^2(\omega_T)}(x_d)\\
	&=\sum_{\substack{(\mathbf{n}_p,\xi_p,\mathbf{n}_q,\xi_q,\\\mathbf{n}_r,\xi_r)\in\mathcal{R}_1}}-i\,\Gamma\big((\mathbf{n}_p,\xi_p),(\mathbf{n}_r,\xi_r)\big)\bar{\prodscal{\omega_{\mathbf{n}_p,\xi_p}\,\sigma_{\mathbf{n}_q,\xi_q}}{\sigma_{-\mathbf{n}_r,-\xi_r}}}_{L^2(\omega_T)}(x_d).
\end{align*}
We have used here the fact that $ \bar{\omega_{\mathbf{n}_p,\xi_p}}=\omega_{-\mathbf{n}_p,-\xi_p} $, the profile $ V^{\osc} $ being real, a change of variables $ (\mathbf{n}_p,\allowbreak\mathbf{n}_r,\mathbf{n}_q,\allowbreak\xi_p,\xi_r,\xi_q)=-(\mathbf{n}_p,\mathbf{n}_q,\mathbf{n}_r,\xi_p,\xi_q,\xi_r) $, the fact that $ -\mathcal{R}_1=\mathcal{R}_1 $ and the identity \eqref{eq propr homogeneite Gamma}. Thus we obtain
\begin{align*}
 2\Re &\prodscal{\eqref{eq estim E i L U R1}}{\tilde{U}^{\osc}_{\res}}_{\inc}(x_d) \\
  =&\sum_{\substack{(\mathbf{n}_p,\xi_p,\mathbf{n}_q,\xi_q,\\\mathbf{n}_r,\xi_r)\in\mathcal{R}_1}}i\Big\{\Gamma\big((\mathbf{n}_p,\xi_p),(\mathbf{n}_q,\xi_q)\big)+\Gamma\big((\mathbf{n}_p,\xi_p),(\mathbf{n}_r,\xi_r)\big)\Big\}\prodscal{\omega_{\mathbf{n}_p,\xi_p}\,\sigma_{\mathbf{n}_q,\xi_q}}{\sigma_{-\mathbf{n}_r,-\xi_r}}_{L^2(\omega_T)}(x_d).
\end{align*}
Using the uniform estimate \eqref{eq propr res type 1} given by Assumption \ref{hypothese type resonance}, one obtains
\begin{align}
\nonumber\left|2\Re \prodscal{\eqref{eq estim E i L U R1}}{\tilde{U}^{\osc}_{\res}}_{\inc}(x_d)\right|&\leq C\sum_{\substack{(\mathbf{n}_p,\xi_p,\mathbf{n}_q,\xi_q,\\\mathbf{n}_r,\xi_r)\in\mathcal{R}_1}} \left|(\mathbf{n}_p,\xi_p)\right|\left|\prodscal{\omega_{\mathbf{n}_p,\xi_p}\,\sigma_{\mathbf{n}_q,\xi_q}}{\sigma_{-\mathbf{n}_r,-\xi_r}}_{L^2(\omega_T)}(x_d)\right|\\
&\leq C\norme{V_{\mathcal{R}^{\inc}}^{\osc}}_{\mathcal{E}_{s,T}}\prodscal{U_{\res}^{\osc}}{U_{\res}^{\osc}}_{\inc}(x_d).\label{eq estim prod R1}
\end{align}
Only the profiles $ V_{\res}^{\osc} $ and $ U_{\res}^{\osc} $ appear in the estimate since only frequencies of $ \F^{\inc}_{\res}$ occur in $ \mathcal{R}_1 $.

Equations \eqref{eq estim E i L U} and \eqref{eq estim prod trsp} and estimates \eqref{eq estim prod F}, \eqref{eq estim prod AI}, \eqref{eq estim prod R2} and \eqref{eq estim prod R1} finally lead to
\begin{equation}\label{eq estim ineq diff F etoile}
\frac{d}{dx_d}\prodscal{U_{\res}^{\osc}}{U_{\res}^{\osc}}_{\inc}(x_d)
\leq C\prodscal{F_{\res}^{\osc}}{F_{\res}^{\osc}}_{\inc}(x_d)+C\left(1+\norme{V_{\res}^{\osc}}_{\mathcal{E}_{s,T}}\right)\prodscal{U_{\res}^{\osc}}{U_{\res}^{\osc}}_{\inc}(x_d),
\end{equation}
which is the expected differential inequality.
\end{proof}

\subsubsection{Proof of Lemma \ref{lemme propagation vitesse finie}}

All requisite techniques to show Lemma \ref{lemme propagation vitesse finie} have now been developed, so the proof is given here. It follows \cite[Section 1.3.1]{BenzoniSerre2007Multi}. Recall that at this stage, we have considered a solution $ U $ to \eqref{eq obtention U 1} regular enough, and we have shown that its mean value is zero.

\begin{proof}[Proof (Lemma \ref{lemme propagation vitesse finie})]

It has been shown that if  $ U $ is a solution to \eqref{eq obtention U 1}, then, with already introduced notations, we have
\begin{subequations}\label{eq ppgat evolution 1}
	\begin{align}
	&\sum_{\substack{\mathbf{n}_0\in\B_{\Z^m}\\\xi_0\in\mathcal{C}(\mathbf{n}_0)}}\sum_{\lambda\in\Z^*}\tilde{X}_{(\mathbf{n}_0\cdot\boldsymbol{\zeta},\xi_0)}\,\sigma_{\lambda,\mathbf{n}_0,\xi_0}\,e^{i\,\lambda\mathbf{n}_0\cdot\theta}\,e^{i\,\lambda\xi_0\,\psi_d}\,\tilde{\pi}_{(\mathbf{n}_0\cdot\boldsymbol{\zeta},\xi_0)}\,E(\mathbf{n}_0,\xi_0)\label{eq ppgat evolution 1 trsp}\\
	+&\sum_{\substack{\mathbf{n}_0\in\B_{\Z^m}\\\xi_0\in\mathcal{C}(\mathbf{n}_0)}}\sum_{\lambda\in\Z^*}\sum_{\substack{\lambda_1,\lambda_2\in\Z^*\\\lambda_1+\lambda_2=\lambda}}i\,\lambda_2\,\sigma_{\lambda_1,\mathbf{n}_0,\xi_0}\,\sigma_{\lambda_2,\mathbf{n}_0,\xi_0}\,\Gamma\big((\mathbf{n}_0,\xi_0),(\mathbf{n}_0,\xi_0)\big)\label{eq ppgat evolution 1 AI}\\
	\nonumber&\qquad\qquad e^{i\,\lambda\mathbf{n}_0\cdot\theta}\,e^{i\,\lambda\xi_0\,\psi_d}\,\tilde{\pi}_{(\mathbf{n}_0\cdot\boldsymbol{\zeta},\xi_0)}\,E(\mathbf{n}_0,\xi_0)\\
	+&\sum_{\substack{\mathbf{n}_0\in\B_{\Z^m}\\\xi_0\in\mathcal{C}(\mathbf{n}_0)}}\sum_{\substack{(\lambda_p,\lambda_q,\lambda_r,\mathbf{n}_p,\mathbf{n}_q,\\\xi_p,\xi_q)\in\mathcal{R}_1(\mathbf{n}_0,\xi_0)\cup\mathcal{R}_2(\mathbf{n}_0,\xi_0)}}\sum_{\ell\in\Z^*}i\,\ell\,\sigma_{\ell\lambda_p,\mathbf{n}_p,\xi_p}\,\sigma_{\ell\lambda_q,\mathbf{n}_q,\xi_q}\,\label{eq ppgat evolution 1 res}\\
	\nonumber&\qquad\qquad \Gamma\big(\lambda_p(\mathbf{n}_p,\xi_p),\lambda_q(\mathbf{n}_q,\xi_q)\big)\,e^{i\,\ell\,\lambda_r\mathbf{n}_0\cdot\theta}\,e^{i\,\ell\,\lambda_r\xi_0\,\psi_d}\,\tilde{\pi}_{(\mathbf{n}_0\cdot\boldsymbol{\zeta},\xi_0)}\,E(\mathbf{n}_0,\xi_0)=0.
	\end{align}
\end{subequations}
Let us point out that despite formula \eqref{eq ppgat evolution 1} looks like formula \eqref{eq sortants evolution 4}, the former only uses the fact that the mean value $ U^{*} $ is zero, and involves both incoming and outgoing modes.

For $ 0\leq t_0\leq T $ and $ x_d^0\geq 0 $, consider the domain $ \K(t_0,x_d^0) $, bounded with respect to $ x_d $, given by
\begin{equation*}
	\K(t_0,x_d^0):=\ensemble{(t,y,x_d)\in\Omega_T\,\middle|\,\mathcal{V}^*t\leq x_d\leq x_d^0+\V^*(t_0-t), \;0\leq t\leq t_0},
\end{equation*}
 see Figure \ref{figure zone de propagation}. Let us prove that $ U^{\osc} $ is zero on the upper boundary of this domain, namely for $ t=t_0 $ and $ \mathcal{V}^*t_0\leq x_d\leq x_d^0 $, which suffices to prove that $ U^{\osc} $ is zero outside $ \ensemble{0\leq x_d\leq \V^*t} $ for all $ t $ in $ [0,T] $.
Take the scalar product \eqref{eq def prod scal ppgat} of expression \eqref{eq ppgat evolution 1} with the modified profile
\begin{equation*}
-\sum_{\substack{\mathbf{n}_0\in\B_{\Z^m}\\\xi_0\in\mathcal{C}(\mathbf{n}_0)}}\sum_{\lambda\in\Z^*}\partial_{\xi}\tau_{k(\mathbf{n}_0,\xi_0)}(\mathbf{n}_0\cdot\boldsymbol{\eta},\xi_0)\,\sigma_{\lambda,\mathbf{n}_0,\xi_0}\,e^{i\,\lambda\mathbf{n}_0\cdot\theta}\,e^{i\,\lambda\xi_0\,\psi_d}\,\tilde{\pi}_{(\mathbf{n}_0\cdot\boldsymbol{\zeta},\xi_0)}\,E(\mathbf{n}_0,\xi_0),
\end{equation*}
to obtain, according to \eqref{eq prod scal pol trigo ppgat}, 
\begin{subequations}\label{eq ppgat prod scal}
	\begin{align}
	&\sum_{\substack{\mathbf{n}_0\in\B_{\Z^m}\\\xi_0\in\mathcal{C}(\mathbf{n}_0)}}\sum_{\lambda\in\Z^*}\prodscal{X_{(\mathbf{n}_0\cdot\boldsymbol{\zeta},\xi_0)}\,\sigma_{\lambda,\mathbf{n}_0,\xi_0}}{\sigma_{\lambda,\mathbf{n}_0,\xi_0}}_{L^2(\K(t_0,x_d^0))}\left|\tilde{\pi}_{(\mathbf{n}_0\cdot\boldsymbol{\zeta},\xi_0)}\,E(\mathbf{n}_0,\xi_0)\right|^2\label{eq ppgat prod scal trsp}\\
	+&\sum_{\substack{\mathbf{n}_0\in\B_{\Z^m}\\\xi_0\in\mathcal{C}(\mathbf{n}_0)}}\sum_{\lambda\in\Z^*}\sum_{\substack{\lambda_1,\lambda_2\in\Z^*\\\lambda_1+\lambda_2=\lambda}}i\,\lambda_2\,\Gamma\big((\mathbf{n}_0,\xi_0),(\mathbf{n}_0,\xi_0)\big)\prodscal{\sigma_{\lambda_1,\mathbf{n}_0,\xi_0}\,\sigma_{\lambda_2,\mathbf{n}_0,\xi_0}}{\sigma_{\lambda,\mathbf{n}_0,\xi_0}}_{L^2(\K(t_0,x_d^0))}\label{eq ppgat prod scal AI}\\
	\nonumber&\qquad\qquad \big(-\partial_{\xi}\tau_{k(\mathbf{n}_0,\xi_0)}(\mathbf{n}_0\cdot\boldsymbol{\eta},\xi_0)\big)\left|\tilde{\pi}_{(\mathbf{n}_0\cdot\boldsymbol{\zeta},\xi_0)}\,E(\mathbf{n}_0,\xi_0)\right|^2\\
	+&\sum_{\substack{\mathbf{n}_0\in\B_{\Z^m}\\\xi_0\in\mathcal{C}(\mathbf{n}_0)}}\sum_{\substack{(\lambda_p,\lambda_q,\lambda_r,\mathbf{n}_p,\mathbf{n}_q,\\\xi_p,\xi_q)\in\mathcal{R}_1(\mathbf{n}_0,\xi_0)\cup\mathcal{R}_2(\mathbf{n}_0,\xi_0)}}\sum_{\ell\in\Z^*}i\,\ell\,\Gamma\big(\lambda_p(\mathbf{n}_p,\xi_p),\lambda_q(\mathbf{n}_q,\xi_q)\big)\big(-\partial_{\xi}\tau_{k(\mathbf{n}_0,\xi_0)}(\mathbf{n}_0\cdot\boldsymbol{\eta},\xi_0)\big)\label{eq ppgat prod scal R1}\\
	\nonumber&\qquad\qquad \prodscal{\sigma_{\ell\lambda_p,\mathbf{n}_p,\xi_p}\,\sigma_{\ell\lambda_q,\mathbf{n}_q,\xi_q}}{\sigma_{\lambda,\mathbf{n}_0,\xi_0}}_{L^2(\K(t_0,x_d^0))}\left|\tilde{\pi}_{(\mathbf{n}_0\cdot\boldsymbol{\zeta},\xi_0)}\,E(\mathbf{n}_0,\xi_0)\right|^2=0.
	\end{align}
\end{subequations}
Term \eqref{eq ppgat prod scal trsp} is obtained by noting that, with the notations of Definition \ref{def sortant rentrant alpha X alpha} and Lemma \ref{lemme Lax}, we have $ \big(-\partial_{\xi}\tau_k(\eta,\xi)\big)\,\tilde{X}_{\alpha}=X_{\alpha} $. First term \eqref{eq ppgat prod scal trsp} is investigated. According to Green's formula, for all $ \mathbf{n}_0 $ in $ \B_{\Z^m} $, $ \xi_0 $ in $ \mathcal{C}(\mathbf{n}_0) $ and $ \lambda $ in $ \Z^* $, we obtain
\begin{equation*}
2\Re\prodscal{X_{(\mathbf{n}_0\cdot\boldsymbol{\zeta},\xi_0)}\,\sigma_{\lambda,\mathbf{n}_0,\xi_0}}{\sigma_{\lambda,\mathbf{n}_0,\xi_0}}_{L^2(\K(t_0,x_d^0))}=2\Re\int_{\partial\K(t_0,x_d^0)}\big(n_t+\vec{n}_x\cdot\mathbf{v}_{(\mathbf{n}_0\cdot\boldsymbol{\zeta},\xi_0)}\big)|\sigma_{\lambda,\mathbf{n}_0,\xi_0}|^2\,dS,
\end{equation*}
where the notation $ \mathbf{v}_{\alpha} $ has been introduced in Definition \ref{def sortant rentrant alpha X alpha}, $ \vec{n}:=(n_t,\vec{n}_x) $ is the outward normal vector associated with $ \partial\K(t_0,x_d^0) $, and $ dS $ is the surface measure. The vector $ \vec{n} $ is given (see Figure \ref{figure zone de propagation}), for the upper boundary by $ \vec{n}=(1,0,\dots,0) $, for the lower boundary by $ \vec{n}=(-1,0,\dots,0) $, for the left boundary by $ \vec{n}=(\V^*,0,\dots,0,-1)/\sqrt{1+(\V^*)^2} $ and for the right boundary by $ \vec{n}=(\V^*,0,\dots,0,1)/\sqrt{1+(\V^*)^2} $. Thus we get
\begin{align*}
	2\Re&\prodscal{X_{(\mathbf{n}_0\cdot\boldsymbol{\zeta},\xi_0)}\,\sigma_{\lambda,\mathbf{n}_0,\xi_0}}{\sigma_{\lambda,\mathbf{n}_0,\xi_0}}_{L^2(\K(t_0,x_d^0))}\\
	=&2\norme{\sigma_{\lambda,\mathbf{n}_0,\xi_0}}^2_{L^2(\R^{d-1}\times[\V^*t_0,x_d^0])}(t_0)-2\norme{\sigma_{\lambda,\mathbf{n}_0,\xi_0}}^2_{L^2(\R^{d-1}\times[0,x_d^0])}(0)\\
	&+\frac{2}{\sqrt{1+(\V^*)^2}}\int_{\ensemble{(t,y,\V^*t),0\leq t \leq t_0}}\big(\V^*-\partial_{\xi}\tau_{k(\mathbf{n}_0,\xi_0)}(\mathbf{n}_0\cdot\boldsymbol{\eta},\xi_0)\big)|\sigma_{\lambda,\mathbf{n}_0,\xi_0}|^2\,dS\\
	&+\frac{2}{\sqrt{1+(\V^*)^2}}\int_{\ensemble{(t,y,x_d^0+\V^*(t_0-t)),0\leq t \leq t_0}}\big(\V^*+\partial_{\xi}\tau_{k(\mathbf{n}_0,\xi_0)}(\mathbf{n}_0\cdot\boldsymbol{\eta},\xi_0)\big)|\sigma_{\lambda,\mathbf{n}_0,\xi_0}|^2\,dS.
\end{align*}
Then note that on one hand we have $ \norme{\sigma_{\lambda,\mathbf{n}_0,\xi_0}}^2_{L^2(\R^{d-1}\times\R_+)}(0)=0 $ according to the initial condition \eqref{eq obtention cond initiale U}, and on the other hand, according to Lemma \ref{lemme vitesse de groupe bornees}, the quantities $ \big(\V^*-\partial_{\xi}\tau_{k(\mathbf{n}_0,\xi_0)}(\mathbf{n}_0\cdot\boldsymbol{\eta},\xi_0)\big) $ and $ \big(\V^*+\partial_{\xi}\tau_{k(\mathbf{n}_0,\xi_0)}(\mathbf{n}_0\cdot\boldsymbol{\eta},\xi_0)\big) $ are non-negative. Therefore,
\begin{equation}\label{eq ppgat ineq 1}
2\Re \eqref{eq ppgat prod scal trsp}\geq 2\sum_{\substack{\mathbf{n}_0\in\B_{\Z^m}\\\xi_0\in\mathcal{C}(\mathbf{n}_0)}}\sum_{\lambda\in\Z^*}\norme{\sigma_{\lambda,\mathbf{n}_0,\xi_0}}^2_{L^2(\R^{d-1}\times[\V^*t_0,x_d^0])}(t_0)\left|\tilde{\pi}_{(\mathbf{n}_0\cdot\boldsymbol{\zeta},\xi_0)}\,E(\mathbf{n}_0,\xi_0)\right|^2.
\end{equation} 

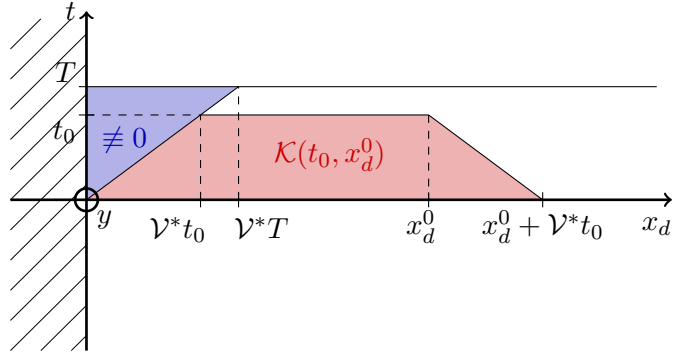
\begin{figure}
	\centering
	\begin{tikzpicture}
		\fill[altblue!30] (0,0)--(2,1.5) -- (0,1.5) -- cycle;
		%\fill[altgreen!30] (0,0)--(2,1.5) -- (6,1.5) --(6,-2) --(0,-2) -- cycle;
		\fill[altred!30] (0,0)--(1.5,1.125) -- (4.5,1.125) --(6,0) -- cycle;
		\draw[line width = 1pt,->] (-1,0) -- (7.7,0);
		\draw[line width = 1pt,->] (0,-2) -- (0,2.5);
		\draw[line width = 1pt] (0,0) circle (0.15);
		\draw (0,0) --(2,1.5);
		\draw (-0.1,1.5)--(7.5,1.5);
		\draw (1.5,1.125) -- (4.5,1.125) ;
		\draw (4.5,1.125) -- (6,0);
		\draw[below] (7.5,-0.1) node{$ x_d $};
		\draw[left] (0,2.5) node{$ t $};
		\draw[below right] (0,0) node{$ y $};
		%	\draw[altred] (4.5,0.4) node{0};
		%	\draw (5.5,1) node{0} ;
		%	\draw (3,-1) node{0};
		\draw[altblue] (0.5,0.8) node{$ \not\equiv 0 $};
		\draw[altred] (3.2,0.6) node{$ \mathcal{K}(t_0,x_d^0) $};
		\draw[left] (0,1.7) node{$ T $};
		\draw (2,-0.1) -- (2,0.1);
		\draw[left] (0,0.9) node{$ t_0 $};
		\draw[dashed] (0,1.125) -- (1.5,1.125) ;
		\draw (-0.1,1.125) -- (0.1,1.125);
		\draw[below] (2.3,-0.1) node{$ \mathcal{V}^*T $};
		\draw[dashed] (2,0) -- (2,1.5);
		\draw (1.5,-0.1) -- (1.5,0.1);
		\draw[below] (1.2,-0.1) node{$ \mathcal{V}^*t_0 $};
		\draw[dashed] (1.5,0) -- (1.5,1.125);
		\draw (4.5,-0.1) -- (4.5,0.1);
		\draw[below] (4.4,0) node{$ x_d^0 $};
		\draw[dashed] (4.5,1.125) -- (4.5,0) ;
		\draw (6,-0.1) -- (6,0.1);
		\draw[below] (6,0) node{$ x_d^0+\V^*t_0 $};
		\draw[line width = 0.2pt] (0,-1.8) -- (-0.2,-2);
		\draw[line width = 0.2pt] (0,-1.4) -- (-0.6,-2);
		\draw[line width = 0.2pt] (0,-1) -- (-1,-2);
		\draw[line width = 0.2pt] (0,-0.6) -- (-1,-1.6);
		\draw[line width = 0.2pt] (0,-0.2) -- (-1,-1.2);
		\draw[line width = 0.2pt] (0,0.2) -- (-1,-0.8);
		\draw[line width = 0.2pt] (0,0.6) -- (-1,-0.4);
		\draw[line width = 0.2pt] (0,1) -- (-1,0);
		\draw[line width = 0.2pt] (0,1.4) -- (-1,0.4);
		\draw[line width = 0.2pt] (0,1.8) -- (-1,0.8);
		\draw[line width = 0.2pt] (0,2.2) -- (-1,1.2);
		\draw[line width = 0.2pt] (-0.2,2.4) -- (-1,1.6);
		\draw[line width = 0.2pt] (-0.6,2.4) -- (-1,2);
	\end{tikzpicture}
	\caption{Propagation zone.}
	\label{figure zone de propagation}
\end{figure}

As for them, terms \eqref{eq ppgat prod scal AI} and \eqref{eq ppgat prod scal R1} are treated the same way as before. For the self-interaction term \eqref{eq ppgat prod scal AI}, it is proved in the same manner than term \eqref{eq sortants prod 2 AI} that it satisfies
\begin{equation*}
2\Re \eqref{eq ppgat prod scal AI}=0.
\end{equation*}
For the resonance term \eqref{eq ppgat prod scal R1}, the same techniques as for terms \eqref{eq estim prod R2} and \eqref{eq estim prod R1} are used. According to Lemma \ref{lemme vitesse de groupe bornees}, the group velocities $ \big(-\partial_{\xi}\tau_{k(\mathbf{n}_0,\xi_0)}\allowbreak(\mathbf{n}_0\cdot\boldsymbol{\eta},\xi_0)\big) $ can be uniformly bounded to obtain
\begin{equation}\label{eq ppgat ineq 2}
\left|2\Re \eqref{eq ppgat prod scal R1}\right| \leq C \norme{ U^{\osc}}_{\mathcal{E}_{s,T}}\sum_{\substack{\mathbf{n}_0\in\B_{\Z^m}\\\xi_0\in\mathcal{C}(\mathbf{n}_0)}}\sum_{\lambda\in\Z^*}\norme{\sigma_{\lambda,\mathbf{n}_0,\xi_0}}^2_{L^2(\K(t_0,x_d^0))}\left|\tilde{\pi}_{(\mathbf{n}_0\cdot\boldsymbol{\zeta},\xi_0)}\,E(\mathbf{n}_0,\xi_0)\right|^2.
\end{equation}
Noting that
\begin{multline*}
\sum_{\substack{\mathbf{n}_0\in\B_{\Z^m}\\\xi_0\in\mathcal{C}(\mathbf{n}_0)}}\sum_{\lambda\in\Z^*}\norme{\sigma_{\lambda,\mathbf{n}_0,\xi_0}}^2_{L^2(\K(t_0,x_d^0))}\left|\tilde{\pi}_{(\mathbf{n}_0\cdot\boldsymbol{\zeta},\xi_0)}\,E(\mathbf{n}_0,\xi_0)\right|^2\\
=\int_0^{t_0}\sum_{\substack{\mathbf{n}_0\in\B_{\Z^m}\\\xi_0\in\mathcal{C}(\mathbf{n}_0)}}\sum_{\lambda\in\Z^*}\norme{\sigma_{\lambda,\mathbf{n}_0,\xi_0}}^2_{L^2(\R^{d-1}\times[\V^*t,x_d^0+\V^*(t_0-t)])}(t)\left|\tilde{\pi}_{(\mathbf{n}_0\cdot\boldsymbol{\zeta},\xi_0)}\,E(\mathbf{n}_0,\xi_0)\right|^2\,dt,
\end{multline*}
using equations \eqref{eq ppgat prod scal}, \eqref{eq ppgat ineq 1} and \eqref{eq ppgat ineq 2} and according to the Grönwall's inequality, it follows
\begin{equation*}
\sum_{\substack{\mathbf{n}_0\in\B_{\Z^m}\\\xi_0\in\mathcal{C}(\mathbf{n}_0)}}\sum_{\lambda\in\Z^*}\norme{\sigma_{\lambda,\mathbf{n}_0,\xi_0}}^2_{L^2(\R^{d-1}\times[\V^*t_0,x_d^0])}(t_0)\left|\tilde{\pi}_{(\mathbf{n}_0\cdot\boldsymbol{\zeta},\xi_0)}\,E(\mathbf{n}_0,\xi_0)\right|^2=0.
\end{equation*}
Therefore, for all $ \mathbf{n}_0 $ in $ \B_{\Z^m} $, $ \xi_0 $ in $ \mathcal{C}(\mathbf{n}_0) $ and $ \lambda $ in $ \Z^* $, the function $ \sigma_{\lambda,\mathbf{n}_0,\xi_0} $ is zero on $ \big\{t=t_0,\allowbreak\V^*t_0\leq x_d\leq x_d^0\big\} $, so the profile $ U^{\osc} $ is also zero in this set, concluding the proof of the Lemma.
\end{proof}

\subsubsection{Estimating the derivatives}

Returning to the proof of the a priori estimate, Proposition \ref{prop estim a priori} is proved here using estimate \eqref{eq estim est a priori L2} of Lemma \ref{lemme estim est a priori L2}. Consider a multi-index $ \alpha $ of $ \mathbb{N}^{d+m} $ such that $ |\alpha|\leq s $. Since the operator  $ \partial_{z',\theta}^{\alpha} $ commutes with the projectors $ \E^{\inc}_{\res} $ and $ \tilde{\E^i}^{\inc}_{\res} $, the profile $ \partial_{z',\theta}^{\alpha}U^{\osc} $ satisfies a system of the form \eqref{eq estim syst osc}, with $ \partial_{z',\theta}^{\alpha}G $ as boundary term, and $ \tilde{\E^i}^{\inc}_{\res}\big[\partial_{z',\theta}^{\alpha}F^{\osc}_{\res}+F_{\alpha}\big] $ as source term, where $ F_{\alpha} $ is the following commutator
\begin{equation*}
F_{\alpha}:=\Big[\sum_{j=1}^m\tilde{L}_1(V^{\osc}_{\res},\zeta_j)\,\partial_{\theta_j},\partial_{z',\theta}^{\alpha}\Big]U_{\res}^{\osc}.
\end{equation*}
Thus, according to estimate \eqref{eq estim est a priori L2} and applying the triangle inequality, we get
\begin{multline}\label{eq estim est derivee alpha 1}
\frac{d}{dx_d}\prodscal{\partial_{z',\theta}^{\alpha}U_{\res}^{\osc}}{\partial_{z',\theta}^{\alpha}U_{\res}^{\osc}}_{\inc}(x_d)\leq C\prodscal{\partial_{z',\theta}^{\alpha}F_{\res}^{\osc}}{\partial_{z',\theta}^{\alpha}F_{\res}^{\osc}}_{\inc}(x_d)
+ C\prodscal{F_{\alpha}}{F_{\alpha}}_{\inc}(x_d)\\
+C\left(1+\norme{V_{\res}^{\osc}}_{\mathcal{E}_{s,T}}\right)\prodscal{\partial_{z',\theta}^{\alpha}U_{\res}^{\osc}}{\partial_{z',\theta}^{\alpha}U_{\res}^{\osc}}_{\inc}(x_d).
\end{multline}

Note that according to Lemma \ref{lemme correspondance norme prod scal} we have
\begin{align*}
\prodscal{F_{\alpha}}{F_{\alpha}}_{\inc}(x_d)&\leq \norme{F_{\alpha}}^2_{\mathcal{C}_b(\R^+_{\psi_d},L^2(\omega_T\times\T^m))}(x_d),
\intertext{and in the same way}
\prodscal{\partial_{z',\theta}^{\alpha}F_{\res}^{\osc}}{\partial_{z',\theta}^{\alpha}F_{\res}^{\osc}}_{\inc}(x_d)&\leq\norme{\partial_{z',\theta}^{\alpha}F_{\res}^{\osc}}^2_{\mathcal{C}_b(\R^+_{\psi_d},L^2(\omega_T\times\T^m))}(x_d).
\end{align*}
On an other hand, according to the algebra property of $ H^s(\omega_T\times\T^m) $ (since $ s>(d+m)/2+1 $) and the commutator estimate \cite[Proposition C.13]{BenzoniSerre2007Multi}, we obtain
\begin{align*}
\norme{F_{\alpha}}^2_{\mathcal{C}_b(\R_+,L^2(\omega_T\times\T^m))}(x_d)
&\leq C\norme{V_{\res}^{\osc}}^2_{\mathcal{C}_b(\R_+,H^s(\omega_T\times\T^m))}(x_d)\norme{U_{\res}^{\osc}}^2_{\mathcal{C}_b(\R_+,H^s(\omega_T\times\T^m))}(x_d)\\[5pt]
&\leq C\norme{V_{\res}^{\osc}}^2_{\mathcal{E}_{s,T}}\norme{U_{\res}^{\osc}}^2_{\mathcal{C}_b(\R_+,H^s(\omega_T\times\T^m))}(x_d).
\end{align*}
Finally, by definition of the $ H^s(\omega_T\times\T^m) $ norm and according to Lemma \ref{lemme correspondance norme prod scal}, we get
\begin{align}\label{eq estim maj U osc prod scal}
	\norme{U_{\res}^{\osc}}^2_{\mathcal{C}_b(\R_+,H^s(\omega_T\times\T^m))}(x_d)&=\sum_{|\alpha|\leq s}\norme{\partial_{z',\theta}^{\alpha}U_{\res}^{\osc}}^2_{\mathcal{C}_b(\R_+,L^2(\omega_T\times\T^m))}(x_d)\\\nonumber
	&\leq C\sum_{|\alpha|\leq s}\prodscal{\partial_{z',\theta}^{\alpha}U_{\res}^{\osc}}{\partial_{z',\theta}^{\alpha}U_{\res}^{\osc}}_{\inc}(x_d).
\end{align}
Therefore, by summing equations \eqref{eq estim est derivee alpha 1} for $ |\alpha|\leq s $, one gets
\begin{multline}\label{eq estim est derivee alpha 2}
	\frac{d}{dx_d}\sum_{|\alpha|\leq s}\prodscal{\partial_{z',\theta}^{\alpha}U_{\res}^{\osc}}{\partial_{z',\theta}^{\alpha}U_{\res}^{\osc}}_{\inc}(x_d)
	\leq
	C\norme{F_{\res}^{\osc}}^2_{\mathcal{C}_b(\R^+_{\psi_d},H^s(\omega_T\times\T^m))}(x_d)\\
	+C\left(1+\norme{V_{\res}^{\osc}}^2_{\mathcal{E}_{s,T}}\right)\sum_{|\alpha|\leq s}\prodscal{\partial_{z',\theta}^{\alpha}U_{\res}^{\osc}}{\partial_{z',\theta}^{\alpha}U_{\res}^{\osc}}_{\inc}(x_d).
\end{multline}
Thus, according to Grönwall's inequality, 
\begin{multline}\label{eq estim est a priori 1}
	\sum_{|\alpha|\leq s}\prodscal{\partial_{z',\theta}^{\alpha}U_{\res}^{\osc}}{\partial_{z',\theta}^{\alpha}U_{\res}^{\osc}}_{\inc}(x_d)\leq e^{C(V)\,x_d}\sum_{|\alpha|\leq s}\prodscal{\partial_{z',\theta}^{\alpha}U_{\res}^{\osc}}{\partial_{z',\theta}^{\alpha}U_{\res}^{\osc}}_{\inc}(0)\\
	+\int_0^{x_d}e^{C(V)(x_d-x_d')}\norme{F_{\res}^{\osc}}^2_{\mathcal{C}_b(\R^+_{\psi_d},H^s(\omega_T\times\T^m))}(x_d')\,dx_d',
\end{multline}
where $ C(V)=C(1+\norme{V_{\res}^{\osc}}^2_{\mathcal{E}_{s,T}}) $.
The trace on the boundary $ U_{\res}^{\osc} $ is therefore given by $ \big(U_{\res}^{\osc}\big)_{|x_d=0,\psi_d=0}=H_{\res}^{\osc} $ where $ H_{\res}^{\osc} $ is determined by equation \eqref{eq red def H osc res}.
Remark \ref{remarque B inverse bornee}, Proposition \ref{prop proj bornes} and Lemma \ref{lemme correspondance norme prod scal} ensure that, for $ |\alpha|\leq s $,
\begin{align}\label{eq estim maj U osc 0 prod scal}
\prodscal{\partial_{z',\theta}^{\alpha}U_{\res}^{\osc}}{\partial_{z',\theta}^{\alpha}U_{\res}^{\osc}}_{\inc}(0)&\leq \norme{\partial_{z',\theta}^{\alpha}U_{\res}^{\osc}(0)}^2_{L^2(\omega_T\times\T^m)}=\norme{\partial_{z',\theta}^{\alpha}H_{\res}^{\osc}}^2_{L^2(\omega_T\times\T^m)}\\
&\nonumber\leq C \norme{\partial_{z',\theta}^{\alpha}G}^2_{L^2(\omega_T\times\T^m)}.
\end{align}
It thus follows, with \eqref{eq estim maj U osc prod scal}, \eqref{eq estim est a priori 1} and \eqref{eq estim maj U osc 0 prod scal},
\begin{multline}\label{eq estim est a priori bis}
\norme{U_{\res}^{\osc}}^2_{\mathcal{C}_b(\R^+_{\psi_d},H^s(\omega_T\times\T^m))}(x_d)\leq Ce^{C(V)\,x_d}\norme{G}^2_{H^s(\omega_T\times\T^m)}\\
+\int_0^{x_d}e^{C(V)(x_d-x_d')}\norme{F_{\res}^{\osc}}^2_{\mathcal{C}_b(\R^+_{\psi_d},H^s(\omega_T\times\T^m))}(x_d')\,dx_d'.
\end{multline}
Because of the function $ \beta_T $ in equation \eqref{eq estim syst osc E i U}, it is possible to bound $ x_d $ by $ 2\mathcal{V}^*T $ then to pass to the upper bound with respect to $ x_d $ in estimate \eqref{eq estim est a priori bis} to obtain the required estimate \eqref{eq estim est a priori}, concluding the proof of Proposition \ref{prop estim a priori}.

\subsection{A priori estimate for the linearized Burgers equations}\label{subsection est Burgers}

We prove now a priori estimates for the linearized Burgers equations \eqref{eq red syst osc non res}, corresponding to the non-resonant incoming modes. These a priori estimates will be used to prove the existence of solution to these Burgers equations. However the estimates will have to be summed with respect to $ (\mathbf{n}_0,\xi_0) $, so we wish for constants independent of $ (\mathbf{n}_0,\xi_0) $. This part is devoted to the proof of the following result.

\begin{proposition}\label{prop estim a priori Burgers}
	Consider $ (\mathbf{n}_0,\xi_0)\in(\B_{\Z^m}\times\mathcal{C}_{\inc}(\mathbf{n}_0))\setminus\F^{\inc}_{\res} $, $ s>s_0 $ and let $ S_{\mathbf{n}_0,\xi_0} $, $ W_{\mathbf{n}_0,\xi_0} $ and $ F_{\mathbf{n}_0,\xi_0} $ be in  $ \mathcal{C}(\R^+_{x_d},H^s(\omega_T\times\T)) $ of zero mean, satisfying the scalar boundary value problem
	\begin{subequations}\label{eq estim syst osc non res}
		\begin{align}
			\tilde{X}_{(\mathbf{n}_0\cdot\boldsymbol{\zeta},\xi_0)}S_{\mathbf{n}_0,\xi_0}+\Gamma\big((\mathbf{n}_0,\xi_0),(\mathbf{n}_0,\xi_0)\big)W_{\mathbf{n}_0,\xi_0}\partial_{\Theta}S_{\mathbf{n}_0,\xi_0}&=F_{\mathbf{n}_0,\xi_0}\label{eq estim syst osc Burgers non res}\\
			\big(S_{\mathbf{n}_0,\xi_0}\big)_{|x_d=0}&=h_{\mathbf{n}_0,\xi_0}\label{eq estim syst osc cond bord non res}\\
			\big(S_{\mathbf{n}_0,\xi_0}\big)_{|t\leq0}&=0,\label{eq estim syst osc cond init non res}
		\end{align}
	\end{subequations}
	where $ h_{\mathbf{n}_0,\xi_0} $ is defined by equation \eqref{eq red def H osc non res}.
	Then the function $ S_{\mathbf{n}_0,\xi_0}$
	satisfies the a priori estimate
	\begin{multline}\label{eq estim est a priori Burgers}
		\norme{S_{\mathbf{n}_0,\xi_0}}^2_{\mathcal{C}(\R^+_{x_d},H^s(\omega_T\times\T))}\\
		\leq C_1\,e^{C(W)\,\mathcal{V}^*T}\norme{h_{\mathbf{n}_0,\xi_0}}^2_{H^s(\omega_T\times\T^m)}
		+ \V^*T\,e^{C(W)\mathcal{V}^*T}\norme{F_{\mathbf{n}_0,\xi_0}}^2_{\mathcal{C}(\R^+_{x_d},H^s(\omega_T\times\T))},
	\end{multline} 
	where
	$ C(W):=C_1\big(1+\norme{W_{\mathbf{n}_0,\xi_0}}^2_{\mathcal{C}(\R^+_{x_d},H^s(\Omega_T\times\T))}\big) $, with $ C_1>0 $ a constant depending only on the operator $ L(0,\partial_z) $ and of $ s $, but not on $ \mathbf{n}_0,\xi_0 $. 
	Recall that the real  number $ \mathcal{V}^* $, bounding the group velocities $ \mathbf{v}_{\alpha} $, has been defined in Lemma \ref{lemme vitesse de groupe bornees}.
\end{proposition}

First the $ L^2(\omega_T\times\T) $ estimate, analogous to estimate \eqref{eq estim est a priori L2} for resonant modes, is investigated, and equation \eqref{eq estim syst osc Burgers non res} is rewritten. Write $ S_{\mathbf{n}_0,\xi_0} $, $ W_{\mathbf{n}_0,\xi_0} $ and $ F_{\mathbf{n}_0,\xi_0} $ in $ \mathcal{C}(\R^+_{x_d},H^s(\omega_T\times\T)) $ as
\begin{align*}
	S_{\mathbf{n}_0,\xi_0}(z,\Theta)&=\sum_{\lambda\in\Z^*}\sigma_{\lambda,\mathbf{n}_0,\xi_0}(z)\,e^{i\lambda\Theta},\quad 
	W_{\mathbf{n}_0,\xi_0}(z,\Theta)=\sum_{\lambda\in\Z^*}\omega_{\lambda,\mathbf{n}_0,\xi_0}(z)\,e^{i\lambda\Theta},\\
	F_{\mathbf{n}_0,\xi_0}(z,\Theta)&=\sum_{\lambda\in\Z^*}f_{\lambda,\mathbf{n}_0,\xi_0}(z)\,e^{i\lambda\Theta},
\end{align*}
with $ \sigma_{\lambda,\mathbf{n}_0,\xi_0} $, $ \omega_{\lambda,\mathbf{n}_0,\xi_0} $ and $ f_{\lambda,\mathbf{n}_0,\xi_0} $ scalar functions on $ \Omega_T $. Then equation \eqref{eq estim syst osc Burgers non res} writes
\begin{subequations}\label{eq estim E i L U singletons}
	\begin{align}
		\sum_{\lambda\in\Z^*}f_{\lambda,\mathbf{n}_0,\xi_0}(z)\,e^{i\,\lambda\mathbf{n}_0\cdot\theta}\,e^{i\lambda\Theta}&=\tilde{X}_{(\mathbf{n}_0\cdot\boldsymbol{\zeta},\xi_0)}\sum_{\lambda\in\Z^*}\sigma_{\lambda,\mathbf{n}_0,\xi_0}\,e^{i\lambda\Theta},\label{eq estim E i L U trsp osc singletons}\\
		&+\Gamma\big((\mathbf{n}_0,\xi_0),(\mathbf{n}_0,\xi_0)\big)\sum_{\lambda\in\Z^*}\sum_{\substack{\lambda_1,\lambda_2\in\Z^*\\\lambda_1+\lambda_2=\lambda}}i\,\lambda_2\,\omega_{\lambda_1,\mathbf{n}_0,\xi_0}\,\sigma_{\lambda_2,\mathbf{n}_0,\xi_0}\,e^{i\lambda\Theta}.\label{eq estim E i L U AI singletons}
	\end{align}
\end{subequations}
Recall that the vector field $ \tilde{X}_{(\mathbf{n}_0\cdot\boldsymbol{\zeta},\xi_0)} $, defined in Lemma \ref{lemme Lax}, is given by
\begin{equation*}
	\tilde{X}_{(\mathbf{n}_0\cdot\boldsymbol{\zeta},\xi_0)}=\frac{-1}{\partial_{\xi}\tau_{k(\mathbf{n}_0,\xi_0)}(\mathbf{n}_0\cdot\boldsymbol{\eta},\xi_0)}\partial_t+\inv{\partial_{\xi}\tau_{k(\mathbf{n}_0,\xi_0)}(\mathbf{n}_0\cdot\boldsymbol{\eta},\xi_0)}\nabla_{\eta}\tau_{k(\eta_0,\xi_0)}(\mathbf{n}_0\cdot\boldsymbol{\eta},\xi_0)\cdot\nabla_y+\partial_{x_d}.
\end{equation*} 

By taking the double of the real part of the $ L^2(\omega_T\times\T) $ scalar product of equality \eqref{eq estim E i L U singletons} with the function $ S_{\mathbf{n}_0,\xi_0} $, one obtains an equality, with on one side of it the term
\begin{equation*}
	2\Re \prodscal{F_{\mathbf{n}_0,\xi_0}}{S_{\mathbf{n}_0,\xi_0}}_{L^2(\omega_T\times\T)}(x_d),
\end{equation*}
which is estimated in a similar manner than for the resonant incoming frequencies:
\begin{equation}
	\left|2\Re \prodscal{F_{\mathbf{n}_0,\xi_0}}{S_{\mathbf{n}_0,\xi_0}}_{L^2(\omega_T\times\T)}(x_d)\right|
	\leq
	C\norme{F_{\mathbf{n}_0,\xi_0}}_{L^2(\omega_T\times\T)}(x_d)+C\norme{S_{\mathbf{n}_0,\xi_0}}_{L^2(\omega_T\times\T)}(x_d).\label{eq estim prod F singletons}
\end{equation}
One may note here the interest of having reduced the equations to scalar Burgers equations for the non-resonant modes, since the coefficients $ |\tilde{\pi}_{(\mathbf{n}_0\cdot\boldsymbol{\zeta},\xi_0)}\,E(\mathbf{n}_0,\xi_0)|^{-1} $ no longer appear, these ones being not uniformly bounded for the non resonant modes $ (\mathbf{n}_0,\xi_0) $.

We now focus on the other side of the equality obtained by taking the double of the real part of the scalar product of equation \eqref{eq estim E i L U singletons} with the profile $ S_{\mathbf{n}_0,\xi_0} $. The analysis of terms \eqref{eq estim E i L U trsp osc singletons} and \eqref{eq estim E i L U AI singletons} is analogous to the one made for the outgoing non-resonant modes.

Concerning the transport term \eqref{eq estim E i L U trsp osc singletons}, according to identity \eqref{eq prod scal pol trigo rentrants} and using an integration by parts, we obtain
\begin{multline}\label{eq estim prod trsp singletons}
	2\Re\prodscal{\eqref{eq estim E i L U trsp osc singletons}}{S_{\mathbf{n}_0,\xi_0}}_{L^2(\omega_T\times\T)}(x_d)=\frac{d}{dx_d}\norme{S_{\mathbf{n}_0,\xi_0} }_{L^2(\omega_T\times\T)}(x_d)\\
	-\frac{1}{\partial_{\xi}\tau_{k(\mathbf{n}_0,\xi_0)}(\mathbf{n}_0\cdot\boldsymbol{\eta},\xi_0)}\norme{S_{\mathbf{n}_0,\xi_0} }^2_{L^2(\R^{d-1}\times\T)}(T).
\end{multline}
Note that since $ (\mathbf{n}_0,\xi_0) $ is an incoming mode, the quantity $ -\partial_{\xi}\tau_{k(\mathbf{n}_0,\xi_0)}(\mathbf{n}_0\cdot\boldsymbol{\eta},\xi_0) $ is positive, allowing to omit the second term on the right hand side of the equality in the estimates.

For the self-interaction term \eqref{eq estim E i L U AI singletons}, with computations analogous to the ones used for the incoming resonant modes, we obtain
\begin{multline*}
	2\Re \prodscal{\eqref{eq estim E i L U AI singletons}}{S_{\mathbf{n}_0,\xi_0} }_{L^2(\omega_T\times\T)}(x_d)=\\
	-(2\pi)^m\sum_{\lambda\in\Z^*}\sum_{\substack{\lambda_1,\lambda_2\in\Z^*\\\lambda_1+\lambda_2=\lambda}}i\lambda_1\,\Gamma\big((\mathbf{n}_0,\xi_0),(\mathbf{n}_0,\xi_0)\big)\prodscal{\omega_{\lambda_1,\mathbf{n}_0,\xi_0}\,\sigma_{\lambda_2,\mathbf{n}_0,\xi_0}}{\sigma_{\lambda,\mathbf{n}_0,\xi_0}}_{L^2(\omega_T)}(x_d).
\end{multline*}
Therefore, using the upper bound \eqref{eq est gamma n n}, we get
\begin{multline*}
	\left|2\Re \prodscal{\eqref{eq estim E i L U AI singletons}}{S_{\mathbf{n}_0,\xi_0} }_{L^2(\omega_T\times\T)}(x_d)\right|\\
	\leq
	C\sum_{\lambda\in\Z^*}\sum_{\substack{\lambda_1,\lambda_2\in\Z^*\\\lambda_1+\lambda_2=\lambda}}|\lambda_1\mathbf{n}_0|^h\left|\prodscal{\omega_{\lambda_1,\mathbf{n}_0,\xi_0}\,\sigma_{\lambda_2,\mathbf{n}_0,\xi_0}}{\sigma_{\lambda,\mathbf{n}_0,\xi_0}}_{L^2(\omega_T)}(x_d)\right|.
\end{multline*}
Here the order of regularity must be taken down to $ h $ since we wish for an upper bound independent of $ \mathbf{n}_0 $, in the purpose of summing the inequality with respect to $ \mathbf{n}_0 $. An upper bound of the form $ C(\mathbf{n}_0) |\lambda_1|$ instead of $ C|\lambda_1\mathbf{n}_0|^h $ could be obtained, but where the constant $ C(\mathbf{n}_0) $ depends on $ \mathbf{n}_0 $, and may be arbitrarily large since we consider modes $ \mathbf{n}_0 $ close to the glancing set.
The right hand side of the equality being of the form $ \prodscal{fg}{g} $, the following inequality holds
\begin{align}
	%\nonumber
	\left|2\Re \prodscal{\eqref{eq estim E i L U AI singletons}}{S_{\mathbf{n}_0,\xi_0} }_{L^2(\omega_T\times\T)}(x_d)\right|
	&\leq C
	\norme{W_{\mathbf{n}_0,\xi_0}}_{\mathcal{C}(\R^+_{x_d},H^s(\omega_T\times\T))}\norme{S_{\mathbf{n}_0,\xi_0}}_{L^2(\omega_T\times\T)}(x_d),\label{eq estim prod AI singletons}
\end{align}
using Sobolev inequality, $ s $ being such that $ s>h+(d+m)/2 $.
Using equations \eqref{eq estim E i L U singletons} and \eqref{eq estim prod trsp singletons} and  estimates \eqref{eq estim prod F singletons} and \eqref{eq estim prod AI singletons}, it finally follows the differential inequality
\begin{multline}\label{eq estim ineq diff F etoile singletons}
	\frac{d}{dx_d}\norme{S_{\mathbf{n}_0,\xi_0}}_{L^2(\omega_T\times\T)}(x_d)
	\\
	\leq C\norme{F_{\mathbf{n}_0,\xi_0}}_{L^2(\omega_T\times\T)}(x_d)+C\left(1+\norme{W_{\mathbf{n}_0,\xi_0}}_{\mathcal{C}(\R^+_{x_d},H^s(\omega_T\times\T))}\right)\norme{S_{\mathbf{n}_0,\xi_0}}_{L^2(\omega_T\times\T)}(x_d).
\end{multline}

To obtain the required $ H^s(\omega_T\times\T) $ estimate, we use commutators estimates analogous to the one for resonant incoming modes, which we do not detail here. Finally we obtain the sought estimate \eqref{eq estim est a priori Burgers}.

\bigskip

The a priori estimates \eqref{eq estim est a priori L2} and \eqref{eq estim est a priori Burgers} (for $ s=0 $) as well as the equivalence property of Proposition \ref{prop equivalence systeme} ensure the uniqueness of the solution to \eqref{eq obtention U 1}.

\subsection{Construction of a solution}\label{subsection construction solution}

\subsubsection{Construction of an oscillating solution to the linearized system for the resonant incoming modes}

Thanks to the a priori estimate \eqref{eq estim est a priori} of Proposition \ref{prop estim a priori} on the linearized system \eqref{eq estim syst osc}, a solution to this system can be constructed, proving the following result.

\begin{proposition}\label{prop exist sol syst linearise}
	Consider $ s>s_0 $ and $ T>0 $, and let $ V_{\res}^{\osc} $ be a profile of $ \N^{\osc}_{s,T} $ involving only resonant incoming modes, $ F^{\osc}_{\res} $ be in $ \P_{s,T}^{\osc} $ and $ G $ be in $ H^s(\omega_T\times\T^m) $. Then there exists a solution $ U_{\res}^{\osc} $ in $ \P_{s,T}^{\osc} $ to system \eqref{eq estim syst osc}, involving only resonant incoming modes, that moreover satisfies the following estimate
	\begin{equation*}
		\norme{U_{\res}^{\osc}}^2_{\mathcal{E}_{s,T}}\leq C_1\,e^{C(V)\,\mathcal{V}^*T}\norme{G}^2_{H^s(\omega_T\times\T^m)}
		+ \V^*T\,e^{C(V)\mathcal{V}^*T}\norme{F_{\res}^{\osc}}^2_{\mathcal{E}_{s,T}},
	\end{equation*}
	where $ C(V):=C_1(1+\norme{V_{\res}^{\osc}}^2_{\mathcal{E}_{s,T}}) $, with $ C_1 $ a positive constant depending only on $ L(0,\partial_z) $, on the boundary frequencies $ \zeta_1,\dots,\zeta_m  $, and on $ s $.
\end{proposition}

The proof of such a result using an a priori estimate of the form \eqref{eq estim est a priori} is detailed in \cite[Theorem 6.3.3]{JolyMetivierRauch1995Coherent}. Its main ideas are recalled here.

The uniqueness of the solution follows directly from the a priori estimate \eqref{eq estim est a priori}. Concerning the existence, a finite difference scheme is used. Since the operators $ \partial_{\theta_j} $ for $ j=1,\dots,m $, are skew-symmetric, skew-symmetric finite difference operators must be considered. Denoting $ e_1,\dots,e_m $ the canonical basis of $ \R^m $, we define, for every function $ U $ of $ \Omega_T\times\T^m\times\R_+ $,
\begin{equation*}
	\delta_j^h U(z,\theta,\psi_d):=\big(U(z,\theta+he_j)-U(z,\theta-he_j\big)/2h,
\end{equation*}
for $ j=1,\dots,m $ and $ h>0 $. The proof then consists in showing that there exists, for $ h>0 $, a unique solution $ U^{\osc}_h $ to the regularized system
\begin{subequations}\label{eq constr syst h}
	\begin{align}
		\E^{\inc}_{\res}\, U^{\osc}_h&=U^{\osc}_h  \\
		\tilde{\E^i}^{\inc}_{\res}\Big[\tilde{L}(0,\partial_z)\,U^{\osc}_h+\sum_{j=1}^m\tilde{L}_1(V_{\res}^{\osc},\zeta_j)\,\delta_j^h U^{\osc}_h\Big]
		&=\tilde{\E^i}^{\inc}_{\res}\, F_{\res}^{\osc} \\
		\big(U^{\osc}_h\big)_{|x_d=0,\psi_d=0}&=H_{\res}^{\osc}  \\[5pt]
		\big(U^{\osc}_h\big)_{|t\leq 0}&=0,
	\end{align}
\end{subequations}
and that this solution satisfies the estimate uniform with respect to $ h>0 $,
\begin{equation*}
	\norme{U^{\osc}_h}^2_{\mathcal{E}_{s,T}}\leq C_1\,e^{C(V)\,\mathcal{V}^*T}\norme{G}^2_{H^s(\omega_T\times\T^m)}
	+ \V^*T\,e^{C(V)\mathcal{V}^*T}\norme{F^{\osc}_{\res}}^2_{\mathcal{E}_{s,T}},
\end{equation*}
where $ C(V):=C_1(1+\norme{V_{\res}^{\osc}}^2_{\mathcal{E}_{s,T}}) $, with $ C_1>0 $ a constant depending only on the operator $ L(0,\partial_z) $, on the boundary frequencies $ \zeta_1,\dots,\zeta_m  $, and on $ s $. This uniform estimate allows to extract a sequence $ (U^{\osc}_{h_n})_n $ weakly converging towards $ U_{\res}^{\osc} $ in $ \mathcal{E}_{s,T} $. Passing to the limit in system \eqref{eq constr syst h} leads to the result of Proposition \ref{prop exist sol syst linearise}.

\subsubsection{Construction of an oscillating solution to systems \eqref{eq red syst osc res} and \eqref{eq red syst osc non res}}

This part is devoted to the following result, constituting a part of the result of Theorem \ref{thm existence profils}.

\begin{proposition}\label{prop existence sol osc}
	Consider $ s>s_0 $, and $ G $ in $ H^{\infty}(\R^d\times\T^m) $, zero for negative times $ t $. There exists a time $ T>0 $, depending only on the operator $ L(0,\partial_z) $, on the boundary frequencies $ \zeta_1,\dots,\zeta_m $, on the $ H^s(\R^d\times\T^m) $ norm of $ G $ and on $ s $, such that system \eqref{eq red syst osc res} and, for every $ (\mathbf{n}_0,\xi_0) $ in  $(\B_{\Z^m}\times\mathcal{C}_{\inc}(\mathbf{n}_0))\setminus\F^{\inc}_{\res} $, system \eqref{eq red syst osc non res}, admit solutions $ U_{\res}^{\osc} $ and $ S_{\mathbf{n}_0,\xi_0} $ in $ \P^{\osc}_{s,T} $ and $ \mathcal{C}(\R_+,H^s(\omega_T\times\T)) $, where the functions $ S_{\mathbf{n}_0,\xi_0} $ are of zero mean. Furthermore, if we denote, for $ (\mathbf{n}_0,\xi_0) $ in  $(\B_{\Z^m}\times\mathcal{C}_{\inc}(\mathbf{n}_0))\setminus\F^{\inc}_{\res} $,
	\begin{equation*}
		S_{\mathbf{n}_0,\xi_0}(z,\Theta)=:\sum_{\lambda\in\Z^*}\sigma_{\lambda,\mathbf{n}_0,\xi_0}(z)\,e^{i\lambda\,\Theta},
	\end{equation*}
	then the profile $ U^{\osc} $ defined, for $ (z,\theta,\psi_d)  $ in $ \Omega_T\times\T^m\times\R_+ $, by
	\begin{equation}\label{eq constr def U osc}
		U^{\osc}(z,\theta,\psi_d):=U^{\osc}_{\res}(z,\theta,\psi_d)+\sum_{\substack{(\mathbf{n}_0,\xi_0)\in\\(\B_{\Z^m}\times\mathcal{C}_{\inc}(\mathbf{n}_0))\setminus\F^{\inc}_{\res}}}\sum_{\lambda\in\Z^*}\sigma_{\lambda,\mathbf{n}_0,\xi_0}(z)\,e^{i\lambda\mathbf{n}_0\cdot\theta}\,e^{i\lambda\xi_0\psi_d}\,E(\mathbf{n}_0,\xi_0),
	\end{equation}
	belongs to the space $ \P_{s,T}^{\osc} $.
\end{proposition}

It is classical to deduce from an existence result of a solution to a linearized system with an estimate of the form \eqref{eq estim est a priori}, the existence of a solution to the original system. The main ideas of the method described in \cite[Théorème 10.1]{BenzoniSerre2007Multi} are recalled here.

First system \eqref{eq red syst osc res} is investigated, and the following iterative scheme is considered:
\begin{subequations}\label{eq constr syst osc}
	\begin{align}
		\E^{\inc}_{\res}\, U_{\nu+1}^{\osc}&=U_{\nu+1}^{\osc} \label{eq constr syst osc E U = U}  \\
		\tilde{\E^i}^{\inc}_{\res}\Big[\tilde{L}(0,\partial_z)\,\beta_TU_{\nu+1}^{\osc}+\sum_{j=1}^m\tilde{L}_1(\beta_TU_{\nu}^{\osc},\zeta_j)\,\partial_{\theta_j}\beta_TU_{\nu+1}^{\osc}\Big]
		&=0 \label{eq constr syst osc E i U}\\
		\big(U_{\nu+1}^{\osc}\big)_{|x_d=0,\psi_d=0}&=H_{\res}^{\osc} \label{eq constr syst osc cond bord} \\[5pt]
		\big(U_{\nu+1}^{\osc}\big)_{|t\leq 0}&=0, \label{eq constr syst osc cond init}
	\end{align}
\end{subequations}
initialized with $ U^{\osc}_{0}(.,x_d,.,\psi_d):=H_{\res}^{\osc} $, for all $ x_d,\psi_d $ in $ \R_+ $. Proposition \ref{prop exist sol syst linearise} ensures that the sequence $ (U^{\osc}_{\nu})_{\nu} $ is well defined in $ \P^{\osc}_{s,T} $. Then the proof consists in showing that the sequence $ (U^{\osc}_{\nu})_{\nu} $ is bounded in high norm, and contracting in low norm, in order to deduce its weak convergence in the Banach space $ \P^{\osc}_{s,T} $.

\emph{Bound in high norm.} According to estimate \eqref{eq estim est a priori}, we have, for $ \nu\geq 0 $,
\begin{equation}\label{eq constr est gd norme 1}
	\norme{U_{\nu+1}^{\osc}}^2_{\mathcal{E}_{s,T}}\leq C_1e^{C(U^{\osc}_{\nu})\,\mathcal{V}^*T}\norme{G}^2_{H^s(\omega_T\times\T^m)},
\end{equation}
where $ C(U^{\osc}_{\nu})=C_1\big(1+\norme{U_{\nu}^{\osc}}^2_{\mathcal{E}_{s,T}}\big) $. If the time $ T>0 $ is chosen sufficiently small so that
\begin{equation*}
	\exp \Big[ C_1\big(1+2\,C_1\norme{G}^2_{H^s(\omega_T\times\T^m)}\big)\,\V^*T\Big]\leq 2,
\end{equation*}
then an induction argument shows that $ (U^{\osc}_{\nu})_{\nu} $ is bounded in $ \mathcal{E}_{s,T} $ by $ \sqrt{2\,C_1}\norme{G}_{H^s(\omega_T\times\T^m)} $. Indeed, the initial step is obvious, up to assuming $ C_1>1/2 $. On an other hand, assuming $ \norme{U^{\osc}_{\nu}}_{\mathcal{E}_{s,T}}\leq \sqrt{2\,C_1}\norme{G}_{H^s(\omega_T\times\T^m)} $ for some $ \nu\geq 0 $,  according to \eqref{eq constr est gd norme 1} and the assumption on $ T $, we obtain
\begin{align*}
	\norme{U_{\nu+1}^{\osc}}^2_{\mathcal{E}_{s,T}}&\leq C_1\exp\Big[C_1\big(1+\norme{U_{\nu}^{\osc}}^2_{\mathcal{E}_{s,T}}\big) \,\mathcal{V}^*T\Big]\norme{G}^2_{H^s(\omega_T\times\T^m)}\\
	&\leq C_1\exp \Big[ C_1\big(1+2\,C_1\norme{G}^2_{H^s(\omega_T\times\T^m)}\big)\,\V^*T\Big]\norme{G}^2_{H^s(\omega_T\times\T^m)}\\
	&\leq 2\,C_1\norme{G}^2_{H^s(\omega_T\times\T^m)},
\end{align*}
which is the expected estimate.

\emph{Contraction in low norm.} Denote, for $ \nu\geq1 $, $ W^{\osc}_{\nu}:=U^{\osc}_{\nu}-U^{\osc}_{\nu-1} $, that satisfies the system
\begin{subequations}\label{eq constr syst osc diff}
	\begin{align}
		\E^{\inc}_{\res}\, W_{\nu+1}^{\osc}&=W_{\nu+1}^{\osc} \label{eq constr syst osc diff E U = U}  \\
		\tilde{\E^i}^{\inc}_{\res}\Big[\tilde{L}(0,\partial_z)\,\beta_TW_{\nu+1}^{\osc}+\sum_{j=1}^m\tilde{L}_1(\beta_TU_{\nu}^{\osc},\zeta_j)\,\partial_{\theta_j}\beta_TW_{\nu+1}^{\osc}%&\nonumber\\
		\Big]&=\tilde{\E^i}^{\inc}_{\res}\,F_{\nu+1} \label{eq constr syst osc diff E i U}\\
		\big(W_{\nu+1}^{\osc}\big)_{|x_d=0,\psi_d=0}&=0 \label{eq constr syst osc diff cond bord} \\[5pt]
		\big(W_{\nu+1}^{\osc}\big)_{|t\leq 0}&=0, \label{eq constr syst osc diff cond init}
	\end{align}
\end{subequations}
where $ F_{\nu+1} $ is given by
\begin{align*}
	F_{\nu+1}:=&\sum_{j=1}^m\big(\tilde{L}_1(\beta_TU_{\nu-1}^{\osc},\zeta_j)-\tilde{L}_1(\beta_TU_{\nu}^{\osc},\zeta_j)\big)\partial_{\theta_j}\beta_TU_{\nu}^{\osc}.
\end{align*}
According to estimate \eqref{eq estim est a priori} applied to system \eqref{eq constr syst osc diff} for  $ s=0 $, the following inequality holds
\begin{equation*}
	\norme{W_{\nu+1}^{\osc}}^2_{\mathcal{E}_{0,T}}\leq \V^*T\,e^{C(U_{\nu}^{\osc})\V^*T}\norme{F_{\nu+1}}^2_{\mathcal{E}_{0,T}}.
\end{equation*}
First note that by assumption on $ T $, and since the sequence $ (U^{\osc}_{\nu})_{\nu} $ is bounded in $ \mathcal{E}_{s,T} $ by $ \sqrt{2}\,C_1\norme{G}_{H^s(\omega_T\times\T^m)} $, we have, for $ \nu\geq 0 $,
\begin{equation}\label{eq constr  contract est 1}
	e^{C(U_{\nu}^{\osc})\V^*T}\leq 2.
\end{equation}
Now the $ \mathcal{E}_{0,T} $ norm of $ F_{\nu+1} $ is estimated. 
Thanks to the product estimate and the choice of the index $ s $, we have
\begin{align*}
	\norme{F_{\nu+1}}^2_{\mathcal{E}_{0,T}}\leq &C \sum_{j=1}^m\norme{\big(\tilde{L}_1(\beta_TU_{\nu-1}^{\osc},\zeta_j)-\tilde{L}_1(\beta_TU_{\nu}^{\osc},\zeta_j)\big)}^2_{\mathcal{E}_{0,T}}\norme{U_{\nu}^{\osc}}^2_{\mathcal{E}_{s,T}}.
\end{align*}
According to the mean value inequality, and since the sequence $ (U^{\osc}_{\nu})_{\nu} $ is bounded in $ \mathcal{E}_{s,T} $, one then obtains
\begin{equation}\label{eq constr  contract est 2}
	\norme{F_{\nu+1}}^2_{\mathcal{E}_{0,T}}\leq C \norme{W_{\nu}}^2_{\mathcal{E}_{0,T}}\norme{G}^2_{H^s(\omega_T\times\T^m)}. 
\end{equation}
Therefore, according to estimates \eqref{eq constr  contract est 1} and \eqref{eq constr  contract est 2}, we get
\begin{equation*}
	\norme{W_{\nu+1}^{\osc}}^2_{\mathcal{E}_{0,T}}\leq C\V^*T\norme{G}^2_{H^s(\omega_T\times\T^m)} \norme{W^{\osc}_{\nu}}^2_{\mathcal{E}_{0,T}}.
\end{equation*}
For $ T>0 $ small enough, the sequence $ (U^{\osc}_{\nu})_{\nu} $ is therefore convergent in $ \mathcal{E}_{0,T} $.

Thus the sequence $ (U^{\osc}_{\nu})_{\nu} $ is a Cauchy sequence in the Banach space $ \P^{\osc}_{0,T} $, and therefore converges to a function $ U_{\res}^{\osc} $ of $ \P^{\osc}_{0,T} $. It is possible to show, with arguments that will not be recalled here, that $ U_{\res}^{\osc} $ is actually in $ \P_{s,T}^{\osc} $ and satisfies system \eqref{eq red syst osc res}, see \cite[Theorem 10.1]{BenzoniSerre2007Multi} for similar results.

The proof of the existence of a solution to \eqref{eq red syst osc non res} is identical, and is not detailed here. It relies on a result of existence of a solution to the linearized system \eqref{eq estim syst osc non res}, analogous to Proposition \ref{prop exist sol syst linearise}, that has not been spelled out. One may however note that the existence time $ T $ is indeed independent of $ (\mathbf{n}_0,\xi_0) $, since the constants in estimate \eqref{eq estim est a priori Burgers} are independent of $ (\mathbf{n}_0,\xi_0) $, and since according to estimate \eqref{eq red est H h avec G}, each boundary term $ h_{\mathbf{n}_0,\xi_0} $ is controlled in $ H^s(\omega_T\times\T) $ by $ C\norme{G}_{H^s(\omega_T\times\T^m)} $, uniformly with respect to $ (\mathbf{n}_0,\xi_0) $.

Finally, it is shown that the profile $ U^{\osc} $ defined by \eqref{eq constr def U osc} actually belongs to $ \P^{\osc}_{s,T} $. Indeed, according to Lemma \ref{lemme correspondance norme prod scal}, we have
\begin{align*}
	\norme{U^{\osc}}^2_{\mathcal{E}_{s,T}}&\leq C \sup_{x_d>0}\sum_{|\alpha|\leq s}\prodscal{\partial_{z',\theta}^{\alpha}U^{\osc}}{\partial_{z',\theta}^{\alpha}U^{\osc}}_{\inc}(x_d)\\
	&\leq C \sup_{x_d>0}\sum_{|\alpha|\leq s}\prodscal{\partial_{z',\theta}^{\alpha}U_{\res}^{\osc}}{\partial_{z',\theta}^{\alpha}U_{\res}^{\osc}}_{\inc}(x_d)\\
	&\quad+C\sup_{x_d>0}\sum_{\substack{(\mathbf{n}_0,\xi_0)\in\\(\B_{\Z^m}\times\mathcal{C}_{\inc}(\mathbf{n}_0))\setminus\F^{\inc}_{\res}}}\sum_{\lambda\in\Z^*}\norme{\sigma_{\lambda,\mathbf{n}_0,\xi_0}}_{H^s(\omega_T)}^2(x_d) \\
	&\leq C\norme{U^{\osc}_{\res}}^2_{\mathcal{E}_{s,T}}+C\sum_{\substack{(\mathbf{n}_0,\xi_0)\in\\(\B_{\Z^m}\times\mathcal{C}_{\inc}(\mathbf{n}_0))\setminus\F^{\inc}_{\res}}}\norme{S_{\mathbf{n}_0,\xi_0}}^2_{\mathcal{C}(\R_+,H^s(\omega_T\times\T))},
\end{align*}
so that, using the a priori estimates \eqref{eq estim est a priori} and  \eqref{eq estim est a priori Burgers} as well as the boundary term estimates \eqref{eq red est H h avec G}, one gets
\begin{equation}\label{eq constr est U osc G}
	\norme{U^{\osc}}^2_{\mathcal{E}_{s,T}}\leq C \norme{G}^2_{H^s(\omega_T\times\T^m)}+C\sum_{\substack{(\mathbf{n}_0,\xi_0)\in\\(\B_{\Z^m}\times\mathcal{C}_{\inc}(\mathbf{n}_0))\setminus\F^{\inc}_{\res}}}\norme{h_{\mathbf{n}_0,\xi_0}}^2_{H^s(\omega_T\times\T)}\leq C\norme{G}^2_{H^s(\omega_T\times\T^m)}.
\end{equation}

\subsubsection{Determination of the evanescent part and conclusion}

To conclude as to the proof of Theorem \ref{thm existence profils}, it must be proved that there exists a solution $ U^{\ev} $ in $ \P^{\ev}_{s,T} $ to system \eqref{eq red syst ev}, where the parameters $ s $ and $ T $ are those given in Proposition \ref{prop existence sol osc}. 

The polarization condition \eqref{eq red syst ev E U = U} results, according to Remark \ref{remarque forme U osc polarise}, to
\begin{equation*}
	U^{\ev}(z,\theta,\psi_d)=\sum_{\mathbf{n}\in\Z^m\privede{0}}e^{\psi_d\,\mathcal{A}(\mathbf{n}\cdot\boldsymbol{\zeta})}\,\Pi^e_{\C^N}(\mathbf{n}\cdot\boldsymbol{\zeta})\,U^{\ev}_{\mathbf{n}}(z,0)\,e^{i\mathbf{n}\cdot\theta}.
\end{equation*}
The traces $ \big(\Pi^e_{\C^N}(\mathbf{n}\cdot\boldsymbol{\zeta})\,U^{\ev}_{\mathbf{n}}\big)_{|\psi_d=0} $ for $ \mathbf{n} $ in $ \Z^m\privede{0} $ must therefore be determined to find the profile $ U^{\ev} $. The boundary condition \eqref{eq red syst ev cond bord} gives the double trace on the boundary, for $ \mathbf{n} $ in $ \Z^m\privede{0} $, 
\begin{equation*}
	U^{\ev}_{\mathbf{n}}(z',0,0)=\Pi^e_{\C^N}(\mathbf{n}\cdot\boldsymbol{\zeta})\,U^{\ev}_{\mathbf{n}}(z',0,0)=\Pi_-^e(\freq)\,\big(B_{|E_-(\freq)}\big)^{-1}G_{\mathbf{n}}(z').
\end{equation*}
Then this trace is lifted with respect to $ x_d $ using a function $ \chi $ of $ \mathcal{C}_0^{\infty}(\R_+) $, equaling 1 in 0. Namely we set
\begin{equation*}
	U^{\ev}(z,\theta,\psi_d):=\sum_{\mathbf{n}\in\Z^m\privede{0}}\chi(x_d)\,e^{\psi_d\,\mathcal{A}(\mathbf{n}\cdot\boldsymbol{\zeta})}\,\Pi_-^e(\freq)\,\big(B_{|E_-(\freq)}\big)^{-1}G_{\mathbf{n}}(z')\,e^{i\mathbf{n}\cdot\theta}.
\end{equation*}
Note that, by construction, the profile $ U^{\ev} $ satisfies the polarization condition \eqref{eq red syst ev E U = U} as well as the boundary condition \eqref{eq red syst ev cond bord}. It must be now verified that it belongs to the space of evanescent profiles $ \P^{\ev}_{s,T} $. First we note that the profile $ U^{\ev} $ belongs to $ L^{\infty}\big(\R^+_{x_d}\times\R^+_{\psi_d},H^s_+(\omega_T\times\T^m)\big) $. Indeed, on one hand, the functions $ G_{\mathbf{n}} $ being zero for negative times $ t $, the profile $ U^{\ev} $ is zero for negative times $ t $. On the other hand, since the function $ \chi $ is bounded, the inverse map $ \big(B_{|E_-(\freq)}\big)^{-1} $ is uniformly bounded according to remark \ref{remarque B inverse bornee}, and the terms $ e^{\psi_d\,\mathcal{A}(\mathbf{n}\cdot\boldsymbol{\zeta})}\,\Pi_-^e(\freq) $ are also uniformly bounded according to estimate \eqref{eq controle exp t A Pi - - t positif} of Proposition \ref{prop controle exp t A Pi }, for $ x_d,\psi_d\geq 0 $, the following estimate holds:
\begin{equation}\label{eq constr est U ev G}
	\norme{U^{\ev}}^2_{H^s(\omega_T\times\T^m)}(x_d,\psi_d)\leq C\sum_{\mathbf{n}\in\Z^m\privede{0}}\norme{G_{\mathbf{n}}}^2_{H^s(\omega_T)}=C\norme{G}^2_{H^s(\omega_T\times\T^m)}.
\end{equation}
From now on we denote, for $ \mathbf{n} $ in $ \Z^m\privede{0} $, \begin{equation*}
	U^{\ev}_{\mathbf{n}}(z,\psi_d):=\chi(x_d)\,e^{\psi_d\,\mathcal{A}(\mathbf{n}\cdot\boldsymbol{\zeta})}\,\Pi_-^e(\freq)\,\big(B_{|E_-(\freq)}\big)^{-1}G_{\mathbf{n}}(z'),
\end{equation*}
so that $ U^{\ev}(z,\theta,\psi_d)=\sum_{\mathbf{n}\in\Z^m\privede{0}}U^{\ev}_{\mathbf{n}}(z,\psi_d)\,e^{i\mathbf{n}\cdot\theta} $. 

Then it is proven that the profile $ U^{\ev} $ is continuous with respect to $ (x_d,\psi_d) $ in $ \R_+\times\R_+ $ with values in $ H^s(\omega_T\times\T^m) $. Consider $ (x_d^0,\psi_d^0) $ in $ \R_+\times\R_+ $, and $ \epsilon>0 $. 
There holds, for $ x_d,\psi_d\geq 0 $,
\begin{align*}
	\norme{U^{\ev}(x_d,\psi_d)-U^{\ev}(x_d^0,\psi_d^0)}_{H^s(\omega_T\times\T^m)}\leq& \norme{U^{\ev}(x_d,\psi_d)-U^{\ev}(x_d^0,\psi_d)}_{H^s(\omega_T\times\T^m)}\\
	+&\norme{U^{\ev}(x_d^0,\psi_d)-U^{\ev}(x_d^0,\psi_d^0)}_{H^s(\omega_T\times\T^m)},
\end{align*}
and we seek to estimate the two terms on the right hand side of the inequality. For the first one, according to estimate \eqref{eq controle exp t A Pi - - t positif} of Proposition \eqref{prop controle exp t A Pi } and Remark \ref{remarque B inverse bornee}, for $ \psi_d\geq 0 $, we get
\begin{equation*}
	\norme{U^{\ev}(x_d,\psi_d)-U^{\ev}(x_d^0,\psi_d)}_{H^s(\omega_T\times\T^m)}\leq C \left|\chi(x_d)-\chi(x_d^0)\right|\norme{G}_{H^s(\omega_T\times\T^m)}.
\end{equation*}
By continuity of $ \chi $, there exists therefore $ \delta_1>0 $, depending only on $ \epsilon $, such that for all $ x_d $ such that $ |x_d-x_d^0|<\delta_1 $ and for all $ \psi_d\geq 0 $, we have
\begin{equation*}
	\norme{U^{\ev}(x_d,\psi_d)-U^{\ev}(x_d^0,\psi_d)}_{H^s(\omega_T\times\T^m)}<\epsilon.
\end{equation*}
For the second one, we denote by $ M $ an integer such that
\begin{equation*}
	\norme{\sum_{|\mathbf{n}|>M}G_{\mathbf{n}}\,e^{i\mathbf{n}\cdot\theta}}_{H^s(\omega_T\times\T^m)}<\epsilon.
\end{equation*}
Thus, for $ \psi_d\geq 0  $,
\begin{subequations}
	\begin{align}
		&\norme{U^{\ev}(x_d^0,\psi_d)-U^{\ev}(x_d^0,\psi_d^0)}_{H^s(\omega_T\times\T^m)}\nonumber\\
		&\qquad\qquad\leq\norme{\sum_{0<|\mathbf{n}|\leq M}
			\left[U^{\ev}_{\mathbf{n}}(y,x_d^0,\psi_d)-U^{\ev}_{\mathbf{n}}(y,x_d^0,\psi_d^0)\right]e^{i\mathbf{n}\cdot\theta}}_{H^s(\omega_T\times\T^m)}\label{eq constr cont diff psi 1}\\
		&\qquad\qquad+\norme{\sum_{|\mathbf{n}|> M}\left[U^{\ev}_{\mathbf{n}}(y,x_d^0,\psi_d)-U^{\ev}_{\mathbf{n}}(y,x_d^0,\psi_d^0)\right]e^{i\mathbf{n}\cdot\theta}}_{H^s(\omega_T\times\T^m)}\label{eq constr cont diff psi 2}.
	\end{align}
\end{subequations}
The sum in term \eqref{eq constr cont diff psi 1} being finite and the functions $ U^{\ev}_{\mathbf{n}} $ being continuous with respect to $ \psi_d $ , there exists $ \delta_2>0 $ such that for all $ \psi_d $ such that $ |\psi_d-\psi_d^0|<\delta_2 $, we have $ \eqref{eq constr cont diff psi 1}<\epsilon $. On an other hand, according to estimate \eqref{eq controle exp t A Pi - - t positif}, Remark \ref{remarque B inverse bornee} and since $ \chi $ is bounded, we have, by construction of $ M $, for all $ \psi_d\geq 0 $, $ \eqref{eq constr cont diff psi 2}<C\epsilon $ where $ C>0 $ does not depend on $ \epsilon $. It is then possible to conclude: for all $ (x_d,\psi_d) $ such that $ |(x_d,\psi_d)-(x_d^0,\psi_d^0)|<\min(\delta_1,\delta_2) $, we have
\begin{equation*}
	\norme{U^{\ev}(x_d,\psi_d)-U^{\ev}(x_d^0,\psi_d^0)}_{H^s(\omega_T\times\T^m)}<(2+C)\epsilon,
\end{equation*}
showing the required continuity.

Finally, with similar arguments as above for the continuity property, it is possible to show that the profile $ U^{\ev} $ converges towards zero in the space $ \mathcal{C}\big(\R^+_{x_d},H^s(\omega_T\times\T^m)\big) $ when $ \psi_d $ goes to infinity (so in particular in $ H^s(\omega_T\times\T^m) $ for every fixed $ x_d $).

All points of Definition \ref{def profils ev} of evanescent profiles have therefore been verified, so it has been proven that the profile $ U^{\ev} $ belongs to the space $ \P^{\ev}_{s,T} $ of evanescent profiles.

\begin{remark}
	It has been shown in the previous paragraph, in estimate \eqref{eq constr est U osc G}, that the oscillating part $ U^{\osc} $ is controlled in $ \mathcal{E}_{s,T} $ by the $ H^{s}(\omega_T\times\T^m) $ norm of the boundary term $ G $. On an other hand, according to estimate \eqref{eq constr est U ev G}, the evanescent part $ U^{\ev} $ is also controlled by the $ H^{s}(\omega_T\times\T^m) $ norm of $ G $. Thus the leading profile $ U $ satisfies
	\begin{equation*}
		\norme{U}_{\P_{s,T}}\leq C\norme{G}_{H^{s}(\omega_T\times\T^m)}.
	\end{equation*}
\end{remark}

\subsection{Conclusion and perspectives}\label{subsection conclusion}

It has therefore been proven that for $ s>h+(d+m)/2 $, there exists a time $ T>0 $ small enough such that systems \eqref{eq red syst osc res}, \eqref{eq red syst osc non res} and \eqref{eq red syst ev} admit solutions $ U_{\res}^{\osc} $, $ S_{\mathbf{n}_0,\xi_0} $ and $ U^{\ev} $ in $ \P^{\osc}_{s,T} $, $ \mathcal{C}(\R_+,H^s(\omega_T\times\T)) $ and $ \P^{\ev}_{s,T} $. According to Proposition \ref{prop equivalence systeme}, the profile $ U=U^{\osc}+U^{\ev} $ (where $ U^{\osc} $ is defined from $ U^{\osc}_{\res} $ and $ S_{\mathbf{n}_0,\xi_0} $ by equation \eqref{eq constr def U osc}) is therefore a solution in $ \P_{s,T} $ to system \eqref{eq obtention U 1}. It concludes the proof of Theorem \ref{thm existence profils}.

Estimate \eqref{eq estim est a priori} is not tame since the  norm of $ V^{\osc} $ in the estimate depends on the regularity index $ s $. Therefore, it is a priori not possible to obtain the existence of a solution $ U^{\osc} $ of infinite regularity considering a boundary term $ G $ infinitely regular, since without a tame estimate, the existence time $ T $ a priori depends on the considered index $ s $. It has been chosen not to attempt to keep the estimates tame until the end for the sake of simplicity, for example in estimate \eqref{eq estim est a priori L2}, but it is however conceivable to achieve this more precise statement in further work.

Possible extensions of the result of this article to less restrictive assumptions are now discussed. It seems reasonable to consider a similar result under the assumption that the system under study is hyperbolic with constant multiplicity, and not strictly hyperbolic (Assumption \ref{hypothese stricte hyp}). Similarly, Assumption \ref{hypothese pas de sortant} could be removed to allow outgoing frequencies to exist within the domain. It is a situation of this type which is considered in \cite{CoulombelGuesWilliams2011Resonant}. But in this case it is no longer possible to determine beforehand the traces of incoming modes, as done in Proposition \ref{prop equivalence systeme}. This may also open the way to an infinite number of resonances with outgoing phases, which complicates the functional framework.  The weakening of the uniform Kreiss-Lopatinskii condition Assumption \ref{hypothese UKL} shall be discussed in a future work. Concerning the glancing frequencies, Assumption \ref{hypothese mult vp} stating that all glancing frequencies are of order 2 seems to be crucial, see \cite{Williams2000Boundary}. Likewise, it seems difficult to do without Assumption \ref{hypothese pas de glancing} ensuring that no glancing frequencies are created on the boundary.

Finally, this work raises the question of the justification of the geometric optic expansion that has been constructed, namely to prove that the function
\begin{equation*}
	z\mapsto \epsilon\, U_1(z,z'\cdot\zeta_1/\epsilon,\dots,z'\cdot\zeta_m/\epsilon,x_d/\epsilon)
\end{equation*}
is indeed a good approximation on a fixed time interval of the exact solution to \eqref{eq systeme 1} as $ \epsilon $ goes to zero. To do so, two main methods are practicable. As conducted in \cite{Williams1996Boundary}, if there exists a solution on a time interval independent on the parameter $ \epsilon $, it is conceivable to show that this exact solution and the function defined above draw near each other when $ \epsilon $ goes to 0, see \cite{JolyMetivierRauch1995Coherent} and \cite{CoulombelGuesWilliams2011Resonant}. The problem is that in this work we do not have an exact solution on a fixed time interval. An other strategy relies on using a large number of corrector profiles, which we do not dispose either here (constructing correctors relies on small divisor accurate controls for noncharacteristic modes, which goes even further beyond Assumption \ref{hypothese petits diviseurs 1}).
Both of these points (getting an existence time of the exact solution independent of epsilon and building a large number of correctors) do not seem to be within our reach for the moment, but will be the topics of future studies.

\appendix

\section{Additional proofs}\label{appendix preuve}

\subsection{Proof of Proposition \ref{prop proj bornes}}

We detail here the proof of Proposition \ref{prop proj bornes}, omitted at first because of its length.

Recall that, for $ \zeta $ in $ \Xi_0 $, the projectors $ \Pi_j(\zeta) $, for $ j $ in $ \mathcal{G}(\zeta)\cup\mathcal{I}(\zeta) $, are defined as the projectors from $ E_-(\zeta) $ on $ E^j_-(\zeta) $ according to decomposition \eqref{eq decomp E_-(zeta)}, and that $ \Pi_-^e(\zeta) $ is defined as the projector from $ E_-(\zeta) $ on the elliptic stable component $ E^e_-(\zeta)=\oplus_{j\in\mathcal{P}(\zeta)}E^j_-(\zeta) $ according to the same decomposition. Proposition \ref{prop proj bornes} then reads as follows.

\begin{proposition}[{\cite{Williams1996Boundary}}]
	Under assumption \ref{hypothese stricte hyp} and \ref{hypothese mult vp}, for $\zeta\in\Xi_0$ the projectors $\Pi^j_-(\zeta)$ for $j$ in $\mathcal{G}(\zeta)\cup\mathcal{I}(\zeta)$, and the projectors $\Pi^{e}_-(\zeta)$ are uniformly bounded with respect to $\zeta$ in $ \Xi_0 $. 
\end{proposition}

\begin{proof} In all the proof we indistinctly denote by $e$ every analytic function which, evaluated in a particular point $\underline{\alpha}$ precised below, is nonzero, and which is therefore nonzero in a neighborhood of the point $\underline{\alpha}$. Since the projectors $\Pi^j_-(\zeta)$, $j\in\mathcal{G}(\zeta)\cup\mathcal{I}(\zeta)$ and $\Pi^{e}_-(\zeta)$ are homogeneous of degree 0 with respect to $\zeta$, the claim is proved locally in $\Sigma_0$, and the result follows from the compactness of the sphere $\Sigma_0$. The study is therefore reduced locally in a neighborhood of every point of $\Sigma_0$.
	
	Consider $\underline{\zeta}=(\underline{\tau},\underline{\eta})\in\Sigma_0$. We are interested in the behavior, on a neighborhood of $\underline{\zeta}$ in $\Sigma_0$, of the projectors $\Pi^j_-(\zeta)$, $j\in\mathcal{G}(\zeta)\cup\mathcal{I}(\zeta)$ and $\Pi^{e}_-(\zeta)$, and therefore in the behavior, in a neighborhood of $\underline{\zeta}$, of the eigenvalues of $\mathcal{A}(\zeta)$. According to Proposition \ref{prop struct bloc} there exists a neighborhood  $\mathcal{V}$ of $\underline{\zeta}$ in $\Sigma_0$, an integer $L\geq 0$, and a regular basis $\C^N$ in which the matrix $\mathcal{A}(\zeta)$ is a block diagonal matrix of the form 
	\begin{equation}\label{eq bornitude decomp bloc A}
		\diag\big(\mathcal{A}_-(\zeta),\mathcal{A}_+(\zeta),\mathcal{A}_1(\zeta),\dots,\mathcal{A}_L(\zeta)\big),
	\end{equation}
	where the block $\mathcal{A}_-(\zeta)$ (resp. $\mathcal{A}_+(\zeta)$), eventually of size zero, is of negative definite (resp. positive definite) real part, and where the blocks $\mathcal{A}_j(\zeta)$ are of type \emph{iii)} or \emph{iv)} with the notations of Proposition \ref{prop struct bloc}. According to this proposition, the eigenvalues associated with the blocks of type \emph{iii)} remain imaginary for $\zeta\in\Sigma_0$ in a neighborhood of $\underline{\zeta}$ and therefore do not contribute to the elliptic parts of the stable and unstable subspaces. However, the eigenvalues of the blocks of type \emph{iv)} may have a nonzero real part in a neighborhood of $\underline{\zeta}$ and thus contribute to the elliptic parts. Thus, in a neighborhood of $\underline{\zeta}$ in $\Sigma_0$, the elliptic part $\oplus_{j\in\P(\zeta)}E^j_-(\zeta)$ writes as the direct sum of the stable subspace for $\mathcal{A}(\zeta)$ associated with the block $\mathcal{A}_-(\zeta)$ and of the generalized eigenspaces associated with the potential eigenvalues of negative real part of the blocks $\mathcal{A}_j(\zeta)$ of type \emph{iv)}. 
	The detailed description of these eigenspaces constitutes the central point of the analysis below. 
	
	In the basis adapted to decomposition \eqref{eq bornitude decomp bloc A}, which is analytic with respect to $\zeta\in\Sigma_0$, we consider the first vectors associated with the block $\mathcal{A}_-(\zeta)$ and the aim is to complete this set of vectors into an analytic basis of the stable subspace $E_-(\zeta)$. The purpose is to construct, in a neighborhood of $\underline{\zeta}$, a determination, continuous with respect to $ \zeta $, of the stable eigenvectors of $\mathcal{A}(\zeta)$  associated with the blocks  $\mathcal{A}_j(\zeta)$ of type \emph{iii)} and \emph{iv)} (which are therefore imaginary in $\underline{\zeta}$) and to deduce from that the existence of a linearly independent set of generalized eigenvectors continuously depending on $\zeta$. To this end, the analyis of \cite{Metivier2000Block} is followed.
	
	Let $i\,\underline{\xi}_j$ be an imaginary eigenvalue of $\mathcal{A}(\underline{\zeta})$ of algebraic multiplicity $n_j$.
	By definition of $\mathcal{A}(\underline{\zeta})$ and with the notations of Assumption \ref{hypothese stricte hyp}, there exists a unique index $k_j$ between 1 and $ N $ such that 
	\begin{equation*}
	\underline{\tau}=\tau_{k_j}(\underline{\eta},\underline{\xi}_j).
	\end{equation*}
	Two cases may occur, depending on the cancellation of the quantity $\frac{\partial \tau_{k_j}}{\partial\xi}(\underline{\eta},\underline{\xi}_j)$. In the first case we shall see that there exists a continuous extension of the eigenvalue $i\,\underline{\xi}_j$ which remains imaginary for $\zeta\in\Sigma_0$ in a neighborhood of $\underline{\zeta}$, and that there exists a regular projector on the associated subspace. In the second case, the eigenvalue $i\,\underline{\xi}_j$ is degenerate (i.e. is not semisimple) and extends to a continuous eigenvalue $i\,\xi_j$, which, depending on the position of $\zeta$ in the neighborhood of $\underline{\zeta}$, may become of nonzero real part, or imaginary and simple, or even remains imaginary and degenerate. 
	
	\bigskip
	
	First suppose that
	\[\frac{\partial \tau_{k_j}}{\partial\xi}(\underline{\eta},\underline{\xi}_j)\neq 0, \]
	that is $(\underline{\tau},\underline{\eta},\underline{\xi}_j)$ is incoming or outgoing. According to Assumption \ref{hypothese stricte hyp}, for $(\zeta,\xi)$  in $\R^{d+1}\setminus\ensemble{0}$, we have
	\begin{equation}\label{eq bornitude intermediaire 3}
		\det\big(\mathcal{A}(\zeta)-i \xi I\big)=\det(A_d(0))^{-1}i^N\det L\big(0,(\zeta,\xi)\big)=\big(\tau-\tau_{k_j}(\eta,\xi)\big)\,e(\zeta,\xi),
	\end{equation}
	where $e(\underline{\zeta},\underline{\xi}_j)\neq0$. Since $\frac{\partial \tau_{k_j}}{\partial\xi}(\underline{\eta},\underline{\xi}_j)\neq 0$, according to the Weierstrass preparation theorem \cite{Hormander1990Complex}, there exists a unique real analytic function $\xi_j$ defined in a neighborhood of $\underline{\zeta}$ in $\Sigma_0$ satisfying $\xi_j(\underline{\zeta})=\underline{\xi}_j$ and such that in a neighborhood of $(\underline{\zeta},\underline{\xi}_j)$ in $\Sigma_0\times \R$ we have
	\begin{equation}\label{eq bornitude intermediaire 1}
		\tau-\tau_{k_j}(\eta,\xi)=\big(\xi-\xi_j(\zeta)\big)\,e(\zeta,\xi),
	\end{equation}
	where $e(\underline{\zeta},\underline{\xi}_j)\neq0$. Thus, in a neighborhood of $(\underline{\zeta},\underline{\xi}_j)$ in $\Sigma_0\times \R$ we have
	\[\det\big(\mathcal{A}(\zeta)-i \xi I\big)=\big(\xi-\xi_j(\zeta)\big)\,e(\zeta,\xi), \]
	where $e(\underline{\zeta},\underline{\xi}_j)\neq0$, so in a neighborhood of $\underline{\zeta}$ in $\Sigma_0$, $i\,\xi_j(\zeta)$ is an eigenvalue (analytic with respect to $\zeta$) of $\mathcal{A}(\zeta)$ of algebraic multiplicity 1. On an other hand, according to identity \eqref{eq bornitude intermediaire 1}, we have $\tau=\tau_{k_j}\big(\eta,\xi_j(\zeta)\big)$, thus
	\begin{equation*}
		\mathcal{A}(\zeta)\,\pi_{k_j}(\eta,\xi_j(\zeta))=i\,\xi_j(\zeta)\,\pi_{k_j}(\eta,\xi_j(\zeta)).
	\end{equation*}
	In a neighborhood of $\underline{\zeta}$ in $\Sigma_0$, $i\,\xi_j(\zeta)$ is therefore an eigenvalue of $\mathcal{A}(\zeta)$ of geometric multiplicity 1, thus simple. Furthermore the projector $\pi_{k_j}\big(\eta,\xi_j(\zeta)\big)$ is analytic with respect to $\zeta$ and is a projector on the eigenspace of $\mathcal{A}(\zeta)$ associated with $i\,\xi_j(\zeta)$. Thus, in the block decomposition \eqref{eq bornitude decomp bloc A}, there is a unique scalar block among the blocks $\mathcal{A}_l(\zeta)$ corresponding to the eigenvalue $i\,\xi_j(\zeta)$. In the incoming case, we then obtain associated eigenvectors depending analyticly on $\zeta\in\Sigma_0$ in a neighborhood of $\underline{\zeta}$, contributing to the stable subspace $E_-(\zeta)$.
	
	\bigskip
	
	If now $\frac{\partial \tau_{k_j}}{\partial\xi}(\underline{\eta},\underline{\xi}_j)=0$, then, according to Assumption \ref{hypothese mult vp}, we have $\frac{\partial^2 \tau_{k_j}}{\partial\xi^2}(\underline{\eta},\underline{\xi}_j)\neq 0$ and in that case we say that $\underline{\xi}_j$ is \emph{glancing}. Thus there exists a function $e$ defined in a neighborhood of $\underline{\xi}_j$ with $e(\underline{\xi}_j)\neq 0$ such that for $\xi$ close to $\underline{\xi}_j$, we have\begin{equation*}
	\underline{\tau}-\tau_{k_j}(\underline{\eta},\xi)=(\xi-\underline{\xi}_j)^2\,e(\xi).
	\end{equation*}
	We deduce, according to \eqref{eq bornitude intermediaire 3} that for $\xi$ close to $\underline{\xi}_j$,
	\[\det\big(\mathcal{A}(\underline{\zeta})-i \xi I\big)=\big(\xi-\underline{\xi}_j\big)^{2}e(\xi),\]
	where $e(\underline{\xi}_j)\neq0$. The algebraic multiplicity $n_j$ of the eigenvalue $i\underline{\xi}_j$ is therefore equal to $2$ whereas its geometric multiplicity equals 1 since 
	\begin{equation*}
		\ker \big(\mathcal{A}(\underline{\zeta})-i\,\underline{\xi}_j\,I\big)=\ker L\big(0,(\underline{\zeta},\underline{\xi}_j)\big)=\Im \pi_{k_j}(\underline{\eta},\underline{\xi}_j),
	\end{equation*}
	and since the projector $\pi_{k_j}(\underline{\eta},\underline{\xi}_j)$ is of rank 1. The aim is therefore to find a basis of the generalized eigenspace associated with $i\,\underline{\xi}_j$, which is of dimension $2$. By definition of the analytic function $\tau_{k_j}$ and of the projector $\pi_{k_j}$, we have, for $\xi$ close to $\underline{\xi}_j$,
	\[L\big(0,\tau_{k_j}(\underline{\eta},\xi),\underline{\eta},\xi\big)\,\pi_{k_j}(\underline{\eta},\xi)=0.\]
	Differentiating this equation with respect to $ \xi $ and evaluating in $ \xi=\underline{\xi}_j $, one gets, since $ \tau_{k_j}(\underline{\eta},\underline{\xi}_j)=\underline{\tau} $,
	\begin{equation*}
		\partial_{\xi} \tau_{k_j}(\underline{\eta},\underline{\xi}_j)\,\partial_{\tau} L\big(0,\underline{\tau},\underline{\eta},\underline{\xi}_j\big)\,\pi_{k_j}(\underline{\eta},\underline{\xi}_j)+\partial_{\xi} L\big(0,\underline{\tau},\underline{\eta},\underline{\xi}_j\big)\,\pi_{k_j}(\underline{\eta},\underline{\xi}_j)
		\\
		+L\big(0,\underline{\tau},\underline{\eta},\underline{\xi}_j\big)\,\partial_{\xi} \pi_{k_j}(\underline{\eta},\underline{\xi}_j)=0,
	\end{equation*}
	that is to say, according to the expression of $ L\big(0,(\tau,\eta,\xi)\big) $ and using $ \partial_{\xi} \tau_{k_j}(\underline{\eta},\underline{\xi}_j)=0 $,
	\begin{equation*}
		A_d(0)\,\pi_{k_j}(\underline{\eta},\underline{\xi}_j)
		+iA_d(0)\big(\mathcal{A}(\underline{\zeta})-i\underline{\xi}_j\big)\frac{\partial \pi_{k_j}}{\partial \xi}(\underline{\eta},\underline{\xi}_j)=0.
	\end{equation*}
	Denoting $ \underline{P}_0:=\pi_{k_j}(\underline{\eta},\underline{\xi}_j) $ and $ \underline{P}_1:=\frac{\partial \pi_{k_j}}{\partial \xi}(\underline{\eta},\underline{\xi}_j) $ we obtain
	\begin{equation}\label{eq bornitude intermediaire 2}
		\big(\mathcal{A}(\underline{\zeta})-i\underline{\xi}_j\big)\underline{P}_1=i\underline{P}_0.
	\end{equation}
	We then denote by $ \underline{E}_j $ a nonzero vector of the linear line $ \Ima \pi_{k_j}(\underline{\eta},\underline{\xi}_j) $. Equation \eqref{eq bornitude intermediaire 2} thus leads to
	\begin{equation}\label{eq bornitude base esp propre}
		\big(\mathcal{A}(\underline{\zeta})-i\underline{\xi}_j\big)\underline{P}_1\underline{E}_j=i\underline{E}_j.
	\end{equation}
	One can then verify that $ (\underline{E}_j,\underline{P}_1\underline{E}_j) $ is a family of linearly independent vectors and that it therefore forms a basis of the generalized eigenspace associated with $ i\,\underline{\xi}_j $. In this basis, according to \eqref{eq bornitude base esp propre}, the operator $ \mathcal{A}(\underline{\zeta}) $ restricted to the generalized eigenspace associated with $ i\,\underline{\xi}_j $ is given by the following matrix:
	\begin{equation}\label{eq bornitude blocs Q point base}
		Q(\underline{\zeta})=\left(\begin{array}{cc}
			i\underline{\xi}_j & i \\
			0 & i\underline{\xi}_j
		\end{array}\right).
	\end{equation}
	We have therefore obtained a triangularization of the matrix $\mathcal{A}(\underline{\zeta})$ restricted to the generalized eigenspace associated with $ i\,\underline{\xi}_j $, and we seek to extend this structure in a neighborhood of $\underline{\zeta}$ and to study the behavior of the stable eigenvalues of the matrix $\mathcal{A}(\underline{\zeta})$ restricted to the generalized eigenspace associated with $ i\,\underline{\xi}_j $ in a neighborhood of $\underline{\zeta}$. In \cite{Metivier2000Block} and using a result of \cite{Ralston1971Note}, it is proved that there exists a linearly independent set of vectors $ E^0_j(\zeta),E^1_j(\zeta) $, analytic with respect to $ \zeta\in\Sigma $ in a neighborhood of $\underline{\zeta}$, generating a subspace $ F_j(\zeta) $ which is stable under $\mathcal{A}(\zeta)$, such that  $ E^0_j(\underline{\zeta})=\underline{E}_j $ and  $ E^1_j(\underline{\zeta})=\underline{P}_1\underline{E}_j $ and such that the restriction of $ \mathcal{A}(\zeta) $ to the subspace $ F_j(\zeta) $ is given by 
	\begin{equation}\label{eq bornitude blocs Q}
		Q(\zeta)=i\left(\begin{array}{cc}
			\underline{\xi}_j+q_1(\zeta) & 1 \\
			q_2(\zeta) & \underline{\xi}_j
		\end{array}\right),
	\end{equation} 
	where $ q_1(\underline{\zeta})=q_2(\underline{\zeta})=0 $ and where $ \frac{\partial q_2}{\partial \tau}(\underline{\zeta})\neq 0 $. 
	Among the blocks $\mathcal{A}_l(\zeta)$ of the block diagonalization \eqref{eq bornitude decomp bloc A} of the matrix $\mathcal{A}(\zeta)$ in a neighborhood of $ \underline{\zeta} $, there is therefore a $ 2\times 2 $ block given by $Q(\zeta)$.

	The aim is now to study the eigenvalues of the $ 2\times2 $ block $ Q(\zeta) $ above and to find a continuous determination of the stable eigenvalue in a neighborhood of $\underline{\zeta}$ in $\Sigma$ (and not only in $\Sigma_0$), namely the Laplace parameter $ \gamma $ is allowed to be positive. First the expression of the characteristic polynomial of $Q(\zeta)$ is investigated. It is of degree 2, allowing to obtain an explicit formula for the eigenvalues of $Q(\zeta)$.  According to \eqref{eq bornitude intermediaire 3}, in a neighborhood of $(\underline{\zeta},\underline{\xi}_j)$, we have
	\begin{equation*}
		\det\big(\mathcal{A}(\zeta)-i \xi I\big)=\big(\tau-\tau_{k_j}(\eta,\xi)\big)\,e(\zeta,\xi).
	\end{equation*}
	On an other hand, according to the Weierstrass preparation theorem and since $ \partial_{\xi} \tau_{k_j}(\underline{\eta},\underline{\xi}_j)=0$ and $\partial^2_{\xi} \tau_{k_j}(\underline{\eta},\underline{\xi}_j)\neq 0$, there exists a couple of  functions $(f_0,f_1)$, analytic with respect to $\underline{\zeta}$, satisfying $f_0(\underline{\zeta})=f_1(\underline{\zeta})=0$ and such that for $(\zeta,\xi)$ close to $(\underline{\zeta},\underline{\xi}_j)$,
	\begin{equation}\label{eq bornitude intermediaire 4}
		\tau-\tau_{k_j}(\eta,\xi)=\Big(\xi^2-\big(2\underline{\xi}_j+f_1(\zeta)\big)\,\xi+\underline{\xi}_j^2+ f_0(\zeta)\Big)\,e(\zeta,\xi),
	\end{equation}
	where $e(\underline{\zeta},\underline{\xi}_j)\neq0$. Thus
	\begin{equation*}
		\det\big(\mathcal{A}(\zeta)-i \xi I\big)=\Big(\xi^2-\big(2\underline{\xi}_j+f_1(\zeta)\big)\,\xi+\underline{\xi}_j^2+ f_0(\zeta)\Big)\,e(\zeta,\xi),
	\end{equation*}
	where $e(\underline{\zeta},\underline{\xi}_j)\neq0$. But according to the block decomposition of $\mathcal{A}(\zeta)$ we have
	\begin{equation*}
		\det\big(\mathcal{A}(\zeta)-i \xi I\big)=\det\big(Q(\zeta)-i \xi I\big)\,e(\zeta,\xi),
	\end{equation*}
	where $e(\underline{\zeta},\underline{\xi}_j)\neq0$ so that
	\begin{equation*}
		\xi^2-\big(2\underline{\xi}_j+f_1(\zeta)\big)\,\xi+\underline{\xi}_j^2+ f_0(\zeta)=\det\big(Q(\zeta)-i \xi I\big)\,e(\zeta,\xi),
	\end{equation*}
	where $e(\underline{\zeta},\underline{\xi}_j)\neq0$. Since according to \eqref{eq bornitude blocs Q} the $\xi$ polynomial given by $\det\big(Q(\zeta)-i \xi I\big)$ is of degree 2 and of leading coefficient $-1$, we obtain
	\[\det\big(Q(\zeta)-i \xi I\big)=-\Big(\xi^2-\big(2\underline{\xi}_j+f_1(\zeta)\big)\,\xi+\underline{\xi}_j^2+ f_0(\zeta)\Big).\]
	By identification, according to \eqref{eq bornitude blocs Q}, we get
	$f_1=q_1$ and $f_0=\underline{\xi}_j\,q_1-q_2$.
	
	The interest is now made on the behavior of the eigenvalues of $Q(\zeta)$, and therefore on the roots of the polynomial $\xi^2-\big(2\underline{\xi}_j+f_1(\zeta)\big)\,\xi+\underline{\xi}_j^2+ f_0(\zeta)$, for $\zeta$ in a neighborhood of $\underline{\zeta}$ in $\Sigma$. The Puiseux expansion theory ensure that for $\gamma>0$ small, the eigenvalues of $Q(\zeta)$ with $\zeta=(\underline{\tau}-i\gamma,\underline{\eta})$ admit an expansion of the form
	\[\xi(\zeta)=\underline{\xi}_j + \alpha_{1,2}\,\gamma^{1/2}+O(\gamma),\]
	where the coefficients $\alpha_{1,2}$ are obtained resolving
	\[\alpha_{1,2}^2=i\big(\partial_{\tau}f_0(\underline{\zeta})-\partial_{\tau}f_1(\underline{\zeta})\,\underline{\xi}_j\big).\]
	But since $f_1=q_1$ and $f_0=\underline{\xi}_j\,q_1-q_2$, we have
	\[\partial_{\tau}f_0(\underline{\zeta})-\partial_{\tau}f_1(\underline{\zeta})\,\underline{\xi}_j=-\partial_{\tau}q_2(\underline{\zeta})\neq 0,\]
	so that $\Im \alpha_{1,2}=\pm c$ where $c>0$. Thus for $\gamma>0$, $Q(\underline{\tau}-i\gamma,\underline{\eta})$  admits a unique stable eigenvalue $\xi_j^-(\zeta)$ (namely such that $\Im \xi_j^-(\zeta)>0$) and a unique unstable eigenvalue $\xi_j^+(\zeta)$ (such that $\Im \xi_j^+(\zeta)<0$). 
	It is deduced that for  $\zeta$ in a neighborhood of $\underline{\zeta}$ in $ \Sigma\setminus\Sigma_0$, $Q(\zeta)$ admits a unique stable eigenvalue denoted by $\xi_j(\zeta)$. We then seek to continuously extend the eigenvalue $\xi_j$ for $\gamma=0$, that is to say we are interested in the root $ \xi^2-\big(2\underline{\xi}_j+f_1(\zeta)\big)\,\xi+\underline{\xi}_j^2+ f_0(\zeta) $ that extends $\xi_j(\zeta)$ to a neighborhood of $\underline{\zeta}$ in $\Sigma$. The behavior of this extension $\xi_j$ shall then depends on the sign of the discriminant (real when $\zeta$ is real) $\Delta_j(\zeta):=4\,\underline{\xi}_j\,f_1(\zeta)+f_1(\zeta)^2-4f_0(\zeta)$ which has been represented in Figure \ref{figure discriminant}.
	
	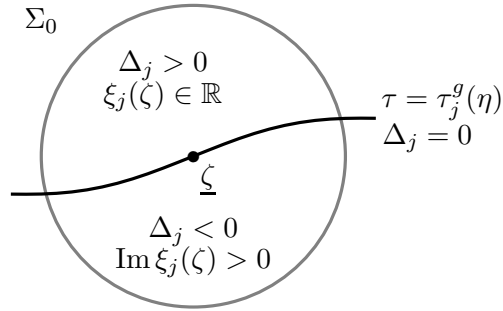
\begin{figure}[!ht]
		\begin{tikzpicture}[scale=0.4]
			\draw[black!50, line width=1.2pt] (0,0) circle(5);
			\draw [line width = 1.2pt, domain=-6:6] plot(\x,0.001*\x^4-0.013*\x^3+0.01*\x^2+0.4*\x);
			\draw (8.2,1.8) node{$\tau=\tau^g_j(\eta)$};
			\draw (7.7,0.6) node{$\Delta_j=0$};
			\draw (-5,4.5) node{$\Sigma_0$};
			\draw (-1,3) node{$\Delta_j>0$};
			\draw (-1,2) node{$\xi_j(\zeta)\in\R$};
			\draw (0,-2.5) node{$\Delta_j<0$};
			\draw (0,-3.5) node{$\Im\xi_j(\zeta)>0$};
			\draw (0,-0.05) node{$\bullet$};
			\draw (-0.1,0.1) node[below right]{$\underline{\zeta}$};
		\end{tikzpicture}
		\caption{Sign of the discriminant $\Delta_j(\zeta)$ in a neighborhood of $\underline{\zeta}$.}
		\label{figure discriminant}
	\end{figure}
	
	When the discriminant $\Delta_j(\zeta)$ of this polynomial is negative, the eigenvalue $\xi_j(\zeta)$ is necessarily given by
	\[\xi_j(\zeta)=\frac{2\,\underline{\xi}_j+f_1(\zeta)+i\sqrt{4f_0(\zeta)-4\,\underline{\xi}_j\,f_1(\zeta)-f_1(\zeta)^2}}{2},\]
	since it must be of non-negative imaginary part.
	When the discriminant $\Delta_j(\zeta)$ is zero, $\xi_j(\zeta)$ is given by
	\[\xi_j(\zeta)=\frac{2\,\underline{\xi}_j+f_1(\zeta)}{2}.\]
	The location of the discriminant roots may even be made precise, since it satisfies $\partial_{ \tau}\Delta_j(\underline{\zeta})=4\big(\partial_{\tau}f_1(\underline{\zeta})\,\underline{\xi}_j-\partial_{\tau}f_0(\underline{\zeta})\big)\neq 0$, so according to the implicit functions theorem, there exists an analytic function $\tau^g_j$ defined in a neighborhood of $\underline{\eta}$ which parameterizes in a neighborhood of  $\underline{\zeta}$ in $\Sigma_0$ the set of the discriminant's roots, see Figure \ref{figure discriminant}. Finally, when the discriminant $\Delta_j(\zeta)$ is positive, we must determine which one of the real roots
	\begin{equation}\label{eq bornitude intermediaire 5}
		\frac{2\,\underline{\xi}_j+f_1(\zeta)\pm\sqrt{4\,\underline{\xi}_j\,f_1(\zeta)+f_1(\zeta)^2-4f_0(\zeta)}}{2},
	\end{equation}
	continuously extends the stable eigenvalue $\xi_j(\zeta)$ when $\gamma=0$. If $\xi_j(\zeta)$  refers to the sought eigenvalue until $\gamma=0$, and if we denote $\zeta=(\sigma,\eta):=(\tau-i\gamma,\eta)$, since $\xi_j(\zeta)$ is real when $\gamma=0$ and $\Im \xi_j(\zeta)>0$ when $\gamma>0$, we have necessarily
	\[\frac{\partial \Im \xi_j}{\partial\Im \sigma}\Big|_{\gamma=0}\leq 0,\]
	so that according to the Cauchy-Riemann equations, we must have
	\[\partial_{\tau}\big(\Re \xi_j\big)|_{\gamma=0}\leq 0.\]
	Thus, if $\partial_{\tau}f_0(\underline{\zeta})-\partial_{\tau}f_1(\underline{\zeta})\,\underline{\xi}_j>0$, the real root
	\[\frac{2\,\underline{\xi}_j+f_1(\zeta)+\sqrt{4\,\underline{\xi}_j\,f_1(\zeta)+f_1(\zeta)^2-4f_0(\zeta)}}{2}\]
	is the one that continuously extends the stable eigenvalue $\xi_j(\zeta)$ when $\gamma=0$, and in the other case, the other root must be chosen.
	We have therefore obtained a continuous determination of the stable eigenvalue $\xi_j(\zeta)$ of the matrix $Q(\zeta)$ in a neighborhood of $\underline{\zeta}$. Note now that an eigenvector of the matrix
	\[\left(\begin{array}{cc}
		\underline{\xi}_j+q_1(\zeta) & 1 \\
		q_2(\zeta) & \underline{\xi}_j
	\end{array}\right)\]  
	associated with the eigenvalue $\xi_j(\zeta)$ writes $\Big(1,\frac{q_2(\zeta)}{\xi_j(\zeta)-\underline{\xi}_j}\Big)$. One thus gets, using the linearly independent vectors $ E^0_j(\zeta),E^1_j(\zeta) $, an eigenvector $\mathcal{A}(\zeta)$ associated with the stable eigenvalue $i\,\xi_j(\zeta)$ continuous with respect to $\zeta$\footnote{Since $\xi_j(\zeta)$ is a root of the polynomial $\xi^2-\big(2\underline{\xi}_j+q_1(\zeta)\big)\,\xi+\underline{\xi}_j^2+ \underline{\xi}_jq_1(\zeta)- q_2(\zeta)$, according to $f_1=q_1$ and $f_0=\underline{\xi}_j q_1-q_2$, we have $\frac{q_2(\zeta)}{\xi_j(\zeta)-\underline{\xi}_j}=\xi_j(\zeta)-\underline{\xi}_j-q_1(\zeta)$. Thus the following limit holds $\frac{q_2(\zeta)}{\xi_j(\zeta)-\underline{\xi}_j}\underset{\zeta\rightarrow \underline{\zeta}}{\longrightarrow 0}$, and the considered eigenvector continuously depends on $ \zeta $ in a neighborhood of $ \underline{\zeta} $}.
	
	In a nutshell, in a neighborhood of $\underline{\zeta}$ in $\Sigma_0$, the degenerate imaginary eigenvalue $i\,\underline{\xi}_j$ continuously extends in a stable eigenvalue $i\,\xi_j(\zeta)$ of which the behavior depends on the sign of the discriminant $\Delta_j(\zeta)$ in the neighborhood of $\underline{\zeta}$, which has been figured in Figure \ref{figure discriminant}. If $\Delta_j(\zeta)$  is negative, then $i\,\xi_j(\zeta)$ is of negative real part so the eigenvalue $i\,\xi_j(\zeta)$ contributes to the elliptic part of the stable subspace $E_-(\zeta)$. If $\Delta_j(\zeta)$ is positive, then $i\,\xi_j(\zeta)$ is a simple imaginary eigenvalue of $\mathcal{A}(\zeta)$ so it contributes to a subspace $E^l_-(\zeta)$ with $l$ in $\mathcal{I}(\zeta)$. Finally if $\Delta_j(\zeta)$ is zero, the eigenvalue $i\,\xi_j(\zeta)$ remains imaginary and degenerate so it contributes to a subspace $E^l_-(\zeta)$ with $l$ in $\mathcal{G}(\zeta)$.
	
	We denote now by $i\,\underline{\xi}_l$, $l=1,\dots,r$ the real incoming eigenvalues and by $i\,\underline{\xi}_l$, $l=r+1,\dots,r+g$ the glancing eigenvalues of the matrix $\mathcal{A}(\underline{\zeta})$. Using the notations of the beginning of the proof, we have found continuous extensions $i\,\xi_l$, $l=1,\dots,r+g$, of these eigenvalues in a neighborhood of $\underline{\zeta}$. Therefore, a continuous determination of the stable eigenvalues of the blocks $\mathcal{A}_1(\zeta),\dots,\mathcal{A}_L(\zeta)$ has been determined, as well as a continuous basis of the stable subspace $E_-(\zeta)$ constituted of generalized eigenvectors of the matrix $\mathcal{A}_-(\zeta)$ and of eigenvectors of the matrix $\mathcal{A}(\zeta)$ associated with the eigenvalues $i\,\xi_l(\zeta)$ for $l=1,\dots,r+g$. Then we denote by $\tilde{\Pi}^e_-(\zeta)$ the analytic projector from $E_-(\zeta)$ to the stable subspace associated with the elliptic block $\mathcal{A}_-(\zeta)$, and, for $l=1,\dots,r+g$, $\tilde{\Pi}^l_-(\zeta)$ the continuous projector from $E_-(\zeta)$ to the eigenspace associated with $i\,\xi_l(\zeta)$. Since these projectors are continuous with respect to $\zeta$ in a neighborhood of $\underline{\zeta}$, they can be assumed to be bounded on this neighborhood.
	
	If $l=1,\dots,r$, the eigenvalue $i\,\xi_l(\zeta)$ is imaginary and simple in a neighborhood of $\underline{\zeta}$, so, for all $\zeta$, the projector $\tilde{\Pi}_-^l(\zeta)$ contributes to a projector $\Pi_-^j(\zeta)$ for some $j$ (depending on $\zeta$) in $\mathcal{I}(\zeta)$. If $l=r+1,\dots,r+g$, then, depending on the sign of $\Delta_l$, the eigenvalue $i\,\xi_l(\zeta)$ may be imaginary and simple, or imaginary and degenerate, or even of nonzero real part, so depending on where $\zeta$ is in a neighborhood of $\underline{\zeta}$, the projector $\tilde{\Pi}_-^l(\zeta)$ contributes to $\Pi^{e}_-(\zeta)$ (when $\Delta_l<0$), to $\Pi^j_-(\zeta)$ for some $j$ in $\mathcal{G}(\zeta)$ (when $\Delta_l=0$) or to $\Pi^j_-(\zeta)$ for some $j$ in $\mathcal{I}(\zeta)$ (when $\Delta_l>0$). As for it, the projector $\tilde{\Pi}^e_-(\zeta)$ always contributes to $\Pi^{e}_-(\zeta)$.
	
	We seek now to explicitly describe the projectors $\Pi^{e}_-(\zeta)$ and $\Pi_-^j(\zeta)$ for $j\in \mathcal{I}(\zeta)\cup\mathcal{G}(\zeta)$. To simplify the notations, we assume that among the imaginary eigenvalues $i\,\underline{\xi}_l$, there is only one of them which is glancing, namely that $g=1$. The expressions of the sought projectors depend on whether the eigenvalue $i\,\xi_{r+1}(\zeta)$ is incoming, glancing, or of negative real part, and therefore on where is $ \zeta $ in the neighborhood of $\underline{\zeta}$, see Figure \ref{figure discriminant}. If $\zeta$ belongs to the area of the neighborhood of $\underline{\zeta}$ where $\Delta_{r+1}(\zeta)>0$, then the eigenvalue $i\,\xi_{r+1}(\zeta)$ is incoming, so $\mathcal{I}(\zeta)$ is of cardinality $r+1$ and $\mathcal{G}(\zeta)$ is empty.
	In this case, for all index $j$ in $\mathcal{I}(\zeta)$, we have
	\[
	\Pi_-^j(\zeta)=\tilde{\Pi}_-^l(\zeta),\]
	for some $l$ between $1$ and $r+1$, and
	\[\Pi^{e}_-(\zeta)=\tilde{\Pi}_-^{e}(\zeta).\] 
	If $\zeta$ is, in the neighborhood of $\underline{\zeta}$, on the hypersurface defined by $\Delta_{r+1}=0$, then $i\,\xi_{r+1}(\zeta)$ is glancing and in that case $\mathcal{I}(\zeta)$ is of cardinality $r$ and $\mathcal{G}(\zeta)$ is if cardinality $1$.
	We have therefore, for $j$ in $\mathcal{I}(\zeta)$,
	\[
	\Pi_-^j(\zeta)=\tilde{\Pi}_-^l(\zeta),\]
	for some $l$ between $1$ and $r$, for the index $j$ of $\mathcal{G}(\zeta)$,
	\[\Pi_-^j(\zeta)=\tilde{\Pi}_-^{r+1}(\zeta),\]
	and
	\[\Pi^{e}_-(\zeta)=\tilde{\Pi}_-^{e}(\zeta).\] 
	Finally, if $\zeta$ belongs to the area of the neighborhood of $\underline{\zeta}$ where $\Delta_j(\zeta)<0$, then $i\,\xi_{r+1}(\zeta)$ is of negative real part so it contributes to the elliptic part $E_-(\zeta)$. Thus $\mathcal{I}(\zeta)$ is of cardinality $r$, $\mathcal{G}(\zeta)$ is empty and the extension of the degenerate eigenvalue $ i\xi_{r+1}(\zeta) $ contributes to the elliptic part. In this case, for all index $j$ in $\mathcal{I}(\zeta)$, we have
	\[
	\Pi_-^j(\zeta)=\tilde{\Pi}_-^l(\zeta),\]
	for some $l$ between $1$ and $r$, and
	\[\Pi^{e}_-(\zeta)=\tilde{\Pi}_-^{e}(\zeta)+\tilde{\Pi}_-^{r+1}(\zeta).\] 
	Therefore, since the projectors $\tilde{\Pi}_-^l(\zeta)$, $l=1,\dots,r+1$ and $\tilde{\Pi}_-^{e}(\zeta)$ are bounded uniformly with respect to $\zeta$ in a neighborhood of $\underline{\zeta}$, we deduce that the projectors $\Pi^{e}_-(\zeta)$ and $\Pi_-^j(\zeta)$ for $j\in \mathcal{I}(\zeta)\cup\mathcal{G}(\zeta)$ are bounded uniformly with respect to $\zeta$ in a neighborhood of $\underline{\zeta}$, which concludes the proof of Proposition \ref{prop proj bornes}. 
	
	In the general case where there are multiple glancing eigenvalues $i\,\underline{\xi}_l$ (namely when $g\geq 1$), the projectors  $\Pi^{e}_-(\zeta)$ and $\Pi_-^j(\zeta)$ for $j\in \mathcal{I}(\zeta)\cup\mathcal{G}(\zeta)$ can still be expressed using the projectors $\tilde{\Pi}_-^l(\zeta)$, $l=1,\dots,r+g$ and $\tilde{\Pi}_-^{e}(\zeta)$. Since the expression of $\Pi^{e}_-(\zeta)$ will be needed in the proof of Proposition \ref{prop controle exp t A Pi } below, it is given here. For $\zeta$ in a neighborhood of $\underline{\zeta}$ in $\Sigma_0$, we have
	\begin{equation}\label{eq bornitude def Pi e}
		\Pi^{e}_-(\zeta)=\tilde{\Pi}^e_-(\zeta)+\sum_{j=r+1}^{r+g}\indicatrice_{\Delta_j(\zeta)<0}\,\tilde{\Pi}^j_-(\zeta).
	\end{equation}
\end{proof}

\subsection{Proof of Lemma \ref{lemme minoration vitesse de groupe}}

The following proof of Lemma \ref{lemme minoration vitesse de groupe} uses results and notations from the previous one, and is therefore given now. First we recall the statement of Lemma \ref{lemme minoration vitesse de groupe}.

\begin{lemma}
	There exists a positive constant $ C>0 $ such that, if the real frequency $ \alpha=(\tau,\eta,\xi) $ in $ \R^{1+d}\privede{0} $ is characteristic, and if $ k $ between 1 and $ N $ is such that $ \tau=\tau_k(\eta,\xi) $, then we have
	\begin{equation*}
		\left|\partial_{\xi}\tau_k(\eta,\xi)\right|\geq C\frac{\dist\big((\tau,\eta),\mathcal{G}\big)^{1/2}}{|(\tau,\eta)|^{1/2}}.
	\end{equation*}  
	Using Lemma \ref{lemme pi tilde E vitesse de groupe}, we therefore obtain the following estimate
	\begin{equation}
		\left|\tilde{\pi}_{\alpha}\,E_k(\eta,\xi)\right|\geq C\frac{\dist\big((\tau,\eta),\mathcal{G}\big)^{1/2}}{|(\tau,\eta)|^{1/2}}.
	\end{equation}  
\end{lemma}

\begin{proof}
	The interest is made at first in the first equality, which is proved using the homogeneity of degree zero of $ \partial_{\xi}\tau_k(\eta,\xi) $ and of degree one of the distance. The analysis is therefore made on the sphere $ \Sigma_0 $, and we denote, for $ \zeta $ in $ \Sigma_0 $,
	\begin{equation*}
		m_g(\zeta) =\left\lbrace\begin{array}{cl}
			1 & \text{ if } \spectre(\mathcal{A}(\zeta)) \cap i\R=\emptyset,  \\[10pt]
			\displaystyle\min_{\substack{j\in\mathcal{G}(\zeta)\cup\\ \mathcal{I}(\zeta)\cup\mathcal{O}(\zeta)}}\big|\partial_{\xi}\tau_{k_j}\big(\eta,\xi_j(\zeta)\big)\big| & \text{ otherwise,}
		\end{array}
		\right.
	\end{equation*}
	where $ \spectre(\mathcal{A}(\zeta)) $ refers to the spectrum of the matrix $ \mathcal{A}(\zeta) $, and where the notations $ k_j $ and $ \xi_j(\zeta) $ has been introduced in Proposition \ref{prop decomp E_-}. Using the compactness of the sphere $ \Sigma_0 $, it will be proved that $ m_g $ satisfies
	\begin{equation}\label{eq freq ineg m g 1}
		m_g(\zeta)\geq C \dist(\zeta,\mathcal{G})^{1/2},
	\end{equation}
	for all $ \zeta $ in $ \Sigma_0 $, where $ C>0 $ is a suitable fixed constant. We thus consider $ \underline{\zeta}:=(\underline{\tau},\underline{\eta}) $ in $ \Sigma_0 $, and we show that there exists a neighborhood $ \mathcal{V} $ of $ \underline{\zeta} $ in which the previous equality \eqref{eq freq ineg m g 1} is satisfied.
	
	We recall the results obtained in the proof of Proposition \ref{prop proj bornes}, in which a continuous determination of the eigenvalues of $ \mathcal{A}(\zeta) $ for $ \zeta $ in a neighborhood of $ \underline{\zeta} $ has been determined. The previous proof focused on describing the stable eigenvalues, but it can be immediately extended to all eigenvalues of $ \mathcal{A}(\zeta) $. Denote by $ i\,\underline{\xi}_j $ the imaginary eigenvalues of $ \mathcal{A}(\underline{\zeta}) $. If the imaginary eigenvalue $ i\,\underline{\xi}_j $ is not glancing, the proof of Proposition \ref{prop proj bornes} yields to a continuous extension $ i\,\xi_j(\zeta) $ in a neighborhood of $ \underline{\zeta} $, which is an eigenvalue of $ \mathcal{A}(\zeta) $. If $ i\,\underline{\xi}_j $ is glancing, then we obtain, in a neighborhood of $ \underline{\zeta} $, two continuous eigenvalues $ i\,\xi_j^-(\zeta) $ and $ i\,\xi_j^+(\zeta)  $ extending $ i\,\underline{\xi}_j $, which are possibly equal (when they are glancing). Finally, the block structure (see Proposition \ref{prop struct bloc}) and Proposition \ref{prop proj bornes} provide a basis of $ \C^N $ in which the matrix $ \mathcal{A}(\zeta) $ is, in a neighborhood of $ \underline{\zeta} $, block diagonal, with a block $ \mathcal{A}_{\pm}(\zeta) $ with eigenvalues of nonzero real part, and scalar blocks corresponding to the eigenvalues $ i\xi_j(\zeta) $. Three cases are then to be investigated.
	\begin{itemize}[label=$ \circ $, leftmargin=0.7cm]
		\item All eigenvalues of $ \mathcal{A}_{\pm}(\zeta) $ are of nonzero real part in a neighborhood of $ \underline{\zeta} $, so they don't contribute to $ m_g(\zeta) $.
		\item If $ i\,\underline{\xi}_j $ is imaginary and $ \partial_{\xi}\tau_{k_j}\big(\underline{\eta},\xi_j(\underline{\tau},\underline{\eta})\big)\neq0 $, namely if the real characteristic frequency $ \alpha_j(\underline{\zeta}) $ is incoming or outgoing, then it has been proven that the eigenvalue $ i\,\xi_j(\zeta) $ is still incoming or outgoing in a neighborhood of $ \underline{\zeta} $. Furthermore, according to equation \eqref{eq bornitude intermediaire 1} differentiated with respect to $ \xi $ and evaluated in $ \xi=\xi_j(\zeta) $, for $ \zeta $ in a neighborhood of $ \underline{\zeta} $, we have
		\begin{equation*}
			\partial_{\xi}\tau_{k_j}\big(\eta,\xi_j(\zeta)\big)=-e\big(\eta,\xi_j(\zeta)\big),
		\end{equation*}
		where $ e $ is an analytic function nonzero in $ \underline{\zeta} $, which is therefore lower bounded in a neighborhood of $ \big(\underline{\zeta},\xi_j(\underline{\zeta})\big) $. Thus, for $ \zeta $ in a neighborhood of $ \underline{\zeta} $, we have
		\begin{equation*}
			\left|\partial_{\xi}\tau_{k_j}\big(\eta,\xi_j(\zeta)\big)\right|\geq C,
		\end{equation*}
		with $ C>0 $.
		\item Finally, if $ i\,\underline{\xi}_j $ is glancing, namely if $ i\,\underline{\xi}_j $ is imaginary and $ \partial_{\xi}\tau_{k_j}(\underline{\eta},\underline{\xi}_j )=0 $, then $ i\,\underline{\xi}_j $ is extended by two eigenvalues $ i\,\xi_j^{\pm}(\zeta) $, of which the behavior depends on where $ \zeta $ is in the neighborhood of $ \underline{\zeta} $, see Figure \ref{figure discriminant}. Denote by $ \Delta_j(\zeta) $ the discriminant of the characteristic polynomial of the $ 2\times2 $ block associated with the glancing eigenvalue. If $ \zeta $ is such that $ \Delta_j(\zeta)<0 $, then the two eigenvalues $ i\,\xi_j^{\pm}(\zeta) $ are of nonzero real part, so they do not contribute to $ m_g(\zeta) $. If $ \Delta_j(\zeta)=0 $, then $ \xi_j^-(\zeta)=\xi_j^+(\zeta) $ and the characteristic frequency $ \big(\zeta,\xi_j^{\pm}(\zeta)\big) $ is glancing, so equality \eqref{eq freq ineg m g 1} is immediately satisfied. Finally, if $ \Delta_j(\zeta)>0 $, then the two distinct eigenvalues $ i\,\xi_j^{\pm}(\zeta) $ are imaginary, and contributes to $ m_g(\zeta) $. According to the relation \eqref{eq bornitude intermediaire 4}, differentiate with respect to $ \xi $ and evaluated in $ \xi=\xi_j^+(\zeta) $, we have, for $ \zeta=(\tau,\eta) $ in a neighborhood of $ \underline{\zeta} $,
		\begin{equation*}
			\partial_{\xi}\tau_{k_j}\big(\eta,\xi_j^+(\zeta)\big)=2\big(f_1(\zeta)/2+\underline{\xi}_j-\xi_j^+(\zeta)\big)\,e\big(\zeta,\xi_j^+(\zeta)\big),
		\end{equation*}
		where $ e $  is an analytic function, nonzero in $ (\underline{\zeta},\underline{\xi}_j) $. According to expression \eqref{eq bornitude intermediaire 5} of the roots $ \xi_j^{\pm}(\zeta) $, we obtain
		\begin{equation*}
			\partial_{\xi}\tau_{k_j}\big(\eta,\xi_j^+(\zeta)\big)=\pm\sqrt{\Delta_j(\zeta)} \,e\big(\zeta,\xi_j^+(\zeta)\big).
		\end{equation*}
		But, according to the proof of Proposition \ref{prop proj bornes} above, one may write, for $ \zeta=(\tau,\eta) $ in a neighborhood of $ \underline{\zeta} $, 
		\begin{equation*}
			\Delta_j(\zeta)=\big(\tau-\tau_j^g(\eta)\big)\,e(\zeta),
		\end{equation*}
		where $ e $ is an analytic function, nonzero in $ \underline{\zeta} $, and where the function $ \tau_j^g $ parameterizes the surface of the zeros of $ \Delta_j $. We finally infer
		\begin{equation*}
			\left|\partial_{\xi}\tau_{k_j}\big(\eta,\xi_j^+(\zeta)\big)\right|\geq C\left|\tau-\tau_j^g(\eta)\right|^{1/2}=C\left|(\tau,\eta)-\big(\tau_j^g(\eta),\eta\big)\right|^{1/2}\geq C \dist\big(\zeta,\mathcal{G}\big)^{1/2},
		\end{equation*}
		since the frequency $ \big(\tau_j^g(\eta),\eta\big) $ is glancing by construction of $ \tau_j^g $. The same arguments apply to $ i\,\xi_j^-(\zeta) $.
	\end{itemize}
	Up to reducing the constant $ C $, we have therefore proved the existence of a neighborhood $ \mathcal{V} $ in which equality \eqref{eq freq ineg m g 1} is satisfied.
	The result follows from the compactness of $ \Sigma_0 $ and by homogeneity.
	
	The second inequality of Lemma \ref{lemme minoration vitesse de groupe} is obtained immediately using the result of Lemma \ref{lemme pi tilde E vitesse de groupe}.
\end{proof}

\subsection{Proof of Proposition \ref{prop controle exp t A Pi }}

The following proof also comes after the one of Proposition \ref{prop proj bornes}.

We recall that $ \Pi^{e}_-(\zeta) $ is the projector from $E_-(\zeta)$ on the elliptic stable component $E_-^e(\zeta)=\oplus_{j\in\P(\zeta)}E^j_-(\zeta)$ according to decomposition \eqref{eq decomp E_-(zeta)} and that, when $ \zeta $ is not glancing, $ \Pi^{e}_{\C^N}(\zeta) $ is the projection from $ \C^N $ on the stable elliptic component $ E^e_-(\zeta) $ according decomposition \eqref{eq decomp C^N E + E -}. The statement of Proposition \ref{prop controle exp t A Pi } reads as follows.

\begin{proposition}
	Under Assumption \ref{hypothese petits diviseurs 1}, there exists a constant $ c_1>0 $ and a real number $ b_1 $ such that, for all $ \zeta $ in $ \F_b\privede{0} $, the following estimates hold
	\begin{subequations}
		\begin{align}
			&\label{eq controle exp t A Pi - - t positif app} \left|e^{t\mathcal{A}(\zeta)} \, \Pi^{e}_-(\zeta)\right|\leq c_1\,e^{-c_1\,t\,|\zeta|^{-b_1}}\leq c_1, &\forall t\geq 0,\\[5pt]
			&\label{eq controle exp t A Pi C^N - t positif app} \left|e^{t\mathcal{A}(\zeta)}\,\Pi^{e}_{\C^N}(\zeta)\right|\leq c_1\,|\zeta|^{b_1}\,e^{-c_1\,t\,|\zeta|^{-b_1}}, &\forall t\geq 0,\\[5pt]
			&\label{eq controle exp t A Pi C^N - t negatif app} \left|e^{t\mathcal{A}(\zeta)}\big(I-\Pi^{e}_{\C^N}(\zeta)\big)\right| \leq c_1\,|\zeta|^{b_1}, &\forall t\leq 0.
		\end{align}
	\end{subequations}
\end{proposition}

\begin{proof}
	The homogeneity of degree $1$ of the matrix $\mathcal{A}(\zeta)$ and of degree zero of the projectors $\Pi^{e}_-(\zeta)$ and $ \Pi_{\C^N}^e(\zeta) $, and the compactness of the unit ball $\Sigma_0$ are used, and we therefore work in a neighborhood of every point $\underline{\zeta}$ of $\Sigma_0$. The result is then extended to a finite conic covering of $ \Xi_0 $. Since the projector $ \Pi_{\C^N}^{e}(\zeta) $ is defined only for $ \zeta $ non glancing, for inequalities \eqref{eq controle exp t A Pi C^N - t positif app} and \eqref{eq controle exp t A Pi C^N - t negatif app} where it occurs, we are only interested in the points of the neighborhood of $ \underline{\zeta} $ which are not glancing. Thus we consider a point $\underline{\zeta}$ of $\Sigma_0$ and we come back to the notations of the proof of Proposition \ref{prop proj bornes}. 
	
	The interest is first made on the first estimate \eqref{eq controle exp t A Pi - - t positif app}.  In the proof of Proposition \ref{prop proj bornes}, we have constructed, in a neighborhood of $\underline{\zeta}$, a continuous basis of $E_-(\zeta)$
	associated with a regular change-of-basis matrix $T(\zeta)$
	in which the matrix $\mathcal{A}(\zeta)$ restricted to $E_-(\zeta)$ is the following block diagonal matrix of size $ p\times p $
	\[\left(\begin{array}{cccc}
		\mathcal{A}_-(\zeta) &&&0\\
		&i\xi_1(\zeta)I_{\omega_1}&&\\
		&&\ddots&\\
		0&&&i\xi_{r+g}(\zeta)I_{\omega_{r+g}}
	\end{array}\right),\]
	constituted of a block $\mathcal{A}_-(\zeta)$ of negative definite real part, of diagonal blocks $i\xi_j(\zeta)\,I_{\omega_j}$, $j=1,\dots,r$ associated with the incoming eigenvalues $\underline{\xi}_j$ and of diagonal blocks $i\xi_j(\zeta)I_{\omega_j}$, $j=r+1,\dots,r+g$ associated with the eigenvalues $\underline{\xi}_j$ which are glancing in $\underline{\zeta}$. In that case, according to expression  \eqref{eq bornitude def Pi e} of the projector $ \Pi^{e}_-(\zeta) $ in a neighborhood of $\underline{\zeta}$ in $\Sigma_0$, the linear map $ e^{t\mathcal{A}(\zeta)}\,\Pi^{e}_-(\zeta) $ from $ E_-(\zeta) $ to itself is given, in the basis associated with the matrix $ T(\zeta) $, by the following $ p\times p $ block diagonal matrix
	\begin{equation*} 
		\left(\begin{array}{ccrc}
			e^{t\mathcal{A}_-(\zeta)} &&&\\
			&0&&\\
			&&e^{i\,t\,\xi_{r+1}(\zeta) }\,\indicatrice_{\Delta_{r+1}(\zeta)<0}\,I_{\omega_{r+1}}&\\[5pt]
			&&\ddots&\\[5pt]
			&&&e^{i\,t\,\xi_{r+g}(\zeta)}\,\indicatrice_{\Delta_{r+g}(\zeta)<0}\,I_{\omega_{r+g}} 
		\end{array}\right).
	\end{equation*}
	On one hand, the block $\mathcal{A}_-(\zeta)$ is of negative definite real part, uniformly with respect to $\zeta$. On the other hand, one can check that $$ \Im \xi_{r+l}(\zeta)=|\Delta_{r+l}(\zeta)|^{1/2}/2, $$ for $ l=1,\dots,g $, where $ \Delta_{r+l}(\zeta) $ refers to the discriminant of the characteristic polynomial associated with the glancing eigenvalue $ \underline{\xi}_{r+l} $ defined in the proof of Proposition \ref{prop proj bornes} and is depicted in Figure \ref{figure discriminant}. But one can write, in a neighborhood of  $ \underline{\zeta} $ in $ \Sigma_0 $, $$ \Delta_j(\zeta)=\big[\tau-\tau_j^g(\eta)\big]e(\zeta),  $$ with $ e(\underline{\zeta})\neq0 $, where we recall that $ \tau_j^g $ parameterizes  in $ \Sigma_0 $ the surface of cancellation of $ \Delta_j $. It yields to the following estimate on $ \Delta_j $,
	\begin{equation}\label{eq controle preuve Delta dist}
		|\Delta_j(\zeta)|\geq C\,\big|\tau-\tau_j^g(\eta)\big|=C\,\big|\zeta-(\tau_j^g(\eta),\eta)\big|\geq C\,\dist (\zeta,\G).
	\end{equation}
	Since the matrix $ T $ is regular, and therefore uniformly bounded with respect to $ \zeta $ in a neighborhood of $ \underline{\zeta} $ in $ \Sigma_0 $, according to \eqref{eq controle preuve Delta dist}, we get
	\begin{equation}\label{eq controle preuve avant homogene}
		\left|e^{t\mathcal{A}(\zeta)} \, \Pi^{e}_-(\zeta)\right|\leq C e^{-C\dist(\zeta,\mathcal{G})^{1/2}\,t}.
	\end{equation}
	We consider now $ \zeta $ in a conic neighborhood of $ \underline{\zeta} $ in $ \Xi_0 $ with $ \zeta=\lambda\zeta^* $ where $ \lambda=|\zeta|\in \R_+^* $ and $ \zeta^* $ is in a neighborhood of $ \underline{\zeta} $ in $ \Sigma_0 $. Then, by homogeneity and using \eqref{eq controle preuve avant homogene} and Assumption \ref{hypothese petits diviseurs 1}, we obtain
	\begin{multline*}
		\left|e^{t\mathcal{A}(\zeta)} \, \Pi^{e}_-(\zeta)\right|=\left|e^{\lambda t\mathcal{A}(\zeta^*)} \, \Pi^{e}_-(\zeta^*)\right|\leq C e^{-C\dist(\zeta^*,\mathcal{G})^{1/2}\,\lambda t}\\
		=C e^{-C\dist(\lambda\zeta^*,\mathcal{G})^{1/2}\,\lambda^{1/2}t}\leq Ce^{-C|\zeta|^{(a_1+1)/2}\,t}.
	\end{multline*}
	Finally the inequality extends to the whole space $ \Xi_0 $ by compactness of $ \Sigma_0 $, yielding to the required inequality \eqref{eq controle exp t A Pi - - t positif app} for all $ \zeta $ in $ \F_b\privede{0} $.
	
	\bigskip
	
	Concerning estimate \eqref{eq controle exp t A Pi C^N - t positif app}, note that, for $ t\geq 0 $ and for $ \zeta $ non glancing,
	\begin{equation*}
		e^{t\mathcal{A}(\zeta)}\,\Pi^{e}_{\C^N}(\zeta)=e^{t\mathcal{A}(\zeta)}\,\Pi^{e}_-(\zeta)\,\Pi^-_{\C^N}(\zeta),
	\end{equation*}
	where $ \Pi_{\C^N}^-(\zeta) $ is the projector from $ \C^N $ to the stable subspace $ E_-(\zeta) $ according to decomposition \eqref{eq decomp C^N E + E -}, defined for $ \zeta $ non glancing. The aim is therefore to control the projector $ \Pi_{\C^N}^-(\zeta) $, and then use inequality \eqref{eq controle exp t A Pi - - t positif app} to conclude. We still work in a neighborhood of $ \underline{\zeta} $ in $ \Sigma_0 $ and, to simplify the notations, we assume that $ \mathcal{A}(\underline{\zeta}) $ admits a unique glancing eigenvalue $ \underline{\xi}_{g} $ of algebraic multiplicity 2.
	
	Applying the arguments of the proof of Proposition \ref{prop proj bornes} to the unstable part $ E^+(\zeta) $, one obtain an analytic basis
	\begin{equation*}
		E_1(\zeta),\dots,E_N(\zeta)
	\end{equation*}
	of $ \C^N $ associated with a change-of-basis matrix $ \tilde{T}(\zeta) $, analytic in a neighborhood of $ \underline{\zeta} $ in $ \Sigma_0 $, such that in this basis, the linear map $ \mathcal{A}(\zeta) $ writes
	\begin{equation*}
		\tilde{T}(\zeta)^{-1}\,\mathcal{A}(\zeta)\,\tilde{T}(\zeta)
		=\diag\big(\mathcal{A}_-(\zeta),\mathcal{A}_1(\zeta),Q(\zeta),\mathcal{A}_+(\zeta),\mathcal{A}_2(\zeta)\big),
	\end{equation*}
	where $ \mathcal{A}_-(\zeta) $ is of negative definite real part, $ \mathcal{A}_1(\zeta) $ is the diagonal block associated with the incoming eigenvalues, $ \mathcal{A}_+(\zeta) $ is of positive definite real part, $ \mathcal{A}_2(\zeta) $  is the diagonal block associated with the outgoing eigenvalues, and the unique $ 2\times2 $ block $ Q(\zeta) $, associated with the glancing eigenvalue $ \underline{\xi}_g $, writes
	\begin{equation*}
		Q(\zeta)=i\left(\begin{array}{cc}
			\underline{\xi}_g+q_1(\zeta) & 1 \\[5pt]
			q_2(\zeta) & \underline{\xi}_g
		\end{array}\right),
	\end{equation*} 
	where $ q_1(\underline{\zeta})=q_2(\underline{\zeta})=0 $ and $ \frac{\partial q_2}{\partial \tau}(\underline{\zeta})\neq 0 $.
	
	We want now to construct, using the basis $ E_1,\dots,E_N $, a new basis $ F_1,\dots,F_N $ adapted to the decomposition
	\begin{equation*}
		\C^N=E_-(\zeta)\oplus E_+(\zeta)
	\end{equation*} 
	for $ \zeta $ non glancing. If $ E_1(\zeta),\dots,E_{p-1}(\zeta) $ are the $ p-1 $ first vectors corresponding to the blocks $ \mathcal{A}_-(\zeta) $ and $ \mathcal{A}_1(\zeta) $ of the basis of $ \C^N $ associated with $ T(\zeta) $, we set, for $ j=1,\dots,p-1 $, $ F_j(\zeta):=E_j(\zeta) $. Note that $ \big(F_1(\zeta),\dots,F_{p-1}(\zeta)\big) $ is therefore a set of linearly independent vectors of $ E_-(\zeta) $. We set as well $ F_j(\zeta):=E_j(\zeta) $ for $ j=p+2,\dots,N $, where $ E_{p+2}(\zeta),\dots,E_N(\zeta) $ are the vectors of the basis of  $ \C^N $ defined by $ T(\zeta) $ associated with the blocks $ \mathcal{A}_+(\zeta) $ and $ \mathcal{A}_2(\zeta) $, constituting a set of linearly independent vectors of $ E_+(\zeta) $. 
	
	The two vectors $ F_p(\zeta) $ and $ F_{p+1}(\zeta) $ are now to be determined, which are the stable and unstable eigenvectors of $ \mathcal{A}(\zeta) $ associated with the block $ Q(\zeta) $. If $ \xi_-(\zeta) ,\xi_+(\zeta)$ are the two stable and unstable eigenvalues (equal for $ \zeta $ glancing) associated with the glancing eigenvalue $ \underline{\xi}_g $, then the stable and unstable eigenvectors of $ \mathcal{A}(\zeta) $ associated with $ \xi_-(\zeta) $ and $ \xi_+(\zeta) $ are given by
	\begin{equation*}
		F_{p}(\zeta):=E_p(\zeta)+\frac{q_2(\zeta)}{\xi_-(\zeta)-\underline{\xi}_g}\,E_{p+1}(\zeta),\qquad 
		F_{p+1}(\zeta):= E_p(\zeta)+\frac{q_2(\zeta)}{\xi_+(\zeta)-\underline{\xi}_g}\,E_{p+1}(\zeta).
	\end{equation*}
	Indeed, an eigenvector of the matrix $ \tilde{T}(\zeta)^{-1}\,\mathcal{A}(\zeta)\,\tilde{T}(\zeta) $ associated with the eigenvalue $ \xi_{\pm}(\zeta) $ is given by $ \,^t(0,\dots,0,1,\frac{q_2(\zeta)}{\xi_{\pm}(\zeta)-\underline{\xi}_g},0,\dots,0) $.
	Note that when $ \zeta $ is glancing, namely when $ \xi_-(\zeta)=\xi_+(\zeta) $ and $ q_2(\zeta)=0 $, we have $ F_p(\zeta)=F_{p+1}(\zeta)=E_p(\zeta) $. 
	
	The change-of-basis matrix from the canonical basis of $ \C^N $ to the basis $ F_1,\dots,F_N $ is therefore given by the product of the matrix $ T(\zeta)  $ and the block diagonal matrix
	\begin{equation*}
		P(\zeta):=\diag\left(I_{p-1},\left(\begin{array}{cc}
			1 & 1 \\
			\dfrac{q_2(\zeta)}{\xi_-(\zeta)-\underline{\xi}_g} & \dfrac{q_2(\zeta)}{\xi_+(\zeta)-\underline{\xi}_g} 
		\end{array}\right),I_{N-p-1}\right).
	\end{equation*}
	Thus the projector $ \Pi_{\C^N}^- $ writes
	\begin{equation*}
		\Pi_{\C^N}^- =T(\zeta)\,P(\zeta)\left(\begin{array}{cc}
			I_p & 0 \\
			0 & 0 
		\end{array}\right) P(\zeta)^{-1}\,T(\zeta)^{-1}.
	\end{equation*}
	
	The matrix $ T(\zeta) $ is analytic and therefore bounded as well as its inverse in a neighborhood of $ \underline{\zeta} $. Since it has already been proven that $ \frac{q_2(\zeta)}{\xi_-(\zeta)-\underline{\xi}_g}  $ and $ \frac{q_2(\zeta)}{\xi_+(\zeta)-\underline{\xi}_g}  $ are bounded in a neighborhood of $ \underline{\zeta} $, the matrix $ P(\zeta) $ and $ ^t\!\com P(\zeta)$ are bounded. The determinant $ \det P(\zeta)  $ shall now be estimated.
	Since $ \xi_-(\zeta) $ and $ \xi_+(\zeta) $ are the two (possibly equal) roots of the polynomial $ \xi^2-\big(2\underline{\xi}_g+q_1(\zeta)\big)\,\xi+\underline{\xi}_g^2+ \underline{\xi}_gq_1(\zeta)- q_2(\zeta) $, we obtain
	\begin{equation*}
		\big(\xi_-(\zeta)-\underline{\xi}_g\big)\big(\xi_+(\zeta)-\underline{\xi}_g\big)=-q_2(\zeta),
	\end{equation*}
	so that
	\begin{equation*}
		\det P(\zeta) =\frac{q_2(\zeta)\big(\xi_-(\zeta)-\xi_+(\zeta)\big)}{\big(\xi_-(\zeta)-\underline{\xi}_g\big)\big(\xi_+(\zeta)-\underline{\xi}_g\big)}=-\big(\xi_-(\zeta)-\xi_+(\zeta)\big).
	\end{equation*}
	It yields to
	\begin{equation*}
		|\det P(\zeta)|=|\Delta_g(\zeta)|^{1/2}\geq C \dist(\zeta,\G)^{1/2},
	\end{equation*}
	according to estimate \eqref{eq controle preuve Delta dist}. The control \eqref{eq controle exp t A Pi C^N - t positif app} follows in the same way as the one of \eqref{eq controle exp t A Pi - - t positif app} using estimate \eqref{eq controle exp t A Pi - - t positif app} and Assumption \ref{hypothese petits diviseurs 1}.
	
	Finally, for estimate \eqref{eq controle exp t A Pi C^N - t negatif app}, taking back the notations and the results of the previous point, the matrix $ e^{t\mathcal{A}(\zeta)}\big(I-\Pi^{e}_{\C^N}(\zeta)\big) $ can be written as
	\begin{multline*}
		e^{t\mathcal{A}(\zeta)}\big(I-\Pi^{e}_{\C^N}(\zeta)\big)=\\
		\tilde{T}(\zeta)^{-1}\,P(\zeta)^{-1}\left(\begin{array}{ccccc}
			0 &&&&\\
			& e^{t\mathcal{A}_1(\zeta)} &&0& \\
			&& e^{tQ(\zeta)}\indicatrice_{\Delta_g(\zeta)>0} && \\
			&0&& e^{t\mathcal{A}_+(\zeta)} &\\
			&&&& e^{t\mathcal{A}_2(\zeta)}
		\end{array}\right)P(\zeta)\,\tilde{T}(\zeta).
	\end{multline*}
	Note that the eigenvalues of the matrices $ \mathcal{A}_1(\zeta) $, $ \mathcal{A}_2(\zeta) $ and of the matrix $ Q(\zeta) $ when $ \Delta_g(\zeta)>0 $, are imaginary, and that the matrix $ \mathcal{A}_+(\zeta) $ is of positive definite real part uniformly with respect to $ \zeta $. Thus, using the estimate on the change-of-basis matrix $ P(\zeta)^{-1} $ proved above, one may conclude as for estimate \eqref{eq controle exp t A Pi C^N - t negatif app}.
\end{proof}

\emph{Acknowledgments.} Warm thanks go to Jean-François Coulombel, for his many ideas on the subject, his proofreading, advice and corrections. His help has been crucial for this paper.

%\nocite{*}
\bibliography{Bibliographie.bib}
\bibliographystyle{alpha}

%
%\makeatletter
%\providecommand\@dotsep{5}
%\makeatother
%\listoftodos\relax

\end{document}